\documentclass[reqno]{amsart}
\pdfoutput=1

\usepackage[all,cmtip,line]{xy}
\usepackage[margin=1in]{geometry}
\usepackage{tikz}

\usepackage{amsmath, amstext, amssymb, amsthm, amscd}
\usepackage{microtype}

\usepackage[colorlinks,bookmarks=false]{hyperref}
\def\arXiv#1{arXiv:\href{http://arXiv.org/abs/#1}{#1}}
\usepackage{url}

\hypersetup{citecolor=blue,linkcolor=blue,urlcolor=blue}

\newtheorem{theorem}{Theorem}
\newtheorem{corollary}[theorem]{Corollary}
\newtheorem{lemma}[theorem]{Lemma}
\newtheorem{proposition}[theorem]{Proposition}

\theoremstyle{definition}

\newtheorem{remark}[theorem]{Remark}

\newcommand{\R}{{\mathbb R}}
\newcommand{\C}{{\mathbb C}}
\newcommand{\Z}{{\mathbb Z}}
\newcommand{\Q}{{\mathbb Q}}
\newcommand{\F}{{\mathbb F}}
\newcommand{\Ha}{{\mathcal H}}
\newcommand{\A}{{\mathcal A}}
\newcommand{\Proj}{{\mathbb P}}
\newcommand{\sO}{{\mathcal O}}

\newcommand{\Km}{{\mathrm{Km}}}

\newcommand{\NS}{\mathrm{NS}}
\newcommand{\MW}{\mathrm{MW}}

\newcommand{\End}{\mathrm{End}}
\newcommand{\M}{\mathcal M}
\newcommand{\sF}{{\mathcal F}}
\newcommand{\tr}{\mathrm{tr}}
\newcommand{\sH}{\mathcal{H}}
\newcommand{\sS}{\mathcal{S}}
\newcommand{\hor}{\textrm{hor}}
\newcommand{\ver}{\textrm{ver}}
\newcommand{\I}{\mathrm{I}}
\newcommand{\II}{\mathrm{II}}
\newcommand{\IV}{\mathrm{IV}}
\newcommand{\III}{\mathrm{III}}

\newcommand{\MoW}{Mordell-\mbox{\kern-.12em}Weil}
\newcommand{\PSL}{\mathop{\rm PSL}\nolimits}
\newcommand{\kbar}{{\kern.06ex\overline{\kern-.06ex k}}}
\newcommand{\Qbar}{{\kern.1ex\overline{\kern-.1ex\Q\kern-.1ex}\kern.1ex}}
\newcommand{\disc}{\mathop{\rm disc}\nolimits}
\newcommand{\Br}{\mathop{\rm Br}\nolimits}
\newcommand{\SL}{\mathop{\rm SL}\nolimits}
\newcommand{\Sp}{\mathop{\rm Sp}\nolimits}

\newcommand{\0}{^{\phantom0}}

\DeclareMathOperator{\Tr}{Tr}
\DeclareMathOperator{\Imag}{Im}


\title{K3 surfaces and equations for Hilbert modular surfaces}

\author{Noam Elkies}
\address{Department of Mathematics\\
Harvard University\\
Cambridge, MA 02138}
\email{elkies@math.harvard.edu}

\author{Abhinav Kumar}
\address{Department of Mathematics\\
Massachusetts Institute of Technology\\
Cambridge, MA 02139}
\email{abhinav@math.mit.edu}

\thanks{Elkies was supported in part by NSF grants DMS-0501029 and DMS-1100511.
  Kumar was supported in part by NSF grants DMS-0757765 and
  DMS-0952486, and by a grant from the Solomon Buchsbaum Research
  Fund. The research was started when Kumar was a postdoctoral
  researcher at Microsoft Research. He also thanks Princeton
  University for its hospitality during Fall 2009.}

\date{January 26, 2015}

\begin{document}

\begin{abstract}
We outline a method to compute rational models for the Hilbert modular
surfaces $Y_{-}(D)$, which are coarse moduli spaces for principally
polarized abelian surfaces with real multiplication by the ring of
integers in $\Q(\sqrt{D})$, via moduli spaces of elliptic K3 surfaces
with a Shioda-Inose structure. In particular, we compute equations for
all thirty fundamental discriminants $D$\/ with $1 < D < 100$,
and analyze rational points and curves on these Hilbert modular surfaces,
producing examples of genus-$2$ curves over~$\Q$ whose Jacobians have
real multiplication over~$\Q$.
\end{abstract}

\maketitle

\section{Introduction}

Hilbert modular surfaces have been objects of extensive investigation
in complex and algebraic geometry, and more recently in number
theory. Since Hilbert modular varieties are moduli spaces for abelian
varieties with real multiplication by an order in a totally real
field, they have intrinsic arithmetic content. Their geometry is
enriched by the presence of modular subvarieties.

In the 1970s, Hirzebruch, van de Ven, and Zagier \cite{Hi, HVdV, HZ}
computed the geometric invariants of many of these surfaces, and
placed them within the Enriques-Kodaira classification.  A chief aim
of the present work is to compute equations for birational models of some
of these surfaces over the field of rational numbers.

More precisely, let $D$\/ be a positive fundamental discriminant,
i.e.\ the discriminant of the ring of integers $\sO_D$ of the real
quadratic field $\Q(\sqrt{D})$. The quotient $\PSL_2(O_D)\backslash
(\mathcal{H}^+ \times \mathcal{H}^-)$ (where $\mathcal{H}^+$ and
$\mathcal{H}^{-}$ are the complex upper and lower half planes)
parametrizes abelian surfaces with an action of $\sO_D$.
It has a natural compactification $Y_{-}(D)$, obtained by
adding finitely many points and desingularizing these cusps.

These surfaces $Y_{-}(D)$ have models defined over $\Q$, and the main
goal of this paper is to describe a method to compute explicit
equations for these models, as well as to carry out this method for
all fundamental discriminants $D$\/ with $1 < D < 100$.
 This felt like a good place to stop for now,
though  these calculations may be extended to some higher~$D$,
as well as to non-fundamental discriminants.

We briefly summarize the method, which we describe in more
detail in later sections. The method relies on being able to
explicitly parametrize K3 surfaces that are related by
Shioda-Inose structure to abelian surfaces with real multiplication by
some $\sO_D$.  The K3 surface corresponding to such an abelian surface has
N\'{e}ron-Severi lattice containing $L_D$, a specific indefinite lattice of
signature $(1,17)$ and discriminant $-D$.  In all our examples, we obtain
the moduli space $\mathcal{M}_D$ of $L_D$-polarized K3 surfaces
as a family of elliptic surfaces with a specific configuration of
reducible fibers and sections.

We then use the $2$- and $3$-neighbor method
to transform to another elliptic fibration, with two reducible fibers
of types $\mathrm{II}^*$ and $\mathrm{III}^*$ respectively.
This lets us read off the map (generically one-to-one) of moduli
spaces from $\mathcal{M}_D$ into the $3$-dimensional moduli space
$\mathcal{A}_2$ of principally polarized abelian surfaces, using the
formulae from \cite{Kum1}. The image of $\mathcal{M}_D$ is the
Humbert surface corresponding to discriminant~$D$. The Hilbert modular
surface $Y_{-}(D)$ itself is a double cover of the Humbert surface,
branched along a union of modular curves.  We use simple lattice
arguments to obtain the branch locus, and pin down the exact twist for
the double cover by counting points on reductions of the related
abelian surfaces modulo several primes. In all our examples, the
Humbert surface happens to be a rational surface (i.e.\ birational to
$\Proj^2$ over~$\Q$), and we display the equation of $Y_{-}(D)$ as a
double cover of $\Proj^2$ branched over a curve of small degree. We
analyze these equations in some detail, attempting to produce rational
or elliptic curves on them, with the intent of producing several
(possibly infinitely many) examples of genus $2$ curves whose
Jacobians have real multiplication. When $Y_{-}(D)$ is a K3 surface, it
often has very high Picard number ($19$ or $20$), and we attempt to
compute generators for the Picard group. When $Y_{-}(D)$ is an honestly
elliptic surface, we analyze the singular fibers and the \MoW\ group,
and attempt to compute a basis for the sections.

To our knowledge, this is the first algebraic description of most of
these surfaces by explicit equations. We outline some related work in
the literature. Wilson \cite{Wi1, Wi2} obtained equations for the
Hilbert modular surface $Y_{-}(5)$ corresponding to the smallest
fundamental discriminant $D>1$.  In~\cite{vdG}, van der Geer gives a
few examples of algebraic equations for Hilbert modular surfaces
corresponding to a congruence subgroup of the full modular group (in
other words, abelian surfaces with some level structure).  Humbert
surfaces have also been well-studied in the literature, and \cite{Ru,
  Gr} have obtained equations for some of these.  However, these
equations are quite complicated, and do not shed as much light on the
geometry of Hilbert modular surfaces. While the methods are simpler,
involving theta functions and $q$-expansions, the result is analogous
to exhibiting the modular polynomial whose zero locus in $\mathbb{A}^1
\times \mathbb{A}^1$ is a singular model of the complement of the
cusps in the modular curve $X_0(N)$.  The coefficients of these
polynomials can quickly become enormous.  We believe that our
approach, giving simpler equations for these surfaces together with
their maps to $\A_2$, is more conducive to an investigation of
arithmetic properties.

It is our hope that these equations will be of much help in
subsequent arithmetic investigation of these surfaces. For instance,
they should provide a testing ground for many conjectures in
Diophantine geometry, because of the abundance of rational curves and
points. Another direction of future investigation is to use these
equations to investigate modularity of the corresponding abelian surfaces.
Modularity of abelian varieties with real multiplication over~$\Q$
is now proven, by combining results of Ribet~\cite{Ri} with the recent
proof of Serre's conjecture by Khare and Wintenberger \cite{KW}.
However, unlike the case of dimension~$1$, where one has modular
parametrizations and very good control of the moduli spaces,
the situation in dimensions $2$ and above is much less clear.
For instance, it is not at all clear how to find a modular form
corresponding to a given abelian surface with real multiplication.\footnote{
  Suppose $A/\Q$ is an abelian surface with \hbox{$\Q$-endomorphisms} by $\sO_D$,
  and let $\varphi = \sum_n a_n q^n$ be an eigenform with every $a_n \in O_D$.
  If $\varphi$ corresponds to~$A$ then counting points over
  $\F_p$ and $\F_{p^2}$ determines each $a_p$ up to Galois conjugation.
  But conceivably there might be some eigenform $\varphi' = \sum_n a'_n q^n$,
  different from both $\varphi$ and its Galois conjugate,
  such that each $a'_p$ equals either $a_p$ or the Galois conjugate of $a_p$;
  if that happens, we do not know how to decide which eigenform corresponds
  to~$A$.  Likewise for abelian varieties of dimension $3$ and higher.
  }
We hope that the abundance of examples provided by
these equations will help pave the path for a better understanding of
the $2$-dimensional case.  For example,
in \cite{DK}, our formulas are combined with efficient computation of
Hilbert modular forms to find examples of simple abelian surfaces
over real quadratic fields, with everywhere good reduction.
An example of such an abelian surface is the Jacobian of the genus $2$ curve
\begin{eqnarray*}
2y^2 & \!\! = \!\! &
   x^6 - \tau x^5 + 74 x^4 - 14 \tau x^3 + 267 x^2 - 13 \tau x + 46
\cr
   & \!\! = \!\! &
     \bigl(x^3 - \frac{1+\tau}{2} x^2 + 13 x + \frac{3-\tau}{2}\bigr)
     \bigl(x^3 + \frac{1-\tau}{2} x^2 + 13 x - \frac{3+\tau}{2}\bigr)
\end{eqnarray*}
where $\tau = \sqrt{193}$ (the curve and its Jacobian can in fact be
defined over~$\Q$, but the Jacobian attains everywhere good reduction
and real multiplication by $\sO_{17}$ only over $\Q(\tau)$).

The outline for the rest of the paper is as follows. In section
\ref{K3}, we describe the relevant background on K3 surfaces and their
moduli spaces, and their connection to moduli spaces of abelian
surfaces via Shioda-Inose structures. In section \ref{realmult}, we
describe the Hilbert modular surfaces and the corresponding moduli
spaces of K3 surfaces. In section \ref{method}, we precisely describe
our methods to compute their equations. Section \ref{neighbors}
describes the $2$- and $3$-neighbor method in more detail. The rest of
the paper consists of detailed examples of Hilbert modular
surfaces for the discriminants less than $100$, as well as an
arithmetic investigation of these surfaces. The accompanying auxiliary
files contain formulas for the Igusa-Clebsch invariants, as well as a
description of the parametrizations exhibited for the moduli spaces of
K3 surfaces in the paper, and the details of the neighbor steps to
transform to a fibration with $\II^*$ and $\III^*$ fibers.

The auxiliary computer files containing the equations for these
Hilbert modular surfaces, as well as formulas for the Igusa-Clebsch
invariants, are available from
\url{http://arxiv.org/abs/1209.3527}. To access these, download the
source file for the paper. This will produce not only the \LaTeX{}
file for this paper, but also the computer code. The file
\texttt{README.txt} gives an overview of the various computer files.

\section{Moduli spaces of abelian surfaces and lattice polarized K3 surfaces}
\label{K3}

Throughout this section, we work with K3 surfaces over a field $k$ of
characteristic $0$. When convenient, we will suppose $k \subseteq \C$,
and use transcendental methods.

\subsection{K3 surfaces}

A K3 surface $X$\/ over $k$ is a projective algebraic nonsingular
surface with $h^1(X,\sO_X) = 0$ and $K_X \cong \sO_X$. For such a
surface, $H^2(X,\Z)$ is torsion-free, and when endowed with the
cup-product form becomes a $22$-dimensional lattice, abstractly
isomorphic with $\Lambda := E_8(-1)^2 \oplus U^3$.  Here $E_8$ is the
even unimodular lattice in eight dimensions, $U$ is the hyperbolic
plane with Gram matrix $\left(\begin{array}{cc} 0 & 1 \\ 1 &
  0 \end{array}\right)$, and for any lattice $\Lambda$ and real number
$\alpha$, the lattice $\Lambda(\alpha)$ consists of the same
underlying abelian group with the form multiplied by $\alpha$. The
N\'{e}ron-Severi group $\NS(X)$ of algebraic divisors defined
over~$\kbar$ modulo algebraic equivalence, which for a K3 surface is
the same as linear or numerical equivalence, is a primitive sublattice
of~$\Lambda$ of signature $(1,\rho -1)$, where $\rho \in \{1,\dots,
20\}$ is the Picard number of $X$. The orthogonal complement of
$\NS(X)$ is the transcendental lattice $T_X$. There is a Torelli
theorem for K3 surfaces, due to Piatetski-Shapiro and Shafarevich
\cite{PS} and Friedman~\cite{Fr}, which describes the moduli space of
K3 surfaces with a fixed polarization. More generally, let $L$ be an
even non-degenerate lattice of signature $(1,r-1)$, with $r \in
\{1,\dots, 20\}$. Assume that $L$ has a unique primitive embedding in
$\Lambda$, up to isometries of $\Lambda$. Then there is a coarse
moduli space $\sF_L$ of $L$-polarized K3 surfaces $(X, j)$, where $j:
L \rightarrow \NS(X)$ is a primitive lattice embedding such that
$j(L)$ contains a pseudo-ample class on $X$.  The space $\sF_L$ is
isomorphic to the quotient of an appropriate fundamental domain
$$
\Omega_L = \Proj\big(\{\omega \in L^{\perp} \otimes \C \, | \, \langle
\omega , \omega \rangle = 0, \langle \omega , \bar{\omega} \rangle > 0
\}\big)
$$ by an arithmetic subgroup $\Gamma_L$, which is the image of
\[
\Gamma(L) = \{\sigma \in O(\Lambda) \, | \,
 \forall \, x \in L, \, \sigma(x) = x \, \}
\]
in $O(\Lambda^{\perp}$). (Here, we have fixed an embedding
$i: L \rightarrow \Lambda$, so we may use its orthogonal complement
$L^{\perp}$.) Therefore $\sF_L$ is a quasiprojective variety.

In fact, there is a fine moduli space $\mathcal{K}_L$ of marked
pseudo-ample $L$-polarized K3 surfaces, i.e.\ $(X, \phi)$ where
$\phi:H^2(X, \Z) \rightarrow \Lambda$ is an isomorphism (a {\em marking})
such that $\phi^{-1}(L) \subseteq \NS(X)$. There is a
period map which associates to such a marked K3 surface the
class of the global algebraic $2$-form up to scaling, giving a point
$[\omega] \in \Proj(L^{\perp} \otimes \C)$. Furthermore, $\omega \cup
\omega = 0$ and $\omega \cup {\overline{\omega}} > 0$. This domain
$\Omega_L$ consists of two copies of a bounded Hermitian domain of
type $\mathrm{IV}_{20-r}$. The period map $(X, \phi) \rightarrow
[\omega_X]$ sets up an isomorphism between the moduli space $\mathcal{K}_L$
and the period domain $\Omega_L$, using the Torelli theorem and the
surjectivity of the period map \cite{Kul, PP}. The quotient $\Gamma_L
\backslash \mathcal{K}_L \cong \Gamma_L \backslash \Omega_L$ forgets the
marking, and describes a coarse moduli space of $L$-polarized K3
surfaces. For details, the reader may consult Nikulin~\cite{Ni1}
and Dolgachev~\cite{Do}.

\subsection{Elliptic K3 surfaces}

We shall be especially interested in moduli spaces of elliptic K3 surfaces.
In this paper, an elliptic K3 surface will be a K3 surface $X$
with a relatively minimal genus-$1$ fibration $\pi : X \rightarrow \Proj^1$,
together with a section. In other words, we may write a
Weierstrass equation of $X$\/ over $\Proj^1_t$ as
\begin{equation}
y^2 + a_1(t) x y + a_3(t)y = x^3 + a_2(t) x^2 + a_4(t)x + a_6(t)
\label{eq:Weier}
\end{equation}
with $a_i(t)$ a polynomial in $t$ of degree at most $2i$.
(More canonically, each $a_i$ is a homogeneous polynomial of degree~$2i$
in the two homogeneous coordinates of $\Proj^1_t$.)
Of course, this Weierstrass equation describes the
generic fiber of $X$; to understand the reducible special fibers, one
can use Tate's algorithm \cite{Ta1} to blow up the singular points
and describe the minimal proper model.  (This will also detect when
a Weierstrass equation~(\ref{eq:Weier}) is equivalent to one with
each $a_i$ vanishing to order at least $i$ at some $t_0$, that is,
when the equation gives not a K3 elliptic surface but a rational or
constant one.)  The singular fibers are classified by Kodaira and N\'{e}ron,
and the non-identity components of any reducible fiber $\pi^{-1}(v)$
contribute an irreducible root lattice (scaled by $-1$), say $L_v$,
to the N\'{e}ron-Severi lattice.  The {\em trivial lattice} is defined to be
$$
T = \Z O \oplus \Z F \oplus \big( \bigoplus_v L_v \big)
$$
(note that $O$ and $F$\/ span a copy of the hyperbolic plane~$U$\/).

A theorem of Shioda and Tate \cite{Sh1, Sh2, Ta2} shows that as long
as the elliptic fibration is non-trivial (equivalently, it has at
least one singular fiber), the \MoW\ group of $X$\/ over
$\Proj^1$ is isomorphic to $\NS(X)/T$. In particular, we have the
Shioda-Tate formula
$$
\rho(X) = 2 + \textrm{rank } \MW(X/\Proj^1) + \sum_v \textrm{rank } L_v.
$$
One may also compute the discriminant of the N\'{e}ron-Severi lattice:
$$
|\mathrm{disc}(\NS(X))| = \frac{\mathrm{det}(H(X/\Proj^1)) \cdot
  \prod_v \mathrm{disc}(L_v) }{|\MW(X/\Proj^1)_{\mathrm{tors}}|^2}
$$
where $H(X/\Proj^1)$ is the height pairing matrix for a basis of
the torsion-free part of the \MoW\ group of $X$\/ over $\Proj^1$.

If $F$\/ is the class of the fiber for an elliptic K3 surface $X$, then
$F$\/ is primitive and nef, with $F^2 = 0$. Conversely, suppose that $F \in
\NS(X)$ is a nonzero divisor class which is primitive and nef with
$F^2 = 0$. Then a simple application of the Riemann-Roch
theorem shows that $F$\/ or $-F$\/ must be effective, and since $F$\/ is
nef, it must be represented by an effective divisor. Then a theorem of
Piatetski-Shapiro and Shafarevich \cite[page 559]{PS} shows that $F$\/
defines a genus $1$ fibration.

\begin{lemma}
\label{lem:section_criterion}
Let $F = \sum a_i E_i$ be a positive linear combination of smooth
rational curves on a K3 surface $X$\/ such that $F \cdot E_i = 0$ for
all $i$, and such that $F$\/ is a primitive class in $\NS(X)$. Then $F$
defines a genus $1$ fibration.
\end{lemma}

\begin{proof}
We have $F^2 = \sum a_i (F \cdot E_i) = 0$. By the above discussion,
it is enough to show that $F$\/ is nef. Let $E'$ be an irreducible curve
on $X$. If $E'$ is distinct from the $E_i$, then $E_i \cdot E' \geq 0$
for every $i$, and so $F \cdot E' \geq 0$. On the other hand, if $E' =
E_i$ say, then $F \cdot E' = F \cdot E_i = 0$. Therefore $F$\/ is nef.
\end{proof}

Let us define an {\em elliptic divisor} to be a divisor satisfying the
conditions of Lemma \ref{lem:section_criterion}. We will frequently use
this lemma, displaying an elliptic divisor by finding a subdiagram of
the set of roots of $\NS(X)$, represented by smooth rational curves,
which is an extended Dynkin diagram for a root lattice. Then the class
of the appropriate linear combination of roots $F$\/ will define a
genus $1$ fibration. We need to know when such a fibration has a section.

\begin{lemma}
  Suppose $F$\/ is an elliptic divisor, defining a genus $1$ fibration
  $\pi: X \rightarrow \Proj^1$. Suppose $D \in \NS(X)$ satisfies $D
  \cdot F = 1$. Then $\pi$ has a section.
\end{lemma}

\begin{proof}
Consider the divisor $D' = D + mF$, for some large integer $m$. Then
$(D')^2 = D^2 + 2m$, while $K \equiv 0$, so the Riemann-Roch theorem implies
$$
h^0(D') - h^1(D') + h^2(D') = (D')^2/2 + \chi(\sO_X) = D^2 + m + 2.
$$ Also, $h^2(D') = h^0(K-D') = h^0(-D') = 0$ by Serre duality, and
since $(-D')\cdot H = -D \cdot H - m F \cdot H < 0$ for any ample
divisor $H$, as long as $m$ is large enough. Therefore we see that for
large $m$, the divisor class $D'$ can be represented by an effective
divisor, which we may call $D'$, by abuse of notation. Note that we
still have $D' \cdot F = 1$. Decompose $D'$ as
$D'_{\mathrm{vert}} + D'_{\mathrm{hor}}$,
where the first term contains all the components which
lie along fibers of the genus $1$ fibration defined by $F$,
and the second contains the other components.
 Then $D'_{\mathrm{hor}} \cdot F = 1$.
Therefore, $D'_{\mathrm{hor}}$ must be reduced and irreducible,
and thus defines a section of the genus $1$ fibration.
\end{proof}

\begin{corollary}
  Let $F$\/ be an elliptic divisor, and let $D_1$ and $D_2$ be two
  divisor classes such that $D_1 \cdot F$\/ and $D_2 \cdot F$\/ are
  coprime. Then the fibration has a section.
\end{corollary}

\begin{proof}
There exist integers $a_1, a_2$ such that $(a_1 D_1 + a_2 D_2) \cdot F = 1$.
Now take $D = a_1 D_1 + a_2 D_2$ in Lemma~\ref{lem:section_criterion}.
\end{proof}

Finally, we note a lattice-theoretic result which allows us to deduce
that in all of the cases studied in this paper, the genus $1$
fibration defined by an elliptic divisor $F$ has a section.

\begin{proposition}\label{automaticsections}
  Let $D$ be a fundamental discriminant, and let $L = U \oplus N(-1)$,
  where $U$ is the hyperbolic plane, and $N$ a positive definite
  lattice of rank $16$ and discriminant $D$. Suppose in addition that
  $N$ contains a sublattice isomorphic to $E_8 \oplus E_7$. If $v \in
  L$ is a primitive vector with $v \cdot v = 0$, then there exists $w
  \in L$ such that $v \cdot w = 1$.
\end{proposition}

\begin{proof}
  Suppose not. Then $\{v \cdot w : w \in L\} = c\Z$ for some integer
  $c > 1$. Since $v$ is primitive, we can take a basis $v_1 = v,
  \dots, v_{16}$ of $L$. Then $L' = \Z (v/c) + \Z v_2 + \dots + \Z
  v_{16}$ is an integral lattice containing $L$ with index $c >
  1$. Since $L' \supset L \supset U$, we have $L' = U \oplus N'(-1)$
  (since $U$ is unimodular), with $N'$ a positive definite lattice
  containing $N$ with index $c$. Then $N'$ must be generated by $E_8
  \oplus E_7$ and a vector $x$ whose projection $x^{\perp}$ to the
  orthogonal complement of $E_8 \oplus E_7$ has norm $D/(2c^2)$. The
  dual lattice of $E_8 \oplus E_7$ has norms congruent to $0$ or $3/2$
  modulo $2$, so $x^{\perp}$ has norm $0$ or $1/2$ mod $2$. Therefore
  $D/c^2$ must be an integer congruent to $0$ or $1$ modulo $4$. Since
  $D$ is a fundamental discriminant, this is impossible.
\end{proof}

For the examples in this paper, we will often draw a Dynkin-type
diagram indicating some of the roots of $\NS(X)$ for an elliptic K3
surface $X$, which will always be nodal classes (i.e.\ classes of
smooth rational curves on $X$). We will outline an elliptic divisor
$F$\/ by drawing a subdiagram in bold which cuts out an extended
Dynkin diagram of an irreducible root lattice (the multiplicities will
be omitted). Where convenient, we will also indicate the class of a
divisor $D$\/ such that $D \cdot F = 1$ (in some cases, there is an
obvious node in the Dynkin diagram satisfying this property), or the
classes of two divisors $D_1$ and $D_2$ such that $D_1 \cdot F$\/ and
$D_2 \cdot F$\/ are coprime. This is not strictly necessary, because
Proposition \ref{automaticsections} guarantees the existence of such a
divisor $D$, but having an explicit divisor might be useful for
further calculations. Then $F$\/ defines another elliptic fibration
with section on $X$, and we may proceed as in Section \ref{neighbors}
to write down its Weierstrass equation.

\subsection{Kummer surfaces, Nikulin involutions and Shioda-Inose structures}

Let $A$ be an abelian surface, and consider the involution $\iota$ on~$A$
defined by multiplication by $-1$ in the group law. The
quotient of $A$ by the group $\{1, \iota\}$ is a surface $Y'$ with
sixteen singularities, the images of the $2$-torsion points of $A$.
In fact, $Y'$ may be realized as a quartic surface in $\Proj^3$ with
sixteen ordinary double points (which is the maximum number of
singularities possible for a quartic surface in $\Proj^3$ \cite[p.15]{Hu}),
by considering the linear system on $A$ corresponding to twice the theta
divisor. The corresponding map is $2$-to-$1$ from $A$ to $Y'$.

Taking the minimal desingularization of $Y'$ gives a K3 surface $Y$,
the Kummer surface of $A$, which contains sixteen disjoint nodal
classes coming from the blowups of the singular points. Note that if
$A$ is defined over some number field $k$, then so is $Y =
\Km(A)$. The N\'{e}ron-Severi lattice of the surface $Y$ contains the
saturation of the lattice spanned by the sixteen special nodal
classes; this is a lattice $\Lambda_{\Km}$ of signature $(0,16)$ and
discriminant $2^6$. Conversely, Nikulin showed that a K3 surface whose
N\'{e}ron-Severi lattice contains $\Lambda_{\Km}$ must be the Kummer
surface of some complex torus. Of course, since $A$ is an abelian
variety, $Y$ is a projective surface, so $\NS(Y)$ contains an ample
divisor as well.

We will be especially concerned with the case when $A = J(C)$ is the
Jacobian of a curve of genus $2$. Let $x_0, \dots, x_5$ be the
Weierstrass points of $C$. The embedding $\eta_0: C \rightarrow A$
given by $x \mapsto [x] - [x_0]$ gives a particular theta divisor
on~$A$, and the translates $\eta_0(C) + [x_i] - [x_j]$ with $0 \leq i
< j \leq 5$ give fifteen more special divisors. The images of these
sixteen divisors (tropes) on the Kummer surface of $A$ are disjoint
rational curves, and each intersects six rational curves coming from
the blowups of the singular points (i.e.\ the nodes). This classical
configuration of tropes and nodes on the Kummer surface is called the
$(16,6)$ configuration, and the intersection pairing describes a
vertex- and edge-transitive bipartite graph of degree~$6$ on $32$
vertices, isomorphic with the quotient of the \hbox{$6$-cube} by
central reflection.

Next, consider a K3 surface $X$\/ with a symplectic involution
$\iota$, i.e.\ an involution $\iota$ that multiplies the algebraic
$2$-forms on~$X$\/ by~$+1$ (such an involution of~$X$\/ is also known
as a Nikulin involution).  Then $\iota$ has eight fixed points on~$X$,
and the minimal desingularization $Y$\/ of the quotient $X/\{1,
\iota\}$ is again a K3 surface. If in addition $Y$\/ is a Kummer
surface $\Km(A)$ and the quotient map $\pi: X \rightarrow Y$ induces a
Hodge isometry $\pi_*: T_X(2) \cong T_Y$, we say that $X$\/ and $A$ are
related by a Shioda-Inose structure.  We have a diagram
$$
\xymatrix{
X \ar@{-->}[dr] & & A \ar@{-->}[dl] \\
& Y &
}
$$ of rational maps of degree $2$, and Hodge isometries $T_X(2) \cong
T_Y \cong T_A(2)$, thus inducing a Hodge isometry $T_X \cong T_A$.
(Note: A Hodge isometry is an isometry of cohomology lattices
compatible with the Hodge decomposition.)

Conversely, a theorem of Morrison \cite{Mo} shows that any Hodge isometry
between $T_X$ and $T_A$ for a K3 surface $X$\/ and an abelian surface
$A$ is induced by a Shioda-Inose structure.

\subsection{Elliptic K3 surfaces with $\mathrm{II}^*$ and
$\mathrm{III}^*$ fibers, and curves of genus $2$} \label{thesis}

We shall exploit such a Shioda-Inose correspondence between Jacobians
of genus $2$ curves and elliptic K3 surfaces with singular fibers of
type $\mathrm{II}^*$ and $\mathrm{III}^*$; equivalently, whose root lattices
$L_v$ are $E_8$ and $E_7$ respectively.
Let $C$ be a curve of genus $2$ over~$k$, and let
$$
y^2 = f(x) = f_6 x^6 + \dots + f_0 = f_6 \prod (x - \alpha_i)
$$
be a Weierstrass equation for $C$, with $f_i \in k$ and $\alpha_i \in \kbar$
(though in general $\alpha_i \notin k$).  There exist polynomial functions
$I_2(f), I_4(f), I_6(f)$ and $I_{10}(f)=\disc(f)$ of degrees $2,4,6,10$
in the coefficients of $f$ (the Igusa-Clebsch invariants of~$f$\/),
 giving a well-defined point $(I_2: I_4: I_6: I_{10})$ in
weighted projective space $\Proj^3_{1,2,3,5}$ which does not depend
on the choice of Weierstrass equation. The complement of the
hyperplane $z_4 = 0$ (where $z_4$ is the last coordinate on
$\Proj^3_{1,2,3,5}$) yields a coarse moduli space $\M_2$ of
curves of genus $2$ \cite{Ig}.  Note that $I_{10}$ is the discriminant of the
sextic polynomial $f$, and therefore it cannot vanish.
Also, the space $\M_2$ has a singular point at $(0:0:0:1)$,
corresponding to the curve $y^2 = x^5 + 1$. Given $\alpha_1, \alpha_2,
\alpha_3, \alpha_5 \in k$, with $\alpha_5 \neq 0$, it is not
necessarily the case that one can construct a genus $2$ curve over $k$
with invariants $I_d = \alpha_{2d}$.  There is an obstruction in $\Br_2(k)$:
when it vanishes, the construction of $C$ is made explicit by work of
Mestre~\cite{Me}. In any case, $C$ may always be defined over a
quadratic extension of $k$. When $k$ is a finite field, the Brauer
obstruction vanishes, and we may define $C$ over $k$. Also note that
such a curve $C$ is unique only up to $\kbar$-isomorphism,
since $\M_2$ is only a coarse moduli space.

The main result of \cite{Kum1} is the following.

\begin{theorem}\label{thesisthm}
 The elliptic K3 surface given by the Weierstrass equation
$$
y^2 = x^3 - t^3 \left(\frac{I_4}{12} t + 1 \right) x + t^5 \left(
 \frac{I_{10}}{4} t^2 + \frac{I_2 I_4 - 3 I_6}{108} t + \frac{I_2}{24}
 \right)
$$
which has elliptic fibers of type $E_8$ and $E_7$ respectively at
$t = \infty$ and $t = 0$, is related by a Shioda-Inose structure to
the Jacobian of the genus $2$ curve $C$ whose Igusa-Clebsch invariants
are $(I_2: I_4: I_6: I_{10})$.
\end{theorem}

Let $L$ be $U \oplus E_8(-1) \oplus E_7(-1)$. Then the above theorem
gives an isomorphism
$$
\psi: \M_2 \rightarrow \mathcal{E}_{E_8, E_7}
$$
between the coarse moduli space $\M_2$ of genus $2$ curves and the
moduli space of elliptic K3 surfaces with an $E_8$ fiber at $\infty$
and an $E_7$ fiber at $0$. Furthermore, this correspondence is
Galois-invariant: the Igusa-Clebsch invariants of $C$ and the Weierstrass
coefficients of the K3 surface are defined over the same field. This
is the key fact which leads to number-theoretic applications, such as
computation of models of Shimura curves over $\Q$ in~\cite{E1} or of
Hilbert modular surfaces in this paper. However, note that $C$ may
not be itself defined over the ground field, even though its
Igusa-Clebsch invariants are.

Now, let $\A_2$ be the moduli space of principally polarized abelian surfaces.
Note that the space $\M_2$ is the complement of the divisor in~$\A_2$
consisting of points corresponding to the product of two elliptic curves.
On the other hand, the moduli space $\mathcal{E}_{E_8, E_7}$
is an open subset of the moduli space $\sF_L$
of K3 surfaces polarized by $L = U \oplus E_8(-1) \oplus E_7(-1)$.
We may write such a K3 surface in Weierstrass form
$$
y^2 = x^3 + t^3 (at + a') x + t^5 (b'' t^2 + bt + b').
$$
It has an $E_8$ fiber at $t = \infty$ and at least an $E_7$ fiber
at $t = 0$. The discriminant of the cubic polynomial is
\begin{align*}
d &= -t^9 \Bigl(27{b''}^2 t^5 + 54 b b'' t^4 +(4a^3 + 27 b^2 + 54 b' b'') t^3 \\
  & \qquad +(12 a^2 a' + 54 b b') t^2 + (12a{a'}^2 + 27{b'}^2) t + 4{a'}^3
    \Bigr).
\end{align*}

By Tate's algorithm, $b'' \neq 0$ (otherwise we would have a rational elliptic
surface), and the fiber at $t = \infty$ must be of type $E_8$, while
the fiber at $t=0$ is of type $E_7$ if and only if $a' \neq 0$,
in which case one may set $a' = -1$ by scaling $(x,y,t)$ appropriately.
If on the other hand $a' = 0$ then the $E_7$ fiber gets promoted to
an $E_8$ fiber (and no further, since $b'$ cannot vanish, else we would have
a rational elliptic surface). Therefore, the moduli space
$\mathcal{E}_{E_8, E_7}$ is the complement in $\sF_L$ of the hypersurface
$a' = 0$ that corresponds to polarization by $U \oplus E_8(-1) \oplus E_8(-1)$.

When we do have $a' = 0$, it was shown by Inose that the K3 surface
$$
y^2 = x^3 + a t^4 x + t^5 (b'' t^2 + bt + b')
$$
has a Shioda-Inose structure, making it $2$-isogenous with a product
of two elliptic curves \cite{In, Sh3, KS, CD, E2}. Recall that $b'b'' \neq 0$,
and it follows from the formulas in \cite{In, Sh3, Kum1} that the
$j$-invariants $j_1,j_2$ of the two elliptic curves are determined by
$$
\frac{j_1}{1728}\cdot \frac{j_2}{1728} = \frac{-a^3}{27 b' b''}, \quad\
\left(1-\frac{j_1}{1728}\right)\left(1-\frac{j_2}{1728}\right) = \frac{b^2}{4b'b''}.
$$
Note that again the map is Galois invariant and invertible, since $a$
is only defined up to a cube root of unity (one may scale $x$ by a
cube root of unity), and $b$ is only defined up to sign (one may scale
$t$ by $-1$).

Putting everything together, we have the following proposition.

\begin{proposition}
  There is a Galois invariant isomorphism $\phi: \sF_L \rightarrow \A_2$,
  which on the open subset $\mathcal{E}_{E_8, E_7}$ restricts to the
  inverse of the explicit isomorphism $\psi: \M_2 \rightarrow
  \mathcal{E}_{E_8, E_7}$ given by Theorem \ref{thesisthm}.
\end{proposition}

For a Hodge-theoretic approach to this isomorphism of moduli spaces,
see \cite[pp. 186--188]{GN}.

\section{Humbert surfaces and Hilbert modular surfaces}\label{realmult}

We next discuss moduli spaces of abelian surfaces with real
multiplication.  As above, let $D > 0$ be a fundamental discriminant,
i.e.\ $D = d$ for $d \equiv 1 \pmod 4$ or $D = 4d$ for $d \equiv 2,3
\pmod 4$, where $d > 1$ is squarefree in both cases. Then $D$\/ is the
discriminant of the ring of integers $\sO_D = \Z + \Z(D+\sqrt{D})/2$
of the real quadratic field $K = \Q(\sqrt{D})$.
Let $\sigma_1, \sigma_2$ be the two embeddings of $K$\/ into $\C$.
Then $\SL_2(\sO_D)/\{\pm 1\}$ acts on $\Ha^+ \times \Ha^-$ by

$$ \left(\begin{array}{cc} a & b \\ c & d
  \end{array} \right) :
(z_1, z_2) \mapsto \left( \frac{\sigma_1(a) z +
  \sigma_1(b)}{\sigma_1(c) z + \sigma_1(d)} , \frac{\sigma_2(a) z +
  \sigma_2(b)}{\sigma_2(c) z + \sigma_2(d)} \right),
$$
where $\Ha^+ = \{z \in \C \, | \, \Imag z > 0 \}$ is the complex
upper half plane and $\Ha^- = - \Ha^+$ is the lower half plane.

Let
\[
\SL_2(\sO_D\0, \sO_D^*) = \left\{
\left(\begin{array}{cc} a & b \\ c &
  d \end{array}\right) \in \SL_2(K)
 \mid
  a, d \in \sO_D\0,
  c \in \sO_D^*,
  b \in (\sO_D^*)^{-1}
\right\}.
\]
We claim that this action is equivalent to the action of
$\SL_2(\sO_D\0, \sO_D^*)$ on $\Ha^+ \times \Ha^+$.
Here $\sO_D^*$ is the dual of $\sO_D$ with respect to the trace form on~$K$\/
(that is, $\sO_D^*$ is the inverse different of~$K$\/).
It is an invertible $\sO_D$-module of rank $1$;
in fact, it is easily checked to be $\frac{1}{\sqrt{D}} \sO_D$.
Assume, without loss of generality, that $\sigma_1(\sqrt{D}) > 0$
and $\sigma_2(\sqrt{D}) < 0$. Then if we let
$$
\psi: \Ha^+ \times \Ha^- \rightarrow \Ha^+ \times \Ha^+
$$ be the biholomorphic map $(z_1, z_2) \mapsto \big(z_1
\sigma_1(\sqrt{D}), z_2 \sigma_2(\sqrt{D})\big)$, and
$$
\phi: \SL_2(\sO_D) \rightarrow \SL_2(\sO_D\0, \sO_D^*)
$$
be the group isomorphism given by
$$ \left(\begin{array}{cc} a & b \\ c & d \end{array}\right) \mapsto
\left(\begin{array}{cc} a & b\sqrt{D} \\ c/\sqrt{D} &
  d \end{array}\right),
$$
an easy check shows that the following diagram commutes.
$$
\begin{CD}
\SL_2(\sO_D) \times \Ha^+ \times \Ha^- @>>> \Ha^+ \times \Ha^- \\
@VV{\phi \times \psi}V @VV{\psi}V \\
\SL_2(\sO_D\0, \sO_D^*) \times \Ha^+ \times \Ha^+ @>>> \Ha^+ \times \Ha^+
\end{CD}
$$
Therefore, $\psi$ induces an isomorphism on the quotients, as desired.

Next, we outline the proof that $\SL_2(\sO_D\0, \sO_D^*)$ is the coarse
moduli space of principally polarized abelian surfaces with real
multiplication by $\sO_D$, closely following \cite{HvdG}.
Let $M = \sO_D\0 \oplus \sO_D^*$,
and define an alternating $\Z$-valued form on~$M$\/ by
$$
E_M ((\alpha_1, \beta_1), (\alpha_2, \beta_2))
= \Tr_{K/\Q} (\alpha_1 \beta_2 - \alpha_2 \beta_1).
$$
Now, for $z = (z_1, z_2) \in \Ha^2$, consider the embedding
\begin{align*}
L_z& : K \oplus K \rightarrow V = \C^2 \\ & (\alpha, \beta) \mapsto
\alpha z + \beta = (\sigma_1(\alpha) z + \sigma_1(\beta),
\sigma_2(\alpha) z + \sigma_2(\beta)),
\end{align*}
which gives us a lattice $L_z(M)$ in $V$. We use this to transfer
$E_M$ to $L_z(M)$ and extend the form $\R$-linearly. This gives an
alternating form
$$
E_{M,z}: V \times V \rightarrow \R,
$$
which can be described in coordinates on $\C^2$ as
$$ E_{M,z} ((\zeta_1, \zeta_2), (\eta_1, \eta_2)) = \frac{\Imag
  \zeta_1 \overline{\eta_1}}{\Imag z_1} + \frac{\Imag \zeta_2
  \overline{\eta_2}}{\Imag z_2}.
$$

This gives a Riemann form on the resulting complex torus $V/L_z(M)$,
and since the form $E_M$ on the lattice $L_z(M)$ is unimodular, we
obtain an abelian variety with principal polarization. The action of
$\sO_D$ is as follows:
$$ \iota(\alpha) (\zeta_1, \zeta_2) \mapsto (\sigma_1(\alpha) \zeta_1,
\sigma_2(\alpha) \zeta_2).
$$

Conversely, it is not hard to show that any principally polarized
abelian surface with real multiplication by $\sO_D$ can be identified
with some $V/L_z(M)$. Finally, we note that the abelian surfaces
corresponding to two different $z \in \Ha^+ \times \Ha^+$ are
isomorphic (with an $\sO_D$-equivariant isomorphism) exactly when these
two points differ by an element of $\SL_2(\sO_D\0, \sO_D^*)$.
Therefore, it follows that the moduli space of abelian surfaces with real
multiplication is $\SL_2(\sO_D\0, \sO_D^*) \backslash (\Ha^+ \times \Ha^+)$,
which as we have shown above, is biholomorphic with
$\SL_2(\sO_D) \backslash (\Ha^+ \times \Ha^-)$.

The construction above yields a map from $\SL_2(\sO_D) \backslash
(\Ha^+ \times \Ha^-)$ to the quotient of $\sS_2$, the Siegel upper
half space of degree $2$, by the arithmetic group $\Sp_4(\Z)$.  In
other words, we get a holomorphic map to $\A_2$, the moduli space of
principally polarized abelian surfaces. Its image (or its closure in
$\A_2$) is the {\em Humbert surface} $\sH_D$ for discriminant $D$.
We next show that the map from the Hilbert modular surface $\SL_2(\sO_D)
\backslash (\Ha^+ \times \Ha^-)$ to the Humbert surface is generically
$2$ to $1$. We may compute this degree above a very general point on
the Humbert surface. Such a point corresponds to a principally
polarized abelian surface $A$, where one has forgotten the action of
$\sO_D$ by endomorphisms. Generically, there are exactly two ways to
extend the obvious map $\Z \rightarrow \End(A) \cong \sO_D$ to
$\sO_D$, corresponding to the choice of image of $(D + \sqrt{D})/2$.

Our approach to computing equations of Hilbert modular surfaces begins
as follows. Fix a discriminant $D$, which we assume to be a
fundamental discriminant.  First, we need to compute a model of the
Humbert surface, i.e.\ the subvariety of $\A_2$ corresponding to
abelian surfaces with real multiplication by $\sO_D$. Via the inverse
of the isomorphism $\phi: \sF_L \rightarrow \A_2$ of subsection
\ref{thesis} above, the Humbert surface corresponds to a surface
inside the three-dimensional moduli space of $L$-polarized K3 surfaces.

Define a pairing on $\sO_D$ by $(\alpha, \beta) \mapsto \tr(\alpha \beta^*)$,
where $\beta^*$ is the Galois conjugate of $\beta$.  This gives $\sO_D$
the structure of an indefinite lattice, which we next identify with
the N\'{e}ron-Severi lattice of~$A$.

\begin{proposition}
Let $A$ be a principally polarized abelian surface with $\End(A) \cong \sO_D$.
Then $\NS(A) \cong \End(A)$.  The lattice $\NS(A)$ has a basis with Gram matrix
\begin{equation}
\left(\begin{array}{cc}
\label{eq:ODGram}
2 & D \\
D & (D^2 - D)/2
\end{array}\right)
\end{equation}
of signature $(1,1)$ and discriminant $-D$.
\end{proposition}

\begin{proof}
For a principally polarized abelian surface $A$, there is an isomorphism
$$
\NS(A) \rightarrow (\End(A))^\dagger
$$
induced naturally by the polarization, where $\dagger$ is the
Rosati involution (also arising from the polarization). Note that for
a general polarization one only gets a weaker isomorphism, between the
$\Q$-spans of both sides. For the proof of both assertions see
Proposition 5.2.1 in \cite{BL}. Now, the Rosati involution is a
positive involution of the real quadratic field $\End(A) \otimes \Q$,
and hence cannot be the nontrivial element of the Galois group.
It must therefore be the identity, whence the subring of~$\End(A)$
fixed by~$\dagger$ is $\End(A)$ itself, proving the first statement.
The isomorphism of groups $\NS(A) \rightarrow \End(A)$ is an isometry,
taking the intersection form on $\NS(A)$ to the form
$(\phi, \psi) \mapsto \tr(\phi \psi^{\vee})$ on $\End(A)$,
which becomes $(\alpha, \beta) \mapsto \tr(\alpha \beta^*)$ on $\sO_D$.
Computing the matrix of this form on the basis $(1,(D+\sqrt{D})/2)$ of~$\sO_D$,
we obtain the claimed Gram matrix~(\ref{eq:ODGram}).
\end{proof}

We henceforth use $\sO_D$ also to denote the lattice with underlying group
 $\sO_D \cong \Z^2$ and form $(\alpha, \beta) \mapsto \tr(\alpha \beta^*)$,
 with Gram matrix~(\ref{eq:ODGram}).

\begin{proposition}
  There is a primitive embedding, unique up to isomorphism, of the lattice
  $\sO_D$ into $U^3$.  Let $T_D$ be the orthogonal complement of
  $\sO_D$ in $U^3$.  Then there is a primitive embedding,
  unique up to isomorphism, of $T_D$ into the K3 lattice $\Lambda$.
\end{proposition}

{\em Remark.}
  In the present paper we analyze only fundamental discriminants~$D$,
  for which every embedding into an even lattice of $\sO_D$ and $T_D$
  (and of the lattice $L_D$, to be introduced before the next theorem)
  is automatically primitive.  But the condition of primitivity is necessary
  to extend our theoretical analysis also to non-fundamental~$D$.

\begin{proof}
  Since $\sO_D$ is two-dimensional, the number of generators
  $\ell(\sO_D^*/\sO_D)$ of the discriminant group is at most $2$. In fact,
  the discriminant group is cyclic of order $D$\/ if $D$\/ is odd, and
  isomorphic to $(\Z/2\Z) \oplus (\Z/(\frac{D}{2} \Z))$ if $D$\/ is even.
  Therefore, by~\cite[Theorem 1.14.4]{Ni3}, $\sO_D$ has a
  unique embedding into the even unimodular lattice~$U^3$.
  Let the orthogonal complement be $T_D$; it has signature $(2,2)$ and the same
  discriminant group. By another application of~\cite[Theorem 1.14.4]{Ni3},
  we see that $T_D$ has a unique embedding into $\Lambda_{K3}$.
\end{proof}

Let $q_D$ be the discriminant form of $\sO_D$, and let $L_D$ be the
orthogonal complement of $T_D$ in $\Lambda$. Then $L_D$ has signature
$(1,17)$ and discriminant form $q_D$.  By \cite[Theorem 1.13.2]{Ni3}
$L_D$ is characterized uniquely by its signature and discriminant form.
Since $\Lambda \cong U^3 \oplus E_8(-1)^2$, it is clear that
$L_D \cong E_8(-1)^2 \oplus \sO_D$.  Finally, $L_D$ has a primitive
embedding in~$\Lambda$, which (again by~\cite[Theorem 1.14.4]{Ni3})
is unique up to isomorphism.

\begin{theorem} \label{humbert}
  Let $\sF_{L_D}$ be the moduli space of K3 surfaces that are lattice
  polarized by $L_D$.  Then the isomorphism $\phi: \sF_L \rightarrow
  \A_2$ of Section~\ref{thesis} induces a birational surjective
  morphism $\sF_{L_D} \rightarrow \sH_D$.
\end{theorem}

\begin{proof}
  First, we note that $\Z \subset \sO_D$ induces
  an embedding of lattices $\langle 2 \rangle \subset \sO_D$, and therefore
  an embedding  $T_D \subset U^2 \oplus \langle -2 \rangle$ of orthogonal
  complements.  Taking orthogonal complements once more in
  $\Lambda_{K3}$, we deduce $L \subset L_D$.

  Fix embeddings $L \subset L_D \subset \Lambda_{K3}$. Then we have
  $L_D^{\perp} \hookrightarrow L^{\perp}$, which induces a map of
  period domains $\Omega_{L_D} \hookrightarrow \Omega_{L}$. We also
  have a map $\Gamma_{L_D} \subset \Gamma_L$. These induce a map
  $\sF_{L_D} = \Gamma_{L_D} \backslash \Omega_{L_D}
  \stackrel{\beta_D}{\rightarrow} \Gamma_L \backslash \Omega_L = \sF_L$.
  We will use the fact that this morphism has degree $1$ onto
  the image (i.e.\ it is an embedding on the generic point of $\sF_{L_D}$).
  We postpone the proof of this fact until the conclusion of the present
  argument.

  Let $\phi_D: \sF_{L_D} \stackrel{\beta_D}{\rightarrow} \sF_L
  \stackrel{\phi}{\rightarrow} \A_2$ be the composition of the maps
  above.  Now suppose $X$\/ is a K3 surface corresponding to a point
  $p$ in the image of $\beta_D$. Then we have $\NS(X) \supset L_D
  \cong E_8(-1)^2 \oplus \sO_D$, and therefore $T_X \subset T_D$. If
  $A$ (corresponding to $\phi(p)$) is the abelian surface connected to
  $X$ through a Shioda-Inose structure, we have $T_A \cong T_X \subset
  T_D$. Therefore, $\NS(A) \supset \sO_D$, the orthogonal complement
  of $T_D$ in $U^3$. Therefore $\End(A)^{\dagger} \supset \sO_D$, and
  $\phi(p)$ must lie on the Humbert surface $\sH_D$. This proves that
  the image of $\sF_{L_D}$ lands in $\sH_D$.

  Conversely, suppose $q$ is a point on $\sH_D$ corresponding to an
  abelian surface $A$ with real multiplication by $\sO_D$. Then
  $\End(A) \supset \sO_D$, and since the Rosati involution (being
  positive) can only act as the trivial element of $\Q(\sqrt{D})$, we
  must have $\End(A)^{\dagger} \supset \sO_D$. Retracing the argument
  in the previous paragraph, we see $p = \phi^{-1}(q)$ must have
  $\NS(X) \supset L_D$, and therefore is in the image of
  $\beta_D$. This proves that $\phi \circ \beta_D$ is surjective onto $\sH_D$.

  Since $\phi$ is an isomorphism and $\beta_D$ has degree $1$ onto
  its image, so does $\phi \circ \beta_D$.
\end{proof}

We now prove that fact used above: $\beta_D$ is generically an embedding.

\begin{proposition}
The map $\sF_{L_D} \rightarrow \sF_L$ is generically an embedding.
\end{proposition}

\begin{proof}
  To see this, we claim that it is enough to show that there is an
  element $\gamma_0 \in O(\Lambda)$ such that $\gamma_0$ fixes $L$
  pointwise, preserves the sublattice $L_D$, and acts by $-1$ on
  $L^{\perp}$. For suppose we have $x = \gamma y$ with $x, y \in
  \Omega_{L_D}$ and $\gamma \in \Gamma_L$, with $x$ and $y$ very
  general. Then we would like to show that in fact we already have $x
  = \gamma' y$ with $\gamma' \in \Gamma_{L_D}$. Then $\Lambda^{1,1}_x
  := \{ z \in x \, | \, \langle z, x \rangle = 0 \}$ and
  $\Lambda^{1,1}_y$ both equal $L_D$, since $x$ and $y$ are
  very general, and therefore $\gamma$ preserves $L_D$ (though it might
  not fix this lattice pointwise). If $\gamma$ acts trivially on
  $L_D$, we may take $\gamma' = \gamma$, and are done. If not, then
  $\gamma$ must act by $-1$ on $L_D/L \cong \Z$. Then $\gamma' =
  \gamma_0 \gamma$ will suffice.

  To prove the claim, we may consider specific embeddings $L \subset
  L_D \subset \Lambda$. If $D \equiv 0 \pmod 4$, then $L_D \cong
  E_7(-1) \oplus \langle -D/2 \rangle \oplus U \oplus E_8(-1) $. We
  may embed it inside $\big(E_8(-1) \oplus U \oplus U \big) \oplus
  \big(U \oplus E_8(-1) \big)$ by embedding $E_7(-1)$ inside $E_8(-1)$
  as the orthogonal complement of a root $\delta$, and $ \langle -D/2
  \rangle$ primitively inside $U \oplus U$. Let $\gamma_0$ be the
  automorphism of $\Lambda$ which acts on the first copy of $E_8(-1)$
  as the reflection $s_\delta$ in the hyperplane orthogonal to
  $\delta$, by $-1$ on $U \oplus U$, and as the identity on the part
  $U \oplus E_8(-1)$.

  If $D \equiv 1 \pmod 4$, consider a system of simple roots in
  $E_8(-1)$ whose intersections give the standard Dynkin matrix. Let
  $\alpha$ be the simple root such that if we remove the corresponding
  vertex from the Dynkin diagram, we obtain the Dynkin diagram for
  $E_7(-1)$. The orthogonal complement of this copy of $E_7(-1)$ is
  some root~$\delta$. Let $\beta \in U$ be any element of norm $(1-D)/2$.
  It is easy to check that the copy of $E_7(-1)$ and $\alpha + \beta$
  generate a negative definite lattice $M$\/ of discriminant~$D$,
  and that $U \oplus M \oplus E_8(-1)$ is isometric
  to $L_D$ (since it is has the right signature and discriminant). In
  other words, the diagram of embeddings
$$
E_7(-1) \hookrightarrow M \subset E_8(-1) \oplus U \hookrightarrow
E_8(-1) \oplus U^2
$$
extends on adding a factor of $U \oplus E_8(-1)$ to
$$
L = E_7(-1) \oplus U \oplus E_8(-1) \hookrightarrow M \oplus U \oplus
E_8(-1) = L_D \subset E_8(-1) \oplus U^2 \oplus E_8(-1)
\rightarrow E_8(-1) \oplus U^3 \oplus E_8(-1) = \Lambda.
$$
Now, consider the automorphism $\gamma_0$ of $\Lambda \cong E_8(-1)
\oplus U^3 \oplus E_8(-1)$ given by
$(s_{\delta}\0, -1|_U\0, -1|_U\0, 1_U\0, 1_{E_8(-1)}\0)$.
By construction, $\gamma_0$ fixes $L$ pointwise,
and acts by $-1$ on $L^{\perp}$. We need to show that
$\gamma_0$ preserves $L_D$, or equivalently, that it preserves $M$. We
know that $M$\/ is spanned by $E_7(-1)$ (which is fixed by $\gamma_0$)
and $\alpha + \beta$. Now, $\gamma_0$ takes $\beta$ to $-\beta$, by
construction. Also, $\alpha + \gamma_0(\alpha) = \alpha +
s_\delta(\alpha) \in E_8(-1)$ is invariant by $s_\delta$, and is
therefore in the orthogonal complement $E_7(-1)$ of
$\delta$. Therefore $\gamma_0(\alpha) \in -\alpha + E_7(-1)$ and
$\gamma_0(\alpha + \beta) \equiv -(\alpha + \beta) \mod E_7(-1)$.
So $\gamma_0$ preserves the lattice~$M$.
\end{proof}

Next, we want to understand the Hilbert modular surface $Y_{-}(D)$,
which is a double cover of $\sH_D$. First, we must identify the branch
locus. Since the map $Y_{-}(D) \rightarrow \sH_D$ is obtained by
simply forgetting the action of $e_D = (D+\sqrt{D})/2$ on the abelian
surface, the branch locus is the subvariety $W$ of $\sH_D$
corresponding to abelian surfaces $A$ such that $e_D = (D+\sqrt{D})/2$
and $e'_D = (D - \sqrt{D})/2$ are conjugate in the endomorphism ring,
say by $\iota \in \End(A)$. It follows that $\iota^2$ fixes $\sO_D$
pointwise.  Generically this implies $\iota^2 = \pm 1$, and
$\iota e_D\0 = e'_D \iota$. This shows that $\End(A)$ for such~$A$
is generically an order in a quaternion algebra $B$.

In fact, the branch locus corresponds to the case when we have a split
quaternion algebra, i.e.\ $\iota^2 = 1$.  Then $A$ is isogenous
over~$\Qbar$ to the square of an elliptic curve.  To see this fact,
observe that the map $\Gamma \backslash \sH^2 \rightarrow \Sp_4(\Z)
\backslash \sS_2$ (where $\Gamma = \SL_2(\sO_D\0, \sO_D^*)$) factors
through the quotient $(\Gamma \cup \Gamma \sigma) \backslash \sH^2$,
where $\sigma$ is the involution $(z_1, z_2) \mapsto (z_2, z_1)$
exchanging the two factors of $\sH^2$, and the induced map on this
quotient is generically one to one \cite[page 158]{HvdG}.  Therefore,
to understand the branch locus, we need to understand the
one-dimensional part of the fixed point set of~$\sigma$ on $\Gamma
\backslash \sH^2$. This (and more) was done by Hausmann in \cite[page
  35]{Ha}. The result is that the fixed point set consists of a small
number of explicit modular curves $F_w$ for $w \in \{1,4,D,D/4\}$
(some of these may be empty, and when they are non-empty, these curves
are irreducible). A simple explicit analysis of the condition relating
$z_1$ and $z_2$ on these curves reveals that the generic point on each
of these corresponds to an abelian surface whose ring of endomorphisms
contains zero divisors, and is therefore a split quaternion
algebra. For instance, the image of the diagonal of $\sH^2$ (given by
$z_1 = z_2$) in $\Gamma \backslash \sH^2$ is an obvious component of
the branch locus. For the corresponding abelian surface, as
constructed earlier in this section, the map $(\zeta_1, \zeta_2)
\mapsto (\zeta_2, \zeta_1)$ is a holomorphic involution, showing that
the endomorphism algebra is split. For completeness, we prove this
next.

\begin{proposition}
  Let $\Gamma = \SL_2(\sO_D, \sO_D^*)$ and let $C$ be one of the curves
  $F_w$ for $w \in \{1,D\}$ if $D$ is odd, $w \in \{1,4,D/4,D\}$ if
  $D$ is even. The generic point on $C$ corresponds to an abelian
  surface whose algebra of endomorphisms is a split quaternion
  algebra.
\end{proposition}

\begin{proof}
  We use the notation of \cite[Chapter V]{vdG}. The ideal $\mathfrak{a} =
  \sO_D^*$ has norm $A = 1/D$. In the reference, it is assumed that
  $\mathfrak{a}$ is an integral ideal, but we may for instance replace
  $\sO_D^*$ by the integral $D\sO_D^* = \sqrt{D} \sO_D$ without loss
  of generality in Hausmann's proof. This would replace $A$ by $D^2 A
  = D$, and would not affect any of the arguments below, which depend
  only on the square class of $A$. Proposition $\mathrm{V}.1.5$ of \cite{vdG}
  states that a point on $F_N$ corresponds to an abelian surface whose
  endomorphism algebra is isomorphic to the indefinite quaternion
  algebra $Q_N = \left(\frac{ D, -N/(AD)}{\Q} \right)$, while Lemma
  $\mathrm{V}.1.4$ in the book says that $F_N$ is non-empty iff for each prime
  $q$ dividing $D$ and not dividing $N$, we have $\chi_{D(q)}(N) =
  (A,D)_q$. For an explanation of the notation, see
  \cite[pg.~2--3]{vdG}. In our situation, we have $AD = 1$. If $N =
  D$, we get the algebra $\left( \frac{D,-D}{\Q} \right)$, which is
  obviously split. This argument also takes care of $N = D/4$ when $D$
  is even. Next, suppose $N = 1$. Then $F_1$ is non-empty iff for
  every prime dividing $D$, we have $(D,D)_q = 1$. So
$$ 
1 = (D,D)_q = (D,-D)_q \cdot (D,-1)_q = (D,-1)_q.
$$
It follows that the quaternion algebra $Q_1 = \left( \frac{D,-1}{\Q}
\right)$ is split. The proof for $N = 4$ is similar.
\end{proof}

\begin{remark}
  A further analysis of the proof above reveals the number of
  components of the branch locus, which is corroborated by the
  calculations in this paper.
\end{remark}

Coming back to our analysis of the endomorphisms of the abelian
surfaces corresponding to points on the branch locus, we note that the
Rosati involution must fix $\sO_D$ and $\iota$, by positivity.
Consider the form $(x,y) \mapsto \Tr(x\bar{y})$, where $\bar \, :
B \rightarrow B$ is the natural involution taking $\sqrt{D}$ to its
negative and $\iota$ to its negative, and $\Tr$ is the reduced
trace. The matrix of this form acting on $\sO_D + \Z \iota$ is
$$
\left(\begin{array}{ccc} 2 & D & 0 \\ D & (D^2 - D)/2 & 0 \\ 0 & 0 & -2
 \end{array}\right),
$$
which therefore gives an even lattice of signature $(1,2)$ and
discriminant $2D$.
We claim that when $D$\/ is fundamental the index, call it $c$, of
$\sO_D \oplus \Z \iota$ as a sublattice of $\End(A)$ is $1$ or $2$,
with $c=2$ possible only if $4|D$.  Indeed $\End(A)$ is an even
lattice of rank~$3$ and discriminant $2D/c^2$.  But the discriminant
is an even integer because the rank is odd, and $D$\/ has no repeated
prime factors except possibly $2^2$ or~$2^3$.  Thus as claimed the
index~$c$ must be either $1$ or $2$, with $c=2 \Rightarrow 4|D$.
Therefore, the only possibilities for the discriminant of the
resulting N\'{e}ron-Severi group for a point corresponding to the
branch locus are $2D$\/ and $D/2$. This is also the discriminant of
the transcendental lattice of the abelian surface, and by the
Shioda-Inose structure, the corresponding K3 surface also has a
N\'{e}ron-Severi lattice of rank~$19$ and discriminant $2D$\/ or
$D/2$, and contains $L_D \cong E_8(-1)^2 \oplus \sO_D$.

Finally, we mention that a group of involutions acts naturally on the
Hilbert modular surface $Y_{-}(D)$. This group is isomorphic to
$(\Z/2\Z)^{t-1}$, where $t$ is the number of distinct primes dividing
$D$. These extra involutions arise from the Hurwitz-Maass extension of
$\PSL_2(\sO_D\0, \sO_D^*)$. For more details see \cite[Section
I.4]{vdG}.

\section{Method of computation} \label{method}

We now outline our method to obtain equations for the Hilbert modular
surfaces $Y_{-}(D)$ for fundamental discriminants $D$.

\subsection*{Step 1}
We compute the unique lattice $L_D$, writing it in the form $U \oplus
N$ (we know $L_D$ contains a copy of $U$, since $L = U \oplus E_8(-1)
\oplus E_7(-1)$ does, and $L \supset L_D$). Note that $N$ is not
uniquely determined, and it turns out to be useful to choose $N$ such
that its root lattice $R$ has small codimension in $N$, and such that
$N/R$ has a basis consisting of vectors of small (fractional)
norm. One may then attempt to realize $L_D$ as the N\'{e}ron-Severi
lattice of an elliptic K3 surface with reducible fibers corresponding
to the irreducible root lattices in $R$. The remaining generators of
the N\'{e}ron-Severi lattice should then arise as sections (of small
height) in the \MoW\ group of the elliptic fibration. Of course, we
need to make sure that $N$ contains $E_8 \oplus E_7$, but this is
easily achieved by showing that the dual lattice $N^*$ contains a
vector of norm $2/D$.

For most of the examples in this paper, we were able to describe a
family of elliptic K3 surfaces with N\'{e}ron-Severi lattice $L_D$ and
with \MoW\ rank $0$ or $1$.  This construction gives us a Zariski-open
subset of $\sF_{L_D}$, with some possible missing curves corresponding
to jumps in the \MoW\ ranks or the ranks of the reducible fibers, or
to denominators introduced in the parametrization process.

For all the examples we considered, the moduli space $\sF_{L_D}$ turns
out to be rational (i.e.\ birational with~$\Proj^2$ over~$\Q$).
This property will fail to hold once $D$\/ is large enough,
but there is still some hope of writing down usable equations
for $\sF_{L_D}$ if it is not too complicated.

Suppose that we have exhibited a Zariski open subset of $\sF_{L_D}$ as
a Zariski open subset $U = \Proj_{r,s}^2 \backslash V$ of the
projective plane.

\subsection*{Step 2}
We find a different elliptic fibration on the generic member of the
family of K3 surfaces in Step~1, with reducible fibers of types $E_8$
at $t = \infty$ and $E_7$ at $t = 0$. This is accomplished by $2$- and
$3$-neighbor steps, which we shall describe in the next section.

Once we obtain the alternate elliptic fibration, we may compute the
map from $U$ to $\A_2$ by the explicit formulas of section
\ref{thesis}. This gives us the Igusa-Clebsch invariants of the
associated genus $2$ curve, in terms of the two parameters $r$ and
$s$.

\subsection*{Step 3}
Consider the base change diagram
$$
\begin{CD}
  \widetilde{Y}_{-}(D) @>{\widetilde{\eta}}>> \sF_{L_D} \\
  @V{\widetilde{\phi}}VV          @VV{\phi}V        \\
  Y_{-}(D)  @>{\eta}>> \sH_D
\end{CD}
$$
Since $\phi$ is a birational map, so is $\widetilde{\phi}$. We want to
describe the degree $2$ map $\widetilde{\eta}$. In the current
situation, we must have a model for $\widetilde{Y}_{-}(D)$ (which is
birational to $Y_{-}(D)$) of the form
$$
z^2 = f(r,s)
$$
where $f(r,s) = 0$ describes the branch locus of
$\widetilde{\eta}$.  As we have seen, this locus $f(r,s) = 0$
must be the union of some irreducible curves, which are either
components of the excluded locus $V$ above, or belong to the sublocus
of $U$ where the Picard number jumps by $1$, the discriminant changing
to $2D$\/ or $D/2$. There are finitely many possibilities for how this
may happen. For instance, the elliptic fibration may have an extra $A_1$
fiber, or a $D_6$ fiber may get promoted to an $E_7$ fiber, or the surface
may have a new section that raises the \MoW\ rank.  Similarly,
in any instance, we may easily list the allowed changes in the
reducible fibers, or the allowed heights and intersections with
components of reducible fibers of a new section of the
fibration. Hence we get a list of polynomials $f_1(r,s), f_2(r,s),
\dots, f_k(r,s) \in \Z[r,s]$ whose zero loci are the possible
components of the branch locus. We may assume that each $f_i(r,s)$ has
content $1$.

We therefore have the following model for $\widetilde{Y}_{-}(D)$ as a
double cover of $\Proj_{r,s}^2$:
$$
z^2 = f(r,s) = C f_{i_1}(r,s) f_{i_2}(r,s) \dots f_{i_m}(r,s)
$$
for some subset $\{i_1, \dots, i_m \} \subseteq \{1, \dots, k \}$ and
some squarefree integer $C$.

\subsection*{Step 4}
In the final step, we show how to compute the subset $I = \{i_1,
\dots, i_m \}$ and the twist $C$. First, we note that the Hilbert
modular surface $Y_{-}(D)$ has good reduction outside primes dividing
$D$, by \cite{Ra, DP}. Therefore, there are only finitely many choices
for $C$ as well. Now we check each of these choices by computer, using
the method described in the next paragraph, and rule out all but one.

So suppose we have a putative choice of $C$ and $I$. We proceed to
check whether this choice of twist is correct by reduction modulo
several odd primes $p$ that do not divide the discriminant~$D$. We
specialize $r,s$ to elements of $\F_p$, and by the formulas of Step 2,
obtain Igusa-Clebsch invariants for in weighted projective space over~$\F_p$.
Since the Brauer obstruction vanishes for finite fields, we can construct
a curve $C_{r,s}$ over~$\F_p$.

We may use the following lemma to detect whether the Jacobian of
$C_{r,s}$ has real multiplication defined over $\F_p$ or not.

\begin{lemma} \label{twistrealmult} Let $A$ be an abelian surface
  over~$\F_p$, and let $\phi$ be the Frobenius endomorphism of $A$
  relative to~$\F_p$.  Suppose that the characteristic polynomial
  $P(T)$ of~$\phi$ is irreducible over $\Q$.  Let $Q$\/ be the
  symmetric characteristic polynomial of $\phi + p \phi^{-1}$, defined
  by $P(T) = T^2 Q(T + pT^{-1})$.
\begin{enumerate}
\item If $A$ has real multiplication by an order in~$\sO_D$
  defined over~$\F_p$, then $Q$\/ is a quadratic polynomial
  of discriminant $c^2 D$\/ for some integer~$c$.
\item If $A$ has real multiplication by an order in~$\sO_D$, defined
  over~$\F_{p^2}$, but not over~$\F_p$, then we have $Q(X) = X^2 - n$
  for some $n \in \Z$ not of the form $c^2D$ for any integer $c$.
\item If $Q(X)$ is a quadratic polynomial of discriminant $D$
  (resp. $c^2 D$ for some positive integer $c$), then $A$ has real
  multiplication by~$\sO_D$ (resp.\ an order in $\sO_D$ of conductor
  dividing $c$), defined over~$\F_p$\,.
\end{enumerate}
\end{lemma}

Note that $A$ might simultaneously have real multiplication $\iota:
\sO \hookrightarrow \End(A)$ and $\iota' : \sO' \hookrightarrow
\End(A)$ by two orders of $\sO_D$, if its ring of endomorphisms is a
quaternion algebra. Furthermore, if one of these is defined over
$\F_p$ and the other only over $\F_p^2$, then by parts (1) and (2) of
the Lemma, $Q(X)$ must be of the form $X^2 - c^2D$ for some integer~$c$.

We will give the proof of the lemma below, but first we indicate how
to use it to find the correct twist. We choose a suitable prime $p
\nmid D$, and for some $r,s$ in $\F_p$, we calculate the number of
points mod $p$ and $p^2$ of the resulting curve $C_{r,s}$. This is
enough to describe the polynomial $Q(X)$ for the abelian surface $A =
J(C)$, by the Lefschetz-Grothendieck trace formula.  Suppose $P(T) =
T^2 Q(T+pT^{-1})$ is irreducible.  Then the first hypothesis of the
lemma is satisfied.  If in addition $Q(X) = X^2 + aX + b $ has
discriminant $D$ and non-zero linear term $a$, then $A$ must have real
multiplication defined over~$\F_p$, by part (3) of the Lemma (and the
comments following it). In fact, most of the time we can even make do
with the weaker assumption that $Q(X)$ has discriminant in the square
class of $D$, since (3) guarantees that $A$ has real multiplication by
an order in $\Q(\sqrt{D})$. We can compute the discriminant of
$\NS(A)$ by computing that of $\NS(Y)$, where $Y$ is the elliptic K3
surface related to $A$ by a Shioda-Inose structure. If this
discriminant is $D$, it follows that $A$ must have real multiplication
by the full ring of integers $\sO_D$.

Now, by the property of a coarse moduli space, $A$ gives rise to an
$\F_p$-point $(r,s)$ of the Hilbert modular surface $Y_{-}(D)$,
i.e.\ the corresponding point on the Humbert surface must actually
lift to the double cover.  Therefore, if $Cf_{i_1}(r,s) \dots
f_{i_m}(r,s)$ is not a square, we must have the wrong quadratic
twist. We can run this test for many such $(r,s) \in \F_p \times
\F_p$. Note that a single large prime is usually enough to pin down
the correct choice of $\{i_1, \dots, i_m\}$, by eliminating all but
one possibility. However, to pin down the correct choice of $C$, we
may need to use several primes, until we find one for which $C/C'$ is
not a quadratic residue, where $C'$ is the correct twist. However, in
any case, this procedure is guaranteed to terminate, and in practice,
it terminates fairly quickly.

At the end of step $4$ of the algorithm, we have determined a
birational model of $Y_{-}(D)$ over $\Q$.
We now give the proof of the lemma above.

\begin{proof}[Proof of Lemma \ref{twistrealmult}]
We will use \cite[Theorem 2]{Ta3}. Since $P(T)$ is irreducible, it
follows that $F = \Q[\phi]$ is simple, and also that $F = E
:= \End_{\F_p}(A) \otimes \Q$ is a quartic number field.  First,
assume that real multiplication by an order $\sO \subset \sO_D$ is
defined over $\F_{p^2}$ but not over $\F_p$.  Say $\sO$ contains
$f\sqrt{D}$ for some positive integer~$f$.  Then consider the base
change of $A$ to $\F_{p^2}$.  The Frobenius over the new field is
$\phi^2$, and therefore, $\phi^2$ commutes with $f\sqrt{D}$, but
$\phi$ must anticommute with $f\sqrt{D}$.  Therefore $\Q(\phi^2)$ is a
strict subfield of $\Q(\phi)$, and must be quadratic.  Hence $\phi^2$
satisfies a quadratic equation
$$
\phi^4 + a \phi^2 + b = 0
$$
and since this must be the characteristic polynomial of $\phi$, we
must have $b = p^2$. Therefore
$$
\phi^2 + p^2\phi^{-2} + a = (\phi + p\phi^{-1})^2 + a - 2p = 0,
$$
proving the assertion (1). 

Now, assume that $A$ has real multiplication by $\sO \subset \sO_D$,
defined over $\F_p$, and let $f \sqrt{D} \subset \sO$, as before. Then
$\Q(\phi + p \phi^{-1})$ is a quadratic subfield of $\Q(\phi)$. Also,
$f\sqrt{D} \in \Q(\phi)$, so we have $f\sqrt{D} = g(\phi)$ for some
polynomial $g \in \Q[T]$. Then $f\sqrt{D}$ is its own dual isogeny, and
so $f\sqrt{D} = g(p \phi^{-1})$ as well. Therefore,
$$
f \sqrt{D} = \frac{1}{2} \Big( g(\phi) + g(p \phi^{-1}) \Big)
$$
can be expressed as $h(\phi + p \phi^{-1})$ for some $h \in \Q(T)$. It
follows that $\Q(\sqrt{D}) = \Q(\phi + p \phi^{-1})$ and therefore the
minimal polynomial $Q(T)$ of $\phi + p \phi^{-1}$ has discriminant
equal to $D$\/ times a square. This proves (2).

Finally, if $Q$ is a quadratic polynomial of discriminant $c^2D$, then
$\eta = \phi + p \phi^{-1}$ is an endomorphism of $A$ satisfying
$Q(\eta) = 0$.  Therefore $\Z[\eta] \cong \sO := \Z + c\, \sO_D$, and we
conclude that $A$ has real multiplication by $\sO$ defined over
$\F_p$, proving (3).
\end{proof}

We conclude this section with a few comments on the models of Hilbert
modular surfaces computed in this paper. We first give formulas for
the family of elliptic K3 surfaces over $\sF_{L_D}$, and then describe
the $2$- and $3$-neighbor steps necessary to reach the alternate
fibration of Step 2. The details of the parametrization will not be
included in the paper, but they are available in the auxiliary
computer files. Steps 3 and 4 are relatively easy to automate. We
simply write down the result, which is the equation of $Y_{-}(D)$ as a
double cover of the Humbert surface. We give a table of points of
small height for which the Brauer obstruction vanishes, and the
associated curves of genus two. We also analyze the Hilbert modular
surface further in the cases that we can describe it as a K3 or
honestly elliptic surface. In particular, we determine the (geometric)
Picard number and generators for the Mordell-Weil group of sections in
most of the cases. We also analyze the branch locus of Step 3,
identifying it with a union of quotients of classical modular curves
in several cases, with the help of explicit formulas given
in~\cite{E3} or obtained by the methods of that paper.  Finally, for
many discriminants, we are able to exhibit curves of low genus on the
surface, possessing infinitely many rational points.

We found the method of van Luijk \cite{vL} quite useful in determining
the ranks of the N\'eron-Severi lattices of these surfaces. Briefly,
the method is as follows. Let $X$ be a smooth projective surface over
$\Q$. For a prime $p$ of good reduction, let $X_p$ be the reduction of
a good model of $X$ at $p$. By counting points on $X_p(\F_q)$ for a
small number of prime powers $q = p^e$, we obtain the characteristic
polynomial of the Frobenius $\phi_p$ on some $\ell$-adic \'etale
cohomology group $H^2(X \times \overline{\F}_p, \Q_l)$. The number of
roots $\rho_0(p)$ of this polynomial which are $p$ times a root of
unity is an upper bound on the geometric Picard number of
$X_p$. Therefore $\rho(X) \leq \rho_0(p)$ for such primes. If we have
$\rho_0(p_1) = \rho_0(p_2) = \rho_0$, but the (expected) square
classes of the discriminant of the N\'eron-Severi groups modulo these
primes (as predicted by the Artin-Tate formula \cite{Ta2}) are distinct, we may
even deduce $\rho(X) < \rho_0$. For if $X_p$ does not satisfy the Tate
conjecture for some $p \in \{ p_1, p_2 \}$, then $\rho(X) \leq \rho(X_p) <
\rho_0$. On the other hand, if both these reductions satisfy the Tate
conjecture, then they also satisfy the Artin-Tate conjecture
\cite{Mi1, Mi2}, and since the size of the Brauer group of $X_p$ is a
square \cite{Mi1, LLR}, the N\'eron-Severi groups must have
discriminants in the same square class.

\section{Neighbor method}\label{neighbors}

Finally, we describe how to transform from one elliptic fibration to
another, using 2- and 3-neighbor steps. We start with an elliptic K3
surface over a field $k$, which we assume has characteristic different
from $2$ or $3$. Let $F$\/ be the class of the fiber, and $O$ be the
class of the zero section. The surface $X$\/ is a minimal proper model
of a given Weierstrass equation
$$
y^2 = x^3 + a_2(t) x^2 + a_4(t) x + a_6(t).
$$ 
Here the elliptic fibration is $\pi: X \rightarrow \Proj^1_t$. Now,
let $F'$ be an elliptic divisor (i.e. $F'$ is effective, $F'^2 = 0$,
the components of $F'$ are smooth rational curves, and $F'$ is
primitive).

We would like to write down the Weierstrass equation for this new
elliptic fibration on $X$.  The space of global sections $H^0(X, \sO_X(F'))$
has dimension $2$ over $k$. The ratio of any two linearly
independent sections gives us the new elliptic parameter $u$.
To compute the space of global sections, we proceed as follows.
Any global section gives a section of the generic fiber $E = \pi^{-1}(\eta)$,
which is an elliptic curve over $k(t) = k(\eta)$.
Therefore, if we have a basis of global sections of $D = F'_{\hor, \eta}$
over $k(t)$, say $\{s_1, \dots, s_r\}$
(where $r = h^0(\sO_E(D)) = \deg(D) = F'_{\hor} \cdot F = F' \cdot F$),
we can assume that any global section of $\sO_X(F')$ is of the form
$$
b_1(t) s_1 + \dots + b_r(t) s_r.
$$
We can now use the information from $F_{\ver}$, which gives us
conditions about the zeroes and poles of the functions $b_i$, to find
the linear conditions cutting out $H^0(X, \sO_X(F'))$, which will be
two-dimensional.

When $F \cdot F' = r$, we say that going between these elliptic
fibrations is an $r$-neighbor step. We will explain the reason for this
terminology shortly. In this paper we use only $2$- and $3$-neighbor steps.
First, we describe how to convert to a genus~$1$
curve in the case when $E = F_{\hor, \eta} = 2O$ or $3O$. These are
the most familiar cases of the $2$- and $3$-neighbor steps, and the
other cases of the two-neighbor step that are needed are more
exhaustively described in \cite{Kum2}.

Suppose $D = 2O$. Then $\{1,x\}$ is a basis of the global sections of
$\sO_E(D)$.  Therefore, on~$X$\/ we obtain two global sections $1$ and
$c(t) + d(t) x$ for some $c(t), d(t) \in k(t)$. The ratio of these two
sections gives the elliptic parameter $u$. We set $x = (u -
c(t))/d(t)$, and substitute into the Weierstrass equation to obtain an
equation
$$
y^2 = g(t,u).
$$
Because $F'$ is an elliptic divisor, the generic fiber of this surface
over $\Proj^1_u$ is a curve of genus $1$.  Thus, once we absorb
square factors into $y^2$ we obtain an equivalent $g$ that is a
polynomial of degree $3$ or~$4$ in~$t$.
Then $y^2 = g(t,u)$ is standard form of the equation of a genus-$1$ curve
as a branched double cover of $\Proj^1$.

Suppose $D = 3O$. Then $\{1,x,y\}$ is a basis of global sections of
$\sO_E(D)$. On~$X$\/ we obtain two global sections $1$ and $c(t) +
d(t) x + e(t) y$.  We set the ratio equal to $u$, solve for $y$, and
substitute into the Weierstrass equation to obtain
$$
\Big( u - c(t) + d(t) x \Big)^2 = e(t)^2 \Big(x^3 + a_2(t) x^2 +
a_4(t) x + a_6(t) \Big).
$$
This equation is of degree $3$ in $x$, and after some simple algebra
(scaling and shifting $x$), we may arrange it to have degree $3$ in
$t$ as well. We end up with a plane cubic curve, which is the standard
model of a genus-$1$ curve with a degree $3$ line bundle.

Finally, if we also know that the elliptic fibration corresponding to
$F'$ has a section (which follows in each of our examples by
Proposition \ref{automaticsections}), then it is isomorphic to the
Jacobian of the genus-$1$ curve we have computed.  We may use standard
formulas to write down the Weierstrass equation of the Jacobian (see
\cite{AKMMMP} and the references cited therein), and this is the
desired Weierstrass equation for the elliptic fibration with fiber
$F'$.

This neighbor step from $F$\/ to~$F'$ corresponds to
computing an explicit isomorphism between two presentations
$\Z O + \Z F + T$ and $\Z O' + \Z F' + T'$ of $\NS(X)$.
Note that $T \cong
F^{\perp}/\Z F$, where $^{\perp}$ refers to the orthogonal complement
in $\NS(X)$. The sublattice $(\Z F + \Z F')^{\perp}$ projects to an
index $r$ sublattice of both $F^{\perp}/\Z F$\/ and $F'^{\perp}/\Z
F'$, where $r = F \cdot F'$. Therefore these lattices are $r$-neighbors.

\section{Discriminant $5$}

\subsection{Parametrization}

This is the smallest fundamental discriminant for real multiplication,
and it is small enough that we do not need any $2$- or $3$-neighbor
steps: we can instead just start with a K3 surface with $E_8$ and
$E_7$ fibers, and ask for the extra condition which allows a section
of height $5/2 = 4 - 3/2$. 

\begin{proposition}
The moduli space of K3 surfaces lattice polarized by $L_5$, the unique
even lattice of signature $(1,17)$ and discriminant $5$ containing $U
\oplus E_8(-1) \oplus E_7(-1)$, is birational to the projective plane
$\Proj^2_{g,h}$. The family of K3 surfaces is given by the following
equation
$$
y^2 = x^3 + \frac{1}{4} t^3 \, \big( -3g^2 \, t + 4 \big) \, x -
\frac{1}{4} t^5 \, \big( 4 h^2 t^2 + (4h + g^3)\, t + ( 4g + 1) \big).
$$
\end{proposition}

\begin{proof}
We start with a family of elliptic K3 surfaces with $E_8$ and $E_7$ fibers:
$$
y^2 = x^3 + x t^3 (a_0 + a_1 t) + t^5 (b_0 + b_1 t + b_2 t^2).
$$ 

To have discriminant $-5$ for the Picard group, the elliptic K3
surfaces in this family must have a section of height $5/2 = 4 - 3/2$.
The $x$-coordinate for such a section must have the form $t^2(x_0 +
x_1t + h^2 t^2)$. Substituting $x$ into the right-hand side, and
completing the square, we may solve for $b_0, b_1, b_2$ and $a_1$ in
terms of $a_0, h, x_0$ and $x_1$. We then set $x_1 = eh$ and $x_0 = g
+ e^2/4$ to simplify the expressions. Finally, we note that scaling
$x,y,t$ by $\lambda^2, \lambda^3, \mu/\lambda$ gives $a_0, e, g, h$
weights $(1,3), (0,1), (0,2)$ and $(-1,2)$ respectively with respect
to $(\lambda, \mu)$. Therefore, we may scale $a_0$ and $e$ to equal
$1$ independently (at most removing hypersurfaces in the moduli
space), and get the parameterization in the statement of the
proposition. We note that the section $P$ of height $5/2$ is given by
$$
x(P) = \frac{1}{4}t^2 \big( (1 + 2ht)^2 + 4g \big), \qquad y(P) =
\frac{1}{8} t^3 (1 + 2ht) \big( (1 + 2ht)^2 + 6g \big).
$$
\end{proof}

\begin{corollary}
The Humbert surface $\sH_5$ is birational to $\Proj^2_{g,h}$, with the
map to $\A_2$ given by the Igusa-Clebsch invariants 
$$
(I_2: I_4: I_6: I_{10}) = \big( 6(4g+1),9g^2,9(4h+9g^3+2g^2),4h^2 \big).
$$
\end{corollary}

\begin{proof}
  This follows immediately from Theorem \ref{humbert}. The
  Igusa-Clebsch invariants may be read out directly from the
  Weierstrass equation above.
\end{proof}

\begin{theorem} 
A birational model over $\Q$ of the Hilbert modular surface $Y_{-}(5)$
is given by the following double cover of $\Proj^2_{g,h}$:
$$
z^2 = 2 \,(6250 \, h^2-4500 \, g^2 \, h-1350 \, g \, h-108
\,h-972\,g^5-324\,g^4-27\,g^3).
$$
It is a rational surface (i.e. birational to $\Proj^2$). 
\end{theorem}

\begin{proof}
We follow the method of section \ref{method}. The possible factors of
the branch locus are $g$, $h$, $8h - 9g^2$ (the zero locus of this
polynomial defines a subvariety of the moduli space for which the
corresponding elliptic K3 surfaces acquire an extra $\I_2$ fiber),
$64h^2+48g^2h+48g^5+9g^4$ (extra $\II$ fiber), and
$6250h^2-4500g^2h-1350gh -108h-972g^5-324g^4-27g^3$ (extra $\I_2$
fiber). By Step 4 of the method, we deduce that only the last factor
occurs, and the correct quadratic twist is by $2$. It was already
well-known that the Hilbert modular surface is a rational surface, but
we give an explicit parametrization in the following analysis.
\end{proof}

\subsection{Analysis}

This is a rational surface: to obtain a parametrization, we
complete the square in $h$, writing $h = k + 9\,(250\,g^2+75
\,g+6)/6250$. Following this up with $k = 3 \,m\,(10g + 3)(15g +
2)/6250$, and removing square factors by writing $z = 3 n (10 g+3) (15
g+2)/25$, we obtain the equation $5n^2-m^2+9 + 30g =0$, which we can
solve for $g$. This gives an explicit birational map between
$Y_{-}(5)$ and $\Proj^2_{m,n}$.

The branch locus is the curve obtained by setting $z = 0$, or
alternatively $n = 0$ in the above parametrization. It is parametrized
by one variable $m$; we have
$$
(g,h) =  \big( (m^2 - 9)/30, (m-2)^2 (m+3)^3/12500 \big).
$$

\subsection{Examples}

We list some points of small height (at most $100$) and
corresponding genus $2$ curves.

\begin{tabular}{l|c}
Rational point $(g,h)$ & Sextic polynomial $f_6(x)$ defining the genus $2$ curve $y^2 = f_6(x)$. \\
\hline\hline \\ [-2.5ex]
$(-8/3, 47/2)$ & $ -x^5 + x^4 - x^3 - x^2 + 2x - 1 $ \\
$(-37/6, 67)$ & $ x^6 + 2x^5 + x^4 - 2x^3 + 2x^2 - 4x + 1 $ \\
$(0, 27/50)$ & $ 3x^5 + 5x^3 + 1 $ \\
$(-1/24, 1/100)$ & $ 4x^6 + 4x^5 + 5x^4 - 5x^2 - 2x - 2 $ \\
$(-25/54, 59/81)$ & $ -x^6 - 2x^4 - 6x^3 - 5x^2 - 6x - 1 $ \\
$(-2/3, -14/25)$ & $ -7x^6 - 7x^5 - 5x^4 + 5x^2 - x - 1 $ \\
$(47/54, 71/81)$ & $ 3x^6 - 6x^5 + 7x^4 - 2x^3 - 2x^2 - 1 $ \\
$(-1/6, 1/25)$ & $ x^6 - 4x^5 + 10x^4 - 10x^3 + 5x^2 + 2x - 3 $ \\
$(-4/3, 16/25)$ & $ -x^6 - x^5 + 5x^2 - 7x - 12 $ \\
$(-4/3, 49/22)$ & $ 7x^5 + 5x^4 + 3x^3 - 9x^2 - 14x - 7 $ \\
$(0, -54/11)$ & $ 2x^6 - 6x^5 - 6x^4 + 3x^3 - 18x^2 - 6x - 2 $ \\
$(11/6, 53/11)$ & $ 9x^6 - 14x^5 + 13x^4 - 2x^3 - 22x^2 - 8x - 7 $ \\
$(-5/24, 1/64)$ & $ 6x^6 - 2x^5 - 15x^4 - 16x^3 - 25x^2 - 8x - 4 $ \\
$(11/6, -89/25)$ & $ -x^6 + 2x^5 - 5x^4 + 30x^3 - 10x^2 + 8x - 1 $ \\
$(-2/3, 68/11)$ & $ -2x^6 - 2x^5 - 11x^4 - 29x^3 - 31x^2 - 26x - 6 $ \\
$(-2/3, 26/25)$ & $ -2x^6 + 36x^5 + 5x^4 + 35x^3 - 10x^2 - 21x - 17 $ 
\end{tabular}

The second entry in the list is the modular curve $X_0(67)/\langle w \rangle$,
where $w$ is the Atkin-Lehner involution.  We find several other modular
curves with real multiplication by $\sO_5$ (and also a few for
discriminants other than~$5$) corresponding to points of larger height.

\subsection{Brumer's and Wilson's families of genus $2$ curves}

Genus $2$ curves whose Jacobians have real multiplication by $\sO_5$
have been studied earlier, by Brumer \cite{Br} and Wilson \cite{Wi1,
  Wi2}. Brumer describes a three-dimensional family of genus $2$
curves, given by
$$
y^2 + \big(1+x+x^3+c(x+x^2)\big)y = -bdx^4 + (b-d-2bd)x^3 + (1-3b-bd)x^2 + (1+3b)x + b.
$$ 
In the auxiliary files, we give formulas for the corresponding
values of our parameters $g$ and $h$.

Wilson describes a family of genus two curves by their
Igusa-Clebsch invariants. His moduli space is two-dimensional, though
he uses three coordinates $z_6, s_2$ and $\sigma_5$ of weights $1,2,5$
respectively, in a weighted projective space. These coordinates are
related to ours via
$$
(g,h) = \left( -\frac{2z_6^2+s_2}{12z_6^2} , \frac{\sigma_5}{64z_6^5} \right).
$$

\section{Discriminant $8$}

\subsection{Parametrization}

We start with an elliptic K3 surface with fibers of type $D_9$ and
$E_7$. A Weierstrass equation for such a family is given by
$$
y^2 = x^3 + t \, \big( (2r+1)\,t+r \big)\,x^2 + 2rst^4 (t+1) \,x + rs^2t^7.
$$

We then identify a fiber of type $E_8$, and transform to it by a
$2$-neighbor step.

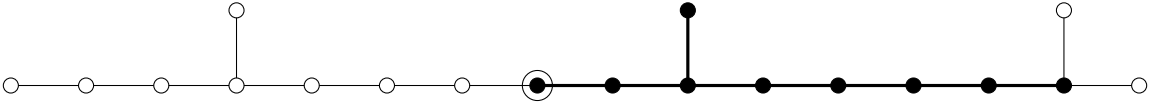
\begin{center}
\begin{tikzpicture}

\draw (0,0)--(15,0);
\draw (3,0)--(3,1);
\draw (9,0)--(9,1);
\draw (14,0)--(14,1);
\draw [very thick] (7,0)--(14,0);
\draw [very thick] (9,0)--(9,1);

\fill [white] (0,0) circle (0.1);
\fill [white] (1,0) circle (0.1);
\fill [white] (2,0) circle (0.1);
\fill [white] (3,0) circle (0.1);
\fill [white] (3,1) circle (0.1);
\fill [white] (4,0) circle (0.1);
\fill [white] (5,0) circle (0.1);
\fill [white] (6,0) circle (0.1);
\fill [black] (7,0) circle (0.1);
\fill [black] (8,0) circle (0.1);
\fill [black] (9,0) circle (0.1);
\fill [black] (10,0) circle (0.1);
\fill [black] (11,0) circle (0.1);
\fill [black] (12,0) circle (0.1);
\fill [black] (13,0) circle (0.1);
\fill [black] (9,1) circle (0.1);
\fill [white] (14,1) circle (0.1);
\fill [black] (14,0) circle (0.1);
\fill [white] (15,0) circle (0.1);

\draw (7,0) circle (0.2);
\draw (0,0) circle (0.1);
\draw (1,0) circle (0.1);
\draw (2,0) circle (0.1);
\draw (3,0) circle (0.1);
\draw (3,1) circle (0.1);
\draw (4,0) circle (0.1);
\draw (5,0) circle (0.1);
\draw (6,0) circle (0.1);
\draw (7,0) circle (0.1);
\draw (8,0) circle (0.1);
\draw (9,0) circle (0.1);
\draw (10,0) circle (0.1);
\draw (11,0) circle (0.1);
\draw (12,0) circle (0.1);
\draw (13,0) circle (0.1);
\draw (9,1) circle (0.1);
\draw (14,1) circle (0.1);
\draw (14,0) circle (0.1);
\draw (15,0) circle (0.1);

\end{tikzpicture}
\end{center}

The resulting elliptic fibration has $E_8$ and $E_7$ fibers, and we
may read out the Igusa-Clebsch invariants, and then compute the branch
locus of the double cover that defines the Hilbert modular surface.
It corresponds to elliptic K3 surfaces with an extra $\I_2$ fiber.
We obtain the following result for $Y_{-}(8)$.
\begin{theorem}
The Humbert surface $\sH_8$ is birational to $\Proj^2_{r,s}$, with the
explicit map to $\A_2$ given by the Igusa-Clebsch invariants
\begin{align*}
I_2 &= -4(3s+8r-2) & \qquad I_4 &= 4(9rs+4r^2+4r+1) \\
I_6 &= -4(36rs^2+94r^2s-35rs+4s+48r^3+40r^2+4r-2) & \qquad I_{10} &= -8s^2r^3.
\end{align*}
A birational model over $\Q$ for the Hilbert modular surface
$Y_{-}(8)$ as a double cover of $\Proj^2$ is given by the following
equation:
$$
z^2 = 2 \,(16rs^2+32 r^2s-40rs-s+16r^3+24r^2+12r+2).
$$
It is also a rational surface.
\end{theorem}

\subsection{Analysis}

This Hilbert modular surface is a rational surface. To see this, we
may complete the square in $s$, setting $s = s_1 - r + 5/4 + 1/(32r)$.
Then setting $s_1 = m(16r-1)/(32r)$ and $z = n(16r-1)$, we remove
square factors from the equation, which becomes linear in~$r$.  We
find $r = (m^2-1)/(32n^2-16)$, thus obtaining an explicit
parametrization of $Y_{-}(8)$ by $\Proj^2_{m,n}$.

The branch locus is a genus $0$ curve; we obtain a parametrization by
setting $z = 0$:
$$
(r,s) = \big( (1-t^2)/16, (t+3)^3/(16(t+1)) \big).
$$

\subsection{Examples}

We list some points of small height and corresponding genus $2$ curves.

\begin{tabular}{l|c}
Rational point $(r,s)$ & Sextic polynomial $f_6(x)$ defining the genus $2$ curve $y^2 = f_6(x)$. \\
\hline \hline  \\ [-2.5ex]
$(2/25, 83/50)$ & $ x^6 - x^5 + 3x^4 + x^3 + x^2 + 2x + 1 $ \\
$(-34/9, 50/9)$ & $ x^6 - 2x^5 - 2x^4 - 4x^3 + 2x^2 + 4x - 3 $ \\
$(-22, 59/2)$ & $ x^6 + 2x^5 - 3x^4 + 5x^3 + x^2 - x - 1 $ \\
$(2/49, 58/49)$ & $ -x^6 - 4x^5 - 6x^4 + 2x^2 + 2x - 1 $ \\
$(-52, 83/2)$ & $ -4x^6 + 2x^5 + 5x^4 - 7x^3 - x + 1 $ \\
$(1/8, 59/32)$ & $ -x^6 - x^5 - 7x^2 + 5x - 4 $ \\
$(-4/9, 23/8)$ & $ -2x^6 + 6x^5 - x^4 - 7x^3 - x^2 - 3x $ \\
$(80, 83/2)$ & $ x^6 + 5x^5 + 8x^4 + 5x^3 + 5x^2 - 8x + 4 $ \\
$(19/2, 22)$ & $ -x^6 - 4x^5 - 8x^4 - 8x^3 - 8x^2 + 4x - 4 $ \\
$(94, -54)$ & $ -x^6 + 6x^4 - 8x^3 + 6x^2 + 6x - 9 $ \\
$(13/8, -2/9)$ & $ -x^6 + 6x^5 - 9x^4 + 3x^2 - 6x - 2 $ \\
$(86/9, -13/2)$ & $ x^6 - 7x^4 - 7x^3 - x^2 + 9x + 9 $ \\
$(-3/8, 9/4)$ & $ -3x^6 - 9x^5 - 2x^4 + 10x^3 + x^2 + 3x $ \\
$(1/14, 10/7)$ & $ -x^6 - 2x^5 - 7x^4 + 4x^3 - 3x^2 + 10x - 5 $ \\
$(1/18, 31/18)$ & $ -3x^6 - 10x^5 - 7x^4 - 5x^2 + 2x - 1 $ \\
$(1/32, 19/16)$ & $ 4x^6 - 7x^5 - 3x^4 - 2x^3 + 10x^2 + 5x + 1 $
\end{tabular}

\section{Discriminant $12$}

\subsection{Parametrization}

We start with an elliptic K3 surface with fibers of types $E_8$, $D_6$
and $A_2$. The Weierstrass equation for such a family is given by

$$
y^2 = x^3 + \big((1-f^2)\,(1-t) + t \big)\,t\,x^2 +
2\,e\,t^3\,(t-1)\,x + e^2\,(t-1)^2\,t^5.
$$

We identify a fiber of type $E_7$, and move to the associated elliptic
fibration by a $2$-neighbor step.

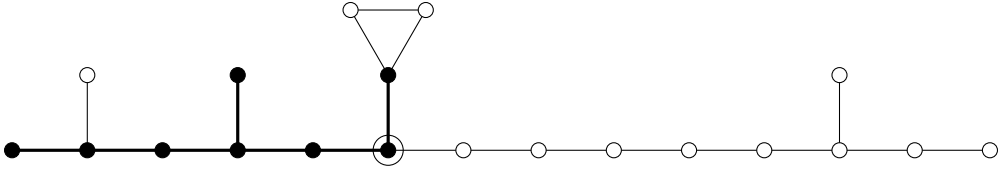
\begin{center}
\begin{tikzpicture}

\draw (0,0)--(13,0);
\draw (1,0)--(1,1);
\draw (3,0)--(3,1);
\draw (11,0)--(11,1);
\draw (5,0)--(5,1)--(4.5,1.866)--(5.5,1.866)--(5,1);
\draw [very thick] (0,0)--(5,0)--(5,1);
\draw [very thick] (3,0)--(3,1);

\draw (5,0) circle (0.2);
\fill [black] (0,0) circle (0.1);
\fill [black] (1,0) circle (0.1);
\fill [black] (2,0) circle (0.1);
\fill [black] (3,0) circle (0.1);
\fill [black] (4,0) circle (0.1);
\fill [black] (5,0) circle (0.1);
\fill [white] (6,0) circle (0.1);
\fill [white] (1,1) circle (0.1);
\fill [black] (3,1) circle (0.1);
\fill [white] (7,0) circle (0.1);
\fill [white] (8,0) circle (0.1);
\fill [white] (9,0) circle (0.1);
\fill [white] (10,0) circle (0.1);
\fill [white] (11,0) circle (0.1);
\fill [white] (12,0) circle (0.1);
\fill [white] (13,0) circle (0.1);
\fill [white] (11,1) circle (0.1);
\fill [black] (5,1) circle (0.1);
\fill [white] (4.5,1.866) circle (0.1);
\fill [white] (5.5,1.866) circle (0.1);

\draw (0,0) circle (0.1);
\draw (1,0) circle (0.1);
\draw (2,0) circle (0.1);
\draw (3,0) circle (0.1);
\draw (4,0) circle (0.1);
\draw (5,0) circle (0.1);
\draw (6,0) circle (0.1);
\draw (1,1) circle (0.1);
\draw (3,1) circle (0.1);
\draw (7,0) circle (0.1);
\draw (8,0) circle (0.1);
\draw (9,0) circle (0.1);
\draw (10,0) circle (0.1);
\draw (11,0) circle (0.1);
\draw (12,0) circle (0.1);
\draw (13,0) circle (0.1);
\draw (11,1) circle (0.1);
\draw (5,1) circle (0.1);
\draw (4.5,1.866) circle (0.1);
\draw (5.5,1.866) circle (0.1);

\end{tikzpicture}
\end{center}

The new elliptic fibration has $E_8$ and $E_7$ fibers, and so we may
read out the Igusa-Clebsch invariants, and determine the branch locus
for the Hilbert modular surface. 

\begin{theorem}
A birational model over $\Q$ for the Hilbert modular surface
$Y_{-}(12)$ as a double cover of $\Proj^2_{e,f}$ is given by the
equation
$$
z^2 = (f-1)\,(f+1)\,(f^6-f^4-18\,e\,f^2+27\,e^2+16\,e).
$$
It is a rational surface.
\end{theorem}

\subsection{Analysis}

Note the extra involution $(e,f) \mapsto (e,-f)$ arising from the
Hurwitz-Maass extension (as described at the end of Section
\ref{realmult}), since $12$ is not a prime power. In fact, there are
two independent involutions evident in the diagram above, but one of
them has been used up to fix the Weierstrass scaling of $x$ (namely,
the coefficient of $x^2$ evaluates to $1$ at $t = 1$). The other
involution is reflected in the Weierstrass equation for the universal
family of elliptic K3 surfaces as $f \mapsto -f$. It preserves the
branch locus of the map $Y_{-}(12) \to \sH_{12}$, and therefore lifts
to $Y_{-}(12)$.

The branch locus has three components; the two simple components $f =
\pm 1$ correspond to the $D_6$ fiber getting promoted to an $E_7$
fiber, while the remaining component corresponds to having an extra
$\I_2$ fiber.  This last component is a rational curve; completing the
square with respect to~$e$, we find after some easy algebraic
manipulation the parametrization
$$
(e,f) = \big( 16 (h^2-1)/(h^2+3)^3, -4h/(h^2+3) \big).
$$

This Hilbert modular surface is rational as well. To obtain an
explicit parametrization, note that the right hand side of the above
equation is quadratic in $e$, and $e = 0$ makes it a square. Therefore
the conic bundle over $\Proj^1_f$ has a section. Setting $z = ge + f^2
(f^2 - 1)$, we may solve for $e$, obtaining a birational
parametrization by $\Proj^2_{f,g}$.

We will not list the Igusa-Clebsch invariants for this (and higher)
discriminants, as they are complicated expressions. They are available
in the auxiliary computer files.

\subsection{Examples}
We list some points of small height and corresponding genus $2$ curves.

\begin{tabular}{l|c}
Rational point $(e,f)$ & Sextic polynomial $f_6(x)$ defining the genus $2$ curve $y^2 = f_6(x)$. \\
\hline\hline \\ [-2.5ex]
$(34/27, 5/3)$ & $ -2x^6 - 2x^5 + x^4 - 3x^2 + 2x - 1 $ \\
$(34/27, -5/3)$ & $ x^5 - x^4 + x^3 - 3x^2 - x + 5 $ \\
$(51/100, 2)$ & $ -3x^6 + 6x^5 + 4x^4 - 2x^3 - 8x^2 - 6x - 6 $ \\
$(-11/3, -2)$ & $ -x^6 - 8x^3 + 12x - 12 $ \\
$(-11/3, 2)$ & $ -x^6 + 12x^4 + 8x^3 - 12x^2 + 12x - 4 $ \\
$(4/3, -2)$ & $ -8x^6 + 12x^4 + 8x^3 - 6x^2 - 12x - 3 $ \\
$(4/3, 2)$ & $ -5x^6 + 12x^5 - 6x^4 + 8x^3 - 12x^2 - 8 $ \\
$(-14/27, 1/3)$ & $ -x^5 + 13x^4 - 6x^3 - 2x^2 + 3x - 7 $ \\
$(-29/14, -15/14)$ & $ x^6 + 2x^5 + 13x^4 - 16x^3 + 17x^2 + 4x + 8 $ \\
$(80/81, 2)$ & $ -8x^6 - 12x^5 + 15x^4 + 5x^3 - 21x^2 + 15x - 5 $ \\
$(51/100, -2)$ & $ -x^6 - 6x^5 - 11x^4 - 14x^3 - 23x^2 + 6x + 3 $ \\
$(-5/2, -3/2)$ & $ -5x^6 - 10x^5 - x^4 + 24x^3 - 5x^2 - 8x - 20 $ \\
$(-23/54, -1/2)$ & $ x^6 - 6x^5 - 3x^4 + 24x^3 - 3x^2 - 4 $ \\
$(25/18, -3/2)$ & $ 4x^6 - 24x^5 + 27x^4 - 28x^3 + 21x^2 - 5 $ \\
$(-5/54, 2/3)$ & $ 3x^6 - 26x^5 + 31x^4 + 12x^3 - 3x^2 - 10x - 15 $ \\
$(13/64, -5/4)$ & $ -5x^6 + 3x^5 - 12x^4 + 28x^3 + 12x^2 - 36x $ 
\end{tabular}

We mention a few special curves on the Hilbert modular surface, which
one may recognize by looking at a plot of the rational points. First,
specializing $f$ to some $f_0 \in \Q$ gives a conic in $z$ and $e$,
which always has a rational point (since, as we noted above, $e =0$
makes the right hand side a square). These may be used to produce many
rational points and examples of genus-$2$ curves when the Brauer
obstruction happens to vanish.

Another noticeable feature of the plot is the parabola $e = 4(f^2-1)/9$.
We check that the Brauer obstruction vanishes identically
on this locus, giving a family of genus-$2$ curves whose Jacobians
have real multiplication by $\sO_{12}$. Note that this is a
non-modular curve: generically the endomorphism ring is no larger
(i.e.\ not a quaternion algebra).

\section{Discriminant $13$}

\subsection{Parametrization}

We start with an elliptic K3 surface with reducible fibers of types
$E_8$, $E_6$ and $A_1$, and with a section of height $13/6 = 4 - 4/3 -
1/2$. A Weierstrass equation for such a family is given by
$$
y^2 = x^3 + (4\,g+1)\,t^2\,x^2 - 4\,g\,(h-g-1)\,(t-1)\,t^3\,x +
4\,g^2\,(t-1)^2\,t^4\,(h^2\,t+1).
$$

We identify an $E_7$ fiber below and move to it by a $2$-neighbor
step.

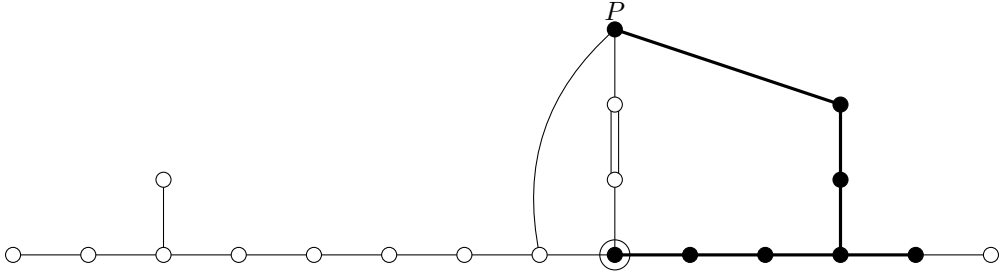
\begin{center}
\begin{tikzpicture}

\draw [very thick] (8,0)--(11,0)--(11,2);
\draw [very thick] (11,2)--(8,3);
\draw [very thick] (11,0)--(12,0);

\draw (0,0)--(13,0);
\draw (2,0)--(2,1);
\draw (11,0)--(11,2);
\draw (8,2)--(8,3)--(11,2);
\draw (8.05,1)--(8.05,2);
\draw (7.95,1)--(7.95,2);
\draw (8,0)--(8,1);
\draw [bend right] (8,3) to (7,0);

\draw (8,0) circle (0.2);
\fill [white] (0,0) circle (0.1);
\fill [white] (1,0) circle (0.1);
\fill [white] (2,0) circle (0.1);
\fill [white] (3,0) circle (0.1);
\fill [white] (4,0) circle (0.1);
\fill [white] (5,0) circle (0.1);
\fill [white] (6,0) circle (0.1);
\fill [white] (7,0) circle (0.1);
\fill [black] (8,0) circle (0.1);
\fill [black] (9,0) circle (0.1);
\fill [black] (10,0) circle (0.1);
\fill [black] (11,0) circle (0.1);
\fill [black] (12,0) circle (0.1);
\fill [white] (13,0) circle (0.1);
\fill [black] (11,1) circle (0.1);
\fill [black] (11,2) circle (0.1);
\fill [white] (2,1) circle (0.1);
\fill [black] (8,3) circle (0.1);
\fill [white] (8,1) circle (0.1);
\fill [white] (8,2) circle (0.1);

\draw (0,0) circle (0.1);
\draw (1,0) circle (0.1);
\draw (2,0) circle (0.1);
\draw (3,0) circle (0.1);
\draw (4,0) circle (0.1);
\draw (5,0) circle (0.1);
\draw (6,0) circle (0.1);
\draw (7,0) circle (0.1);
\draw (8,0) circle (0.1);
\draw (9,0) circle (0.1);
\draw (10,0) circle (0.1);
\draw (11,0) circle (0.1);
\draw (12,0) circle (0.1);
\draw (13,0) circle (0.1);
\draw (11,1) circle (0.1);
\draw (11,2) circle (0.1);
\draw (2,1) circle (0.1);
\draw (8,1) circle (0.1);
\draw (8,2) circle (0.1);
\draw (8,3) circle (0.1);

\draw (8,3) [above] node{$P$};

\end{tikzpicture}
\end{center}

The resulting elliptic fibration has $E_8$ and $E_7$ fibers, so we may
read out the Igusa-Clebsch invariants and work out the branch locus of
the double cover defining the Hilbert modular surface. It corresponds
to the elliptic K3 surface having an extra $\I_2$ fiber. 

\begin{theorem}
A birational model over $\Q$ for the Hilbert modular surface
$Y_{-}(13)$ as a double cover of $\Proj^2_{g,h}$ is given by the
equation
$$
z^2 = 108\,g\,h^3 -(27\,g^2+468\,g-4)\,h^2 +
8\,(82\,g^2+71\,g-2)\,h -16\,(2\,g-1)^3.
$$
It is a rational surface.
\end{theorem}

\subsection{Analysis}

The surface $Y_{-}(13)$ is rational.  To show this, we proceed as follows.
The substitutions $g = g_1 - 1/54$ and $h = h_1 + 64/27$, followed by
$h_1 = mg_1$, make the right hand side quadratic in $g_1$, up to
a square factor.  Removing this factor by setting $z = z_1 g_1/18$,
and then setting $g_1 = g_2/54$, we get the following conic bundle
over~$\Q(m)$:
$$
z_1^2 = 3\,m^2\,(4\,m-1)\,g_2^2
-6\,(2\,m^3-301\,m^2-528\,m+128)\,g_2 -507\,(m^2-96\,m-1024).
$$
Furthermore, $g_2 = 1$ makes the resulting expression a square,
giving us a rational point on the conic over $\Q(m)$. Therefore, we
may set $z_1 = 36(m+20) + n(g_2 - 1)$ and solve for $g_2$. This gives
us a birational parametrization of $Y_{-}(13)$ by $\Proj^2_{m,n}$.

We can also deduce that the branch locus is a rational curve, as
follows. By setting $z_1 = 0$, we obtain a quadratic equation in $g_2$
whose discriminant, up to a square factor, is $m^2 - 44m + 16$. Setting it
equal to $n^2$ and noting that $m = 0$ makes the expression a square,
we obtain a parametrization of this conic, as $m = -4(2r+11)/(r^2-1)$.
Working backwards, we obtain a parametrization of the
branch locus as
$$
(g,h) = \left(\frac{-2(r-2)^2(r+1)}{27(r+7)},  \frac{2(2r+5)^3}{27(r+1)(r+7)}\right).
$$

\subsection{Examples}
We list some points of small height and corresponding genus $2$ curves.

\begin{tabular}{l|c}
Rational point $(g,h)$ & Sextic polynomial $f_6(x)$ defining the genus $2$ curve $y^2 = f_6(x)$. \\
\hline \hline \\ [-2.5ex]
$(-17, 1/3)$ & $ -3x^6 - 6x^5 + 4x^3 + 3x^2 + 6x - 5 $ \\
$(-13/2, -11/2)$ & $ x^5 + 5x^4 + 5x^3 - 5x^2 + 6x - 1 $ \\
$(-17/2, 7/2)$ & $ -x^5 - 2x^4 - 3x^3 - 6x^2 - 7 $ \\
$(11/5, 1/5)$ & $ x^6 + 4x^5 + 2x^4 - 8x^3 - 5x^2 - 5 $ \\
$(-1/3, 11/9)$ & $ -x^6 + 3x^4 + 12x^3 + 6x^2 - 11 $ \\
$(-14/11, -2/11)$ & $ 13x^6 + 12x^5 + 6x^4 + 10x^3 - 7x^2 - 2x + 1 $ \\
$(-2/17, 2/17)$ & $ -x^6 - 2x^5 - x^4 + 14x^3 + 2x^2 - 8x - 9 $ \\
$(1/5, 17/15)$ & $ 3x^6 + 12x^5 + 6x^4 - 4x^3 - 15x^2 + 5 $ \\
$(-10/3, -10/9)$ & $ -x^6 + 18x^5 + 3x^4 - 6x^3 + 6x^2 - 5 $ \\
$(-10, 6)$ & $ -x^6 - 6x^4 - 10x^3 - 9x^2 - 30x + 11 $ \\
$(-18/5, 14/5)$ & $ -x^6 - 2x^4 - 10x^3 + 7x^2 - 30x - 9 $ \\
$(-7/13, 3/13)$ & $ -31x^6 + 12x^5 - 30x^4 + 4x^3 - 33x^2 - 12x - 1 $ \\
$(10/3, 10/3)$ & $ 7x^6 + 18x^5 - 9x^4 - 34x^3 + 18x^2 - 5 $ \\
$(-11, -11/9)$ & $ -5x^6 + 6x^5 + 3x^4 - 4x^3 + 18x^2 - 36x - 9 $ \\
$(-7/8, 10/3)$ & $ -x^6 + 3x^4 - 20x^3 + 30x^2 - 36x + 9 $ \\
$(-13/9, 13/9)$ & $ -x^6 - 9x^4 + 8x^3 - 30x^2 + 36x - 3 $ 
\end{tabular}

On a plot of rational points we observe the line $h = (g+4)/3$,
along which the Brauer obstruction vanishes, leading to a family of
genus-$2$ curves whose Jacobians have ``honest'' real multiplication
by $\sO_{13}$.

\section{Discriminant $17$}

\subsection{Parametrization}

We start with an elliptic K3 surface with a $\I_{17}$ fiber.
A Weierstrass equation for such a family of surfaces is given by
\begin{align*}
y^2 &= x^3+ \Big(1+2\,g\,t+\big(2\,h+(g+1)^2
\big)\,t^2+2\,(g\,h+g+2\,g^2+h)\,t^3+ \big((g+h)^2+2\,g^3 \big)\,t^4
\Big)\,x^2 \\
& \qquad - 4\,h^2\,t^5\,\big(1 + g\,t + (h +
2\,g+1)\,t^2 + (h+2\,g^2+g)\,t^3 \big)\,x \\
& \qquad + 4\,h^4\,t^{10}\,\big( (2\,g+1)\,t^2 + 1\big).
\end{align*}

We first identify an $E_7$ fiber and go to the associated elliptic
fibration via a $2$-neighbor step.

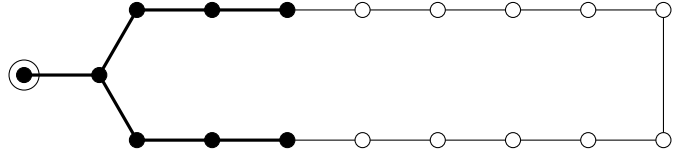
\begin{center}
\begin{tikzpicture}

\draw (0,0)--(1,0)--(1.5,0.866)--(8.5,0.866)--(8.5,-0.866)--(1.5,-0.866)--(1,0);
\draw [very thick] (3.5,0.866)--(1.5,0.866)--(1,0)--(1.5,-0.866)--(3.5,-0.866);
\draw [very thick] (0,0)--(1,0);

\draw (0,0) circle (0.2);
\fill [black] (0,0) circle (0.1);
\fill [black] (1,0) circle (0.1);
\fill [black] (1.5,0.866) circle (0.1);
\fill [black] (1.5,-0.866) circle (0.1);
\fill [black] (2.5,0.866) circle (0.1);
\fill [black] (2.5,-0.866) circle (0.1);
\fill [black] (3.5,0.866) circle (0.1);
\fill [black] (3.5,-0.866) circle (0.1);
\fill [white] (4.5,0.866) circle (0.1);
\fill [white] (4.5,-0.866) circle (0.1);
\fill [white] (5.5,0.866) circle (0.1);
\fill [white] (5.5,-0.866) circle (0.1);
\fill [white] (6.5,0.866) circle (0.1);
\fill [white] (6.5,-0.866) circle (0.1);
\fill [white] (7.5,0.866) circle (0.1);
\fill [white] (7.5,-0.866) circle (0.1);
\fill [white] (8.5,0.866) circle (0.1);
\fill [white] (8.5,-0.866) circle (0.1);

\draw (0,0) circle (0.1);
\draw (1,0) circle (0.1);
\draw (1.5,0.866) circle (0.1);
\draw (1.5,-0.866) circle (0.1);
\draw (2.5,0.866) circle (0.1);
\draw (2.5,-0.866) circle (0.1);
\draw (3.5,0.866) circle (0.1);
\draw (3.5,-0.866) circle (0.1);
\draw (4.5,0.866) circle (0.1);
\draw (4.5,-0.866) circle (0.1);
\draw (5.5,0.866) circle (0.1);
\draw (5.5,-0.866) circle (0.1);
\draw (6.5,0.866) circle (0.1);
\draw (6.5,-0.866) circle (0.1);
\draw (7.5,0.866) circle (0.1);
\draw (7.5,-0.866) circle (0.1);
\draw (8.5,0.866) circle (0.1);
\draw (8.5,-0.866) circle (0.1);

\end{tikzpicture}
\end{center}

The resulting elliptic fibration has $E_7$ and $A_8$ fibers, and a
section $P$ of height $17/18 = 4 - 3/2 - 14/9$. We next identify a
fiber $F'$ of type $E_8$ and perform a $3$-neighbor step to move to
the associated elliptic fibration.

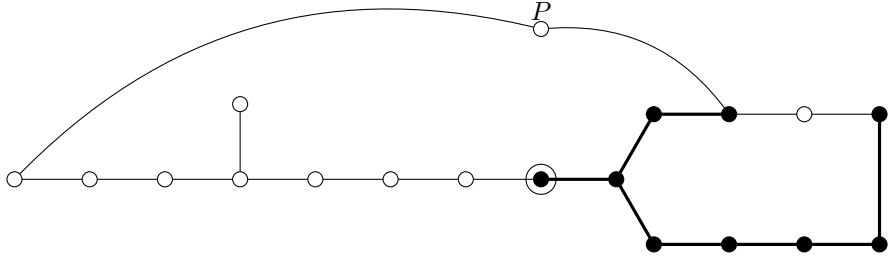
\begin{center}
\begin{tikzpicture}

\draw (-1,0)--(7,0)--(7.5,0.866)--(7.5,0.866)--(10.5,0.866)--(10.5,-0.866)--(7.5,-0.866)--(7,0);
\draw (2,0)--(2,1);
\draw [bend right] (6,2) to (-1,0);
\draw [bend left] (6,2) to (8.5,0.866);
\draw [very thick] (8.5, 0.866)--(7.5,0.866)--(7,0)--(7.5,-0.866)--(10.5,-0.866)--(10.5,0.866);
\draw [very thick] (6,0)--(7,0);

\draw (6,0) circle (0.2);
\fill [white] (-1,0) circle (0.1);
\fill [white] (0,0) circle (0.1);
\fill [white] (1,0) circle (0.1);
\fill [white] (2,0) circle (0.1);
\fill [white] (3,0) circle (0.1);
\fill [white] (4,0) circle (0.1);
\fill [white] (5,0) circle (0.1);
\fill [black] (6,0) circle (0.1);
\fill [white] (2,1) circle (0.1);
\fill [black] (7,0) circle (0.1);
\fill [black] (7.5,0.866) circle (0.1);
\fill [black] (7.5,-0.866) circle (0.1);
\fill [black] (8.5,0.866) circle (0.1);
\fill [black] (8.5,-0.866) circle (0.1);
\fill [white] (9.5,0.866) circle (0.1);
\fill [black] (9.5,-0.866) circle (0.1);
\fill [black] (10.5,0.866) circle (0.1);
\fill [black] (10.5,-0.866) circle (0.1);
\fill [white] (6,2) circle (0.1);

\draw (-1,0) circle (0.1);
\draw (0,0) circle (0.1);
\draw (1,0) circle (0.1);
\draw (2,0) circle (0.1);
\draw (3,0) circle (0.1);
\draw (4,0) circle (0.1);
\draw (5,0) circle (0.1);
\draw (6,0) circle (0.1);
\draw (2,1) circle (0.1);
\draw (7,0) circle (0.1);
\draw (7.5,0.866) circle (0.1);
\draw (7.5,-0.866) circle (0.1);
\draw (8.5,0.866) circle (0.1);
\draw (8.5,-0.866) circle (0.1);
\draw (9.5,0.866) circle (0.1);
\draw (9.5,-0.866) circle (0.1);
\draw (10.5,0.866) circle (0.1);
\draw (10.5,-0.866) circle (0.1);
\draw (6,2) circle (0.1);
\draw (6,2) [above] node{$P$};

\end{tikzpicture}
\end{center}

Since $P \cdot F' = 2$, while the remaining component of the
$A_8$ fiber intersects $F'$ with multiplicity~$3$, the new elliptic
fibration has a section. We may therefore convert to the Jacobian;
this has $E_8$ and $E_7$ fibers, and we may read out the Igusa-Clebsch
invariants.

\begin{theorem}
A birational model over $\Q$ for the Hilbert modular surface
$Y_{-}(17)$ as a double cover of $\Proj^2$ is given by the following
equation:
$$
z^2 = -256\,h^3 + (192\,g^2+464\,g+185)\,h^2
-2\,(2\,g+1)\,(12\,g^3-65\,g^2-54\,g-9)\,h + (g+1)^4\,(2\,g+1)^2.
$$
It is a rational surface.
\end{theorem}

\subsection{Analysis}

This is evidently a rational elliptic surface over $\Proj^1_g$. In
fact, it is a rational surface over $\Q$. We exhibit a birational
parametrization as follows: set $h = \big(4 (2\,g + 1) + m
\,(2\,g+1)\,(27\,g+13) \big)/27$ and absorb square factors by setting
$z = z_1\,(2\,g+1)\,(27\,g+13)/243$, to get (after scaling $g = g_1/3$)
the following equation
$$
z_1^2 = -9\,(8\,m-1)^3\,g_1^2
-2\,(16\,m+7)\,(424\,m^2-385\,m+1)\,g_1
-3\,(3328\,m^3-1923\,m^2-3138\,m-803)
$$
which is a conic bundle over $\Proj^1_m$. We see that $g_1 = -3/2$
makes the right hand side a square, and so setting
$z_1 = 9\,(2\,m+11)/2 + n\,(g_1 + 3/2)$ and solving for $g_1$,
we get a birational map from $\Proj^2_{m,n}$.

The branch locus is a curve of genus 0. To produce a parametrization,
we set $z_1 = 0$ and note that the discriminant of the resulting
quadratic equation in $g_1$ is a square times $64m^2 + 218 m - 8$.
Setting $64m^2 + 218 m - 8 = (8m+r)^2$ and solving for $m$,
we ultimately obtain:
$$
(g,h) = \left( \frac{-(8r^3-111r^2+1212r-8146)}{2(2r-7)^3}, \frac{(r-17)^2(r+10)^4}{4(2r-7)^6} \right).
$$

\subsection{Examples}

We list some points of small height and corresponding genus $2$ curves.

\begin{tabular}{l|c}
Rational point $(g,h)$ & Sextic polynomial $f_6(x)$ defining the genus $2$ curve $y^2 = f_6(x)$. \\
\hline \hline \\ [-2.5ex]
$(0, 13/32)$ & $ -2x^6 - x^5 - 6x^4 - 5x^3 - 4x^2 - 4x $ \\
$(-5/11, 1/88)$ & $ 3x^6 + 4x^5 + 4x^4 - 6x^3 - 5x^2 - 4x + 4 $ \\
$(6, 26)$ & $ -2x^6 + 2x^5 + x^4 + 8x^3 + 7x^2 + 4x $ \\
$(9/4, 77/64)$ & $ -2x^6 - x^5 + 8x^4 - 5x^3 - 4x^2 + 4x - 4 $ \\
$(3/5, -11/50)$ & $ -8x^6 + 8x^5 + 7x^4 + 2x^3 - x $ \\
$(5, -11)$ & $ 4x^5 + 9x^4 + 2x^3 - 8x^2 - 2x $ \\
$(5, 22)$ & $ 2x^6 - 4x^5 + 9x^4 - 10x^3 + 4x^2 - 4x $ \\
$(1/5, 28/25)$ & $ -10x^6 - 10x^5 - 2x^4 - 7x^3 - x $ \\
$(-1/4, -11/64)$ & $ 4x^5 + 3x^4 + 11x^3 - 7x^2 + x $ \\
$(5/4, -35/64)$ & $ x^5 - 7x^3 + 2x^2 - 8x + 12 $ \\
$(-5/2, -13/8)$ & $ 4x^5 - x^4 - 13x^3 - 3x^2 + 13x $ \\
$(1/5, -7/20)$ & $ 3x^6 + 7x^5 + 6x^4 + 16x^3 + 14x^2 - 8x - 8 $ \\
$(-12, -23/2)$ & $ -7x^6 - 19x^5 - 7x^4 + 14x^3 - x $ \\
$(0, -1/16)$ & $ -4x^6 + 19x^5 - 20x^4 + 11x^3 + 15x^2 - 8x - 4 $ \\
$(2, -35/4)$ & $ -4x^5 + 20x^4 - 12x^3 - 15x^2 - 4x $ \\
$(4, 51/8)$ & $ -4x^6 - 12x^5 - 27x^4 + 2x^3 + 12x^2 + 18x $ 
\end{tabular}

A plot of the rational points on $Y_{-}(17)$ reveals two
special curves on the surface. First, there is the line $h = -g/2$.
Substituting this in to the equation for the Hilbert modular
surface, we obtain a conic, which can be parametrized as
$(g,h) = \big( -(m^2-4)/(2(m-6)), (m^2-4)/(4(m-6)) \big)$. The Brauer
obstruction vanishes along this locus. However, this curve is modular:
the endomorphism ring is a split quaternion algebra, strictly containing
$\sO_{17}$.  The second curve is the parabola $h = -(6g^2 + g - 1)/8$.
The Brauer obstruction vanishes along this curve too, giving a $1$-parameter
family of genus-$2$ curves whose Jacobians have real multiplication by
$\sO_{17}$.

\section{Discriminant $21$}

\subsection{Parametrization}

We start with an elliptic K3 surface with fibers of type $E_8$, $A_6$
and $A_2$ at $t = \infty$, $0$ and $1$ respectively.

A Weierstrass equation for such a family is
$$
y^2 = x^3 + (a_0 + a_1 t + a_2 t^2)\, x^2 + 2 t^2 (t-1) (b_0 + b_1
t) \, x + t^4 (t-1)^2 (c_0 + c_1 t)
$$
with
\begin{align*}
a_0 &= 1, &
a_1 &= -r^2 + 2rs - 1, &
a_2 &= (r-s)^2; \\
b_0 &= (r^2-1) (s-r)^2, &
b_1 &=  (r^2-1) (s-r)^2 (r s-1); & \\
c_0 &= (r^2-1)^2 (s-r)^4, &
c_1 &= (r^2-1)^3 (s-r)^4. &
\end{align*}

We identify a fiber of type $E_7$, and a $3$-neighbor step gives us
the desired $E_8 E_7$ fibration.

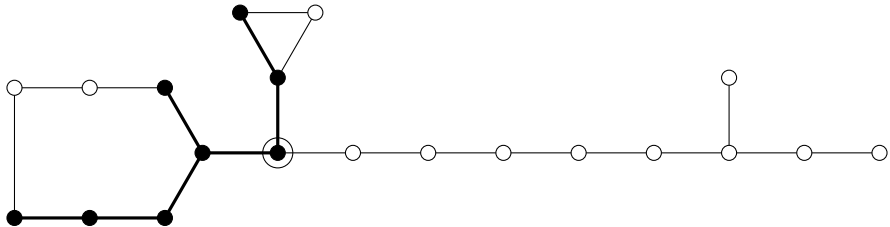
\begin{center}
\begin{tikzpicture}

\draw (0,0)--(9,0);
\draw (7,0)--(7,1);
\draw (0,0)--(-0.5,0.866)--(-2.5,0.866)--(-2.5,-0.866)--(-0.5,-0.866)--(0,0);
\draw (1,0)--(1,1)--(0.5,1.866)--(1.5,1.866)--(1,1);
\draw [very thick] (-2.5,-0.866)--(-0.5,-0.866)--(0,0)--(1,0)--(1,1)--(0.5,1.866);
\draw [very thick] (0,0)--(-0.5,0.866);

\draw (1,0) circle (0.2);
\fill [black] (0,0) circle (0.1);
\fill [black] (1,0) circle (0.1);
\fill [white] (2,0) circle (0.1);
\fill [white] (3,0) circle (0.1);
\fill [white] (4,0) circle (0.1);
\fill [white] (5,0) circle (0.1);
\fill [white] (6,0) circle (0.1);
\fill [white] (7,0) circle (0.1);
\fill [white] (8,0) circle (0.1);
\fill [white] (7,1) circle (0.1);
\fill [white] (9,0) circle (0.1);

\fill [black] (-0.5,0.866) circle (0.1);
\fill [white] (-1.5,0.866) circle (0.1);
\fill [white] (-2.5,0.866) circle (0.1);
\fill [black] (-0.5,-0.866) circle (0.1);
\fill [black] (-1.5,-0.866) circle (0.1);
\fill [black] (-2.5,-0.866) circle (0.1);

\fill [black] (1,1) circle (0.1);
\fill [black] (0.5,1.866) circle (0.1);
\fill [white] (1.5,1.866) circle (0.1);

\draw (0,0) circle (0.1);
\draw (1,0) circle (0.1);
\draw (2,0) circle (0.1);
\draw (3,0) circle (0.1);
\draw (4,0) circle (0.1);
\draw (5,0) circle (0.1);
\draw (6,0) circle (0.1);
\draw (7,0) circle (0.1);
\draw (8,0) circle (0.1);
\draw (9,0) circle (0.1);
\draw (7,1) circle (0.1);

\draw (1,1) circle (0.1);
\draw (0.5,1.866) circle (0.1);
\draw (1.5,1.866) circle (0.1);

\draw (-0.5,0.866) circle (0.1);
\draw (-1.5,0.866) circle (0.1);
\draw (-2.5,0.866) circle (0.1);
\draw (-0.5,-0.866) circle (0.1);
\draw (-1.5,-0.866) circle (0.1);
\draw (-2.5,-0.866) circle (0.1);

\end{tikzpicture}
\end{center}

We read out the Igusa-Clebsch invariants, and the branch locus of
$Y_{-}(21)$ as a double cover of $\Proj^2_{r,s}$ corresponds to the
subvariety of the moduli space where the elliptic K3 surfaces have an
extra $\I_2$ fiber.

\begin{theorem}
A birational model over $\Q$ for the Hilbert modular surface
$Y_{-}(21)$ as a double cover of $\Proj^2_{r,s}$ is given by the following
equation:
$$ 
z^2 = 16 s^4 - 8 r (27 r^2-23) s^3 + (621 r^4-954 r^2+349) s^2 - 18
(r^3-r) (33 r^2-29) s + (r^2-1) (189 (r^4-r^2)+16).
$$
It is a singular K3 surface (i.e. of Picard number $20$).
\end{theorem}

\subsection{Analysis}

The extra involution (corresponding to $21 = 3 \cdot 7$) here is given
by $(r,s) \mapsto (-r,-s)$.

The branch locus is a rational curve; a parametrization is given by
$$
(r,s) = \left( \frac{h^4+72h^2-81}{18h(h^2+3)}, \frac{(h^2-9)(h^4-126h^2+189)}{432h(h^2+3)} \right).
$$

The equation of $Y_{-}(21)$ above expresses it as a surface fibered by
genus-$1$ curves over $\Proj^1_r$. In fact, since the coefficient of
$s^4$ is a square, this fibration has a section, and we may convert it
to its Jacobian form, which after some simple Weierstrass
transformations may be written as
\begin{align*}
z^2 &= w^3 + (-27 r^4 + 43) \, w^2
  + 4 (r^2-1)  (8127 r^4-18459 r^2+9740)  \,w \\
  & \qquad + 4 (r^2-1)^2 w (-186624 r^6+1320813 r^4-1817964 r^2+705679).
\end{align*}

This is a K3 surface with an elliptic fibration to $\Proj^1_r$. The
discriminant of the cubic polynomial is
$$
(r^2-1)^3  (27 r^2-25)^2  (27 r^4+342r^2-289)^3,
$$
from which we deduce that we have three $I_2$ fibers (including $r =
\infty$) and six $I_3$ fibers, which contribute $15$ to the Picard
number.  We find the following sections.

\begin{align*}
P_0 &= \bigl( 6(r^2-1)(6r^2-7), 4(r-1)(r+1)(27r^4+342r^2-289) \bigr) \\
P_1 &= \bigl( 2 (324 r^4 - 1503 r^2 + 1019)/21, 8(27r^2-25)(27r^4+342r^2-289)/(21\mu )\bigr) \\
P_2 &= \bigl( (-102 + 32\nu )(r^2 - 1), 32\nu r(r^2-1)(-27r^2 + 31 + 22\nu )\bigr) \\
P_3 &= \bigl( (-102 - 32\nu )(r^2 - 1), -32\nu r(r^2-1)(-27r^2 + 31 - 22\nu )\bigr).
\end{align*}
Here $\mu = \sqrt{21}$ and $\nu = \sqrt{-1}$. Note that $P_0$ is a
$3$-torsion section, whereas the height matrix for $P_1, P_2, P_3$ is
$$
\frac{1}{6} \left(
\begin{array}{ccc}
2 & 0 & 0 \\ 0 & 13 & 1 \\ 0 & 1 & 13
\end{array} \right).
$$ 
Therefore, the Picard number of the surface is $20$, and the
discriminant of the lattice spanned by these sections and the trivial
lattice is $1008 = 2^4 3^2 7$. We showed that this is the entire
N\'eron-Severi group by checking that the above subgroup of the
Mordell-Weil group is $2$- and $3$-saturated.

Using $P_0$, we may rewrite the equation in the much simpler form
$$
z^2 + (9 r^2-13) w z + (r^2-1) (27 r^4+342 r^2-289)  z = w^3.
$$

The quotient of the Hilbert modular surface by the involution $(r,s,z)
\mapsto (-r,-s,-z)$ is the rational elliptic surface
\begin{align*}
z^2 &= w^3 + (-27t^2 + 43)w^2  + 4(t-1)  (8127t^2-18459t+9740)  w \\
 & \quad  + 4(t-1)^2  (-186624t^3+1320813t^2-1817964t+705679).
\end{align*}

It has three reducible fibers of type $\I_3$ and one of type $\I_2$, a
$3$-torsion section, and \MoW\ rank $1$, generated by the
following section of height $1/6$.
$$
(z,w) = \bigl( 2 (324 r^2 - 1503 r + 1019)/21, 8 (27 r-25) (27 r^2+342 r-289)/(21 \mu) \bigr).
$$

\subsection{Examples}

We list some points of small height and corresponding genus $2$ curves.

\begin{tabular}{l|c}
Rational point $(r,s)$ & Sextic polynomial $f_6(x)$ defining the genus $2$ curve $y^2 = f_6(x)$. \\
\hline \hline \\ [-2.5ex]
$(3/2, 1/44)$ & $ -5x^6 - 8x^5 + 20x^4 + 5x^3 - 13x $ \\
$(21/34, -43/68)$ & $ 13x^6 + 26x^5 + 33x^4 + 9x^3 - 5x^2 - 11x $ \\
$(-36/13, -49/26)$ & $ 7x^6 - 42x^5 - 44x^4 - 12x^3 - 11x^2 - 14x - 7 $ \\
$(0, 1/2)$ & $ 13x^6 + 54x^5 + 32x^4 - 28x^3 - 25x^2 + 14x - 1 $ \\
$(0, -1/2)$ & $ -x^6 + 2x^5 + 4x^4 + 36x^3 + 25x^2 + 42x - 59 $ \\
$(45/46, -25/92)$ & $ 9x^6 + 16x^5 + 12x^4 - 73x^3 - 41x^2 + 14x + 77 $ \\
$(-3/2, -1/44)$ & $ -13x^5 + 39x^4 + 31x^3 - 115x^2 - 50x - 125 $ \\
$(-21/34, 43/68)$ & $ -13x^5 + 58x^4 - 83x^3 - 90x^2 + 100x + 125 $ \\
$(-45/46, 25/92)$ & $ 5x^5 + 30x^4 - 61x^3 - 122x^2 + 112x + 161 $ \\
$(3, -1/28)$ & $ 28x^6 + 52x^5 - 149x^4 - 174x^3 + 235x^2 - 60x - 100 $ \\
$(55/63, -4/63)$ & $ 8x^6 + 192x^5 + 237x^4 + 238x^3 - 15x^2 - 60x - 76 $ \\
$(-1/3, -11/9)$ & $ -56x^6 + 132x^5 + 102x^4 + 195x^3 - 240x^2 - 204x - 178 $ \\
$(-3, 1/28)$ & $ 4x^6 - 52x^5 + 133x^4 + 34x^3 - 171x^2 - 308x - 292 $ \\
$(1/9, 11/9)$ & $ -2x^6 - 54x^5 + 135x^4 + 120x^3 - 135x^2 - 324x - 108 $ \\
$(-4, -7/2)$ & $ -15x^6 + 36x^5 + 5x^4 + 52x^3 + 50x^2 + 156x + 375 $ \\
$(-35/69, 11/69)$ & $ 4x^6 - 12x^5 - 219x^4 - 460x^3 + 15x^2 - 246x - 230 $ 
\end{tabular}

The torsion section $P_0$ on the Jacobian model of $Y_{-}(21)$ pulls
back to the curve
$$
s = \frac{45r^3+9r^2-45r-17}{2(27r^2+6r-17)}.
$$
The Brauer obstruction always vanishes for this family, and yields
a family of genus $2$ curves over $\Q(r)$, whose Jacobians have real
multiplication by $\sO_{21}$.

Another special curve which we observe from a plot of the rational
points is the hyperbola $s^2 = (3r^2 + 1)/4$. It may be parametrized as 
$$
(r,s) = \left( \frac{4m}{4m^2 - 3} , \frac{3 + 4m^2}{2(3-4m^2)} \right).
$$ 
The Brauer obstruction also vanishes for this family.

\section{Discriminant $24$}

\subsection{Parametrization}

We start with an elliptic K3 surface with fibers of type $E_6$, $E_7$
and $A_3$ at $t = \infty$, $0$ and $1$ respectively.

A Weierstrass equation for such a family is given by
$$
y^2 = x^3 + t^2 \, x^2 + a (t-1) t^3 \big(2+ (d^2 - a +
1)(t-1)\big) \, x + a^2 (t-1)^2 t^5 \big(1 + d^2 (t-1)\big).
$$

We identify a fiber of type $E_8$, and this leads us by a $3$-neighbor
step to an $E_8 E_7$ fibration.

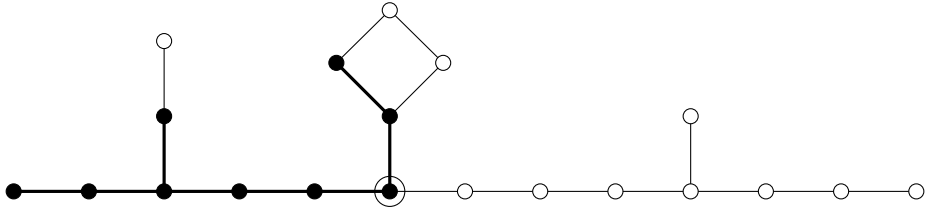
\begin{center}
\begin{tikzpicture}
\draw (0,0)--(12,0);
\draw (2,0)--(2,2);
\draw (5,0)--(5,1)--(4.29,1.71)--(5,2.41)--(5.71,1.71)--(5,1);
\draw (9,0)--(9,1);
\draw [very thick] (0,0)--(5,0)--(5,1)--(4.29,1.71);
\draw [very thick] (2,0)--(2,1);

\draw (5,0) circle (0.2);
\fill [black] (0,0) circle (0.1);
\fill [black] (1,0) circle (0.1);
\fill [black] (2,0) circle (0.1);
\fill [black] (3,0) circle (0.1);
\fill [black] (4,0) circle (0.1);
\fill [black] (2,1) circle (0.1);
\fill [white] (2,2) circle (0.1);
\fill [black] (5,0) circle (0.1);
\fill [black] (5,1) circle (0.1);
\fill [white] (6,0) circle (0.1);
\fill [white] (7,0) circle (0.1);
\fill [white] (8,0) circle (0.1);
\fill [white] (9,0) circle (0.1);
\fill [white] (10,0) circle (0.1);
\fill [white] (11,0) circle (0.1);
\fill [white] (12,0) circle (0.1);
\fill [white] (9,1) circle (0.1);
\fill [black] (4.29,1.71) circle (0.1);
\fill [white] (5.71,1.71) circle (0.1);
\fill [white] (5,2.41) circle (0.1);

\draw (0,0) circle (0.1);
\draw (1,0) circle (0.1);
\draw (2,0) circle (0.1);
\draw (3,0) circle (0.1);
\draw (4,0) circle (0.1);
\draw (2,1) circle (0.1);
\draw (2,2) circle (0.1);
\draw (5,0) circle (0.1);
\draw (5,1) circle (0.1);
\draw (6,0) circle (0.1);
\draw (7,0) circle (0.1);
\draw (8,0) circle (0.1);
\draw (9,0) circle (0.1);
\draw (10,0) circle (0.1);
\draw (11,0) circle (0.1);
\draw (12,0) circle (0.1);
\draw (9,1) circle (0.1);
\draw (4.29,1.71) circle (0.1);
\draw (5.71,1.71) circle (0.1);
\draw (5,2.41) circle (0.1);

\end{tikzpicture}
\end{center}

From the new elliptic fibration we read out the Igusa-Clebsch
invariants as usual, and then obtain the branch locus of $Y_{-}(24)$
as a double cover of $\Proj^2_{a,d}$, which is a union of two curves:
one corresponding to an extra $\I_2$ fiber, and one corresponding to
the locus where the $E_7$ fiber promotes to an $E_8$ fiber.

\begin{theorem}
A birational model over $\Q$ for the Hilbert modular surface
$Y_{-}(24)$ as a double cover of $\Proj^2_{a,d}$ is given by the
following equation:
$$
z^2 = (d^2 - a - 1) \, (16 a d^4-8 a^2 d^2-20 a d^2+d^2+a^3-3 a^2+3 a-1).
$$ 
It is a singular K3 surface.
\end{theorem}

\subsection{Analysis}

Note the extra involution $(a,d) \mapsto (a,-d)$.

The branch locus has two components. The first is the zero locus of
$d^2 - a - 1$, and is obviously a rational curve (i.e.\ of genus $0$)
in the moduli space. It parametrizes the K3 surfaces in the family for
which the $E_7$ fiber gets promoted to an $E_8$ fiber. The other
component parametrizes elliptic K3 surfaces for which there is an
extra $\I_2$ fiber. It is also a genus $0$ curve, though this fact is
less obvious. A parametrization is given by
$$
(a,d) = \left( \frac{1-g^2}{2g^2-1} , \frac{g^3}{2g^2 - 1} \right).
$$

The equation of the Hilbert modular surface describes it as a family
of curves of genus one fibered over $\Proj^1_d$. In fact, we readily
check that $(a,z) = (0, d^2 - 1)$ gives a section. So in fact, we have
an elliptic surface, and by using the formula for the Jacobian, we can
write it in Weierstrass form as
\begin{align*}
y^2 &= x^3 -x \, d^2 (144 d^6-324 d^4 + 235 d^2-54)/48  \\
& \qquad -d^2 (3456 d^{10}- 22032 d^8+ 50625 d^6-54866 d^4+28647 d^2-5832)/1728.
\end{align*}
This is an elliptic K3 surface $E$, and in fact, it is the base change
(by $d \mapsto d^2$) of a rational elliptic surface with reducible
fibers of types $\I_2$, $\I_3$ and $\I_4$ (and therefore with
\MoW\ rank $2$). We can use this to readily compute one section
$$
(x_0,y_0) = \big(11 d^4 - 239/12 d^2 + 9, (d^2-1) (9 d^2 -8) (32 d^2-27)/8 \big)
$$
of height $1/6$. Translating $x$ by $x_0$ and scaling to get
rid of denominators, we get the following nicer form for $E$\/:
\begin{align*}
y^2 &= x^3 + (132 d^4 - 239 d^2 + 108) \, x^2 \\
  & \quad + 2   (d^2 - 1)  (9 d^2 - 8)  (32 d^2 - 27)   (10 d^2 - 9) \, x \\
  & \quad + ((d^2 - 1)  (9 d^2 - 8)   (32 d^2 - 27))^2.
\end{align*}
This has reducible fibers of type $\I_2$ at $d = \infty$ and $d^2 =
8/9$, type $\I_3$ at $d^2 = 27/32$, type $\IV$ at $d = 0$, and type
$\I_4$ at $d = \pm 1$. This gives a root system of type $A_1^3 \oplus
A_2^3 \oplus A_3^2$, which has rank $15$ and discriminant $1152$.

In addition to the section $P_1 = \big(0 ,(d^2-1) (9 d^2 -8) (32 d^2-27)
\big)$, we also find the sections
\begin{align*}
P_2 &= \big(-5 (d^2-1) (9 d^2-8), \sqrt{-1}  (d^2-1)^2  (9 d^2-8) \big) \\
P_3 &= \big(-(32 d^2-27) (9 d^2-8)/6 , d (32 d^2-27) (9 d^2-8) / {\sqrt{216}} \big).
\end{align*}

This shows that the elliptic K3 surface $Y_{-}(24)$ has geometric
Picard number $20$, i.e.\ is a singular K3 surface.  The discriminant of
the span of the algebraic divisors exhibited is $96$.  We showed that
this is the entire N\'{e}ron-Severi group by checking that our subgroup
of the \MoW\ group is $2$-saturated.

As mentioned above, the quotient by the involution $d \mapsto -d$ is a
rational elliptic surface
\begin{align*}
y^2 &= x^3 + (132 t^2 - 239 t + 108) \, x^2 \\
  & \quad + 2   (t - 1)  (9 t - 8)  (32 t - 27)   (10 t - 9) \, x \\
  & \quad + ((t - 1)  (9 t - 8)   (32 t - 27))^2.
\end{align*}
This surface has an $\I_2$ fiber at $t = 8/9$, an $\I_3$ fiber at $t =
27/32$ and an $\I_4$ fiber at $t = 1$.

The \MoW\ group is generated by the sections
\begin{align*}
P_1 &= \big(0 ,(t-1) (9 t -8) (32 t-27) \big), \\
P_2 &= \big(-5 (t-1) (9 t-8), \sqrt{-1}  (t-1)^2  (9 t-8) \big),
\end{align*}
with height matrix
$$
\left(\begin{array}{cc}
1/12 &  0 \\ 0 & 1/2
\end{array}\right).
$$
\subsection{Examples}

We list some points of small height and corresponding genus $2$ curves.

\begin{tabular}{l|c}
Rational point $(a,d)$ & Sextic polynomial $f_6(x)$ defining the genus $2$ curve $y^2 = f_6(x)$. \\
\hline \hline \\ [-2.5ex]
$(77/36, -1/6)$ & $ -7x^6 - 18x^4 - 10x^3 + 3x^2 + 10x + 22 $ \\
$(21/16, 1/8)$ & $ -12x^5 + 8x^4 + 24x^3 - 16x^2 - 21x + 14 $ \\
$(9/4, -1/2)$ & $ 3x^6 - 12x^5 + 20x^4 - 79x^3 + 11x^2 - 84x + 60 $ \\
$(-7/81, -17/18)$ & $ 27x^6 - 56x^5 - 21x^4 + 52x^3 + 57x^2 - 84x + 49 $ \\
$(77/36, 1/6)$ & $ -32x^6 - 80x^5 + 94x^4 + 115x^3 - 91x^2 - 55x + 33 $ \\
$(9/4, -5/4)$ & $ -24x^6 + 48x^5 + 100x^4 + 120x^3 - 50x^2 - 72x - 41 $ \\
$(-3/4, -1/4)$ & $ 25x^6 + 70x^4 + 24x^3 + 124x^2 + 48x + 72 $ \\
$(9/4, 1/2)$ & $ -72x^6 + 84x^5 + 127x^4 - 123x^3 - 83x^2 + 51x + 25 $ \\
$(1, 3/2)$ & $ 9x^6 + 9x^4 - 60x^3 - 45x^2 + 132x - 53 $ \\
$(33/50, -1/5)$ & $ 50x^5 - 50x^4 + 35x^3 - 35x^2 - 31x + 139 $ \\
$(-3/4, 1/4)$ & $ 33x^6 - 36x^5 + 110x^4 - 120x^3 + 140x^2 - 96x + 72 $ \\
$(21/16, -1/8)$ & $ 7x^5 + 3x^4 + 78x^3 - 2x^2 + 63x + 147 $ \\
$(21, -5/2)$ & $ -8x^6 - 24x^5 + 80x^4 - 100x^3 + 170x^2 - 84x + 63 $ \\
$(-4/9, 2/3)$ & $ -3x^6 + 14x^5 - 63x^4 + 96x^3 - 171x^2 - 48x - 188 $ \\
$(-7/81, 17/18)$ & $ -7x^6 + 56x^5 - 95x^4 - 20x^3 - 205x^2 - 44x - 93 $ \\
$(21, 5/2)$ & $ -24x^6 - 24x^5 - 280x^4 - 260x^3 + 170x^2 - 24x + 1 $ 
\end{tabular}

For instance, for the genus two curves corresponding to $(1,3/2)$ and
$(1,-3/2)$, the point counts modulo $p$ match the twist by $\Q(\sqrt{2})$
of a modular form of level $2592$.

We describe a few curves on the Hilbert modular surface, which can be
used to produce rational points. Setting $a = -1/9$ gives a rational
curve of genus $0$, with infinitely many points. Sections of the
fibration will also lead to curves birational to $\Proj^1$ over~$\Q$.
For instance, $P_1$ and $2P_1$ describe the rational curves given by
$a = 4(d^2-1)/5$ and $a = (4d^2 - 5)/13$, respectively.  The Brauer
obstruction does not vanish identically on any of these.

\section{Discriminant $28$}

\subsection{Parametrization}

We start with an elliptic K3 surface with fibers of type $E_6$, $D_5$
and $A_4$ at $t = \infty$, $0$ and $1$ respectively, and a section of
height $28/60 = 7/15 = 4 - 6/5 - 1 - 4/3$.

A Weierstrass equation for this family is given by
$$
y^2 = x^3 + a t x^2 + b t^2 (t-1)^2 x + c t^3 (t-1)^4
$$
where
\begin{align*}
a &= 2 (f^2 - g^2) (t-1) + t, \\
b &= (f^2 - g^2)^2 (1-t) - 2 (f^2 - g^2)(f+1) t, \\
c &= (f+1)^2 (f^2 - g^2) t.
\end{align*}

We identify the class of an $D_8$ fiber and carry out a $2$-neighbor
step to convert to an elliptic fibration with $D_8$ and $E_7$ fibers,
and a section of height $7/2 = 4 + 2\cdot 1 - 1 - 3/2$.

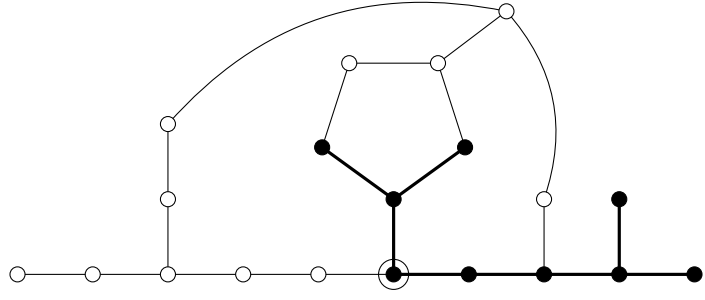
\begin{center}
\begin{tikzpicture}

\draw (0,0)--(9,0);
\draw (2,0)--(2,2);
\draw (5,0)--(5,1)--(4.05,1.69)--(4.41,2.81)--(5.59,2.81)--(5.95,1.69)--(5,1);
\draw (7,0)--(7,1);
\draw (8,0)--(8,1);
\draw (6.5,3.5)--(5.59,2.81);
\draw [bend right] (6.5,3.5) to (2,2);
\draw [bend left] (6.5,3.5) to (7,1);
\draw [very thick] (5,1)--(5,0)--(9,0);
\draw [very thick] (8,0)--(8,1);
\draw [very thick] (4.05,1.69)--(5,1)--(5.95,1.69);

\draw (5,0) circle (0.2);
\fill [white] (0,0) circle (0.1);
\fill [white] (1,0) circle (0.1);
\fill [white] (2,0) circle (0.1);
\fill [white] (2,1) circle (0.1);
\fill [white] (2,2) circle (0.1);
\fill [white] (3,0) circle (0.1);
\fill [white] (4,0) circle (0.1);
\fill [black] (5,0) circle (0.1);
\fill [black] (6,0) circle (0.1);
\fill [black] (7,0) circle (0.1);
\fill [black] (8,0) circle (0.1);
\fill [black] (9,0) circle (0.1);
\fill [white] (7,1) circle (0.1);
\fill [black] (8,1) circle (0.1);
\fill [black] (5,1) circle (0.1);
\fill [black] (4.05,1.69) circle (0.1);
\fill [black] (5.95,1.69) circle (0.1);
\fill [white] (4.41,2.81) circle (0.1);
\fill [white] (5.59,2.81) circle (0.1);
\fill [white] (6.5,3.5) circle (0.1);

\draw (0,0) circle (0.1);
\draw (1,0) circle (0.1);
\draw (2,0) circle (0.1);
\draw (2,1) circle (0.1);
\draw (2,2) circle (0.1);
\draw (3,0) circle (0.1);
\draw (4,0) circle (0.1);
\draw (5,0) circle (0.1);
\draw (6,0) circle (0.1);
\draw (7,0) circle (0.1);
\draw (8,0) circle (0.1);
\draw (9,0) circle (0.1);
\draw (7,1) circle (0.1);
\draw (8,1) circle (0.1);
\draw (5,1) circle (0.1);
\draw (4.05,1.69) circle (0.1);
\draw (5.95,1.69) circle (0.1);
\draw (4.41,2.81) circle (0.1);
\draw (5.59,2.81) circle (0.1);

\draw (6.5,3.5) circle (0.1);

\end{tikzpicture}
\end{center}

Then we identify the class of an $E_8$ fiber, and carry out a
$2$-neighbor step to get the desired $E_8 E_7$ fibration.

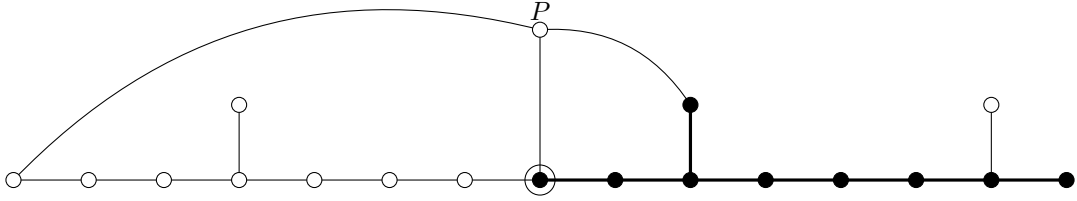
\begin{center}
\begin{tikzpicture}

\draw (0,0)--(14,0);
\draw (3,0)--(3,1);
\draw (13,0)--(13,1);
\draw (9,0)--(9,1);
\draw (7,2)--(7,0);
\draw [bend right] (7,2) to (0,0);
\draw [bend left] (7,2) to (9,1);
\draw [very thick] (7,0)--(14,0);
\draw [very thick] (9,0)--(9,1);

\draw [above] (7,2) node{$P$};

\draw (7,0) circle (0.2);
\fill [white] (0,0) circle (0.1);
\fill [white] (1,0) circle (0.1);
\fill [white] (2,0) circle (0.1);
\fill [white] (3,0) circle (0.1);
\fill [white] (4,0) circle (0.1);
\fill [white] (5,0) circle (0.1);
\fill [white] (6,0) circle (0.1);
\fill [black] (7,0) circle (0.1);
\fill [black] (8,0) circle (0.1);
\fill [black] (9,0) circle (0.1);
\fill [black] (10,0) circle (0.1);
\fill [black] (11,0) circle (0.1);
\fill [black] (12,0) circle (0.1);
\fill [black] (13,0) circle (0.1);
\fill [black] (14,0) circle (0.1);
\fill [white] (3,1) circle (0.1);
\fill [white] (13,1) circle (0.1);
\fill [black] (9,1) circle (0.1);

\fill [white] (7,2) circle (0.1);

\draw (0,0) circle (0.1);
\draw (1,0) circle (0.1);
\draw (2,0) circle (0.1);
\draw (3,0) circle (0.1);
\draw (4,0) circle (0.1);
\draw (5,0) circle (0.1);
\draw (6,0) circle (0.1);
\draw (7,0) circle (0.1);
\draw (8,0) circle (0.1);
\draw (9,0) circle (0.1);
\draw (10,0) circle (0.1);
\draw (11,0) circle (0.1);
\draw (12,0) circle (0.1);
\draw (13,0) circle (0.1);
\draw (14,0) circle (0.1);
\draw (3,1) circle (0.1);
\draw (13,1) circle (0.1);
\draw (9,1) circle (0.1);

\draw (7,2) circle (0.1);

\end{tikzpicture}
\end{center}

The new elliptic fibration has a section, since $P \cdot F' = 5$,
while the remaining component of the $D_8$ fiber has intersection
number $2$ with $F'$. We now read out the Igusa-Clebsch invariants.

\begin{theorem}
A birational model over $\Q$ for the Hilbert modular surface
$Y_{-}(28)$ as a double cover of $\Proj^2_{g,h}$ is given by the following
equation:
$$
z^2 = -(g-f-2)(g+f+2)(8g^4+92f^2g^2+180fg^2+71 g^2-100 f^4-180 f^3-71 f^2+4f+4).
$$ 
It is a singular K3 surface.
\end{theorem}

\subsection{Analysis}

It has a second involution $(f,g) \mapsto (f, -g)$.  The branch locus
consists of three components.  The factors $g \pm (f + 2)$
correspond to the locus where the Picard number of the K3 surface
jumps to~$19$, due to the presence of an $\I_2$ fiber, but the
discriminant decreases to~$14$, because the non-trivial section of
height~$7/15$ becomes divisible by~$2$. The more complicated factor
corresponds to just the presence of an extra $\I_2$ factor, which
makes the discriminant~$56$. This component of the branch locus is
also a genus~$0$ curve; a parametrization is given by
$$
(f,g) = \left(
  -\frac{2m^4+17m^3+57m^2+85m+47}{2(m+1)(m+2)(m^2+6m+11)},
  -\frac{(m^2+6m+7)(2m^2+7m+7)}{2(m+1)(m+2)(m^2+6m+11)}
\right).
$$

Since the double cover is branched along a sextic, the Hilbert modular
surface is itself a K3 surface. Setting $f = h-2$ and then using the
invertible substitution $h = t(1+1/x), g = t(1-1/x)$ (and absorbing
square factors) converts it to an elliptic fibration over $\Proj^1_t$,
which we can write in Weierstrass form as

\begin{align*}
y^2 &= x^3 -t (108 t^3-176 t^2+63 t+4) x^2 \\
  &\quad + 32 (t-1)^2 t^3 (135 t^2-36 t-106) x -64 (t-1)^4 t^4 (6075 t^2-6075 t-196) .
\end{align*}

This has fibers of type $\I_2$ at $t = \infty$, $4/5$ and $28/27$,
$\I^*_1$ at $t = 0$, $\I_5$ at $t = 1$,
and $\I_3$ at $t = (19 \pm 7 \sqrt{7})/36$,
giving a contribution of $D_5 \oplus A_4 \oplus A_1^3 \oplus A_2^2$
to the N\'{e}ron-Severi lattice. Therefore the Picard
number is at least $18$. We identify the following sections:
\begin{align*}
P_1 &= \Big( 12 t^2 (t-1) (36 t - 37), \, 4 t^2 (t-1)(27 t-28) (72
t^2-76 t+1) \Big), \\
P_2 &= \Big( t (1080 t^3-2064 t^2+953 t+28)/7, \,
2 t^2 (5 t-4) (27 t-28) (72 t^2-76 t+1)/7^{3/2} \Big),
\end{align*}
with height matrix
$$
\left(\begin{array}{cc} 7/60 & 0 \\ 0 & 2/3
\end{array} \right).
$$
Therefore the Picard number is $20$. An easy lattice-theoretic
argument (see the auxiliary files) shows that these sections must
generate the \MoW\ group, and therefore the N\'{e}ron-Severi
lattice has discriminant $112$.

The quotient by the involution $g \mapsto -g$ has equation
\begin{align*}
y^2 &= x^3 -(84f^2+148f+39)x^2/4 -(96f^4+364f^3+615f^2+500f+140)x \\
& \quad -(f+2)^2(5f+2)^2(4f^2+4f-1).
\end{align*}

This is a rational elliptic surface with an $\I_4$ fiber at $f = -1$,
an $\I_3$ fiber at $f = -17/18$, and $\I_2$ fibers at $t = -26/27$ and
$t = \infty$. The \MoW\ group is generated by the $2$-torsion
section $\big(-2(f+2)^2, 0 \big)$ and the non-torsion section
$\big(-2(f^2 + 3f + 3), (f+1)(18f + 17) \big)$ of height $1/12$.

\subsection{Examples}

We list some points of small height and corresponding genus $2$ curves.

\begin{tabular}{l|c}
Rational point $(f,g)$ & Sextic polynomial $f_6(x)$ defining the genus $2$ curve $y^2 = f_6(x)$. \\
\hline \hline \\ [-2.5ex]
$(-31/16, -9/16)$ & $ 9x^6 - 15x^5 + 39x^4 - 25x^3 + 36x^2 + 11 $ \\
$(-56/65, -61/65)$ & $ -37x^6 + 42x^5 + 17x^4 - 33x^3 - 7x^2 - 3x - 3 $ \\
$(-31/16, 9/16)$ & $ 15x^5 - 39x^4 - 11x^3 + 45x^2 + 27x $ \\
$(-35/24, -31/24)$ & $ -27x^6 + 54x^5 - 36x^4 + 30x^3 - 24x^2 - 6x - 2 $ \\
$(-12/11, -1/11)$ & $ x^6 + 14x^5 + 61x^4 + 73x^3 - 49x^2 + 3x + 1 $ \\
$(-14/9, -5/9)$ & $ 37x^6 + 69x^5 + 75x^4 + 79x^3 + 69x^2 + 12x + 20 $ \\
$(-2/5, -7/5)$ & $ -4x^6 + 12x^5 - x^4 + 87x^3 + 68x^2 + 48x + 3 $ \\
$(-13/8, -7/8)$ & $ 3x^6 - 6x^5 + 5x^4 - 10x^3 + 41x^2 + 30x + 87 $ \\
$(-42/41, 1)$ & $ -92x^6 - 76x^5 + 21x^4 + 9x^3 + 9x^2 - x - 1 $ \\
$(-13/8, 7/8)$ & $ 18x^6 + 78x^5 + 44x^4 - 94x^3 - 76x^2 + 30x + 25 $ \\
$(5/16, -19/16)$ & $ -48x^6 + 45x^5 + 11x^4 + 87x^3 - 97x^2 - 63x + 56 $ \\
$(-30/41, -36/41)$ & $ -27x^6 + 57x^5 - 100x^4 - 68x^3 - 76x^2 + 36x $ \\
$(-11/3, 7/3)$ & $ 16x^6 - 24x^5 - 111x^4 + 9x^3 + 102x^2 - 27x - 33 $ \\
$(26/31, -64/31)$ & $ -9x^6 - 6x^5 + 32x^4 + 32x^3 - 112x^2 - 6x + 99 $ \\
$(13/80, -15/16)$ & $ -9x^6 - 15x^5 + 85x^3 - 135x + 54 $ \\
$(26/31, 64/31)$ & $ -25x^6 + 30x^5 - 64x^4 + 72x^3 - 136x^2 + 102x - 69 $ 
\end{tabular}

Next, we describe some special curves on the surface, which may be
used to produce rational points.
First, $f = -17/18$ gives a rational curve, which can be parametrized
as $g = -19(h^2-2)/(18(h^2+2))$. The Brauer obstruction vanishes
identically for this family, giving a family of genus $2$ curves with
real multiplication by $\sO_{28}$.
Next, the specialization $f = -26/27$ gives another rational curve, which
can be parametrized as $g = -2(13h^2+729h-75816)/(27(h^2+5832))$.
The Brauer obstruction does not vanish identically for this family.
Finally, the section $P_1$ on the Jacobian of $Y_{-}(28)$ can be described as
$$
(f,g) = \left(  \frac{4t^2-8t+1}{4t-1} , \frac{4t^2-2t+1}{4t-1} \right).
$$
The Brauer obstruction vanishes identically on this family as well.

From a plot of the rational points, we observe many rational points on
the lines $g = \pm (5f + 2)/3$. However, the Brauer obstruction does
not vanish identically along these lines.

\section{Discriminant $29$}

\subsection{Parametrization}

We start with an elliptic K3 surface with fibers of type $E_7$ and
$A_8$, and a section $P$ of height $29/18 = 4 - 3/2 - 8/9$. A
Weierstrass equation for this family is given by
$$
y^2 = x^3 + \big(-(4\,f-1)\,t^2+(g-2)\,t+1\big)\,x^2
-2\,g\,t^3\,\big(2\,f^2\,t^2+(-g+2\,f+1)\,t-1\big)\,x +
g^2\,t^6\,\big( (g-4\,f)\,t+1\big).
$$
We identify a fiber $F$\/ of type $E_8$, and go to the associated
elliptic fibration by a $3$-neighbor step. Note that $P \cdot F = 4$,
while the remaining component of the $A_8$ fiber intersects $F$\/ with
multiplicity $3$. Therefore the new fibration has a section.

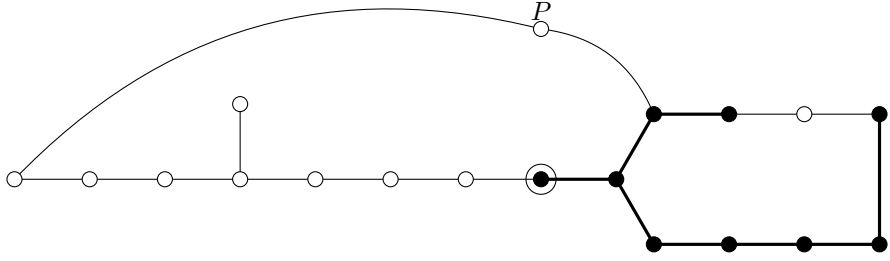
\begin{center}
\begin{tikzpicture}

\draw (-1,0)--(7,0)--(7.5,0.866)--(7.5,0.866)--(10.5,0.866)--(10.5,-0.866)--(7.5,-0.866)--(7,0);
\draw (2,0)--(2,1);
\draw [bend right] (6,2) to (-1,0);
\draw [bend left] (6,2) to (7.5,0.866);
\draw [very thick] (6,0)--(7,0);
\draw [very thick] (8.5,0.866)--(7.5,0.866)--(7,0)--(7.5,-0.866)--(10.5,-0.866)--(10.5,0.866);

\draw (6,0) circle (0.2);
\fill [white] (-1,0) circle (0.1);
\fill [white] (0,0) circle (0.1);
\fill [white] (1,0) circle (0.1);
\fill [white] (2,0) circle (0.1);
\fill [white] (3,0) circle (0.1);
\fill [white] (4,0) circle (0.1);
\fill [white] (5,0) circle (0.1);
\fill [black] (6,0) circle (0.1);
\fill [white] (2,1) circle (0.1);
\fill [black] (7,0) circle (0.1);
\fill [black] (7.5,0.866) circle (0.1);
\fill [black] (7.5,-0.866) circle (0.1);
\fill [black] (8.5,0.866) circle (0.1);
\fill [black] (8.5,-0.866) circle (0.1);
\fill [white] (9.5,0.866) circle (0.1);
\fill [black] (9.5,-0.866) circle (0.1);
\fill [black] (10.5,0.866) circle (0.1);
\fill [black] (10.5,-0.866) circle (0.1);
\fill [white] (6,2) circle (0.1);

\draw (-1,0) circle (0.1);
\draw (0,0) circle (0.1);
\draw (1,0) circle (0.1);
\draw (2,0) circle (0.1);
\draw (3,0) circle (0.1);
\draw (4,0) circle (0.1);
\draw (5,0) circle (0.1);
\draw (6,0) circle (0.1);
\draw (2,1) circle (0.1);
\draw (7,0) circle (0.1);
\draw (7.5,0.866) circle (0.1);
\draw (7.5,-0.866) circle (0.1);
\draw (8.5,0.866) circle (0.1);
\draw (8.5,-0.866) circle (0.1);
\draw (9.5,0.866) circle (0.1);
\draw (9.5,-0.866) circle (0.1);
\draw (10.5,0.866) circle (0.1);
\draw (10.5,-0.866) circle (0.1);
\draw (6,2) circle (0.1);
\draw (6,2) [above] node{$P$};

\end{tikzpicture}
\end{center}

The new elliptic fibration has $E_8$ and $E_7$ fibers, and we may read
out the Igusa-Clebsch invariants. The branch locus for the double
cover of $\Proj^2_{f,g}$ corresponds to elliptic K3 surfaces having an
extra $\I_2$ fiber. 

\begin{theorem}
A birational model over $\Q$ for the Hilbert modular surface
$Y_{-}(29)$ as a double cover of $\Proj^2$ is given by the following
equation:
\begin{align*}
z^2 &= -g^4+(6\,f+11)\,g^3 -(27\,f^4-18\,f^3+5\,f^2+102\,f-1)\,g^2 \\
  & \qquad   +8\,f\,(36\,f^3-25\,f^2+35\,f-1)\,g -16\,f^2\,(4\,f-1)^3.
\end{align*}
It is a K3 surface of Picard number $19$.
\end{theorem}

\subsection{Analysis}

The branch locus is a rational curve, and a parametrization is given by
$$
(f,g) = \left(  \frac{(m+2)^2(m^2+2m-4)}{m^2(m^2+6m+12)} , \frac{4(m+2)^2(m+4)^3}{m(m^2+6m+12)^2} \right).
$$

The substitution $g = hf$ makes the right hand side of the equation a
quartic in $f$, after absorbing a factor of $f^2$ into $z^2$ on the
left. Also, $f = 0$ makes the right hand side a square (namely
$(h-4)^2$), giving us a point on the genus-$1$ curve over
$\Proj^1_h$. We may then replace the genus-$1$ curve by its Jacobian,
which has the following equation after the change of parameter $h = 4
- t$, and some Weierstrass transformations:
\begin{align*}
y^2 &= x^3 -(t^4-10\,t^3-t^2-t-16)\,x^2  \\
& \qquad -t\,(20\,t^5-140\,t^4-3740\,t^3+14234\,t^2+3349\,t-1120)\,x \\
& \qquad -t^2\,(3475\,t^6-46300\,t^5+180355\,t^4-10362\,t^3-849193\,t^2+276269\,t-19600).
\end{align*}

This is an elliptic K3 surface, with bad fibers of type $\I_4$ at $t =
\infty$, $\I_3$ at $t = 0$ and at the roots of the polynomial
$15t^3+12t^2+160t-64$ (which gives the cubic field of discriminant
$-87$), and $\I_2$ at the roots of $2t^3-38 t^2+255 t-28$ (which gives
the cubic field of discriminant $-116$).  The trivial lattice
therefore has rank $16$, leaving room for \MoW\ rank at most~$4$. We
find the following linearly independent sections:
\begin{align*}
P_1 &= \big( t(125t^3-700t^2+800t-76)/4, \quad t(1375t^5-11450t^4+29240t^3-23616t^2+7296t-512)/8 \big), \\
P_2 &= \big(  (30t^4-306t^3+636t^2+1833t-448)/29, \quad (2t^3-38t^2+255t-28)(15t^3+12t^2+160t-64)/(29b) \big), \\
P_3 &= \Big( 2(3a-5)t^2-3(8a-23)t , \quad at\big(-99t^3+ 18(a+37)t^2-144(3a+7)t+416a-288 \big)/3 \Big).
\end{align*}
Here $a = \sqrt{-3}$ and $b = \sqrt{29}$. Therefore, the Picard number
of the Hilbert modular surface is at least $19$.  We showed by
counting points modulo $11$ and $13$ that the Picard number must be
exactly~$19$.  This agrees with Oda's calculations
\cite[pg.~109]{Oda}. The height matrix of the sections above is
$$
\frac{1}{6} \left(\begin{array}{rrr}
20 & 0 & -10 \\
0 & 3 & 0 \\
-10 & 0 & 14
\end{array}\right).
$$ 
Therefore, the sublattice of the N\'eron-Severi lattice spanned by the
sections above together with the trivial lattice has discriminant $2^4
\cdot 3^4 \cdot 5 = 6480$. By checking that it is $2$- and
$3$-saturated, we showed that it is the entire N\'eron-Severi lattice,
and therefore the sections $P_1, P_2, P_3$ generate the Mordell-Weil
group.

\subsection{Examples}

We list some points of small height and corresponding genus $2$ curves.

\begin{tabular}{l|c}
Rational point $(f,g)$ & Sextic polynomial $f_6(x)$ defining the genus $2$ curve $y^2 = f_6(x)$. \\
\hline \hline \\ [-2.5ex]
$(-37, -37/2)$ & $ -98x^5 - 56x^4 - 131x^3 - 114x^2 + 10x - 68 $ \\
$(20/67, 16/67)$ & $ 264x^6 + 760x^5 + 183x^4 - 630x^3 - 53x^2 - 20x - 4 $ \\
$(6/7, 48/7)$ & $ 12x^6 - 12x^5 - 409x^4 + 1062x^3 + 287x^2 - 588x - 252 $ \\
$(-40, 40)$ & $ -53x^6 + 227x^5 + 374x^4 + 1191x^3 + 669x^2 + 680x + 900 $ \\
$(-1, -29/10)$ & $ 1210x^5 - 110x^4 + 511x^3 - 17x^2 + 53x + 1 $ \\
$(-4, -16/5)$ & $ -200x^6 + 1360x^5 - 995x^4 + 242x^3 - 191x^2 - 4x - 12 $ \\
$(-1/7, -37/98)$ & $ -1588x^6 + 986x^5 - 122x^4 + 221x^3 - 68x^2 - 2x - 8 $ \\
$(-6, 8)$ & $ 540x^6 + 2052x^5 - 1149x^4 + 1724x^3 - 39x^2 - 894x - 506 $ \\
$(5/43, 19/43)$ & $ -2x^6 + 80x^5 - 786x^4 + 2265x^3 + 74x^2 + 80x - 2 $ \\
$(8/9, 16/9)$ & $ -4x^6 + 204x^5 - 837x^4 - 160x^3 + 2451x^2 + 1620x - 228 $ \\
$(-6/13, -16/13)$ & $ 236x^6 - 796x^5 + 2293x^4 - 2178x^3 + 1525x^2 + 2492x - 764 $ \\
$(-4/9, 16/9)$ & $ 552x^6 - 2232x^5 + 3183x^4 + 562x^3 - 4713x^2 - 1248x - 88 $ \\
$(-3, 36/5)$ & $ 1024x^6 - 1920x^5 - 2252x^4 + 1065x^3 - 2288x^2 - 6195x + 66 $ \\
$(18/19, 48/19)$ & $ -768x^6 - 2560x^5 + 1571x^4 + 7838x^3 - 2133x^2 - 6912x + 2376 $ \\
$(-4/5, -32/25)$ & $ -1412x^6 - 1372x^5 + 2149x^4 + 8226x^3 + 4889x^2 - 896x - 4096 $ \\
$(2/9, 7/9)$ & $ 3189x^6 + 4599x^5 - 6897x^4 - 9331x^3 + 5424x^2 + 5040x - 1968 $ 
\end{tabular}

The section $P_1$ corresponds to the rational curve given by
$$
g = \frac{4uf}{5} \textrm{ with } f = \frac{22u^5-321u^4+1651u^3-3377u^2+1980u+50}{25u^6-400u^5+2375u^4-6050u^3+4813u^2+2325u+225}.
$$
The Brauer obstruction vanishes identically, yielding a
$1$-parameter family of genus $2$ curves with real multiplication by
$\sO_{29}$.

\section{Discriminant $33$}

\subsection{Parametrization}

We start with an elliptic K3 surface with fibers of type $A_{10}$ and
$E_6$ at $t = \infty$ and $t = 0$ respectively. A Weierstrass equation
for such a family is given by
$$
y^2 = x^3 + (c+2d + 1) \, t^2 \,x^2 + 2 (c+d)\, t^4 \,x + c \,t^4,
$$
with
\begin{align*}
c &= (s^2 - r^2)^2 \, t^2 - (s^2 - r^2)(s^2 - r^2 + 2r) t + s^2 \\
d &= (s^2 - r^2) (1-t) + r
\end{align*}

We identify the class of an $E_7$ fiber below, and perform a $2$-neighbor step.

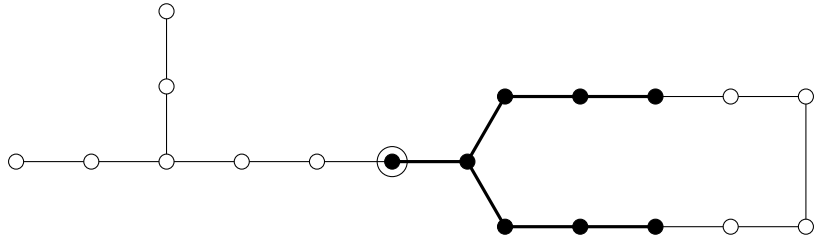
\begin{center}
\begin{tikzpicture}

\draw (0,0)--(6,0)--(6.5,0.866)--(10.5,0.866)--(10.5,-0.866)--(6.5,-0.866)--(6,0);
\draw (2,0)--(2,2);
\draw [very thick] (5,0)--(6,0);
\draw [very thick] (8.5,0.866)--(6.5,0.866)--(6,0)--(6.5,-0.866)--(8.5,-0.866);

\draw (5,0) circle (0.2);
\fill [white] (0,0) circle (0.1);
\fill [white] (1,0) circle (0.1);
\fill [white] (2,0) circle (0.1);
\fill [white] (3,0) circle (0.1);
\fill [white] (4,0) circle (0.1);
\fill [black] (5,0) circle (0.1);
\fill [white] (2,1) circle (0.1);
\fill [white] (2,2) circle (0.1);
\fill [black] (6,0) circle (0.1);
\fill [black] (6.5,0.866) circle (0.1);
\fill [black] (7.5,0.866) circle (0.1);
\fill [black] (8.5,0.866) circle (0.1);
\fill [white] (9.5,0.866) circle (0.1);
\fill [white] (10.5,0.866) circle (0.1);
\fill [black] (6.5,-0.866) circle (0.1);
\fill [black] (7.5,-0.866) circle (0.1);
\fill [black] (8.5,-0.866) circle (0.1);
\fill [white] (9.5,-0.866) circle (0.1);
\fill [white] (10.5,-0.866) circle (0.1);

\draw (0,0) circle (0.1);
\draw (1,0) circle (0.1);
\draw (2,0) circle (0.1);
\draw (3,0) circle (0.1);
\draw (4,0) circle (0.1);
\draw (5,0) circle (0.1);
\draw (2,1) circle (0.1);
\draw (2,2) circle (0.1);
\draw (6,0) circle (0.1);
\draw (6.5,0.866) circle (0.1);
\draw (7.5,0.866) circle (0.1);
\draw (8.5,0.866) circle (0.1);
\draw (9.5,0.866) circle (0.1);
\draw (10.5,0.866) circle (0.1);
\draw (6.5,-0.866) circle (0.1);
\draw (7.5,-0.866) circle (0.1);
\draw (8.5,-0.866) circle (0.1);
\draw (9.5,-0.866) circle (0.1);
\draw (10.5,-0.866) circle (0.1);

\end{tikzpicture}
\end{center}

The new elliptic fibration has fibers of type $E_6, E_7$ and $A_2$, as
well as a section $P$ of height $33/18 = 11/6 = 4 - 0 - 3/2 - 2/3$. We
then identify the class of an $E_8$ fiber and carry out a $3$-neighbor
step to an $E_8 E_7$ elliptic fibration.

\begin{center}
\begin{tikzpicture}

\draw (0,0)--(12,0);
\draw (2,0)--(2,2);
\draw (5,0)--(5,1);
\draw (9,0)--(9,1);
\draw (5,1)--(4.5,1.866);
\draw (5,1)--(5.5,1.866);
\draw (4.5,1.866)--(5.5,1.866);
\draw [bend right] (4,3) to (4,0);
\draw [bend left] (4,3) to (12,0);
\draw (4,3)--(4.5,1.866);
\draw [very thick] (0,0)--(5,0)--(5,1)--(4.5,1.866);
\draw [very thick] (2,0)--(2,1);

\fill [black] (0,0) circle (0.1);
\fill [black] (1,0) circle (0.1);
\fill [black] (2,0) circle (0.1);
\fill [black] (3,0) circle (0.1);
\fill [black] (4,0) circle (0.1);
\fill [black] (5,0) circle (0.1);
\fill [white] (6,0) circle (0.1);
\fill [white] (7,0) circle (0.1);
\fill [white] (8,0) circle (0.1);
\fill [white] (9,0) circle (0.1);
\fill [white] (10,0) circle (0.1);
\fill [white] (11,0) circle (0.1);
\fill [white] (12,0) circle (0.1);
\fill [black] (2,1) circle (0.1);
\fill [white] (2,2) circle (0.1);
\fill [black] (5,1) circle (0.1);
\fill [white] (9,1) circle (0.1);
\fill [black] (4.5,1.866) circle (0.1);
\fill [white] (5.5,1.866) circle (0.1);
\fill [white] (4,3) circle (0.1);

\draw (5,0) circle (0.2);
\draw (0,0) circle (0.1);
\draw (1,0) circle (0.1);
\draw (2,0) circle (0.1);
\draw (3,0) circle (0.1);
\draw (4,0) circle (0.1);
\draw (5,0) circle (0.1);
\draw (6,0) circle (0.1);
\draw (7,0) circle (0.1);
\draw (8,0) circle (0.1);
\draw (9,0) circle (0.1);
\draw (10,0) circle (0.1);
\draw (11,0) circle (0.1);
\draw (12,0) circle (0.1);
\draw (2,1) circle (0.1);
\draw (2,2) circle (0.1);
\draw (5,1) circle (0.1);
\draw (9,1) circle (0.1);
\draw (4.5,1.866) circle (0.1);
\draw (5.5,1.866) circle (0.1);
\draw (4,3) circle (0.1);
\draw [above] (4,3) node {$P$};

\end{tikzpicture}
\end{center}

The new elliptic fibration has a section, since $P \cdot F'= 5$,
while the intersection number of the remaining component of the
$E_6$ fiber with $F'$ is $3$, and these are coprime.

From the Weierstrass equation of the $E_8 E_7$ fibration, we determine
the Igusa-Clebsch invariants, and then the equation of the branch
locus.

\begin{theorem}
A birational model over $\Q$ for the Hilbert modular surface
$Y_{-}(33)$ as a double cover of $\Proj^2$ is given by the following
equation:
\begin{align*}
z^2 &=  9 s^6 - (26 r^2 - 80 r + 104) \, s^4  + (25 r^4 - 152 r^3 + 400 r^2 - 408 r + 432) \, s^2 \\
   & \qquad  - (8 r^6 - 72 r^5 + 280 r^4 - 472 r^3 + 336 r^2 - 64 r - 16)
\end{align*}
It is a singular K3 surface.
\end{theorem}

\subsection{Analysis}

This is a double cover of the $(r,s)$-plane branched along a sextic,
and is therefore a K3 surface.  The extra involution is given by
$(r,s) \mapsto (r,-s)$. The equation of the branch locus may be
transformed as follows: setting $t = s^2$ we have
\begin{align*}
 -8 r^6 + 72 r^5 + (25 t - 280) r^4 + (-152 t + 472) r^3+ (-26 t^2 + 400 t - 336) r^2 & \\
  + (80 t^2 - 408 t + 64) r + (9 t^3 - 104 t^2 + 432 t + 16) &= 0
\end{align*}
which, after resolution of singularities, becomes a genus zero curve,
parametrized by
$$
r = \frac{m^3+4 m^2+4 m+4}{m^2 (m+1)},\qquad  t= \frac{8 (m^3+4 m^2+4 m+2)}{m^4 (m+1)^2}.
$$
Then the branch locus can be written as a double cover
$$
s^2 = 8 (m^3+4 m^2+4 m+2 )/\big( m^4 (m+1)^2 \big).
$$ 
After removing square factors and performing a Weierstrass
transformation, it is converted to the elliptic curve $y^2 + y = x^3 -
x^2$. It is isomorphic to $X_1(11) \cong X_0(33)/\langle w_{33} \rangle$,
where $w_{33}$ is the Atkin-Lehner involution.

For the equation of the Hilbert modular surface, the transformation $s
= r + t$ makes the right hand side of the equation a quartic in $r$,
with the coefficient of $r^4$ being a square. Converting to the
Jacobian form, and applying a Weierstrass transformation as well as
scaling $t$, we get an elliptic fibration
\begin{align*}
y^2 &= x^3 + (t^4+24 t^3+58 t^2+84 t+1) \,x^2 \\
    & \qquad  + (280 t^5+5488 t^4+1376 t^3+2192 t^2+72 t) \,x \\
    & \qquad  + 4608 t^7+95632 t^6+32576 t^5+26848 t^4+14656 t^3+1296 t^2.
\end{align*}
This has bad fibers of type $\I_5$ at $t = \infty$, type $\II$ at $t = 1$,
type $\I_3$ at $t = 0, 1/2$ and $(-17 \pm 3 {\sqrt{33}})/2$,
and type $\I_2$ at $t = -21/2 \pm 11 {\sqrt{33}}/6$.
These contribute $A_1^3 \oplus A_2^4 \oplus A_4$ to the N\'{e}ron-Severi lattice.

By finding sections modulo a small prime and attempting to lift them
to $\Q$ or $\Q({\sqrt{33}})$, we find the following sections of small height:
\begin{align*}
P_1 &= \big( 4 t (3 t^2+60 t-13) , \, 4t (t^2+17 t-2) (3t^2+63 t-2) \big) \\
P_2 &= \big( -4 t(3t + 11),\, 4 t (t-1) (3t^2 + 63t - 2) \big) \\
P_3 &= \big(  (4 {\sqrt{33}}+12) t^2 - (2{\sqrt{33}} + 34) t, \, ({\sqrt{33}} + 3) t (2t - 1) (2t + 3 {\sqrt{33}} + 17) (6t + 63 - 11 {\sqrt{33}})/6 \big) .
\end{align*}
These are linearly independent in the \MoW\ group, and the
matrix of N\'{e}ron-Tate heights is
$$
\frac{1}{30} \left(\begin{array}{rrr}
6 & -2 & 3 \\
-2 & 19 & -1 \\
3 &  -1 & 9
\end{array}\right).
$$
It has determinant $11/360$, and so the sublattice of the
N\'{e}ron-Severi group generated by these sections and the trivial
lattice has rank $20$ and discriminant
$-(11/360) \cdot 2^3 \cdot 3^4 \cdot 5 = -99$.  
We show that this is the full Picard group by checking that our
subgroup of the \MoW\ group is $3$-saturated.
We deduce that $Y_{-}(33)$ is a singular K3 surface with
Picard lattice of discriminant $-99$.

The quotient of the $Y_{-}(33)$ by the involution $s \mapsto -s$ is
given by the Weierstrass equation
\begin{align*}
y^2 &= x^3 -2(13r^2-40r+52)x^2 + 9(25r^4-152r^3+400r^2-408r+432)x \\
  & \quad -648(r-1)^3(r^3-6r^2+14r+2).
\end{align*}

It is a rational elliptic surface, with reducible fibers of types
$\I_5$, $\I_3$ and $\I_2$ at $r = \infty, 19$ and $23$
respectively. The \MoW\ group is generated by the section
$\big( 3(3t^2-16t+112), 12(r-23)(r-19) \big)$ of height $1/30$.

\subsection{Examples}

We list some points of small height and corresponding genus $2$ curves.

\begin{tabular}{l|c}
Rational point $(r,s)$ & Sextic polynomial $f_6(x)$ defining the genus $2$ curve $y^2 = f_6(x)$. \\
\hline \hline \\ [-2.5ex]
$(11/3, 8/3)$ & $ -9x^6 - 6x^5 - 7x^4 + 7x^3 + 2x^2 + 3x - 2 $ \\
$(1, 3)$ & $ -4x^6 + 15x^4 - x^3 - 9x^2 + 12x + 5 $ \\
$(-13/5, 27/5)$ & $ -14x^5 + 20x^4 + 2x^3 - 15x^2 - 4x $ \\
$(28/3, 23/3)$ & $ -5x^6 + 6x^5 - 5x^4 + 27x^3 - 11x^2 + 12x - 24 $ \\
$(73/19, -41/19)$ & $ -7x^6 + 15x^5 + x^4 - 31x^3 - 2x^2 + 12x + 12 $ \\
$(1, -3)$ & $ -x^6 - 3x^5 + 9x^4 + 34x^3 - 30x^2 - 9x + 8 $ \\
$(41/51, 7/51)$ & $ -9x^6 - 9x^5 - 35x^4 + 11x^3 - 8x^2 + 12x $ \\
$(16/15, 1/15)$ & $ 2x^6 - 3x^5 + 7x^4 + 13x^3 - 20x^2 + 36x - 15 $ \\
$(-13/5, -27/5)$ & $ -5x^5 - 3x^4 + 13x^3 - 17x^2 - 40x $ \\
$(46/3, 41/3)$ & $ -8x^6 + 36x^5 - 23x^4 - 21x^3 - 47x^2 - 18x - 9 $ \\
$(-17/10, 27/10)$ & $ -20x^5 + 28x^4 - 37x^3 + 60x^2 + 44x $ \\
$(13/3, -11/3)$ & $ 7x^6 - 48x^5 + 68x^4 - 2x^3 - 25x^2 + 24x - 36 $ \\
$(-38/11, 61/11)$ & $ -9x^6 - 6x^5 + 69x^4 - 19x^3 - 39x^2 + 12x - 8 $ \\
$(-29/3, -11)$ & $ -31x^5 + 71x^4 - 32x^3 + 23x^2 - 40x - 8 $ \\
$(1/33, 56/33)$ & $ 3x^6 + 24x^5 + 8x^4 + 41x^3 + 32x^2 - 72x - 36 $ \\
$(13/3, 11/3)$ & $ -12x^6 + 24x^5 - 31x^4 + 81x^3 - 67x^2 + 39x - 70 $ 
\end{tabular}

We may attempt to match up these examples with eigenforms in the
tables of modular forms. For instance, for the points $(1,3)$ and
$(1,-3)$, the corresponding genus $2$ curves (with isogenous
Jacobians)
$$
y^2 = -(x^3-3x-1)(4x^3-3x+5)
$$
and
$$
y^2 =  -(x^3-3x^2+1)(x^3+6x^2+9x-8).
$$
have the property that their traces match those of a newform of
weight $1296 = 2^4 \cdot 3^4$ in the modular forms database.

We also see some simple rational curves on the Hilbert modular
surface: the specialization $r = 19$ gives a rational curve, with a
parametrization $s = -(16m^2 + 41)/(3m)$, while $r = 23$ is also a
rational curve, parametrized by $s = -(16m^2+7)/m$. The Brauer
obstructions do not identically vanish for points on these curves.

We have slightly better luck with sections of the elliptic fibration:
for instance, the sections $P_1, 2P_1, 3P_1$ give rise to rational
curves with parametrizations
\begin{align*}
\left( -\frac{3t^2-112}{2(3t+8)}, \frac{3t^2+16t+112}{2(3t+8)} \right), \quad \left( -\frac{ t^2-12}{2(t+2)}, \frac{t^2+4t+12}{2(t+2)} \right) \quad \textrm{and} \quad \left( -\frac{3t^2-20}{6(t+2)}, \frac{3t^2+12t+20}{6(t+2)} \right)
\end{align*}
respectively. The Brauer obstruction vanishes identically on each of these.

\section{Discriminant $37$}

\subsection{Parametrization}

We start with an elliptic K3 surface with fibers of type $E_6$, $D_5$
and $A_4$ at $t = \infty, 0$ and $1$ respectively, and a section of
height $37/60 = 4 - 4/3 - 5/4 - 4/5$.

A Weierstrass equation for this family is
$$
y^2 = x^3 + a t x^2 + b t^2 (t-1)^2 x + c t^3 (t-1)^4
$$
with
\begin{align*}
a &= (2g - f + 1)(t -1) + g^2 t/4, \\
b &= (f-g-1) \big( (f+g-1) (t-1) + (f-2) g t/2), \\
c &= (g-f+1)^2 \big(f^2 (t -1) + (f-2)^2 \big)/4 .
\end{align*}

We identify the class of an $D_8$ fiber and carry out a $2$-neighbor
step to convert to an elliptic fibration with $D_8$ and $E_6$ fibers.

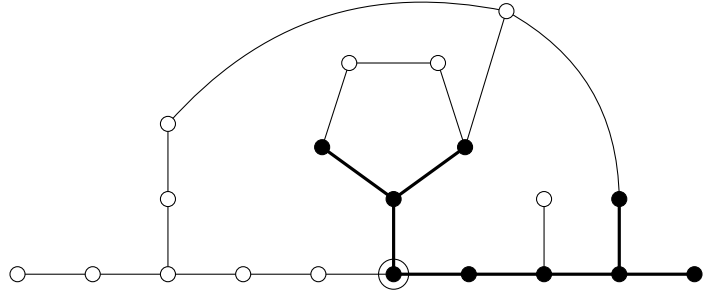
\begin{center}
\begin{tikzpicture}

\draw (0,0)--(9,0);
\draw (2,0)--(2,2);
\draw (5,0)--(5,1)--(4.05,1.69)--(4.41,2.81)--(5.59,2.81)--(5.95,1.69)--(5,1);
\draw (7,0)--(7,1);
\draw (8,0)--(8,1);
\draw (6.5,3.5)--(5.95,1.69);
\draw [bend right] (6.5,3.5) to (2,2);
\draw [bend left] (6.5,3.5) to (8,1);
\draw [very thick] (5,1)--(5,0)--(9,0);
\draw [very thick] (8,0)--(8,1);
\draw [very thick] (4.05,1.69)--(5,1)--(5.95,1.69);

\draw (5,0) circle (0.2);
\fill [white] (0,0) circle (0.1);
\fill [white] (1,0) circle (0.1);
\fill [white] (2,0) circle (0.1);
\fill [white] (2,1) circle (0.1);
\fill [white] (2,2) circle (0.1);
\fill [white] (3,0) circle (0.1);
\fill [white] (4,0) circle (0.1);
\fill [black] (5,0) circle (0.1);
\fill [black] (6,0) circle (0.1);
\fill [black] (7,0) circle (0.1);
\fill [black] (8,0) circle (0.1);
\fill [black] (9,0) circle (0.1);
\fill [white] (7,1) circle (0.1);
\fill [black] (8,1) circle (0.1);
\fill [black] (5,1) circle (0.1);
\fill [black] (4.05,1.69) circle (0.1);
\fill [black] (5.95,1.69) circle (0.1);
\fill [white] (4.41,2.81) circle (0.1);
\fill [white] (5.59,2.81) circle (0.1);
\fill [white] (6.5,3.5) circle (0.1);

\draw (0,0) circle (0.1);
\draw (1,0) circle (0.1);
\draw (2,0) circle (0.1);
\draw (2,1) circle (0.1);
\draw (2,2) circle (0.1);
\draw (3,0) circle (0.1);
\draw (4,0) circle (0.1);
\draw (5,0) circle (0.1);
\draw (6,0) circle (0.1);
\draw (7,0) circle (0.1);
\draw (8,0) circle (0.1);
\draw (9,0) circle (0.1);
\draw (7,1) circle (0.1);
\draw (8,1) circle (0.1);
\draw (5,1) circle (0.1);
\draw (4.05,1.69) circle (0.1);
\draw (5.95,1.69) circle (0.1);
\draw (4.41,2.81) circle (0.1);
\draw (5.59,2.81) circle (0.1);

\draw (6.5,3.5) circle (0.1);

\end{tikzpicture}
\end{center}

This new elliptic fibration has \MoW\ rank $2$. In fact, it is
quite easy to exhibit one non-torsion section $P$\/ (we do so in the
auxiliary files) of height $2/3 = 4 - 4/3 - 2$.  Then we find an
$E_8$ fiber as shown below, and proceed to it by a $2$-neighbor step.
The section $P$ combined with most of the components of the
$E_6$ fiber gives a disjoint $E_7$ configuration, and therefore the
new elliptic fibration has reducible fibers of type $E_8$ and~$E_7$.
The fact that $(-P)$ intersects the multiplicity one component
of the $E_8$ fiber shown below implies that the new fibration has a
section.

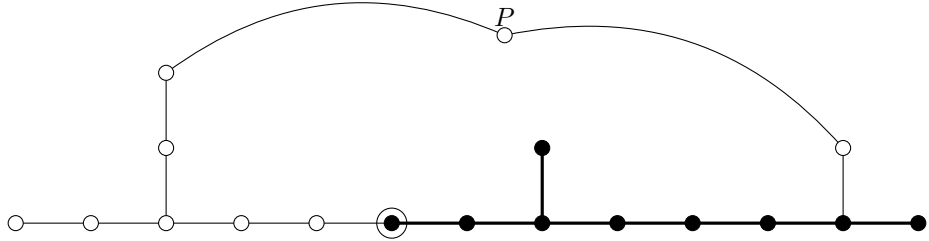
\begin{center}
\begin{tikzpicture}

\draw (0,0)--(12,0);
\draw (2,0)--(2,2);
\draw (7,0)--(7,1);
\draw (11,0)--(11,1);
\draw [bend right] (6.5,2.5) to (2,2);
\draw [bend left] (6.5,2.5) to (11,1);
\draw [very thick] (5,0)--(12,0);
\draw [very thick] (7,0)--(7,1);

\draw (5,0) circle (0.2);
\fill [white] (0,0) circle (0.1);
\fill [white] (1,0) circle (0.1);
\fill [white] (2,0) circle (0.1);
\fill [white] (2,1) circle (0.1);
\fill [white] (2,2) circle (0.1);
\fill [white] (3,0) circle (0.1);
\fill [white] (4,0) circle (0.1);
\fill [black] (5,0) circle (0.1);
\fill [black] (6,0) circle (0.1);
\fill [black] (7,0) circle (0.1);
\fill [black] (7,1) circle (0.1);
\fill [black] (8,0) circle (0.1);
\fill [black] (9,0) circle (0.1);
\fill [black] (10,0) circle (0.1);
\fill [black] (11,0) circle (0.1);
\fill [white] (11,1) circle (0.1);
\fill [black] (12,0) circle (0.1);

\fill [white] (6.5,2.5) circle (0.1);

\draw (0,0) circle (0.1);
\draw (1,0) circle (0.1);
\draw (2,0) circle (0.1);
\draw (2,1) circle (0.1);
\draw (2,2) circle (0.1);
\draw (3,0) circle (0.1);
\draw (4,0) circle (0.1);
\draw (5,0) circle (0.1);
\draw (6,0) circle (0.1);
\draw (7,0) circle (0.1);
\draw (7,1) circle (0.1);
\draw (8,0) circle (0.1);
\draw (9,0) circle (0.1);
\draw (10,0) circle (0.1);
\draw (11,0) circle (0.1);
\draw (11,1) circle (0.1);
\draw (12,0) circle (0.1);

\draw (6.5,2.5) circle (0.1);

\draw [above] (6.5,2.5) node{$P$};
\end{tikzpicture}
\end{center}

We then read out the Igusa-Clebsch invariants and write down the
equation of the Hilbert modular surface.

\begin{theorem}
A birational model over $\Q$ for the Hilbert modular surface
$Y_{-}(37)$ as a double cover of $\Proj^2$ is given by the following
equation:
\begin{align*}
z^2 &= f^2g^4 + 2f(14f - 1)g^3  -(126f^3-142f^2+44f-1)g^2 \\
  &\quad + (f-1)(54f^3-34f^2+17f-10)g - (f-1)^2(27f^2-8f+8).
\end{align*}
It is a K3 surface of Picard number $19$.
\end{theorem}

\subsection{Analysis}

The branch locus corresponds to the locus where the elliptic K3
surface acquires an extra $\I_2$ fiber. The transformation
$$
(f,g)  = \left( \frac{2x^2y+4xy+y+x^4+x^3-3x^2-x}{x^3(x+2)}, \frac{2y+x^2-2x-1}{x^2} \right)
$$
converts it to the elliptic curve
$$
y^2 + y = x^3 - x
$$
which is \textrm{37a} in Cremona's tables.  This is an elliptic
curve of rank $1$, isomorphic to $X_0(37)/\langle w \rangle$,
where $w$ is the Atkin-Lehner involution.

To analyze this Hilbert modular
surface, note that we have a genus $1$ fibration over $\Proj^1_f$,
which has a section because the coefficient of $g^4$ is a square.
Hence $Y_{-}(37)$ is an elliptic K3 surface. Taking the Jacobian of this
genus-$1$ curve over $\Q(f)$, and reparametrizing $f = t/(t+1)$,
we get after some Weierstrass transformations the following equation:
\begin{align*}
y^2 &= x^3 -(t+1)(27t^3+21t^2-19t-1)x^2 -8t^2(t+1)^2(30t^2-235t-1)x \\
& \quad  -16t^3(t+1)^2(3136t^4+5484t^3+1024t^2-2161t-108).
\end{align*}

This surface has reducible fibers of type $\IV$ at $t = -1$, type
$\I_2$ at $t = -1/28$, and type $\I_3$ at $t = 0, \infty$ and the four
roots of $(27 t^4+45 t^3+10 t^2+22 t+3)$.   This quartic
polynomial describes a quadratic extension of $\Q(\sqrt{37})$. The bad
fibers contribute $A_2^7 \oplus A_1$ to the trivial lattice. We also
have a $3$-torsion section
$$
P_0 = \Big(4t(t+1)(9t^2+7t-3), \,  4t(t+1)(27t^4+45t^3+10t^2+22t+3) \Big),
$$
and two non-torsion sections
\begin{align*}
P_1 &= \Big( 4t(t+1)(49t^2+28t-1), \, 4t(t+1)(637t^4+854t^3+276t^2+33t+1) \Big), \\
P_2 &= \Big( 4(252t^4+457t^3+118t^2-177t-9)/37, \, 4(t+3)(28t+1)(27t^4+45t^3+10t^2+22t+3)/37^{3/2} \Big).
\end{align*}
These two sections have height $8/3$ and $5/6$ respectively and are
orthogonal with respect to the N\'{e}ron-Tate height pairing. Therefore,
the N\'{e}ron-Severi lattice contains a sublattice of rank $19$ and
discriminant~$1080$.  Counting points modulo $11$ and $13$ shows that
the Picard number must be exactly~$19$.  This is again confirmed by
Oda's tables \cite[pg.~109]{Oda}.  We checked that the \MoW\ subgroup
generated by $P_0$, $P_1$, and $P_2$ is saturated at~$2$ and~$3$,
and thus that we have the full N\'{e}ron-Severi lattice.

\subsection{Examples}

We list some points of small height and corresponding genus $2$ curves.

\begin{tabular}{l|c}
Rational point $(f,g)$ & Sextic polynomial $f_6(x)$ defining the genus $2$ curve $y^2 = f_6(x)$. \\
\hline \hline \\ [-2.5ex]
$(3/2, 11/7)$ & $ 17x^6 - 24x^5 - 66x^4 + 68x^3 + 81x^2 - 54x - 27 $ \\
$(3/4, 11/9)$ & $ 15x^6 - 6x^5 - 71x^4 + 35x^3 + 94x^2 - 48x - 21 $ \\
$(3/2, 1/3)$ & $ 11x^6 - 54x^5 - 125x^4 - 52x^3 - 32x^2 - 42x - 75 $ \\
$(3/2, 59/65)$ & $ -135x^6 + 108x^5 + 45x^4 - 44x^3 + 130x^2 - 12x + 95 $ \\
$(1/4, 17/5)$ & $ 81x^5 - 135x^4 + 13x^3 - 9x^2 + 70x - 15 $ \\
$(-13/3, 16/13)$ & $ -13x^6 + 156x^5 - 24x^4 - 132x^3 - 45x^2 + 108x - 27 $ \\
$(-4/9, -13/6)$ & $ 36x^6 - 108x^5 + 165x^4 - 124x^3 + 21x^2 + 36x - 28 $ \\
$(5/3, 13/10)$ & $ -54x^6 + 54x^5 - 9x^4 - 84x^3 + 141x^2 - 180x + 100 $ \\
$(3/16, 13/8)$ & $ 52x^6 + 156x^5 - 39x^4 - 180x^3 + 9x^2 + 72x - 16 $ \\
$(-1/2, 15/13)$ & $ -31x^6 + 156x^5 - 195x^4 - 260x^3 + 210x^2 + 156x + 23 $ \\
$(3, 81/11)$ & $ -18x^6 + 122x^5 - 135x^4 - 268x^3 - 25x^2 + 144x + 176 $ \\
$(1/4, 3)$ & $ x^6 + 72x^5 - 18x^4 - 189x^3 - 117x^2 + 270x + 45 $ \\
$(31/15, 16/31)$ & $ -53x^6 + 6x^5 - 21x^4 + 208x^3 - 258x^2 - 276x + 259 $ \\
$(34/27, 14/51)$ & $ -2x^6 + 36x^5 - 138x^4 + 105x^3 - 33x^2 - 153x - 289 $ \\
$(3, 7/2)$ & $ -108x^6 - 324x^5 - 207x^4 - 116x^3 + 105x^2 - 12x - 12 $ \\
$(22/3, 38/11)$ & $ -22x^6 + 72x^5 + 84x^4 - 341x^3 - 441x^2 + 417x + 473 $ 
\end{tabular}

Next, we describe some curves of small genus on the surface, which may
be used to produce rational points. The specialization $f = -1/27$
gives a rational curve, with parametrization $g = 7(h^2-8h+19)/(3(h^2-1))$.
The Brauer obstruction does not vanish identically for this rational curve.

The sections $P_1$ and $-P_1$ give rational curves, parametrized by
$$
g = \frac{13f^2-7f+3}{f(3f+2)} \quad \textrm{and}
\quad  g = \frac{9f^2-2f+2}{7f+1}
$$
respectively. The Brauer obstruction vanishes on both these loci,
yielding families of genus $2$ curves whose Jacobians have real
multiplication by $\sO_{37}$.

\section{Discriminant $40$}

\subsection{Parametrization}

We start with an elliptic K3 surface with fibers of type $E_7$, $D_5$
and $A_4$ at $t = \infty, 0$ and $1$ respectively.

A Weierstrass equation for this family is given by
$$
y^2 = x^3 + t \, \big( e^2 \, t + (4\,d+1) \,(1-t) \big)\, x^2 + 2
\, t^2 \, (t-1)^2 \big( 2 \,d \,e \, t + 2 d (d+1) \, (1-t) \big) \, x
+ 4 \,d^2 \, t^3 \, (t-1)^4
$$
with
$$
d = (f-e+1)\,(f+e-1)/2.
$$

We first identify the class of a $D_8$ fiber, and perform a
$2$-neighbor step to an elliptic fibration with $D_8$ and $E_7$
fibers.

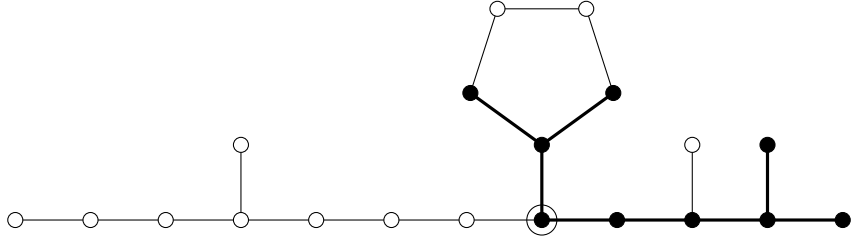
\begin{center}
\begin{tikzpicture}

\draw (0,0)--(11,0);
\draw (3,0)--(3,1);
\draw (7,0)--(7,1)--(6.05,1.69)--(6.41,2.81)--(7.59,2.81)--(7.95,1.69)--(7,1);
\draw (9,0)--(9,1);
\draw (10,0)--(10,1);
\draw [very thick] (7,1)--(7,0)--(11,0);
\draw [very thick] (10,0)--(10,1);
\draw [very thick] (6.05,1.69)--(7,1)--(7.95,1.69);

\draw (7,0) circle (0.2);
\fill [white] (0,0) circle (0.1);
\fill [white] (1,0) circle (0.1);
\fill [white] (2,0) circle (0.1);
\fill [white] (3,0) circle (0.1);
\fill [white] (3,1) circle (0.1);
\fill [white] (4,0) circle (0.1);
\fill [white] (5,0) circle (0.1);
\fill [white] (6,0) circle (0.1);
\fill [black] (7,0) circle (0.1);
\fill [black] (8,0) circle (0.1);
\fill [black] (9,0) circle (0.1);
\fill [black] (10,0) circle (0.1);
\fill [black] (11,0) circle (0.1);
\fill [white] (9,1) circle (0.1);
\fill [black] (10,1) circle (0.1);
\fill [black] (7,1) circle (0.1);
\fill [black] (6.05,1.69) circle (0.1);
\fill [black] (7.95,1.69) circle (0.1);
\fill [white] (6.41,2.81) circle (0.1);
\fill [white] (7.59,2.81) circle (0.1);

\draw (0,0) circle (0.1);
\draw (1,0) circle (0.1);
\draw (2,0) circle (0.1);
\draw (3,0) circle (0.1);
\draw (3,1) circle (0.1);
\draw (4,0) circle (0.1);
\draw (5,0) circle (0.1);
\draw (6,0) circle (0.1);
\draw (7,0) circle (0.1);
\draw (8,0) circle (0.1);
\draw (9,0) circle (0.1);
\draw (10,0) circle (0.1);
\draw (11,0) circle (0.1);
\draw (9,1) circle (0.1);
\draw (10,1) circle (0.1);
\draw (7,1) circle (0.1);
\draw (6.05,1.69) circle (0.1);
\draw (7.95,1.69) circle (0.1);
\draw (6.41,2.81) circle (0.1);
\draw (7.59,2.81) circle (0.1);

\end{tikzpicture}
\end{center}

This fibration has a section $P$ of height $5 = 4 + 2 - 1$. Next, we take a
$2$-neighbor step to go to a fibration with $E_8$ and $E_7$ fibers.
We have $P \cdot F' = 9$, while the intersection number of the $F'$ with
the remaining component of the $D_8$ fiber is $2$. Since these are
coprime, the new elliptic fibration has a section.

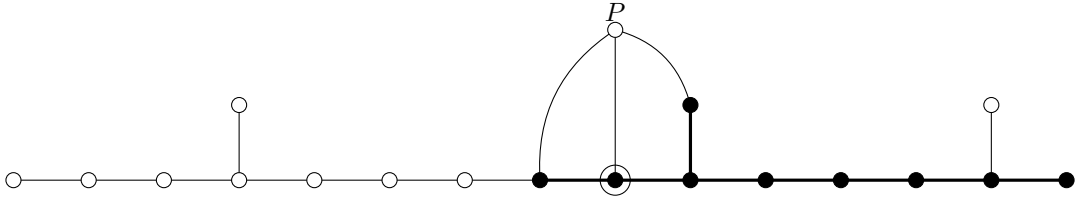
\begin{center}
\begin{tikzpicture}

\draw (0,0)--(14,0);
\draw (3,0)--(3,1);
\draw (13,0)--(13,1);
\draw (9,0)--(9,1);
\draw (8,2)--(8,0);
\draw [bend right] (8,2) to (7,0);
\draw [bend left] (8,2) to (9,1);
\draw [very thick] (7,0)--(14,0);
\draw [very thick] (9,0)--(9,1);

\fill [white] (0,0) circle (0.1);
\fill [white] (1,0) circle (0.1);
\fill [white] (2,0) circle (0.1);
\fill [white] (3,0) circle (0.1);
\fill [white] (4,0) circle (0.1);
\fill [white] (5,0) circle (0.1);
\fill [white] (6,0) circle (0.1);
\fill [black] (7,0) circle (0.1);
\fill [black] (8,0) circle (0.1);
\fill [black] (9,0) circle (0.1);
\fill [black] (10,0) circle (0.1);
\fill [black] (11,0) circle (0.1);
\fill [black] (12,0) circle (0.1);
\fill [black] (13,0) circle (0.1);
\fill [black] (14,0) circle (0.1);
\fill [white] (3,1) circle (0.1);
\fill [white] (13,1) circle (0.1);
\fill [black] (9,1) circle (0.1);

\fill [white] (8,2) circle (0.1);
\draw [above] (8,2) node{$P$};

\draw (8,0) circle (0.2);
\draw (0,0) circle (0.1);
\draw (1,0) circle (0.1);
\draw (2,0) circle (0.1);
\draw (3,0) circle (0.1);
\draw (4,0) circle (0.1);
\draw (5,0) circle (0.1);
\draw (6,0) circle (0.1);
\draw (7,0) circle (0.1);
\draw (8,0) circle (0.1);
\draw (9,0) circle (0.1);
\draw (10,0) circle (0.1);
\draw (11,0) circle (0.1);
\draw (12,0) circle (0.1);
\draw (13,0) circle (0.1);
\draw (14,0) circle (0.1);
\draw (3,1) circle (0.1);
\draw (13,1) circle (0.1);
\draw (9,1) circle (0.1);

\draw (8,2) circle (0.1);

\end{tikzpicture}
\end{center}

Finally, we read out the Igusa-Clebsch invariants, and compute the
branch locus of the double cover, which is a union of two curves,
one corresponding to an extra $\I_2$ fiber, the other to a promotion of
the fiber at $t=\infty$ from $E_7$ to $E_8$.

\begin{theorem}
A birational model over $\Q$ for the Hilbert modular surface
$Y_{-}(40)$ as a double cover of $\Proj^2_{e,f}$ is given by the following
equation:
$$
z^2 = -(f^2 - e^2 + 1) \big(8 f^4 + (-17 e^2+12 e-8) f^2+ 9 e^4-12
e^3+7 e^2+10 e+2\big).
$$
It is a K3 surface of Picard number $19$.
\end{theorem}

\subsection{Analysis}

The extra involution is $(e,f) \mapsto (e,-f)$.  The branch locus is
a union of two curves.  The first, $f^2 - e^2 + 1 = 0$, is a rational
curve, parametrized by say $(e,f) = \big( (t^2+ 1)/(2t), (t^2 -
1)/(2t) \big)$. It corresponds to the sub-family of elliptic K3
surfaces for which the $E_7$ fiber is promoted to an $E_8$ fiber. The
other component is also a rational curve, parametrized by
$$
(e,f) = \left( - \frac{ (m^2-2)^2 }{ m (m-2) (m^2-4m+2)} , \frac{ 2(m^2-2m+2)(m^2-m-1)}{m (m-2)(m^2-4m+2) } \right).
$$
It corresponds to elliptic K3 surfaces with an extra $\I_2$ fiber.

The Hilbert modular surface is a double cover of a plane branched along
a sextic, and is therefore a K3 surface.  The transformation $f = e + t$
makes it a quartic in $e$, whose leading coefficient is a square.
We thus get an elliptic K3 surface whose Weierstrass equation
may be obtained by taking the Jacobian of this genus-$1$ curve.
After some elementary algebra, we get the elliptic K3 surface
\begin{align*}
y^3 &=  x^3 + (t^4+24 t^3+98 t^2+16 t+1) \, x^2 \\
  & \qquad + (128 t^5+2352 t^4+1088 t^3+64 t^2) \,x - (512 t^6+9216 t^5+704 t^4).
\end{align*}
It has reducible fibers of types $\I_6$ at $t = \infty$,
$\I_4$ at $t = 0$,
$\I_3$ at $t = -8 \pm 5 \sqrt{10}/2$, and
$\I_2$ at $t = 1/3$ and $t = -9 \pm 4 \sqrt{5}$,
giving a contribution of $A_5 \oplus A_3 \oplus A_1^3 \oplus A_2^2$
to the N\'{e}ron-Severi lattice. The trivial lattice thus has rank
$17$ and discriminant $1728$.

We also have the two sections
\begin{align*}
P_1 &= \big( 4 t (2 t^2+36 t+3), \, -4 t (t^2+18 t+1) (2 t^2+32 t+3)
\big) \\
P_2 &= \big( -(6 t^4+124 t^3+303 t^2+138 t+9)/10, \,
    {\sqrt{10}} (2t +3) (3t-1) (t^2+18 t+1) (2t^2+32 t+3) / 100 \Big).
\end{align*}
These have heights $1/12$ and $7/6$ respectively, and are orthogonal
with respect to the height pairing. Therefore, the Picard number is at
least $19$. Counting points modulo $7$ and $11$ proves that the Picard
number is $19$ (we thank Ronald van Luijk for carrying out such a
calculation), in agreement with Oda's calculations
\cite[pg.~109]{Oda}. The part of the N\'{e}ron-Severi lattice spanned
by the above sections with the trivial lattice has rank $19$ and
discriminant $168$. We showed that this lattice is $2$-saturated, so
it is the full N\'eron-Severi lattice.

The quotient of the Hilbert modular surface by the involution
$f \mapsto -f$\/ is the rational elliptic surface
$$
y^2 = (x+8e^2-8)\big(x^2+(17e^2-12e+8)x+ 8(3e+1)^2(e^2-2e+2) \big).
$$
It has reducible fibers of types $\I_6, \I_3, \I_2$ at $e = \infty, 8,
9$, respectively, and the \MoW\ group is generated by the
$6$-torsion section $\big(-4(2e^2 + e -11), 4(e-9)(e-8) \big)$;
indeed this is the universal elliptic curve over $X_1(6)$.

\subsection{Examples}

We list some points of small height and corresponding genus $2$ curves.

\begin{tabular}{l|c}
Rational point $(e,f)$ & Sextic polynomial $f_6(x)$ defining the genus $2$ curve $y^2 = f_6(x)$. \\
\hline \hline \\ [-2.5ex]
$(49/8, -47/8)$ & $ -12x^6 - 12x^5 - 21x^4 + 14x^3 + 39x^2 + 48x - 64 $ \\
$(31/6, 5)$ & $ 108x^6 + 108x^5 - 81x^4 + x^3 + 63x^2 - 33x + 9 $ \\
$(-1/12, -5/6)$ & $ 36x^5 + 78x^4 - 41x^3 - 129x^2 + 45x + 27 $ \\
$(-13/10, -4/5)$ & $ -72x^6 + 108x^5 + 135x^4 - 135x^3 - 219x^2 + 135x + 53 $ \\
$(49/8, 47/8)$ & $ -72x^6 + 216x^5 - 315x^4 + 162x^3 + 21x^2 - 72x - 16 $ \\
$(-23/14, -13/7)$ & $ -8x^6 + 60x^5 - 87x^4 - 163x^3 + 288x^2 + 324x + 27 $ \\
$(-29/12, -25/12)$ & $ 12x^6 + 132x^5 + 355x^4 - 90x^3 - 245x^2 + 36x - 44 $ \\
$(87/55, 12/55)$ & $ 44x^5 + 200x^4 - 422x^3 + 180x^2 - 81x $ \\
$(-75/28, 18/7)$ & $ -77x^6 + 147x^5 - 45x^4 - 335x^3 + 186x^2 - 180x - 432 $ \\
$(49/23, 43/23)$ & $ 46x^6 - 24x^5 + 252x^4 - 29x^3 + 468x^2 + 24x + 366 $ \\
$(-37/36, 35/36)$ & $ -18x^6 - 258x^5 - 475x^4 + 220x^3 - 325x^2 + 72x - 48 $ \\
$(19/8, 13/8)$ & $ 432x^6 + 216x^5 - 27x^4 - 502x^3 - 87x^2 + 36x + 116 $ \\
$(55/28, -43/28)$ & $ -496x^6 + 48x^5 - 545x^4 + 90x^3 - 257x^2 + 120x - 80 $ \\
$(-1/28, 15/28)$ & $ 314x^6 + 426x^5 + 555x^4 + 140x^3 - 195x^2 - 264x - 176 $ \\
$(-23/15, 22/15)$ & $ -586x^6 + 330x^5 - 512x^4 + 150x^3 - 110x^2 - 24x - 1 $ \\
$(-5/4, -1/4)$ & $ 8x^6 - 168x^5 - 269x^4 + 466x^3 + 451x^2 - 624x - 376 $ 
\end{tabular}

The section $P_1$ gives a rational curve
$$
(e,f) = \left( - \frac{ 2g^2+11}{ 4g-1}, \frac{ 2g^2-g-11}{4g-1} \right).
$$
However, the Brauer obstruction does not vanish identically on this locus.
The section $2 P_1$ gives a rational curve
$$
(e,f) = \left( -\frac{2g^2+3}{4g}, \frac{2g^2-3}{4g} \right)
$$
on the surface. Here, the Brauer obstruction does vanish
identically, yielding a family of genus $2$ curves with real
multiplication by $\sO_{40}$.

\section{Discriminant $41$}

\subsection{Parametrization}

We start with an elliptic K3 surface with fibers of type $A_5$ at $t = 0$
and $A_{10}$ at $t = \infty$, with a section of height
$41/66 = 4 - (4 \cdot 7)/11 - 5/6$.

A Weierstrass equation for this family is given by
$$
y^2 = x^3 + (t^2 + 2\,d\,f\,t + c\,f^2)\, x^2
  + 2\,t^2\, (d \, t + c\,f) \, x + c \ t^4,
$$
with
\begin{align*}
c &= r^2\,s^2 \big( 16 \,t^2 -8 \,(4 \,r\, s-16\, s-r) \,t
  + (4\,r \,s-16\, s+r)^2 \big), \\
d &= r\,s\,(4\,t-12\,r\,s+16\,s+r), \\
f &= (t+4\, s)/(4 \,r \,s).
\end{align*}

We identify the class of an $E_8$ fiber, and perform a $3$-neighbor
step to an elliptic fibration with $E_8$ and $A_7$ fibers.

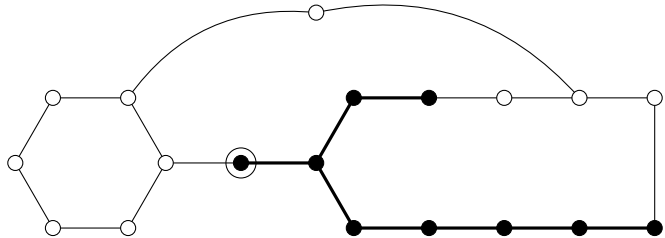
\begin{center}
\begin{tikzpicture}

\draw (0,0)--(0.5,0.866)--(1.5,0.866)--(2,0)--(1.5,-0.866)--(0.5,-0.866)--(0,0);
\draw (2,0)--(4,0)--(4.5,0.866)--(8.5,0.866)--(8.5,-0.866)--(4.5,-0.866)--(4,0);
\draw [bend right] (4,2) to (1.5,0.866);
\draw [bend left] (4,2) to (7.5,0.866);
\draw [very thick] (3,0)--(4,0);
\draw [very thick] (5.5,0.866)--(4.5,0.866)--(4,0)--(4.5,-0.866)--(8.5,-0.866);

\draw (3,0) circle (0.2);
\fill [white] (0,0) circle (0.1);
\fill [white] (0.5,0.866) circle (0.1);
\fill [white] (0.5,-0.866) circle (0.1);
\fill [white] (1.5,0.866) circle (0.1);
\fill [white] (1.5,-0.866) circle (0.1);
\fill [white] (2,0) circle (0.1);
\fill [black] (3,0) circle (0.1);
\fill [black] (4,0) circle (0.1);
\fill [black] (4.5, 0.866) circle (0.1);
\fill [black] (4.5, -0.866) circle (0.1);
\fill [black] (5.5, 0.866) circle (0.1);
\fill [black] (5.5, -0.866) circle (0.1);
\fill [white] (6.5, 0.866) circle (0.1);
\fill [black] (6.5, -0.866) circle (0.1);
\fill [white] (7.5, 0.866) circle (0.1);
\fill [black] (7.5, -0.866) circle (0.1);
\fill [white] (8.5, 0.866) circle (0.1);
\fill [black] (8.5, -0.866) circle (0.1);
\fill [white] (4,2) circle (0.1);

\draw (0,0) circle (0.1);
\draw (0.5,0.866) circle (0.1);
\draw (0.5,-0.866) circle (0.1);
\draw (1.5,0.866) circle (0.1);
\draw (1.5,-0.866) circle (0.1);
\draw (2,0) circle (0.1);
\draw (3,0) circle (0.1);
\draw (4,0) circle (0.1);
\draw (4.5, 0.866) circle (0.1);
\draw (4.5, -0.866) circle (0.1);
\draw (5.5, 0.866) circle (0.1);
\draw (5.5, -0.866) circle (0.1);
\draw (6.5, 0.866) circle (0.1);
\draw (6.5, -0.866) circle (0.1);
\draw (7.5, 0.866) circle (0.1);
\draw (7.5, -0.866) circle (0.1);
\draw (8.5, 0.866) circle (0.1);
\draw (8.5, -0.866) circle (0.1);
\draw (4,2) circle (0.1);

\end{tikzpicture}
\end{center}

This fibration has a section $P$ of height $41/8 = 4 + 2 - 7/8$. Next, we
take a $2$-neighbor step to go to a fibration with $E_8$ and $E_7$ fibers.

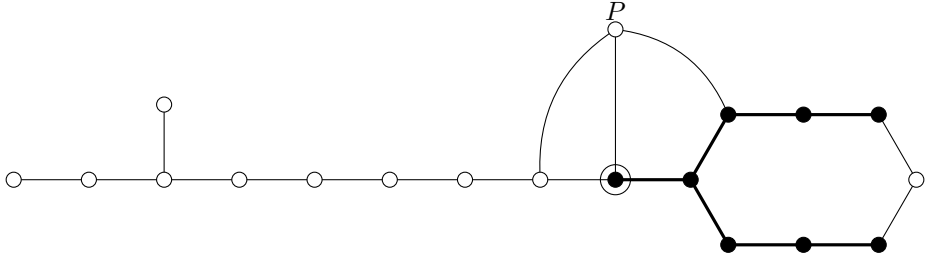
\begin{center}
\begin{tikzpicture}

\draw (7,0)--(7,2);
\draw (-1,0)--(8,0);
\draw (1,0)--(1,1);
\draw (8,0)--(8.5,0.866)--(10.5,0.866)--(11,0)--(10.5,-0.866)--(8.5,-0.866)--(8,0);

\draw [bend right] (7,2) to (6,0);
\draw [bend left] (7,2) to (8.5,0.866);
\draw [very thick] (7,0)--(8,0);
\draw [very thick] (10.5,0.866)--(8.5,0.866)--(8,0)--(8.5,-0.866)--(10.5,-0.866);

\draw (7,0) circle (0.2);
\fill [white] (-1,0) circle (0.1);
\fill [white] (0,0) circle (0.1);
\fill [white] (1,0) circle (0.1);
\fill [white] (2,0) circle (0.1);
\fill [white] (3,0) circle (0.1);
\fill [white] (4,0) circle (0.1);
\fill [white] (5,0) circle (0.1);
\fill [white] (6,0) circle (0.1);
\fill [white] (1,1) circle (0.1);
\fill [black] (7,0) circle (0.1);
\fill [black] (8,0) circle (0.1);
\fill [black] (8.5,0.866) circle (0.1);
\fill [black] (9.5,0.866) circle (0.1);
\fill [black] (10.5,0.866) circle (0.1);
\fill [black] (8.5,-0.866) circle (0.1);
\fill [black] (9.5,-0.866) circle (0.1);
\fill [black] (10.5,-0.866) circle (0.1);
\fill [white] (11,0) circle (0.1);
\fill [white] (7,2) circle (0.1);

\draw (-1,0) circle (0.1);
\draw (0,0) circle (0.1);
\draw (1,0) circle (0.1);
\draw (2,0) circle (0.1);
\draw (3,0) circle (0.1);
\draw (4,0) circle (0.1);
\draw (5,0) circle (0.1);
\draw (6,0) circle (0.1);
\draw (1,1) circle (0.1);
\draw (7,0) circle (0.1);
\draw (8,0) circle (0.1);
\draw (8.5,0.866) circle (0.1);
\draw (9.5,0.866) circle (0.1);
\draw (10.5,0.866) circle (0.1);
\draw (8.5,-0.866) circle (0.1);
\draw (9.5,-0.866) circle (0.1);
\draw (10.5,-0.866) circle (0.1);
\draw (11,0) circle (0.1);
\draw (7,2) circle (0.1);
\draw [above] (7,2) node {$P$};

\end{tikzpicture}
\end{center}
The intersection number of the new fiber $F'$ with the remaining
component of the $A_7$ fiber is $2$ and with the section $P$ is
$5$. Since these number are coprime, the new genus $1$ fibration
defined by $F'$ has a section.

Reading out the map to $\M_2$ from the $E_8 E_7$ fibration, we obtain
the following result.

\begin{theorem}
A birational model over $\Q$ for the Hilbert modular surface
$Y_{-}(41)$ as a double cover of $\Proj^2_{r,s}$ is given by the following
equation:
\begin{align*}
z^2 &= (4s+1)^4  r^4 - 16(8s^2+42s-1)(4s+1)^2  r^3 \\
&\quad   + 32(128s^4+4672s^3+1976s^2+248s+3)  r^2  \\
& \quad - 256(2688s^3+872s^2+82s-1)  r  + 256(16s+1)^3.
\end{align*}
It is a K3 surface of Picard number $19$.
\end{theorem}

\subsection{Analysis}

The branch locus is a rational curve, with a parametrization given by
$$
(r,s) = \left( -\frac{4(u-1)(u+1)^3}{(u^2-u-1)^2} , -\frac{1}{4u^2 (u-1)^2} \right).
$$

The Hilbert modular surface has a genus $1$ fibration to $\Proj^1_r$,
which is in fact an elliptic fibration, since the coefficient of $r^4$
is a perfect square.  Converting to the Jacobian form, we get the
Weierstrass equation
\begin{align*}
y^2 & =  x^3 + (t-1)  (t^3 + 23 t^2 + 96 t - 32) x^2 \\
   & \quad  + 16 (t-1) (4 t^4 + 53 t^3 - 217 t^2 + 112 t - 16) x.
\end{align*}

This is an elliptic K3 surface, with bad fibers of type $\I_6$ at $t = \infty$,
$\I_4$ at $t = 0$, $\III$ at $t =1$, $\I_3$ at $t = -15$, and
$\I_2$ at the four roots of $q(t) = 4 t^4 + 53 t^3 - 217 t^2 + 112 t - 16$.
This quartic $q(t)$ describes a dihedral Galois extension $K$\/ of~$\Q$,
quadratic over $\Q(\sqrt{41})$. Thus we get a trivial lattice of
rank $17$ and discriminant $2304$.  We also have a $2$-torsion
point $(0,0)$, and the following two non-torsion sections:
\begin{align*}
P_1 &=  \Big( -164 (t-1) ,\quad  4 \mu (t-1) (t+15) (5 t-6) \Big), \\
P_2 &= \Big( (5+\mu)(t-1)(8t^2+53t+9t\mu-185-29\mu)/16 , \\
& \qquad  (5+\mu)(t-1)(t+15)(4t-11+\mu)(8t^2+53t+9t\mu-185-29\mu)/64 \Big)
\end{align*}
where $\mu = \sqrt{41}$. The height matrix for $P_1$ and $P_2$ is 
$$
\frac{1}{3} \left(\begin{array}{rr}
4 & -2  \\
-2 & 1 \\
\end{array}\right).
$$ 
Therefore the Picard number is at least $19$. Counting points
modulo $7$ and $11$ shows that the Picard number must be $19$, in
agreement with \cite{Oda}. The sections above and the trivial lattice
therefore generate a lattice of rank $19$ and discriminant $512$.  We
showed that it is $2$-saturated, and is thus the full N\'eron-Severi
lattice.

\subsection{Examples}

We list some points of small height and corresponding genus $2$ curves.

\begin{tabular}{l|c}
Rational point $(r,s)$ & Sextic polynomial $f_6(x)$ defining the genus $2$ curve $y^2 = f_6(x)$. \\
\hline \hline \\ [-2.5ex]
$(52/3, -1/5)$ & $ -6x^6 + 3x^5 - 4x^4 + 3x^3 + 20x^2 - 36x + 20 $ \\
$(7/2, 13/12)$ & $ -x^6 - 15x^5 - 65x^4 - 57x^3 + 40x^2 + 6x + 20 $ \\
$(56/9, -43/40)$ & $ -56x^6 - 42x^5 + 64x^4 - 78x^3 - 29x^2 + 36x - 28 $ \\
$(52/7, -2/3)$ & $ -24x^6 + 36x^5 - 34x^4 - 51x^3 + 86x^2 - 87x + 18 $ \\
$(13, -7/12)$ & $ -15x^6 + 33x^5 + 23x^4 + 73x^3 - 72x^2 - 54x - 108 $ \\
$(32/5, -47/48)$ & $ -18x^6 - 3x^5 + 46x^4 - 111x^3 + 22x^2 + 48x - 84 $ \\
$(92/9, -11/14)$ & $ -4x^6 + 24x^5 - 17x^4 + 3x^3 - 125x^2 - 72x - 95 $ \\
$(60/13, 3)$ & $ -10x^6 + 33x^5 + 40x^4 - 111x^3 - 152x^2 + 60x + 140 $ \\
$(76/13, -4/3)$ & $ -24x^6 - 12x^5 - 50x^4 + 111x^3 + 10x^2 + 129x - 153 $ \\
$(28, -23/24)$ & $ -46x^6 + 69x^5 - 14x^4 + 169x^3 - 134x^2 - 84x - 72 $ \\
$(68/15, -11/4)$ & $ 48x^6 - 48x^5 - 119x^4 - 84x^3 + 145x^2 + 180x + 100 $ \\
$(23/6, 21/4)$ & $ 168x^5 + 85x^4 + 70x^3 + 229x^2 + 24x - 72 $ \\
$(40/3, -55/56)$ & $ -39x^6 + 36x^5 - 116x^4 + 186x^3 - 107x^2 + 240x - 200 $ \\
$(52/25, -1/24)$ & $ -112x^6 - 264x^5 + 25x^4 + 240x^3 - 125x^2 - 114x + 62 $ \\
$(4/13, -11/96)$ & $ 16x^6 + 24x^5 - 223x^4 + 274x^3 - 7x^2 + 216x - 120 $ \\
$(27/2, 11/76)$ & $ 12x^6 - 132x^5 - 219x^4 + 286x^3 + 201x^2 - 264x - 56 $ 
\end{tabular}

Next, we describe some curves which are a source of many rational
points. The specialization $s = -4$ gives a curve of genus $0$, with a
parametrization $r = -2(m^2 + 343)/(9(m-13))$. The specialization $r =
108/25$ also gives a rational curve, parametrized by $s =
-(4m-169)(4m+169)/(16(108m-10813))$. The Brauer obstruction vanishes
on both these loci, giving families of curves whose Jacobians have
real multiplication by $\sO_{41}$.

Finally, sections of the elliptic fibration also give rational curves
on the surface. For instance, the section $P_0$ is parametrized by
$$
r = \frac{4(256s^3-48s^2-16s-1)}{(4s+1)^2(16s-7)}.
$$
The Brauer obstruction vanishes here as well.

\section{Discriminant $44$}

\subsection{Parametrization}

We start with an elliptic K3 surface with fibers of type $A_{10}$ at
$t = 0$ and $D_6$ at $t = \infty$. This has N\'{e}ron-Severi lattice of
discriminant $11 \cdot 4 = 44$ and rank $2 + 10 + 6 = 18$. The
Weierstrass equation for this universal family can be written as
$$
y^2 = x^3 + a\, x^2 + 2\, b \, t^4 x +  c \, t^8,
$$
with
\begin{align*}
a &= s\,(r^2\,s^2-s^2+2\,s+2)\,t^3 -s\,(2\,r^2\,s-3\,s+2)\,t^2 + s\,(r^2-3)\,t   +  1,\\
b &= s^2\,\big( (2\,r^2\,s^2-2\,s^2+4\,s+1)/2\,t^2  -(r^2\,s-2\,s+1)\,t - 1 \big), \\
c &= s^4\,\big( (r^2\,s-s+2)\,t + 1 \big).
\end{align*}

We identify the class of an $E_8$ fiber and move to this fibration via
a $3$-neighbor step.

\begin{center}
\begin{tikzpicture}

\draw (0,0)--(6,0)--(6.5,0.866)--(10.5,0.866)--(10.5,-0.866)--(6.5,-0.866)--(6,0);
\draw (1,0)--(1,1); \draw (3,0)--(3,1);
\draw [very thick] (5,0)--(6,0);
\draw [very thick] (7.5,0.866)--(6.5,0.866)--(6,0)--(6.5,-0.866)--(10.5,-0.866);

\fill [white] (0,0) circle (0.1);
\fill [white] (1,0) circle (0.1);
\fill [white] (2,0) circle (0.1);
\fill [white] (3,0) circle (0.1);
\fill [white] (4,0) circle (0.1);
\fill [black] (5,0) circle (0.1);
\fill [black] (6,0) circle (0.1);
\fill [white] (1,1) circle (0.1);
\fill [white] (3,1) circle (0.1);
\fill [black] (6.5,0.866) circle (0.1);
\fill [black] (7.5,0.866) circle (0.1);
\fill [white] (8.5,0.866) circle (0.1);
\fill [white] (9.5,0.866) circle (0.1);
\fill [white] (10.5,0.866) circle (0.1);
\fill [black] (6.5,-0.866) circle (0.1);
\fill [black] (7.5,-0.866) circle (0.1);
\fill [black] (8.5,-0.866) circle (0.1);
\fill [black] (9.5,-0.866) circle (0.1);
\fill [black] (10.5,-0.866) circle (0.1);

\draw (5,0) circle (0.2);
\draw (0,0) circle (0.1);
\draw (1,0) circle (0.1);
\draw (2,0) circle (0.1);
\draw (3,0) circle (0.1);
\draw (4,0) circle (0.1);
\draw (5,0) circle (0.1);
\draw (6,0) circle (0.1);
\draw (1,1) circle (0.1);
\draw (3,1) circle (0.1);
\draw (6.5,0.866) circle (0.1);
\draw (7.5,0.866) circle (0.1);
\draw (8.5,0.866) circle (0.1);
\draw (9.5,0.866) circle (0.1);
\draw (10.5,0.866) circle (0.1);
\draw (6.5,-0.866) circle (0.1);
\draw (7.5,-0.866) circle (0.1);
\draw (8.5,-0.866) circle (0.1);
\draw (9.5,-0.866) circle (0.1);
\draw (10.5,-0.866) circle (0.1);

\end{tikzpicture}
\end{center}

The new elliptic fibration has $E_8$, $D_6$ and $A_1$ fibers, and a
section $P$ of height $11/2 = 4 + 2\cdot 1 - 1/2$. We identify an
$E_7$ fiber $F'$ below. Note that $P \cdot F' = 5$, while the excluded
component of the $D_6$ fiber intersects $F'$ in $2$. Since $2$ and $5$
are coprime, we see that the fibration defined by $F'$ has a
section. We move to it by a $2$-neighbor step.

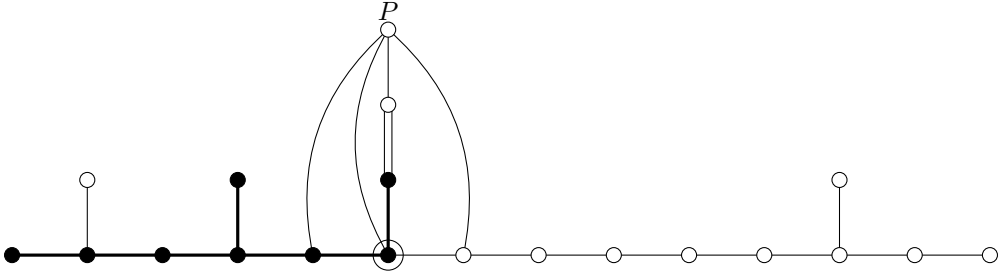
\begin{center}
\begin{tikzpicture}

\draw (0,0)--(13,0);
\draw (1,0)--(1,1);
\draw (3,0)--(3,1);
\draw (11,0)--(11,1);
\draw (5,0)--(5,1);
\draw (4.95,1)--(4.95,2);
\draw (5.05,1)--(5.05,2);
\draw (5,2)--(5,3);
\draw [bend right] (5,3) to (4,0);
\draw [bend right] (5,3) to (5,0);
\draw [bend left] (5,3) to (6,0);
\draw [very thick] (0,0)--(5,0)--(5,1);
\draw [very thick] (3,0)--(3,1);

\draw (5,0) circle (0.2);
\fill [black] (0,0) circle (0.1);
\fill [black] (1,0) circle (0.1);
\fill [black] (2,0) circle (0.1);
\fill [black] (3,0) circle (0.1);
\fill [black] (4,0) circle (0.1);
\fill [black] (5,0) circle (0.1);
\fill [white] (6,0) circle (0.1);
\fill [white] (1,1) circle (0.1);
\fill [black] (3,1) circle (0.1);
\fill [white] (7,0) circle (0.1);
\fill [white] (8,0) circle (0.1);
\fill [white] (9,0) circle (0.1);
\fill [white] (10,0) circle (0.1);
\fill [white] (11,0) circle (0.1);
\fill [white] (12,0) circle (0.1);
\fill [white] (13,0) circle (0.1);
\fill [white] (11,1) circle (0.1);
\fill [black] (5,1) circle (0.1);
\fill [white] (5,2) circle (0.1);
\fill [white] (5,3) circle (0.1);

\draw (0,0) circle (0.1);
\draw (1,0) circle (0.1);
\draw (2,0) circle (0.1);
\draw (3,0) circle (0.1);
\draw (4,0) circle (0.1);
\draw (5,0) circle (0.1);
\draw (6,0) circle (0.1);
\draw (1,1) circle (0.1);
\draw (3,1) circle (0.1);
\draw (7,0) circle (0.1);
\draw (8,0) circle (0.1);
\draw (9,0) circle (0.1);
\draw (10,0) circle (0.1);
\draw (11,0) circle (0.1);
\draw (12,0) circle (0.1);
\draw (13,0) circle (0.1);
\draw (11,1) circle (0.1);
\draw (5,1) circle (0.1);
\draw (5,2) circle (0.1);
\draw (5,3) circle (0.1);
\draw (5,3) [above] node{$P$};

\end{tikzpicture}
\end{center}

Calculating the Igusa-Clebsch invariants and following the rest of the
algorithm in Section \ref{method}, we obtain the following result.

\begin{theorem}
A birational model over $\Q$ for the Hilbert modular surface
$Y_{-}(44)$ as a double cover of $\Proj^2_{r,s}$ is given by the following
equation:
$$
z^2 =  (rs+s-1)(rs-s+1)(r^6s^2-r^4s^2+18r^2s-16s+27).
$$ 
It is an honestly elliptic surface, with arithmetic genus $2$ and
Picard number $29$.
\end{theorem}

\subsection{Analysis}

The extra involution is $\iota: (r,s) \mapsto (-r,s)$. The branch
locus has three components. The two simpler components $rs \pm (s-1) = 0$
are obviously rational curves, and a simple calculation shows that
they correspond to elliptic K3 surfaces for which the $D_6$ fiber gets
promoted to an $E_7$ fiber. The more complicated component of the branch locus,
corresponding to elliptic fibrations with an extra $\I_2$ fiber,
is also a rational curve; a parametrization is given by
$$
(r,s) = \left( \frac{4m}{m^2 + 3}, -\frac{(m^2+3)^3}{16(m^2-1)} \right).
$$

The right hand side of the equation defining the Hilbert modular surface
is quartic in $s$, whence the surface $Y_{-}(44)$ is elliptically
fibered over $\Proj^1_r$. The two linear factors (or the fact that the
coefficient of $s^4$ is a square) imply that there are sections, so we
may convert to the Jacobian form:
$$
y^2 = x^3 + 2(r^6-r^4-9r^2+11)\, x^2 + (r^2-1)^3(r^6+r^4+91r^2-121) \, x.
$$

This is an honestly elliptic surface, with $\chi = 3$. It has
reducible fibers of type $\I_2$ at $r =0$, $\I_6$ at $r = \pm 1$,
$\I_4$ at $r = \infty$, $\I_3$ at $r = \pm 2/\sqrt{3}$, and type
$\I_2$ at the six roots of $(r^3-3r^2+5r+11) \, (r^3+3r^2+5r-11)$
(both factors generate the cubic field $k_{-44}$ of discriminant $-44$). The
trivial lattice has rank $26$, leaving room for \MoW\ rank at
most~$4$.  We find the following sections, of which $P_0$ is
$2$-torsion, while $P_1$, $P_2$ and $P_3$ are linearly independent
non-torsion sections, orthogonal with respect to the height pairing
and of heights $7/6$, $3/2$ and $11/6$ respectively:
\begin{align*}
P_0 &= \big(0 , \, 0 \big), \\
P_1 &= \big( 11(r^2-1) , \quad 4 \sqrt{11}\,r (r^2-1)^3 (3r^2-4) \big), \\
P_2 &= \big( -(r+1)^3(r^3-3r^2+5r+11), \quad 6 \sqrt{-3}\,r (r+1)^3 (r^3-3r^2+5r+11) \big), \\
P_3 &= \big( (r-1)(r+1)^2(r^3-3r^2+5r+11), \quad 2 r^3 (r-1)(r+1)^2(r^3-3r^2+5r+11) \big) .
\end{align*}
Therefore, the Picard number is at least $29$. Analysis of the
associated quotient elliptic K3 surface and its twist (see below)
shows that the \MoW\ rank is exactly $3$ and therefore the Picard
number is exactly~$29$. The sections above together with the trivial
lattice generate a lattice of discriminant $133056 = 2^6 \cdot 3^3
\cdot 7 \cdot 11$. We checked that it is $2$- and $3$-saturated, and
so it is the entire N\'eron-Severi lattice. Therefore these sections
generate the \MoW\ group.

We next analyze the quotient of $Y_{-}(44)$ by the involution $\iota$.
Taking $t=r^2$, we find the equation
$$
z^2 = s^4\,t^4 -s^2\,(2\,s^2-2\,s+1)\,t^3 + s^2\,(s^2+16\,s+1)\,t^2 -s\,(2\,s-3)\,(17\,s-6)\,t + (s-1)^2\,(16\,s-27).
$$
This has an elliptic fibration over $\Proj^1_s$, and since the
coefficient of $t^4$ is a square, we may convert to the Jacobian,
which is
$$
y^2 = x^3 + s^2\,(s+2)\,(s+8)\,x^2 + 2\,s^3\,(6\,s^2+47\,s+9)\,x + s^4\,(36\,s^2+268\,s-27).
$$
This is an elliptic K3 surface, with reducible fibers of type $E_6$ at
$s = 0$, $\I_7$ at $s = \infty$, $\I_3$ at $s = -27/4$, and $\I_2$ at
the roots of $2s^3+14s^2-6s+1$, which generates the cubic field
$k_{-44}$. The trivial lattice has rank $19$. We find a non-torsion
section
$$
P = \big( 1-4\,s,\quad 2\,s^3+14\,s^2-6\,s+1 \big)
$$
of height $4 - 12/7 - 3/2 = 11/14$.
It is easy to show that $P$\/ generates the \MoW\ group:
the configuration of reducible fibers does not allow for either
nontrivial torsion or a section of height $11/14n^2$ for any integer $n>1$.
Therefore the K3 surface is singular, with N\'{e}ron-Severi lattice of
discriminant $-396 = -4 \cdot 9 \cdot 11$.

We may also analyze the quotient by considering it as an elliptic
surface over $\Proj^1_t$, and since the coefficient of $s^4$ is
$t^2(t-1)^2$, which is a square, there is a section. Converting to the
Jacobian, we get the elliptic K3 surface
$$
y^2 = x^3 + 2(t^3-t^2-9t+11) \, x^2 + (t-1)^3(t^3+t^2+91t-121) \, x
$$
which is also obtained by replacing $r^2$ by $t$ in the Weierstrass
equation of $Y_{-}(44)$. This elliptic fibration has bad fibers of
type $\I^*_6$, $\I_6$, and $\I_3$ at $t = \infty$, $1$, and $4/3$
respectively, and of type $\I_2$ at the roots of $t^3 + t^2 + 91t -
121$, which generates $k_{-44}$.  Therefore, the root lattice has rank
$18$. We find the sections
\begin{align*}
P_0 &= \big(0, 0\big), \\
P_1 &= \big((1-t)(t -3\sqrt{-3})^2, 2(3-\sqrt{-3})(t-1)(t-3\sqrt{-3})(9t-6-2\sqrt{-3})/3 \big) \\
P_2 &= \big((1-t)(t+ 3\sqrt{-3})^2, 2(3+\sqrt{-3})(t-1)(t+3\sqrt{-3})(9t-6+2\sqrt{-3})/3 \big),
\end{align*}
with $P_0$ being $2$-torsion, and $P_1$ and $P_2$ having height pairing matrix
$$
\left(\begin{array}{cc} 5/3 & 1/6 \\ 1/6 & 5/3
\end{array} \right).
$$
These sections, along with the trivial lattice, generate a lattice
of rank $20$ and discriminant $-396$, which must therefore be the
entire N\'{e}ron-Severi lattice.  The \MoW\ rank of this elliptic
surface is $2$.

Finally, we consider the quadratic twist of the above elliptic K3
surface, which is the quotient of $Y_{-}(44)$ by the involution
$\iota' : (r,s,z) \mapsto (-r,s,-z)$. It is given by the equation
$$
y^2 = x^3 + 2t(t^3-t^2-9t+11)\, x^2 + t^2(t-1)^3(t^3+t^2+91t-121) \, x.
$$
This is an elliptic K3 surface, with reducible fibers of type
$\I^*_1$, $\I_6$, $\I_3$ and $\I_2$ at $t = 0$, $1$, $4/3$ and $\infty$
respectively, and $\I_2$ at the roots of $t^3 + t^2 + 91t - 121$.
The trivial lattice has rank $18$. We find the sections
\begin{align*}
P_0 &= \big( 0, 0 \big) \\
P_1 &= \big( 11 t (t-1)^3 , 4 \sqrt{11} t^2 (t-1)^3 (3t - 4) \big),
\end{align*}
the first being $2$-torsion, and the second of height $7/12$. The
Picard number is therefore at least~$19$.  Counting points modulo $5$
and $7$ shows that it is exactly $19$; therefore the \MoW\ rank is
exactly $1$. These sections and the trivial lattice span a sublattice
of discriminant $168$ of the N\'eron-Severi lattice. This sublattice
is $2$-saturated, since the configuration of fibers does not allow for
a section of height $7/48$, and we can easily check that the elliptic
surface does not have $4$-torsion or other $2$-torsion
sections. Therefore, we have the entire N\'eron-Severi lattice, and
$P_0$ and $P_1$ generate the \MoW\ group.

The calculation of the \MoW\ ranks of the quotient elliptic
surface and its quadratic twist allows us to conclude that the
\MoW\ rank of the original (honestly) elliptic surface is $1 + 2 = 3$.

\subsection{Examples}

We list some points of small height and corresponding genus $2$ curves.

\begin{tabular}{l|c}
Point $(r,s)$ & Sextic polynomial $f_6(x)$ defining the genus $2$ curve $y^2 = f_6(x)$. \\
\hline \hline \\ [-2.5ex]
$(-2, -7/6)$ & $ 101x^6 - 60x^5 + 2x^4 + 92x^2 + 48x + 24 $ \\
$(2, -7/6)$ & $ 15x^6 - 168x^5 + 170x^4 - 112x^3 + 20x^2 - 8 $ \\
$(4, -7/15)$ & $ -24x^6 + 48x^5 + 52x^4 + 144x^3 + 238x^2 + 588x - 161 $ \\
$(-4, -7/15)$ & $ 56x^6 + 196x^4 - 320x^3 + 250x^2 - 480x + 723 $ \\
$(-5, -13/12)$ & $ -144x^6 + 120x^5 + 265x^4 + 700x^3 - 425x^2 - 750x - 1750 $ \\
$(2, 7/12)$ & $ 696x^6 - 2112x^5 + 7492x^4 - 7032x^3 + 10234x^2 - 756x + 5103 $ \\
$(5, -13/12)$ & $ 12x^6 - 60x^5 - 145x^4 + 400x^3 - 1225x^2 - 1500x + 11500 $ \\
$(2, 13/36)$ & $ -5193x^6 - 5124x^5 - 16906x^4 - 11576x^3 - 17212x^2 - 6240x - 5304 $ \\
$(-3/4, 36/7)$ & $ -4744x^6 - 15552x^5 + 7596x^4 + 42048x^3 - 20310x^2 - 32112x + 33553 $ \\
$(-2, 13/36)$ & $ -18261x^6 + 13668x^5 + 65210x^4 - 41512x^3 - 74284x^2 + 18816x + 24696 $ \\
$(3/4, 36/7)$ & $ -6504x^6 + 22608x^5 + 28428x^4 - 94288x^3 - 66510x^2 + 103092x + 73009 $ \\
$(-2, 7/12)$ & $ 25389x^6 - 97062x^5 - 511x^4 + 240860x^3 - 5989x^2 - 127758x - 83849 $ \\
$(-5, -9/32)$ & $ -691156x^6 + 20220x^5 - 232521x^4 + 19406x^3 - 22521x^2 + 2484x - 564 $ \\
$(2, -61/42)$ & $ -629624x^6 + 272400x^5 - 383596x^4 - 704000x^3 - 60778x^2 - 194100x - 32193 $ \\
$(-2, -61/42)$ & $ -550872x^6 + 1549296x^5 - 1810124x^4 - 2005984x^3 + 3719134x^2 - 1321788x - 4862401 $ \\
$(5, -9/32)$ & $ 1781676x^6 - 5240052x^5 + 5462991x^4 - 5705734x^3 + 1769571x^2 + 1002576x - 1011776 $ 
\end{tabular}

The specializations $r = \pm 1$ give rational curves on the surface,
but points on these correspond to decomposable abelian surfaces, which
therefore have an endomorphism ring strictly larger than $\sO_{44}$.
The section $P_0$ is a rational curve, given by
$s = (r^4+27)/\big(2(r^2-1)(r^2-8)\big)$.
However, the Brauer obstruction does not vanish identically on it.

\section{Discriminant $53$}

\subsection{Parametrization}

We start with an elliptic K3 surface with $A_8$, $A_6$ and $A_1$
fibers, and a section of height $53/126 = 4 - 1/2 - 6/7 - 20/9$. A
Weierstrass equation for this family is given by
$$
y^2 = x^3 + a\,x^2 + 2\,b\,t\,(t-h)\,x + c\,t^2\,(t-h)^2,
$$
with
\begin{align*}
a &= g^2\,t^4 + \big(4\,(h+1)^2 - 2\,g \big)\,t^3 + \big(4\,h^2+4\,g\,h+6\,g-3\big)\,t^2 + 2\,\big(8\,h^2-4\,g\,h+8\,h+1\big)\,t + (2\,h+1)^2, \\
b &= -4\,(h-g+1)\,\big( g\,(2\,h+1)\,t^2 + (6\,h^2-2\,g\,h+6\,h+1)\,t + (2\,h + 1)^2 \big), \\
c &= 16\,(h-g+1)^2\,(2\,h+1)^2.
\end{align*}

To transform to a fibration with $E_8$ and $E_7$ fibers, we first
identify an $E_7$ fiber, and move to the associated elliptic fibration
via a $2$-neighbor step.

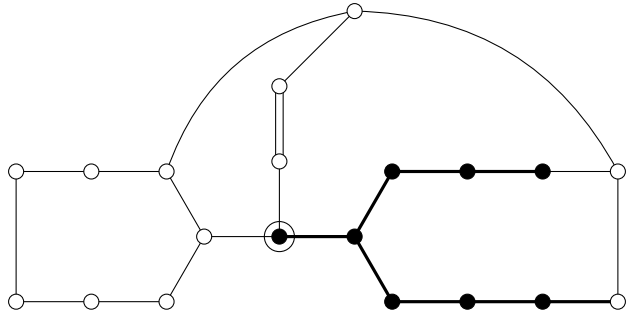
\begin{center}
\begin{tikzpicture}

\draw (0,0)--(2,0)--(2.5,0.866)--(5.5,0.866)--(5.5,-0.866)--(2.5,-0.866)--(2,0);
\draw (0,0)--(-0.5,0.866)--(-2.5,0.866)--(-2.5,-0.866)--(-0.5,-0.866)--(0,0);
\draw (1,0)--(1,1);
\draw (1.05,1)--(1.05,2);
\draw (0.95,1)--(0.95,2);
\draw (1,2)--(2,3);
\draw [bend right] (2,3) to (-0.5,0.866);
\draw [bend left] (2,3) to (5.5,0.866);
\draw [very thick] (1,0)--(2,0);
\draw [very thick] (4.5,0.866)--(2.5,0.866)--(2,0)--(2.5,-0.866)--(5.5,-0.866);

\draw (1,0) circle (0.2);
\fill [white] (0,0) circle (0.1);
\fill [black] (1,0) circle (0.1);
\fill [black] (2,0) circle (0.1);
\fill [black] (2.5,0.866) circle (0.1);
\fill [black] (3.5,0.866) circle (0.1);
\fill [black] (4.5,0.866) circle (0.1);
\fill [white] (5.5,0.866) circle (0.1);
\fill [black] (2.5,-0.866) circle (0.1);
\fill [black] (3.5,-0.866) circle (0.1);
\fill [black] (4.5,-0.866) circle (0.1);
\fill [white] (5.5,-0.866) circle (0.1);

\fill [white] (-0.5,0.866) circle (0.1);
\fill [white] (-1.5,0.866) circle (0.1);
\fill [white] (-2.5,0.866) circle (0.1);
\fill [white] (-0.5,-0.866) circle (0.1);
\fill [white] (-1.5,-0.866) circle (0.1);
\fill [white] (-2.5,-0.866) circle (0.1);
\fill [white] (1,1) circle (0.1);
\fill [white] (1,2) circle (0.1);

\fill [white] (2,3) circle (0.1);

\draw (0,0) circle (0.1);
\draw (1,0) circle (0.1);
\draw (2,0) circle (0.1);
\draw (2.5,0.866) circle (0.1);
\draw (3.5,0.866) circle (0.1);
\draw (4.5,0.866) circle (0.1);
\draw (5.5,0.866) circle (0.1);
\draw (2.5,-0.866) circle (0.1);
\draw (3.5,-0.866) circle (0.1);
\draw (4.5,-0.866) circle (0.1);
\draw (5.5,-0.866) circle (0.1);

\draw (-0.5,0.866) circle (0.1);
\draw (-1.5,0.866) circle (0.1);
\draw (-2.5,0.866) circle (0.1);
\draw (-0.5,-0.866) circle (0.1);
\draw (-1.5,-0.866) circle (0.1);
\draw (-2.5,-0.866) circle (0.1);
\draw (1,1) circle (0.1);
\draw (1,2) circle (0.1);

\draw (2,3) circle (0.1);

\end{tikzpicture}
\end{center}

The resulting elliptic fibration has $E_7$ and $A_8$ fibers, and a
section $P$ of height $53/18 = 4 + 2 \cdot 1 - 3/2 - 4 \cdot 5/9$. We
identify a fiber $F$\/ of type $E_8$ and perform a $3$-neighbor step to
move to the associated elliptic fibration. Note that $P \cdot F = 4$,
while the remaining component of the $A_8$ fiber intersects $F$\/ with
multiplicity $3$. Therefore the new elliptic fibration has a section.

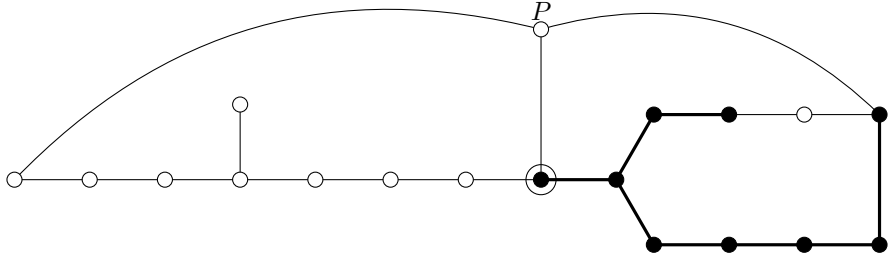
\begin{center}
\begin{tikzpicture}

\draw (-1,0)--(7,0)--(7.5,0.866)--(7.5,0.866)--(10.5,0.866)--(10.5,-0.866)--(7.5,-0.866)--(7,0);
\draw (2,0)--(2,1);
\draw [bend right] (6,2) to (-1,0);
\draw [bend left] (6,2) to (10.5,0.866);
\draw (6,2)--(6,0);
\draw [very thick] (6,0)--(7,0);
\draw [very thick] (8.5,0.866)--(7.5,0.866)--(7,0)--(7.5,-0.866)--(10.5,-0.866)--(10.5,0.866);

\draw (6,0) circle (0.2);
\fill [white] (-1,0) circle (0.1);
\fill [white] (0,0) circle (0.1);
\fill [white] (1,0) circle (0.1);
\fill [white] (2,0) circle (0.1);
\fill [white] (3,0) circle (0.1);
\fill [white] (4,0) circle (0.1);
\fill [white] (5,0) circle (0.1);
\fill [black] (6,0) circle (0.1);
\fill [white] (2,1) circle (0.1);
\fill [black] (7,0) circle (0.1);
\fill [black] (7.5,0.866) circle (0.1);
\fill [black] (7.5,-0.866) circle (0.1);
\fill [black] (8.5,0.866) circle (0.1);
\fill [black] (8.5,-0.866) circle (0.1);
\fill [white] (9.5,0.866) circle (0.1);
\fill [black] (9.5,-0.866) circle (0.1);
\fill [black] (10.5,0.866) circle (0.1);
\fill [black] (10.5,-0.866) circle (0.1);
\fill [white] (6,2) circle (0.1);

\draw (-1,0) circle (0.1);
\draw (0,0) circle (0.1);
\draw (1,0) circle (0.1);
\draw (2,0) circle (0.1);
\draw (3,0) circle (0.1);
\draw (4,0) circle (0.1);
\draw (5,0) circle (0.1);
\draw (6,0) circle (0.1);
\draw (2,1) circle (0.1);
\draw (7,0) circle (0.1);
\draw (7.5,0.866) circle (0.1);
\draw (7.5,-0.866) circle (0.1);
\draw (8.5,0.866) circle (0.1);
\draw (8.5,-0.866) circle (0.1);
\draw (9.5,0.866) circle (0.1);
\draw (9.5,-0.866) circle (0.1);
\draw (10.5,0.866) circle (0.1);
\draw (10.5,-0.866) circle (0.1);
\draw (6,2) circle (0.1);
\draw (6,2) [above] node{$P$};

\end{tikzpicture}
\end{center}

The new elliptic fibration has the requisite $E_8$ and $E_7$ fibers,
and therefore we may read out the Igusa-Clebsch invariants, and
describe the branch locus, which corresponds to elliptic K3 surfaces
having an extra $\I_2$ fiber. 

\begin{theorem}
A birational model over $\Q$ for the Hilbert modular surface
$Y_{-}(53)$ as a double cover of $\Proj^2_{g,h}$ is given by the following
equation:
\begin{align*}
z^2 &= -27\,h^4\,g^4 -2\,h^3\,(13\,h^2+9\,h+9)\,g^3 \\
    & \qquad -(11\,h^6+138\,h^5+383\,h^4+506\,h^3+353\,h^2+120\,h+16)\,g^2 \\
    & \qquad  -2\,(h+1)^2\,(52\,h^4+99\,h^3+65\,h^2+19\,h+2)\,g \\
    & \qquad -(h+1)^4\,(44\,h^3+76\,h^2+40\,h+7).
\end{align*}
It is an honestly elliptic surface, with arithmetic genus $2$ and
Picard number $28$.
\end{theorem}

\subsection{Analysis}

The branch locus is a curve of genus $1$, and the change of coordinates
\begin{align*}
g &=  -\frac{ (x^2-x+1)(x^2+3x+3)}{(x+1)(2x+1)\, y + x^4+3x^3+5x^2+3} \\
h &= \frac{(x^5-x^4+4x^3-4x^2+x+1)\, y + x^2(x^6-x^5+2x^4+2x^3-x+1)}{(x^2-x+1)^2 \big((x+1)(2x+1) \, y + x^4+3x^3+5x^2+3 \big)}
\end{align*}
transforms it to the elliptic curve $y^2 + xy +y = x^3 - x^2$ of
conductor $53$, which is isomorphic to $X_0(53)/ \langle w \rangle$,
where $w$ is the Atkin-Lehner involution.

This surface has a genus $1$ fibration over $\Proj^1_h$,
making it honestly elliptic with $\chi = 3$.  We could find no sections
over~$\Q$, although there certainly exist sections over~$\Qbar$,
since the coefficient of $g^4$ is a square in~$\Q(\sqrt{-3})$.
The Jacobian of this fibration has the following equation
(after making the linear fractional transformation $h = -t/(t+1)$
on the base, and some Weierstrass transformations):
\begin{align*}
y^2 &= x^3 + (t^6-18\,t^5+55\,t^4+106\,t^3-179\,t^2+24\,t-16)\,x^2 \\
    & \qquad -8\,(t-1)\,t^2\,(37\,t^5-471\,t^4-140\,t^3+1121\,t^2-309\,t+248)\,x \\
   & \qquad  -16\,(t-1)^2\,t^4\,(196\,t^5-2797\,t^4+2712\,t^3+8606\,t^2-3084\,t+3115).
\end{align*}
This surface has reducible fibers of type $\I_7$, $\I_4$ and $\I_3$ at
$t = \infty, 0, 1$ respectively, type $\I_2$ at the roots of $7t^3 -
99 t^2 + 104t - 32$ (which generates the cubic field of discriminant
$-2^2 \cdot 53$), and $\I_3$ at the roots of $t^5 - 11t^4 - 11t^3 +
6t^2 - 3t - 9$ (which generates the quintic field of discriminant
$-3^2 \cdot 53^2$, and whose roots generate a dihedral $D_{10}$
extension unramified over its quadratic subfield $\Q(\sqrt{-3 \cdot
  53})$).  The trivial lattice has rank $26$, leaving room for \MoW\
rank at most~$4$. We find the sections 
\begin{align*}
P_1 &= \big( -4(7t^6-106t^5+189t^4+202t^3-778t^2+342t-216)/53, \\
 & \qquad  4(5t-3)(7t^3-99t^2+104t-32)(t^5-11t^4-11t^3+6t^2-3t-9)/(53\sqrt{53}) \big) \\
P_2 &= \big( -4(49t^6-63t^5+99t^4+162t^3-351t^2+162t-108)/27, \\
& \qquad 4t(637t^8-378t^7-1485t^6+3186t^5-1755t^4-1782t^3+3942t^2-243t-972)/(81\sqrt{-3})  \big) 
\end{align*}
of heights $7/6$ and $21/4$ respectively, orthogonal with respect to
the height pairing. Therefore, the Mordell-Weil rank is least $2$. By
Oda's calculations \cite{Oda}, we deduce that the Mordell-Weil rank is
exactly $2$. The sections $P_1$ and $P_2$ and the trivial lattice span
a lattice of discriminant $1000188 = 2^2 \cdot 3^6 \cdot
7^3$. Checking that it is saturated at $2$, $3$ and $7$, we deduce
that it must be the full N\'eron-Severi lattice.

\subsection{Examples}

We list some points of small height and corresponding genus $2$ curves.

\begin{tabular}{l|c}
Rational point $(g,h)$ & Sextic polynomial $f_6(x)$ defining the genus $2$ curve $y^2 = f_6(x)$. \\
\hline \hline \\ [-2.5ex]
$(11/152, -16/19)$ & $ -2332x^6 - 902x^5 - 5060x^4 - 17111x^3 - 5995x^2 - 17545x - 27951 $ \\
$(8/21, -2/3)$ & $ -8788x^6 + 34200x^5 - 22425x^4 - 11907x^3 - 16230x^2 - 35604x + 34024 $ \\
$(72/49, -13/7)$ & $ 816x^6 - 1944x^5 - 13459x^4 + 24712x^3 + 37733x^2 + 7596x - 2300 $ \\
$(75/26, -19/4)$ & $ -106605x^6 - 62661x^5 + 467345x^4 + 193313x^3 - 691816x^2 - 149592x + 346546 $ \\
$(12/5, -10)$ & $ 139968x^6 + 471744x^5 - 1301409x^4 - 77363x^3 + 671633x^2 + 236496x + 18756 $ \\
$(-72/31, -4)$ & $ -138482x^6 + 1643417x^5 + 2645029x^4 - 1507309x^3 - 2429567x^2 $ \\
& $ + 1188320x + 36288 $ \\
$(2/135, -10/11)$ & $ 2175200x^6 - 2750760x^5 + 2725545x^4 + 7678368x^3 - 5205621x^2$ \\
& $ + 3781674x + 9014158 $ \\
$(16/63, -5)$ & $ 3829988x^6 + 11621820x^5 - 19617225x^4 - 25097450x^3 + 29201451x^2$ \\
& $ + 14626080x - 14560512 $ \\
$(18/103, -8/11)$ & $ -4788022x^6 - 21151494x^5 - 18288935x^4 - 20340320x^3 - 61042325x^2$ \\
& $ + 10128456x - 40124160 $ \\
$(-39/76, -29/16)$ & $ 394632000x^6 - 1964113200x^5 - 1523778060x^4 + 4757784967x^3 $  \\
& $ + 148400811x^2 - 3811137819x + 540964447 $ \\
$(55/117, -79/144)$ & $ -5511986931x^6 + 20881501795x^5 + 17115817125x^4 + 19864594645x^3 $ \\
& $ + 1353729618x^2 - 16117938900x + 5833685448 $ \\
$(49/135, -61/75)$ & $ -26148549648x^6 - 278797809744x^5 - 748507062651x^4 + 329438683288x^3$ \\
& $ + 1420002530997x^2 - 1751808944796x + 1934174962804 $ \\
$(-50/117, -13/9)$ & $ -159682912000x^6 - 1077016472800x^5 - 2039981245815x^4 - 762577047304x^3$ \\
& $ + 6811301171385x^2 - 4055008902300x + 449504680500 $ \\
$(81/91, -13/7)$ & $ 134236157214x^6 + 962817170858x^5 - 12198892111873x^4 + 23659009829816x^3 $ \\
& $ + 9649525790385x^2 - 12776814846900x - 8264106337500 $ \\
$(59/153, -17/27)$ & $ -687158622928816x^6 + 23483931596064x^5 - 14038441316573x^4 - 893569395800x^3 $ \\
 & $ - 20141231607x^2 - 200112822x - 748062 $ 
\end{tabular}

We could not find any curves of genus $0$ which were not
contained in the ``bad locus'' where the abelian surfaces have a
strictly larger ring of endomorphisms.

\section{Discriminant $56$}

\subsection{Parametrization}

We start with a K3 elliptic surface with $D_6$, $A_8$ and $A_1$
fibers, and a section $P$ of height $56/72 = 7/9 = 4 - 1 - 20/9$. The
Weierstrass equation of this family is
$$
y^2 = x^3 + a x^2 + 2 b t^2 (\lambda t - \mu) \,x + ct^4 (\lambda t - \mu)^2,
$$
with
\begin{align*}
\lambda &= (2h - g^2 + 1), \qquad \mu = -2h, \\
c &= -(g^2-1)^2\big( 4(g^2-1)t-(h+2)^2 \big), \\
b &= (g^2-1)\big( 2(g^2-1)t^2 -(h^2+4h-4g^2+8)t - (h+2)^2 \big), \\
a &= (g^2-1)^2 t^3 + (h+2)(h-2g^2+4)t^2 + 2(h^2+4h-2g^2+6)t + (h+2)^2. \\
\end{align*}

To find an $E_8 E_7$ fibration on such a K3 surface,
we first identify the class of an $E_7$ fiber below.

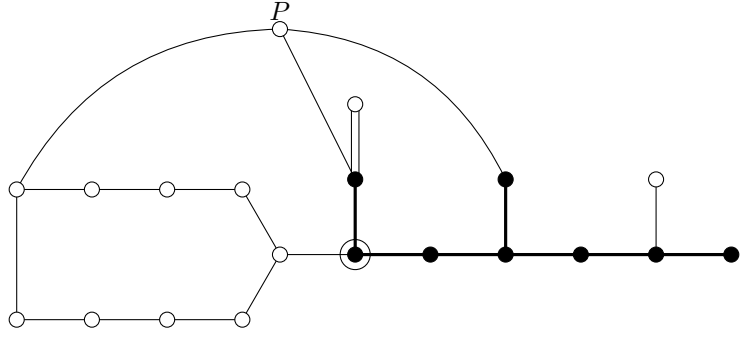
\begin{center}
\begin{tikzpicture}
\draw (2.05,1)--(2.05,2);
\draw (1.95,1)--(1.95,2);
\draw (1,0)--(0.5,0.866)--(-2.5,0.866)--(-2.5,-0.866)--(0.5,-0.866)--(1,0)--(7,0);
\draw (4,0)--(4,1);
\draw (6,0)--(6,1);
\draw (2,0)--(2,1);
\draw [very thick] (2,1)--(2,0)--(7,0);
\draw [very thick] (4,0)--(4,1);
\draw [bend right] (1,3) to (-2.5,0.866);
\draw [bend left] (1,3) to (4,1);
\draw (1,3)--(2,1);

\fill [white] (2,2) circle (0.1);
\fill [black] (2,1) circle (0.1);
\fill [black] (2,0) circle (0.1);
\fill [black] (3,0) circle (0.1);
\fill [black] (4,0) circle (0.1);
\fill [black] (5,0) circle (0.1);
\fill [black] (6,0) circle (0.1);
\fill [black] (7,0) circle (0.1);
\fill [white] (1,0) circle (0.1);
\fill [black] (4,1) circle (0.1);
\fill [white] (6,1) circle (0.1);

\fill [white] (1,3) circle (0.1);
\fill [white] (0.5,0.866) circle (0.1);
\fill [white] (0.5,-0.866) circle (0.1);
\fill [white] (-0.5,0.866) circle (0.1);
\fill [white] (-0.5,-0.866) circle (0.1);
\fill [white] (-1.5,0.866) circle (0.1);
\fill [white] (-1.5,-0.866) circle (0.1);
\fill [white] (-2.5,0.866) circle (0.1);
\fill [white] (-2.5,-0.866) circle (0.1);

\draw (1,3) [above] node{$P$};
\draw (1,3) circle (0.1);
\draw (2,0) circle (0.2);
\draw (2,2) circle (0.1);
\draw (2,1) circle (0.1);
\draw (2,0) circle (0.1);
\draw (3,0) circle (0.1);
\draw (4,0) circle (0.1);
\draw (5,0) circle (0.1);
\draw (6,0) circle (0.1);
\draw (7,0) circle (0.1);
\draw (1,0) circle (0.1);
\draw (4,1) circle (0.1);
\draw (6,1) circle (0.1);

\draw (0.5,0.866) circle (0.1);
\draw (-0.5,0.866) circle (0.1);
\draw (-1.5,0.866) circle (0.1);
\draw (-2.5,0.866) circle (0.1);
\draw (0.5,-0.866) circle (0.1);
\draw (-0.5,-0.866) circle (0.1);
\draw (-1.5,-0.866) circle (0.1);
\draw (-2.5,-0.866) circle (0.1);

\end{tikzpicture}
\end{center}

This converts it to an $A_8 E_7$ fibration. Note that this fibration
has a section, because the section $P$ intersects $F'$ in $3$,
whereas (for instance) the non-identity component of the $A_1$ fiber
intersects $F'$ in $2$, and these two numbers are coprime.

The new fibration has a section of height $56/18 = 28/9 = 4 - 8/9$,
which intersects the identity component of the $E_7$ fiber and
component $1$ of the $A_8$ fiber. We then do a $3$-neighbor step to
get an $E_8 E_7$ fibration.

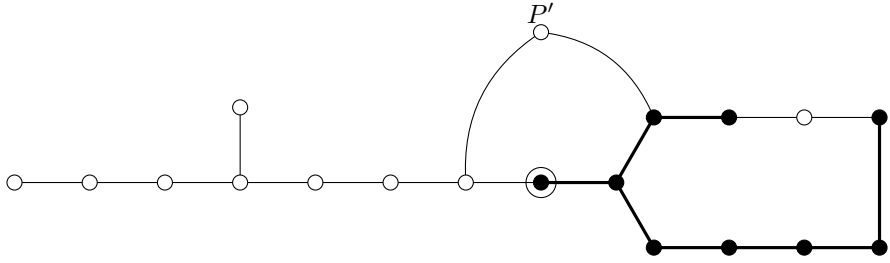
\begin{center}
\begin{tikzpicture}

\draw (-1,0)--(7,0)--(7.5,0.866)--(10.5,0.866)--(10.5,-0.866)--(7.5,-0.866)--(7,0);
\draw (2,0)--(2,1);
\draw [very thick] (7,0)--(6,0);
\draw [very thick] (8.5,0.866)--(7.5,0.866)--(7,0)--(7.5,-0.866)--(10.5,-0.866)--(10.5,0.866);
\draw [bend right] (6,2) to (5,0);
\draw [bend left] (6,2) to (7.5,0.866);

\fill [white] (6,2) circle (0.1);
\fill [white] (-1,0) circle (0.1);
\fill [white] (0,0) circle (0.1);
\fill [white] (1,0) circle (0.1);
\fill [white] (2,0) circle (0.1);
\fill [white] (3,0) circle (0.1);
\fill [white] (4,0) circle (0.1);
\fill [white] (5,0) circle (0.1);
\fill [black] (6,0) circle (0.1);
\fill [white] (2,1) circle (0.1);
\fill [black] (7,0) circle (0.1);
\fill [black] (7.5,0.866) circle (0.1);
\fill [black] (7.5,-0.866) circle (0.1);
\fill [black] (8.5,0.866) circle (0.1);
\fill [black] (8.5,-0.866) circle (0.1);
\fill [white] (9.5,0.866) circle (0.1);
\fill [black] (9.5,-0.866) circle (0.1);
\fill [black] (10.5,0.866) circle (0.1);
\fill [black] (10.5,-0.866) circle (0.1);

\draw [above] (6,2) node{$P'$};
\draw (6,2) circle (0.1);
\draw (6,0) circle (0.2);
\draw (-1,0) circle (0.1);
\draw (0,0) circle (0.1);
\draw (1,0) circle (0.1);
\draw (2,0) circle (0.1);
\draw (3,0) circle (0.1);
\draw (4,0) circle (0.1);
\draw (5,0) circle (0.1);
\draw (6,0) circle (0.1);
\draw (2,1) circle (0.1);
\draw (7,0) circle (0.1);
\draw (7.5,0.866) circle (0.1);
\draw (7.5,-0.866) circle (0.1);
\draw (8.5,0.866) circle (0.1);
\draw (8.5,-0.866) circle (0.1);
\draw (9.5,0.866) circle (0.1);
\draw (9.5,-0.866) circle (0.1);
\draw (10.5,0.866) circle (0.1);
\draw (10.5,-0.866) circle (0.1);

\end{tikzpicture}
\end{center}

The new fiber $F''$ satisfies $P' \cdot F'' = 2$, while the excluded
component of the $A_8$ fiber intersects $F''$ in $3$. Since these are
coprime, the fibration defined by $F''$ has a section.

Now we can read out the Igusa-Clebsch invariants, and describe the
branch locus of $Y_{-}(56)$ as a double over of $\Proj^2_{g,h}$. 

\begin{theorem}
A birational model over $\Q$ for the Hilbert modular surface
$Y_{-}(56)$ as a double cover of $\Proj^2_{g,h}$ is given by the
following equation:
\begin{align*}
z^2 =& (2h-g^2+1) (2h^5+27g^2h^4-11h^4+72g^2h^3-24h^3-40g^2h^2 \\
     & \quad +104h^2-160g^2h+192h+64g^4-144g^2+80).
\end{align*}
It is a surface of general type.
\end{theorem}

\subsection{Analysis}

The extra involution is $g \mapsto -g$.  The branch locus has two
components.  Points on the simpler component $2h-g^2+1=0$ (which is
clearly a rational curve) correspond to elliptic K3 surfaces for which
the $A_1$ and $D_6$ fibers merge and get promoted to a $D_8$ fiber.
The other component corresponds to elliptic K3 surfaces with an extra
$\I_2$ fiber. It is also of genus $0$, and a parametrization is given
by
$$
(g,h) = \left( \frac{s(s^4+22s^2-7)}{(3s^2+1)^2} , -\frac{(s^2-1)(s^2-5)}{2(3s^2+1)} \right).
$$

The Hilbert modular surface $Y_{-}(56)$ is of general type.

We now analyze the quotient of the Hilbert modular surface by the
involution $(g,h,z) \mapsto (-g,h,z)$. Setting $f = g^2$, the right
hand side becomes a cubic in $f$. After some elementary Weierstrass
transformations, we get the equation
$$
y^2 = x^3 -(27h^4+72h^3-40h^2+96h-16) \, x^2 + 512h^3(7h^2+20h-4) \, x.
$$
This is an elliptic K3 surface, with reducible fibers of type
$\I_6, \I_6, \I_4, \I_3$ at $h = 0, \infty, -2, 2/9$ respectively,
and $\I_2$ fibers at $h = (-10 \pm 8 \sqrt{2})/7$. The trivial lattice has rank
$19$ and discriminant $1728$.  There is an obvious $2$-torsion section
$P_0 = (0,0)$, and we find a non-torsion section
$$
P_1 = \big(128 h, \, 128 h(h+2)^2 \big)
$$
of height $2/3$.  We checked that the group generated by $P_0$ and $P_1$
is saturated at $2$ and $3$.  Therefore, this is a singular K3 surface,
with N\'{e}ron-Severi lattice of rank $20$ and discriminant $-288$.

Next, we analyze the twist of the elliptic K3 surface above, obtained
by substituting $z = wg$ in the equation of the Hilbert modular
surface, and then setting $f = g^2$ (it is the quotient of $Y_{-}(56)$
by the involution $(g,h,z) \mapsto (-g,h,-z)$). This twist is an
honestly elliptic surface, with $\chi = 3$. After some simple algebra,
the Weierstrass equation can be written as
$$
y^2 = x^3 + (58h^5+149h^4-56h^3-152h^2-64h+16)x^2 + 8h^3(2h+5)(h^2-4h-4)^2(7h^2+20h-4)x.
$$
It has reducible fibers of type $\I^*_0, \I_6, \I_4, \I_3, \I_2, \I_2
$ at $h = \infty, 0, -2, 2/9, -1/2, -5/2$ respectively, $\I_4$ fibers
at $h = (-10 \pm 8 \sqrt{2})/7$ and $\I_2$ fibers at $h = 2 \pm 2
\sqrt{2}$. The trivial lattice has rank $26$. In addition to the
$2$-torsion section $P_0 = (0,0)$, we find the sections
\begin{align*}
P_1 &= \big( -56h^3(h^2-4h-4), 16 \sqrt{14} h^3(2h+1)(9h-2)(h^2-4h-4) \big) \\
P_2 &= \big(  -8h^3(7h^2+20h-4), 64 \sqrt{-1} h^3(2h+1)(7h^2+20h-4)\big) \\
P_3 &= \big(  4h^2(7h^2+20h-4), 4h^2(h+2)^2(2h+1)(7h^2+20h-4)\big) 
\end{align*}
of heights $5/6$, $2$ and $7/6$ respectively, orthogonal with respect
to the height pairing. On the other hand, counting points on the
reductions modulo $11$ and $29$ shows that the Picard number is at
most $29$. Therefore, it is exactly $29$. The lattice spanned by these
sections and the trivial lattice has discriminant $35840 = 2^{10}
\cdot 5 \cdot 7$. We checked that it is $2$-saturated, and therefore
it must equal the entire N\'eron-Severi lattice.

\subsection{Examples}

We list some points of small height and corresponding genus $2$ curves.

\begin{tabular}{l|c}
Rational point $(g,h)$ & Sextic polynomial $f_6(x)$ defining the genus $2$ curve $y^2 = f_6(x)$. \\
\hline \hline \\ [-2.5ex]
$(-79/61, 28/61)$ & $ -2000x^6 + 2040x^5 - 565x^4 + 628x^3 - 349x^2 - 36x - 68 $ \\
$(23/19, 24/95)$ & $ 480x^6 + 1200x^5 + 2657x^4 + 1264x^3 + 497x^2 - 2220x + 660 $ \\
$(2/13, -6/13)$ & $ -600x^6 - 360x^5 + 2660x^4 + 256x^3 - 2698x^2 - 222x + 639 $ \\
$(-2/13, -6/13)$ & $ 1096x^6 - 24x^5 - 3388x^4 + 608x^3 + 2750x^2 - 930x - 225 $ \\
$(79/61, 28/61)$ & $ 1350x^6 + 270x^5 + 3375x^4 - 3944x^3 + 1669x^2 - 5328x + 3392 $ \\
$(-3/7, -8/7)$ & $ -1340x^6 + 5900x^5 - 2227x^4 + 5096x^3 + 2707x^2 + 10x + 1950 $ \\
$(3/7, -8/7)$ & $ 1440x^6 - 4720x^5 - 13227x^4 + 20x^3 + 7389x^2 - 1080x - 432 $ \\
$(-1/7, 24/7)$ & $ -12600x^6 - 3192x^5 - 16975x^4 - 4442x^3 + 5717x^2 - 516x + 4 $ \\
$(1/91, -40/91)$ & $ 3740x^6 - 6420x^5 - 11789x^4 + 18160x^3 + 7315x^2 - 13356x + 2268 $ \\
$(-11/7, 9/7)$ & $ 35220x^6 + 10548x^5 + 43345x^4 - 10038x^3 + 3313x^2 - 228x + 52 $ \\
$(-23/19, 24/95)$ & $ 47824x^6 + 45048x^5 + 13973x^4 - 11016x^3 + 9341x^2 - 2040x + 400 $ \\
$(2/3, 2/9)$ & $ -6883x^6 + 10038x^5 + 62514x^4 + 31744x^3 - 21780x^2 + 3720x - 200 $ \\
$(37/31, 9/31)$ & $ -1548x^6 - 7732x^5 - 33547x^4 - 51202x^3 - 71163x^2 + 65988x - 11772 $ \\
$(-4/17, -3/8)$ & $ -316x^6 + 4764x^5 + 21121x^4 - 11666x^3 - 75071x^2 + 20364x + 49716 $ \\
$(11/7, 9/7)$ & $ 25092x^6 + 70500x^5 + 71881x^4 - 29834x^3 - 80543x^2 - 25908x + 38700 $ \\
$(55/13, 28/3)$ & $ -94x^6 - 114x^5 - 2497x^4 - 660x^3 - 29263x^2 - 10920x - 170352 $
\end{tabular}

Next, we analyze curves of low genus on the Hilbert modular surface.
The specialization $h = 2/9$ gives a rational curve,
parametrized by $g = (m^2 - 4m - 9)/(m^2 + 9)$. The Brauer obstruction
vanishes identically for rational points on this curve, giving a
$1$-parameter family of genus $2$ curves whose Jacobians have real
multiplication by $\sO_{56}$.

The specializations $h = -1/2$ and $h = -5/2$ give genus-$1$ curves
with rational points, both of whose Jacobians have rank $1$. The
Brauer obstruction does not vanish identically on either of these loci.

We also obtain some genus-$1$ curves by pulling back some sections
from the quotient K3 surface. For instance, the section $P_0 + P_1$ gives
the genus-$1$ curve $g^2 = -(7h^4+20h^3-4h^2-32h-16)/16$ which has
rational points (such as $(h,g)=(1,\pm 1)$), with a Jacobian of
conductor $2^4 \, 211$ and \MoW\ group $\cong \Z^2$.  The section $2P_1$
gives the genus-$1$ curve $g^2 = -(49h^4-112h^3+64h^2-32h-16)/16$ which
also has rational points $(h,g)=(1,\pm1)$, with Jacobian of conductor
$2^6 \, 7 \cdot 23$ and \MoW\ group $\cong (\Z/2\Z) \oplus \Z$.
The Brauer obstruction does not vanish identically on either of these loci.

\section{Discriminant $57$}

\subsection{Parametrization}

We start with an elliptic K3 surface with fibers of type $E_6$, $A_7$
and $A_2$ at $t = \infty, 0, 1$ respectively, and a section of height
$57/72 = 19/24 = 4 - 4/3 - 3 \cdot 5/8$.

The Weierstrass equation for this family is
$$
y^2 = x^3 + a x^2 + 2b t^2 (t-1) x + c t^4(t-1)^2,
$$
with
\begin{align*}
a &= t(t-1) \big((g^2-1/4)(h^2-3)+2gh \big) + \big ((g^2+1/4)h - g \big)^2t^2 - t + 1, \\
b &= -(g^2-1/4)^2 (h^2-1) \Big( (1-t) -\big((g^2-1/4)(h^2-2)+2gh \big)t(1-t) - (g-h/2)\big((g^2+1/4)h - g\big)t^2 \Big), \\
c &= (g^2-1/4)^4(h^2-1)^2 \Big( (1-t) + (g-h/2)^2t^2 + t(1-t)\big((g - h/2)^2  - (gh+1/2)^2 \big)\Big).
\end{align*}

We identify an $E_8$ fiber below, and the resulting $3$-neighbor step
takes us to an elliptic fibration with $E_8$ and $A_7$ fibers.

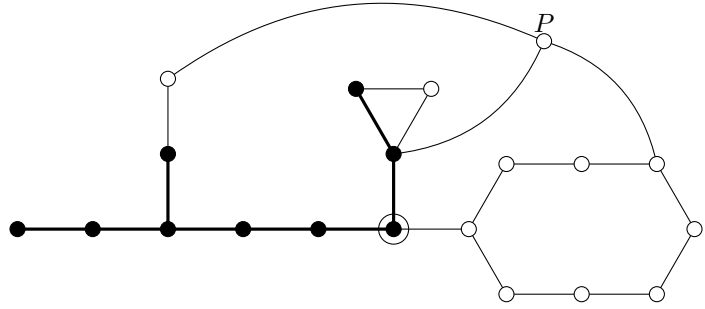
\begin{center}
\begin{tikzpicture}

\draw (0,0)--(6,0)--(6.5,0.866)--(8.5,0.866)--(9,0)--(8.5,-0.866)--(6.5,-0.866)--(6,0);
\draw (2,0)--(2,2);
\draw (5,0)--(5,1);
\draw (5,1)--(4.5,1.866);
\draw (5,1)--(5.5,1.866);
\draw (4.5,1.866)--(5.5,1.866);
\draw [bend right] (7,2.5) to (2,2);
\draw [bend left] (7,2.5) to (5,1);
\draw [bend left] (7,2.5) to (8.5,0.866);
\draw [very thick] (0,0)--(5,0)--(5,1)--(4.5,1.866);
\draw [very thick] (2,0)--(2,1);

\draw (5,0) circle (0.2);
\fill [black] (0,0) circle (0.1);
\fill [black] (1,0) circle (0.1);
\fill [black] (2,0) circle (0.1);
\fill [black] (3,0) circle (0.1);
\fill [black] (4,0) circle (0.1);
\fill [black] (5,0) circle (0.1);
\fill [white] (6,0) circle (0.1);
\fill [white] (6.5,0.866) circle (0.1);
\fill [white] (6.5,-0.866) circle (0.1);
\fill [white] (7.5,0.866) circle (0.1);
\fill [white] (7.5,-0.866) circle (0.1);
\fill [white] (8.5,0.866) circle (0.1);
\fill [white] (8.5,-0.866) circle (0.1);
\fill [white] (9,0) circle (0.1);
\fill [black] (2,1) circle (0.1);
\fill [white] (2,2) circle (0.1);
\fill [black] (5,1) circle (0.1);
\fill [black] (4.5,1.866) circle (0.1);
\fill [white] (5.5,1.866) circle (0.1);
\fill [white] (7,2.5) circle (0.1);

\draw (0,0) circle (0.1);
\draw (1,0) circle (0.1);
\draw (2,0) circle (0.1);
\draw (3,0) circle (0.1);
\draw (4,0) circle (0.1);
\draw (5,0) circle (0.1);
\draw (6,0) circle (0.1);
\draw (6.5,0.866) circle (0.1);
\draw (6.5,-0.866) circle (0.1);
\draw (7.5,0.866) circle (0.1);
\draw (7.5,-0.866) circle (0.1);
\draw (8.5,0.866) circle (0.1);
\draw (8.5,-0.866) circle (0.1);
\draw (9,0) circle (0.1);
\draw (2,1) circle (0.1);
\draw (2,2) circle (0.1);
\draw (5,1) circle (0.1);
\draw (4.5,1.866) circle (0.1);
\draw (5.5,1.866) circle (0.1);
\draw (7,2.5) circle (0.1);
\draw (7,2.5) [above] node{$P$};

\end{tikzpicture}
\end{center}

Since $P \cdot F' = 2$ for the new fiber $F'$, while the intersection
number of the remaining component of the $E_6$ fiber with $F'$ is $3$,
we deduce that the new fibration has a section.

The new fibration has a section of height $57/8 = 4 + 2 \cdot 2 - 1
\cdot 7/8$.  Now we can identify an $E_7$ fiber $F''$ and move to the
$E_8 E_7$ fibration by a $2$-neighbor step. Note that since $Q \cdot
F'' = 7$, while the remaining component of the $A_7$ fiber intersects
$F''$ in $2$, the new fibration will have a section.

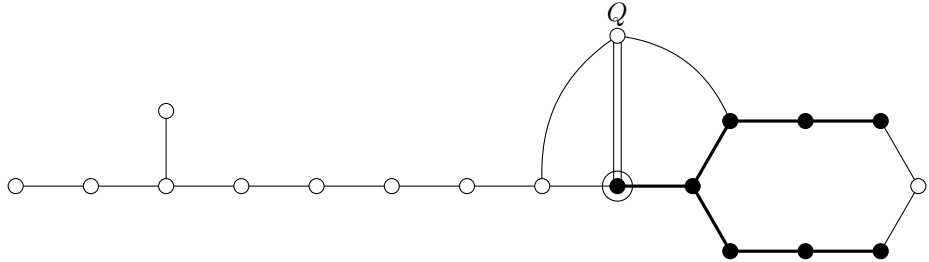
\begin{center}
\begin{tikzpicture}

\draw (6.95,0)--(6.95,2);
\draw (7.05,0)--(7.05,2);
\draw (-1,0)--(8,0);
\draw (1,0)--(1,1);
\draw (8,0)--(8.5,0.866)--(10.5,0.866)--(11,0)--(10.5,-0.866)--(8.5,-0.866)--(8,0);

\draw [bend right] (7,2) to (6,0);
\draw [bend left] (7,2) to (8.5,0.866);
\draw [very thick] (7,0)--(8,0);
\draw [very thick] (10.5,0.866)--(8.5,0.866)--(8,0)--(8.5,-0.866)--(10.5,-0.866);

\fill [white] (-1,0) circle (0.1);
\fill [white] (0,0) circle (0.1);
\fill [white] (1,0) circle (0.1);
\fill [white] (2,0) circle (0.1);
\fill [white] (3,0) circle (0.1);
\fill [white] (4,0) circle (0.1);
\fill [white] (5,0) circle (0.1);
\fill [white] (6,0) circle (0.1);
\fill [white] (1,1) circle (0.1);
\fill [black] (7,0) circle (0.1);
\fill [black] (8,0) circle (0.1);
\fill [black] (8.5,0.866) circle (0.1);
\fill [black] (9.5,0.866) circle (0.1);
\fill [black] (10.5,0.866) circle (0.1);
\fill [black] (8.5,-0.866) circle (0.1);
\fill [black] (9.5,-0.866) circle (0.1);
\fill [black] (10.5,-0.866) circle (0.1);
\fill [white] (11,0) circle (0.1);
\fill [white] (7,2) circle (0.1);

\draw (7,0) circle (0.2);
\draw (-1,0) circle (0.1);
\draw (0,0) circle (0.1);
\draw (1,0) circle (0.1);
\draw (2,0) circle (0.1);
\draw (3,0) circle (0.1);
\draw (4,0) circle (0.1);
\draw (5,0) circle (0.1);
\draw (6,0) circle (0.1);
\draw (1,1) circle (0.1);
\draw (7,0) circle (0.1);
\draw (8,0) circle (0.1);
\draw (8.5,0.866) circle (0.1);
\draw (9.5,0.866) circle (0.1);
\draw (10.5,0.866) circle (0.1);
\draw (8.5,-0.866) circle (0.1);
\draw (9.5,-0.866) circle (0.1);
\draw (10.5,-0.866) circle (0.1);
\draw (11,0) circle (0.1);
\draw (7,2) circle (0.1);
\draw [above] (7,2) node {$Q$};

\end{tikzpicture}
\end{center}

Now we can read out the Igusa-Clebsch invariants and compute the
branch locus, which is the subfamily with an extra $\I_2$ fiber.

\begin{theorem}
A birational model over $\Q$ for the Hilbert modular surface
$Y_{-}(57)$ as a double cover of $\Proj^2_{g,h}$ is given by the
following equation:
\begin{align*}
z^2 &= (256g^6-176g^4+40g^2-3)h^4+(2176g^5-960g^3+104g)h^3 \\
    & \quad  +(4320g^4-688g^2-34)h^2  +(-3456g^5+576g^3+328g)h \\
    & \quad  -2160g^4-648g^2 +361.
\end{align*}
It is an honestly elliptic surface, with arithmetic genus $2$ and
Picard number $29$.
\end{theorem}

\subsection{Analysis}

The extra involution is $(g,h) \mapsto (-g,-h)$.
The branch locus has genus $2$; the transformation
$$
(g,h) = \left(  \frac{y+x^2+1}{4x^3} \quad, \quad \frac{x(4y-x^4-14x^2-1)}{(x^2-1)^2} \right)
$$
converts it to Weierstrass form
$$
y^2 = 3x^6 + 11x^4 + x^2 + 1.
$$
It is isomorphic to the quotient of $X_0(57)$ by the Atkin-Lehner
involution $w_{57}$.

The Hilbert modular surface $Y_{-}(57)$ is an honestly elliptic
surface, with a genus $1$ fibration over $\Proj^1_g$, and in fact,
setting $h = 1$ gives a section. Therefore, we may use the Jacobian form,
which has the Weierstrass equation
\begin{align*}
y^2 &=  x^3 + 4(12g^2-1)(28g^2-5)x^2 -4(2g-1)^3(2g+1)^3(12g^2-5)(108g^2-19)x \\
 & \qquad  + (2g-1)^6(2g+1)^6(108g^2-19)^2.
\end{align*}

It has reducible fibers of type $\I_7$ at $g = \pm 1/2$, $\IV$ at $g = \infty$,
$\I_2$ at $g = \pm {\sqrt{57}}/18$, and $\I_3$ at the four roots of
$432 g^4+216 g^2-49$ (which generate a dihedral extension containing
$\Q({\sqrt{57}})$). The trivial lattice contributes $26$ to
the rank of the N\'{e}ron-Severi lattice, leaving room for \MoW\
rank at most~$4$.  We find the sections
\begin{align*}
P_1 &= \big(0, (2g-1)^3(2g+1)^3(108g^2-19)\big), \\
P_2 &= \big( (2g-1)^2(2g+1)^2(6g-1)(6g+1), 2g(2g-1)^2(2g+1)^2(432g^4+216g^2-49) \big), \\
P_3 &= \big( (\mu+3)(18g-\mu)(2g-1)(2g+1)^3/3, (36g^2+9-2\mu)(2g+6+\mu)(18g-\mu)(2g-1)(2g+1)^3/3 \big),
\end{align*}
(where $\mu = \sqrt{57}$) with non-degenerate height pairing matrix
$$
\frac{1}{42} \left(\begin{array}{rrr}
38 & 0 & -19 \\
0 & 20 & -10 \\
-19 & -10 & 39
\end{array}\right).
$$ 
Therefore, the Picard number is at least $29$. In fact, counting
points modulo $7$ and $11$ shows that the Picard number must be
exactly $29$. Alternatively, analysis of the quotients below gives
another proof. The sections above together with the trivial lattice
span a lattice of discriminant $11970 = 2 \cdot 3^2 \cdot 5 \cdot 7
\cdot 19$. We checked that it is $3$-saturated, and thus it is all of
the N\'eron-Severi lattice. Therefore, these sections generate the
\MoW\ group.

Next, we analyze the quotient by the involution $(g,h,z) \mapsto
(-g,-h,z)$. This turns out to have a genus-$1$ fibration with section
over the $t$-line, where $t = g^2$. The Weierstrass equation may be
written (after a linear change of the base parameter and a Weierstrass
transformation)
$$
y^2 = x^3 + 4(t+1)(3t+2)(7t+2) x^2  -4t^3(t+1)^2(3t-2)(27t+8)x + t^6(t+1)^3(27t+8)^2.
$$
This is an elliptic K3 surface with fibers of type $\I_7$ at $t = 0$,
$\I^*_0$ at $t = -1$, $\I_2$ at $t= -8/27$, $\II$ at $t = \infty$, and
$\I_3$ at $t = -2 \pm 2 \sqrt{57}/9$. Thus the trivial lattice has
rank $17$, leaving room for at most three independent sections.

We find the independent sections
\begin{align*}
P_1 &= \big( t^2(t+1)(9t + 8), t^2(t+1)^2(27t^2+108t+32)\big), \\
P_2 &= \big( (t+1)(27t+8)(3t^2-64t-64)/57,  (t+1)^2(t+40)(27t+8)(27t^2+108t+32)/57^{3/2} \big)
\end{align*}
of heights $5/21$ and $7/6$ respectively, and orthogonal with respect
to the height pairing. Therefore the Picard number is at least $19$.
Counting points modulo $11$ and $13$ shows that the Picard number
cannot be $20$. These sections together with the trivial lattice
generate a lattice of discriminant $140$. We check that it is
$2$-saturated and must therefore be the full N\'eron-Severi lattice.

Replacing $g^2$ by $t$ in the Weierstrass equation for $Y_{-}(57)$
gives a quadratic twist of the above quotient K3 surface, given by the Weierstrass equation 
$$
y^2 =  x^3 + 4(12t-1)(28t-5)x^2 -4(4t-1)^3(12t-5)(108t-19)x + (4t-1)^6(108t-19)^2.
$$
This is an elliptic K3 surface with fibers of type $\I_7$ at $t =
1/4$, $\I_2$ at $t = 19/108$, $\I_3$ at $t = -1/4 \pm \sqrt{57}/18$
and $\IV^*$ at $t = \infty$. The section $P = \big(0, (4t-1)^3(108t -
19)\big)$ has height $19/42$. Therefore, this K3 surface is
singular. Together with the trivial lattice, the section $P$ spans a
lattice of discriminant $171 = 3^2 \cdot 19$. Since there is no
$3$-torsion section, and we cannot have a section of height $19/(3^2
\cdot 42)$ due to the configuration of fibers, this must be the full
N\'eron-Severi lattice.

Therefore, the \MoW\ rank of the original surface must be $2 + 1 = 3$.

\subsection{Examples}

We list some points of small height and corresponding genus $2$ curves.

\begin{tabular}{l|c}
Rational point $(g,h)$ & Sextic polynomial $f_6(x)$ defining the genus $2$ curve $y^2 = f_6(x)$. \\
\hline \hline \\ [-2.5ex]
$(-1/10, -5/3)$ & $ -2x^6 + 15x^5 + 131x^4 + 240x^3 - 61x^2 - 8 $ \\
$(7/10, -29/3)$ & $ 81x^6 + 54x^5 - 286x^4 - 186x^3 + 323x^2 + 120x - 130 $ \\
$(7/18, 9)$ & $ -108x^6 + 324x^5 - 243x^4 + 186x^3 - 279x^2 - 80 $ \\
$(-7/6, 21/5)$ & $ -100x^6 - 390x^5 - 204x^4 + 74x^3 - 69x^2 - 12x + 8 $ \\
$(-1/18, -3)$ & $ -220x^6 + 420x^5 - 111x^4 + 238x^3 + 381x^2 + 168x + 24 $ \\
$(-1/4, -16/3)$ & $ 95x^6 - 114x^5 + 325x^4 + 35x^3 + 10x^2 - 429x - 234 $ \\
$(5/14, 35/3)$ & $ 120x^6 - 192x^5 + 122x^4 + 286x^3 - 448x^2 + 357x - 63 $ \\
$(19/18, -3/5)$ & $ 58x^6 + 39x^5 - 129x^4 + 132x^3 + 519x^2 + 240x - 40 $ \\
$(-7/10, 29/3)$ & $ -390x^6 + 451x^5 - 230x^4 - 593x^3 + 682x^2 + 220x - 200 $ \\
$(4/9, 18)$ & $ -540x^6 + 729x^5 - 135x^4 + 255x^3 - 225x^2 - 36x - 52 $ \\
$(-5/2, -10/3)$ & $ -60x^6 - 156x^5 + 137x^4 + 310x^3 - 351x^2 + 108x + 756 $ \\
$(-17/10, -47/33)$ & $ 60x^5 + 839x^4 + 278x^3 - 652x^2 - 36x - 489 $ \\
$(1/18, 3)$ & $ -60x^6 + 60x^5 - 3x^4 + 184x^3 + 669x^2 + 132x + 868 $ \\
$(5/2, 10/3)$ & $ -819x^5 - 1042x^4 + 61x^3 + 248x^2 - 48x $ \\
$(-1/18, 9/5)$ & $ -40x^6 - 72x^5 - 45x^4 - 534x^3 - 297x^2 - 324x - 1188 $ \\
$(-5/2, 5/3)$ & $ -36x^6 + 84x^5 + 491x^4 - 750x^3 - 337x^2 - 912x - 1200 $ 
\end{tabular}

The specialization $h = 0$ gives a genus-$1$ curve with rational
points, whose Jacobian has rank $1$. Of course, there is a large
supply of genus-$1$ curves, simply by specializing $g$, since we have an
elliptic surface.

The sections $P_1$ and $P_2$ give rational curves $h = -16g/(4g^2 - 1)$
and $h = (72g^3-36g^2+78g-7)/\big((4g^2-1)(6g-7)\big)$ respectively.
The Brauer obstruction vanishes on these curves.

\section{Discriminant $60$}

\subsection{Parametrization}

We start with a K3 elliptic surface with $E_6, D_6, A_4$ fibers at
$\infty, 0, 1$ respectively. The Weierstrass equation for this family is
$$
y^2 = x^3 + a t \, x^2 + 2 b t^3 (t-1) \, x + c t^5 (t-1)^2,
$$
with
\begin{align*}
a &= (h^2-g^2-1)^2 -4 (2 h^2-3 g^2+6) (t-1), \\
b &= 4\big( -4(h^2-g^2-1)^2 + 16 (h^2-2 g^2+1) (t-1)\big), \\
c &= 256\big( (h^2-g^2-1)^2 + 4g^2 (t-1)\big). \\
\end{align*}

We identify an $E_8$ fiber below, and move to this elliptic fibration
by a $3$-neighbor step.

\begin{center}
\begin{tikzpicture}

\draw (4,0)--(4,1);
\draw (4,1)--(3.05,1.69);
\draw (4,1)--(4.95,1.69);
\draw (3.05,1.69)--(3.41,2.81);
\draw (4.95,1.69)--(4.59,2.81);
\draw (3.41,2.81)--(4.59,2.81);
\draw (-1,0)--(9,0);
\draw (7,0)--(7,2);
\draw (0,0)--(0,1);
\draw (2,0)--(2,1);
\draw [very thick] (4.95,1.69)--(4,1)--(4,0)--(9,0);
\draw [very thick] (7,0)--(7,1);

\draw (4,0) circle (0.2);
\fill [black] (4,0) circle (0.1);
\fill [black] (4,1) circle (0.1);
\fill [white] (3.05,1.69) circle (0.1);
\fill [black] (4.95,1.69) circle (0.1);
\fill [white] (3.41,2.81) circle (0.1);
\fill [white] (4.59,2.81) circle (0.1);
\fill [white] (-1,0) circle (0.1);
\fill [white] (0,0) circle (0.1);
\fill [white] (1,0) circle (0.1);
\fill [white] (2,0) circle (0.1);
\fill [white] (3,0) circle (0.1);
\fill [black] (5,0) circle (0.1);
\fill [black] (6,0) circle (0.1);
\fill [black] (7,0) circle (0.1);
\fill [black] (8,0) circle (0.1);
\fill [black] (9,0) circle (0.1);
\fill [black] (7,1) circle (0.1);
\fill [white] (7,2) circle (0.1);
\fill [white] (0,1) circle (0.1);
\fill [white] (2,1) circle (0.1);

\draw (4,0) circle (0.1);
\draw (4,1) circle (0.1);
\draw (3.05,1.69) circle (0.1);
\draw (4.95,1.69) circle (0.1);
\draw (3.41,2.81) circle (0.1);
\draw (4.59,2.81) circle (0.1);
\draw (-1,0) circle (0.1);
\draw (0,0) circle (0.1);
\draw (1,0) circle (0.1);
\draw (2,0) circle (0.1);
\draw (3,0) circle (0.1);
\draw (5,0) circle (0.1);
\draw (6,0) circle (0.1);
\draw (7,0) circle (0.1);
\draw (8,0) circle (0.1);
\draw (9,0) circle (0.1);
\draw (7,1) circle (0.1);
\draw (7,2) circle (0.1);
\draw (0,1) circle (0.1);
\draw (2,1) circle (0.1);

\end{tikzpicture}
\end{center}

The new elliptic fibration has fibers of type $E_8$, $D_6$ and $A_1$,
and a section $P$ of height $60/8 = 15/2 = 4 + 2 \cdot 2 - 1/2$.  We now
identify an $E_7$ fiber $F'$ and perform a $2$-neighbor step to go to
the new fibration. Note that it has a section, since $P \cdot F' = 7$,
while the remaining component of the $D_6$ fiber has intersection number
$2$ with~$F'$.

\begin{center}
\begin{tikzpicture}

\draw (0,0)--(13,0);
\draw (1,0)--(1,1);
\draw (3,0)--(3,1);
\draw (5,0)--(5,1);
\draw (11,0)--(11,1);
\draw (4.95,1)--(4.95,2);
\draw (5.05,1)--(5.05,2);
\draw [bend right] (5.5,3) to (4,0);
\draw (5.5,3)--(5,2);
\draw [bend left] (5.45,3) to (4.95,0);
\draw [bend left] (5.55,3) to (5.05,0);
\draw [bend left] (5.5,3) to (6,0);
\draw [very thick] (0,0)--(5,0)--(5,1);
\draw [very thick] (3,0)--(3,1);

\fill [black] (0,0) circle (0.1);
\fill [black] (1,0) circle (0.1);
\fill [black] (2,0) circle (0.1);
\fill [black] (3,0) circle (0.1);
\fill [black] (4,0) circle (0.1);
\fill [black] (5,0) circle (0.1);
\fill [white] (6,0) circle (0.1);
\fill [white] (7,0) circle (0.1);
\fill [white] (8,0) circle (0.1);
\fill [white] (9,0) circle (0.1);
\fill [white] (10,0) circle (0.1);
\fill [white] (11,0) circle (0.1);
\fill [white] (12,0) circle (0.1);
\fill [white] (13,0) circle (0.1);
\fill [white] (1,1) circle (0.1);
\fill [black] (3,1) circle (0.1);
\fill [black] (5,1) circle (0.1);
\fill [white] (5,2) circle (0.1);
\fill [white] (11,1) circle (0.1);
\fill [white] (5.5,3) circle (0.1);
\draw (5.5,3) [above] node{$P$};

\draw (5,0) circle (0.2);
\draw (0,0) circle (0.1);
\draw (1,0) circle (0.1);
\draw (2,0) circle (0.1);
\draw (3,0) circle (0.1);
\draw (4,0) circle (0.1);
\draw (5,0) circle (0.1);
\draw (6,0) circle (0.1);
\draw (7,0) circle (0.1);
\draw (8,0) circle (0.1);
\draw (9,0) circle (0.1);
\draw (10,0) circle (0.1);
\draw (11,0) circle (0.1);
\draw (12,0) circle (0.1);
\draw (13,0) circle (0.1);
\draw (1,1) circle (0.1);
\draw (3,1) circle (0.1);
\draw (5,1) circle (0.1);
\draw (5,2) circle (0.1);
\draw (11,1) circle (0.1);
\draw (5.5,3) circle (0.1);

\end{tikzpicture}
\end{center}

From the resulting $E_8 E_7$ fibration, we can read out the map to
$\A_2$. 

\begin{theorem}
A birational model over $\Q$ for the Hilbert modular surface
$Y_{-}(60)$ as a double cover of $\Proj^2_{g,h}$ is given by the following
equation:
\begin{align*}
z^2 &= -(h^2-2h-g^2+5)(h^2+2h-g^2+5) \cdot \\
   & \qquad (8h^6-25g^2h^4+24h^4+26g^4h^2-86g^2h^2+24h^2-9g^6+66g^4+47g^2+8).
\end{align*}
It is a surface of general type.
\end{theorem}

\subsection{Analysis}

The surface $Y_{-}(60)$ has two extra commuting involutions,
$\iota_1: (g,h) \mapsto (-g,h)$ and $\iota_2: (g,h) \mapsto (g,-h)$.
The two simpler components $h^2 \pm 2h - g^2 + 5 = 0$ of the branch locus
correspond to the subfamily of elliptic K3 surfaces where the
$D_6$ fiber gets promoted to an $E_7$ fiber, while the more complicated
component corresponds to an extra $\I_2$ fiber.
The simpler components are easily seen to be rational curves,
as they define conics in the $(g,h)$-plane, with rational points.
The last component is also a rational curve; a parametrization is given by
$$
(g,h) = \left( \frac{(t^2+1)^3}{t(t^2-1)(t^2-2t-1)} \quad , \quad
 \frac{(t^2+2t-1)(t^4-t^3+2t^2+t+1)}{t(t^2-1)(t^2-2t-1)} \right).
$$
This Hilbert modular surface is of general type.

We now analyze some of the quotients of this surface by the
involutions. The quotient under both involutions is given by
\begin{align*}
z^2 &= -(h^2-2gh+6h+g^2-10g+25) \\
 & \qquad (8h^3-25gh^2+24h^2+26g^2h-86gh+24h-9g^3+66g^2+47g+8)
\end{align*}
and this is actually a rational surface; the transformation
$(g,h) = (h' + g'^2 + 2g' + 5, h' + g'^2)$ converts the above equation into a conic bundle over $\Proj^1_{h'}$ with a section.

The quotient by the involution $\iota_1$ turns out to be an elliptic
K3 surface, with Weierstrass equation given by
\begin{align*}
y^2 &= x^3 + 2t^2(215t^2+356t+140) \, x^2  -t^3(t+2)^3(5t^2+874t+864) \, x/3  \\
 & \qquad -8t^4(t+2)^6(163t^2-54t-216)/27.
\end{align*}

It has reducible fibers of type $E_6$ at $t = 0$, $\I_6$ at $t = -2$,
$\I_2$ at $t= -1$ and $t = -2/9$, and $\I_3$ at $t = (-7 \pm 5
\sqrt{5})/19$. The trivial lattice therefore has rank $19$. We find a
$3$-torsion section with $x$-coordinate $11 t^2 (t+2)^2/3$ and a
non-torsion section with $x = -t^2(t+2)^2/3$. Therefore the K3 surface
is singular. These sections, together with the trivial lattice,
generate a lattice of rank $20$ and discriminant $60$. It must be the
full N\'{e}ron-Severi lattice, since otherwise there would have to be
a $6$-torsion section or section of height $5/24$,
neither of which is possible with our configuration of reducible fibers.

The quotient by the involution $\iota_2$ is also an elliptic K3
surface, with Weierstrass equation
$$
y^2 = x^3 -(11 t^4-20 t^2+8) \, x^2+ 16 (t-1)^3 (t+1)^3 (4 t^2-5) x.
$$
This has bad fibers of type $\I_6$ at $t = \pm 1$, $\I_2$ at $t =
\pm \sqrt{5}/2$, and $\I_3$ at $t = \pm 2/\sqrt{3}$. Therefore the
trivial lattice has rank $18$, leaving room for at most two independent
sections. We find the following sections, of which the first is $6$-torsion.
\begin{align*}
P_0 &= \big( 2(t^2 - 1)(4t^2 - 5) , \, 4(t^2-1)(3t^2-4)(4t^2-5) \big), \\
P_1 &= \big( (11 - 3 \mu) t^2(t^2 - 1)/2, \, (3+5 \mu)t(t^2-1)/12 (18t^2-15+\mu ) \big), \\
P_2 &= \big( (11 + 3 \mu) t^2(t^2 - 1)/2, \, (3-5 \mu)t(t^2-1)/12 (18t^2-15-\mu ), \big)
\end{align*}
where $\mu = \sqrt{-15}$. The height pairing matrix of $P_1$ and $P_2$
is
$$
\frac{1}{3} \left( \begin{array}{cc} 7 & -2 \\ -2 & 7
\end{array} \right).
$$ 
Therefore the Picard number of this K3 surface is $20$.  The
discriminant of the sublattice of $\NS(X)$ generated by the trivial
lattice and these sections is $180$. We checked that this lattice is
$2$- and $3$-saturated, which proves that it is the entire Picard
group.

\subsection{Examples}

We list some points of small height and corresponding genus $2$ curves.

\begin{tabular}{l|c}
Rational point $(g,h)$ & Sextic polynomial $f_6(x)$ defining the genus $2$ curve $y^2 = f_6(x)$. \\
\hline \hline \\ [-2.5ex]
$(17/6, -13/6)$ & $ 468x^6 + 1332x^5 + 1345x^4 - 20x^3 + 1051x^2 - 150x + 186 $ \\
$(-17/6, 13/6)$ & $ -12x^6 + 132x^5 + 371x^4 + 1506x^3 + 1391x^2 - 528x - 1872 $ \\
$(17/6, 13/6)$ & $ -942x^6 + 3150x^5 - 869x^4 - 4220x^3 + 745x^2 + 2244x + 468 $ \\
$(-17/6, -7/6)$ & $ 48x^6 - 360x^5 + 1907x^4 - 4000x^3 + 5195x^2 + 828x + 2556 $ \\
$(-17/6, -13/6)$ & $ -4500x^6 - 9300x^5 - 5365x^4 + 4106x^3 + 3335x^2 - 2112x - 2648 $ \\
$(-17/6, 7/6)$ & $ -4116x^6 - 6468x^5 + 8617x^4 + 11086x^3 - 12239x^2 - 3708x + 4212 $ \\
$(17/6, 7/6)$ & $ 72x^6 - 2136x^5 + 15869x^4 - 258x^3 - 1759x^2 + 108x - 4 $ \\
$(17/6, -7/6)$ & $ 12x^6 - 36x^5 + 929x^4 - 1458x^3 + 16361x^2 - 4452x + 1476 $ \\
$(51/5, -54/5)$ & $ 9248x^6 - 2312x^5 + 12427x^4 - 29852x^3 - 21811x^2 + 26690x + 21270 $ \\
$(57/10, 43/10)$ & $ 2272x^6 + 35064x^5 + 12877x^4 - 24234x^3 - 37079x^2 + 29700x - 3500 $ \\
$(-51/5, -54/5)$ & $ -6368x^6 - 20760x^5 + 11991x^4 + 29560x^3 - 61443x^2 + 39870x - 12150 $ \\
$(-46/15, -49/15)$ & $ -575x^6 - 3075x^5 - 12269x^4 - 16401x^3 - 56024x^2 - 21792x - 73242 $ \\
$(-61/10, 49/10)$ & $ -36450x^6 - 10530x^5 + 6327x^4 + 78760x^3 - 29879x^2 - 17700x + 2612 $ \\
$(61/10, 49/10)$ & $ 8612x^6 - 4020x^5 - 52381x^4 - 4290x^3 + 91787x^2 + 47220x - 11540 $ \\
$(61/10, -49/10)$ & $ -17092x^6 + 13812x^5 - 101885x^4 + 63210x^3 - 89229x^2 + 69580x - 15092 $ \\
$(-51/5, 54/5)$ & $ -100572x^6 - 102884x^5 - 147679x^4 - 25432x^3 + 27727x^2 + 35870x - 11890 $ 
\end{tabular}

We now describe some curves of genus $1$, possessing infinitely many
rational points, on the Hilbert modular surface. These were obtained
by pulling back rational curves on the quotients by $\iota_1$ and
$\iota_2$ obtained as sections of the elliptic fibrations.  In each case
we give the curve as a double cover of $\Proj^1$, exhibit a coordinate
of a point on $\Proj^1$ that lifts to a rational point, and give the
conductor and \MoW\ group.
$$
\begin{array}{lccc}
\phantom{g^2 = }\textrm{Equation} & \textrm{point}
 & \textrm{conductor} & \textrm{group} \\[5pt]
g^2 = \frac{v^4+132v^3+11784v^2+566280v+20175732}{36(v+9)^2}, \quad h = \frac{v^2+72v+4560}{6(v+9)}
 & v = \infty & 2^4 \, 3^2 \, 11 \cdot 97 & (\Z/2\Z) \oplus \Z \\[5pt]
g^2 = \frac{4913v^4+1990v^2+153}{36(v^2-1)^2}, \quad h = -\frac{25v^2+3}{2(v^2-1)}
 & v = 1 & 2 \cdot 3 \cdot 7 \cdot 17 \cdot 19 & (\Z/2\Z)^2 \oplus \Z \\[5pt]
g^2 = \frac{14049t^4-57248t^3+87462t^2-59840t+15657}{(t+1)^2(9t-11)^2}, \quad h = -\frac{2(t-1)(54t-67)}{(t+1)(9t-11)}
 & t = -1 & 5^2 \, 11 \cdot 17 \cdot 47 & \Z^2 \\[5pt]
h^2 = \frac{833v^4+190v^2+1}{4(v^2-1)^2}, \quad g = \frac{27v^2+5}{2(v^2-1)}
 & v = 0 & 238 = 2 \cdot 7 \cdot 17 & (\Z/2\Z) \oplus \Z
\end{array}
$$

We can obtain a few more such curves from these, by
applying the involution $\iota_2$ to the first three curves,
and the involution $\iota_1$ to the last.  If two genus-$2$ curves
are parametrized by points related by such an involution then
the curves' Jacobians are isogenous.

\section{Discriminant $61$}

\subsection{Parametrization}

Start with an elliptic surface with $D_7, A_6$ and $A_2$ fibers and a
section of height $61/84 = 4 - 2/3 - 6/7 - 7/4$.

The Weierstrass equation for this family can be written as
$$
y^2 =  x^3 + a\, x^2 + 2b t(1-1) \, x + c t^2(t-1)^2,
$$
where
\begin{align*}
a & = 4h^3(h-g)^3 t^3 + (h-g)^2(g^2h^2-4gh^2-8h^2-2g^2h+g^2+12g+12) t^2 \\
& \qquad -2(g+1)(h-g)(g^2h+4h-g^2-6g)t + g^2(g+1)^2, \\
b &= -4(g+1)(h-g)^2 \big( (h-g)^2(2gh^2+4h^2+g^2h-g^2-6g-6)t^2  \\
 & \qquad + (g+1)(h-g)(g^2h+2h-2g^2-6g)t -g^2(g+1)^2 \big), \\
c &= 16(g+1)^2(h-g)^4 \big( (g+2)(h-g)t + g(g+1) \big)^2.
\end{align*}

We first perform a $2$-neighbor step to move to an elliptic fibration
with $E_7$ and $A_7$ fibers, by locating an $E_7$ fiber, as follows.

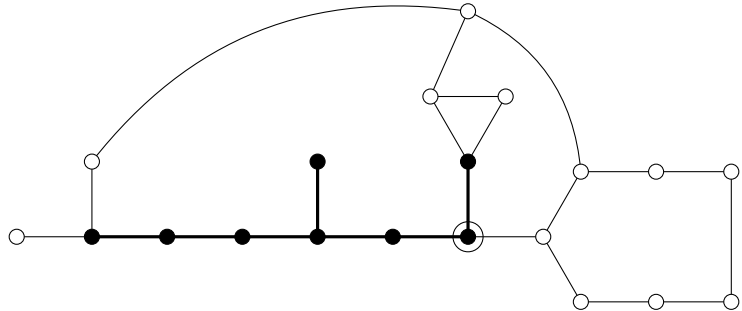
\begin{center}
\begin{tikzpicture}

\draw (0,0)--(7,0);
\draw (1,0)--(1,1);
\draw (4,0)--(4,1);
\draw (6,0)--(6,1);
\draw (6,1)--(5.5,1.866)--(6.5,1.866)--(6,1);
\draw (7,0)--(7.5,0.866)--(9.5,0.866)--(9.5,-0.866)--(7.5,-0.866)--(7,0);

\draw (6,3) to [bend right] (1,1);
\draw (6,3) to (5.5,1.866);
\draw (6,3) to [bend left] (7.5,0.866);
\draw [very thick] (1,0)--(6,0)--(6,1);
\draw [very thick] (4,0)--(4,1);

\draw (6,0) circle (0.2);
\fill [white] (0,0) circle (0.1);
\fill [black] (1,0) circle (0.1);
\fill [black] (2,0) circle (0.1);
\fill [black] (3,0) circle (0.1);
\fill [black] (4,0) circle (0.1);
\fill [black] (5,0) circle (0.1);
\fill [black] (6,0) circle (0.1);
\fill [white] (7,0) circle (0.1);
\fill [white] (7.5,0.866) circle (0.1);
\fill [white] (8.5,0.866) circle (0.1);
\fill [white] (9.5,0.866) circle (0.1);
\fill [white] (7.5,-0.866) circle (0.1);
\fill [white] (8.5,-0.866) circle (0.1);
\fill [white] (9.5,-0.866) circle (0.1);
\fill [black] (6,1) circle (0.1);
\fill [white] (5.5,1.866) circle (0.1);
\fill [white] (6.5,1.866) circle (0.1);
\fill [white] (1,1) circle (0.1);
\fill [black] (4,1) circle (0.1);
\fill [white] (6,3) circle (0.1);

\draw (0,0) circle (0.1);
\draw (1,0) circle (0.1);
\draw (2,0) circle (0.1);
\draw (3,0) circle (0.1);
\draw (4,0) circle (0.1);
\draw (5,0) circle (0.1);
\draw (6,0) circle (0.1);
\draw (7,0) circle (0.1);
\draw (7.5,0.866) circle (0.1);
\draw (8.5,0.866) circle (0.1);
\draw (9.5,0.866) circle (0.1);
\draw (7.5,-0.866) circle (0.1);
\draw (8.5,-0.866) circle (0.1);
\draw (9.5,-0.866) circle (0.1);
\draw (6,1) circle (0.1);
\draw (5.5,1.866) circle (0.1);
\draw (6.5,1.866) circle (0.1);
\draw (1,1) circle (0.1);
\draw (4,1) circle (0.1);
\draw (6,3) circle (0.1);

\end{tikzpicture}
\end{center}

This elliptic fibration has \MoW\ rank $2$, and we can in fact write
down two generators of the \MoW\ group, which intersect the components
of reducible fibers as shown below. Next, we identify the class of an
$E_8$ fiber, and use it to perform $2$-neighbor step to an elliptic
fibration with $E_8$, $A_5$ and $A_1$ fibers, and \MoW\ rank $2$.

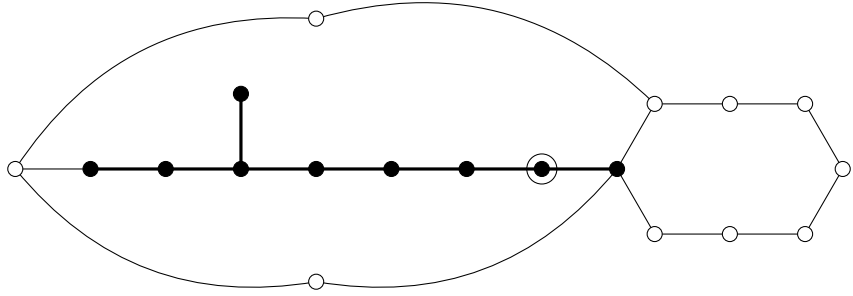
\begin{center}
\begin{tikzpicture}

\draw (0,0)--(8,0);
\draw (3,0)--(3,1);
\draw (8.5,0.866)--(10.5,0.866);
\draw (8.5,-0.866)--(10.5,-0.866);
\draw (8,0)--(8.5,0.866);
\draw (8,0)--(8.5,-0.866);
\draw (11,0)--(10.5,0.866);
\draw (11,0)--(10.5,-0.866);
\draw (0,0) to [bend left] (4,2);
\draw (4,2) to [bend left] (8.5,0.866);
\draw (4,-1.5) to [bend left] (0,0);
\draw (4,-1.5) to [bend right] (8,0);
\draw [very thick] (1,0)--(8,0);
\draw [very thick] (3,0)--(3,1);

\draw (7,0) circle (0.2);
\fill [white] (0,0) circle (0.1);
\fill [black] (1,0) circle (0.1);
\fill [black] (2,0) circle (0.1);
\fill [black] (3,0) circle (0.1);
\fill [black] (4,0) circle (0.1);
\fill [black] (5,0) circle (0.1);
\fill [black] (6,0) circle (0.1);
\fill [black] (7,0) circle (0.1);
\fill [black] (8,0) circle (0.1);
\fill [white] (8.5,0.866) circle (0.1);
\fill [white] (8.5,-0.866) circle (0.1);
\fill [white] (9.5,0.866) circle (0.1);
\fill [white] (9.5,-0.866) circle (0.1);
\fill [white] (10.5,0.866) circle (0.1);
\fill [white] (10.5,-0.866) circle (0.1);
\fill [white] (11,0) circle (0.1);
\fill [black] (3,1) circle (0.1);
\fill [white] (4,2) circle (0.1);
\fill [white] (4,-1.5) circle (0.1);

\draw (0,0) circle (0.1);
\draw (1,0) circle (0.1);
\draw (2,0) circle (0.1);
\draw (3,0) circle (0.1);
\draw (4,0) circle (0.1);
\draw (5,0) circle (0.1);
\draw (6,0) circle (0.1);
\draw (7,0) circle (0.1);
\draw (8,0) circle (0.1);
\draw (8.5,0.866) circle (0.1);
\draw (8.5,-0.866) circle (0.1);
\draw (9.5,0.866) circle (0.1);
\draw (9.5,-0.866) circle (0.1);
\draw (10.5,0.866) circle (0.1);
\draw (10.5,-0.866) circle (0.1);
\draw (11,0) circle (0.1);
\draw (3,1) circle (0.1);
\draw (4,2) circle (0.1);
\draw (4,-1.5) circle (0.1);

\end{tikzpicture}
\end{center}

We show one of the generators $P$ of the Mordell-Weil lattice, which
has height $4 - 2 \cdot 4/6 - 1/2$. Next, we go to a fibration with
$E_8$ and $E_6$ fibers using the fiber class $F'$ of $E_6$
below. Since $P \cdot F' = 1$, the new elliptic fibration has a
section.

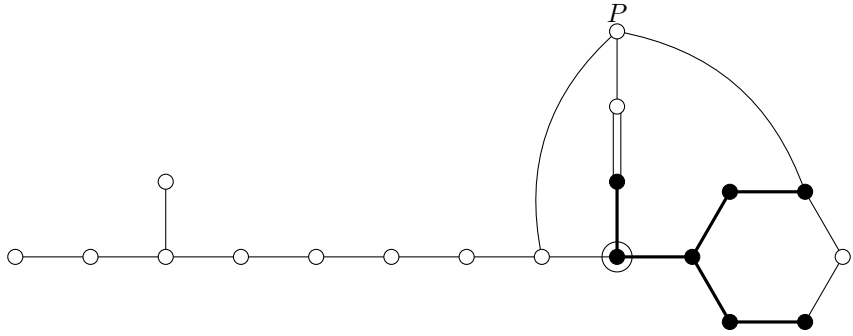
\begin{center}
\begin{tikzpicture}

\draw (0,0)--(9,0);
\draw (9,0)--(9.5,0.866);
\draw (9,0)--(9.5,-0.866);
\draw (9.5,0.866)--(10.5,0.866);
\draw (9.5,-0.866)--(10.5,-0.866);
\draw (11,0)--(10.5,0.866);
\draw (11,0)--(10.5,-0.866);
\draw (2,0)--(2,1);
\draw (8,0)--(8,1);
\draw (7.95,1)--(7.95,2);
\draw (8.05,1)--(8.05,2);
\draw [very thick] (8,1)--(8,0)--(9,0);
\draw [very thick] (10.5,0.866)--(9.5,0.866)--(9,0)--(9.5,-0.866)--(10.5,-0.866);
\draw (8,3)--(8,2);
\draw [bend right] (8,3) to (7,0);
\draw [bend left] (8,3) to (10.5,0.866);

\draw (8,3) [above] node{$P$};
\draw (8,0) circle (0.2);
\fill [white] (8,3) circle (0.1);
\fill [white] (0,0) circle (0.1);
\fill [white] (1,0) circle (0.1);
\fill [white] (2,0) circle (0.1);
\fill [white] (3,0) circle (0.1);
\fill [white] (4,0) circle (0.1);
\fill [white] (5,0) circle (0.1);
\fill [white] (6,0) circle (0.1);
\fill [white] (7,0) circle (0.1);
\fill [black] (8,0) circle (0.1);
\fill [black] (9,0) circle (0.1);
\fill [white] (2,1) circle (0.1);
\fill [black] (8,1) circle (0.1);
\fill [white] (8,2) circle (0.1);
\fill [black] (9.5,0.866) circle (0.1);
\fill [black] (9.5,-0.866) circle (0.1);
\fill [black] (10.5,0.866) circle (0.1);
\fill [black] (10.5,-0.866) circle (0.1);
\fill [white] (11,0) circle (0.1);

\draw (8,3) circle (0.1);
\draw (0,0) circle (0.1);
\draw (1,0) circle (0.1);
\draw (2,0) circle (0.1);
\draw (3,0) circle (0.1);
\draw (4,0) circle (0.1);
\draw (5,0) circle (0.1);
\draw (6,0) circle (0.1);
\draw (7,0) circle (0.1);
\draw (8,0) circle (0.1);
\draw (9,0) circle (0.1);
\draw (2,1) circle (0.1);
\draw (8,1) circle (0.1);
\draw (8,2) circle (0.1);
\draw (9.5,0.866) circle (0.1);
\draw (9.5,-0.866) circle (0.1);
\draw (10.5,0.866) circle (0.1);
\draw (10.5,-0.866) circle (0.1);
\draw (11,0) circle (0.1);

\end{tikzpicture}
\end{center}

We can find a section $P'$ of this elliptic fibration with $E_8$ and
$E_6$ fibers, of height $8/3 = 4 - 4/3$. We use it to go to a fibration
with $E_8$ and $E_7$ fibers as shown.

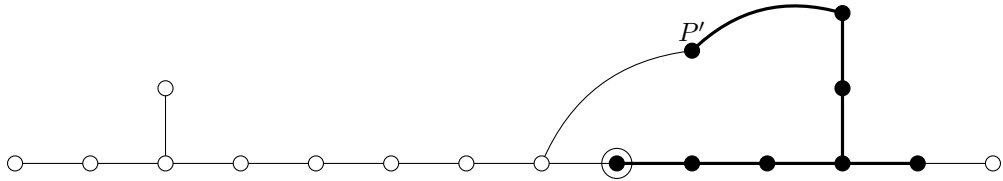
\begin{center}
\begin{tikzpicture}

\draw (0,0)--(13,0);
\draw (2,0)--(2,1);
\draw (11,0)--(11,2);
\draw (9,1.5) to [bend right] (7,0);
\draw [very thick] (9,1.5) to [bend left] (11,2);
\draw [very thick] (8,0)--(12,0);
\draw [very thick] (11,0)--(11,2);

\fill [white] (0,0) circle (0.1);
\fill [white] (1,0) circle (0.1);
\fill [white] (2,0) circle (0.1);
\fill [white] (3,0) circle (0.1);
\fill [white] (4,0) circle (0.1);
\fill [white] (5,0) circle (0.1);
\fill [white] (6,0) circle (0.1);
\fill [white] (7,0) circle (0.1);
\fill [black] (8,0) circle (0.1);
\fill [black] (9,0) circle (0.1);
\fill [black] (10,0) circle (0.1);
\fill [black] (11,0) circle (0.1);
\fill [black] (12,0) circle (0.1);
\fill [white] (13,0) circle (0.1);
\fill [black] (11,1) circle (0.1);
\fill [black] (11,2) circle (0.1);
\fill [white] (2,1) circle (0.1);
\fill [black] (9,1.5) circle (0.1);

\draw (8,0) circle (0.2);
\draw (0,0) circle (0.1);
\draw (1,0) circle (0.1);
\draw (2,0) circle (0.1);
\draw (3,0) circle (0.1);
\draw (4,0) circle (0.1);
\draw (5,0) circle (0.1);
\draw (6,0) circle (0.1);
\draw (7,0) circle (0.1);
\draw (8,0) circle (0.1);
\draw (9,0) circle (0.1);
\draw (10,0) circle (0.1);
\draw (11,0) circle (0.1);
\draw (12,0) circle (0.1);
\draw (13,0) circle (0.1);
\draw (11,1) circle (0.1);
\draw (11,2) circle (0.1);
\draw (2,1) circle (0.1);
\draw (9,1.5) circle (0.1);

\draw (9,1.5) [above] node{$P'$};

\end{tikzpicture}
\end{center}

From the resulting $E_8 E_7$ fibration we read out the Igusa-Clebsch
invariants and calculate the equation of $Y_{-}(61)$ as a double cover
of $\Proj^2_{g,h}$.

\begin{theorem}
A birational model over $\Q$ for the Hilbert modular surface
$Y_{-}(61)$ as a double cover of $\Proj^2_{g,h}$ is given by the following
equation:
\begin{align*}
z^2 =& (h-1)^4 \,g^4 - 2 (h-1) h (h^3-14 h^2-20 h-21) \, g^3 + h (h^5-46 h^4-19 h^3+42 h^2+39 h-44) \, g^2  \\
& \qquad + 2 h^2 (10 h^4+5 h^3-13 h^2-h+12) \, g - h^2 (8 h^4-13 h^2+16).
\end{align*}
It is an honestly elliptic surface, with arithmetic genus $2$ and
Picard number $28$.
\end{theorem}

\subsection{Analysis}

The branch locus has genus $1$. The transformation
\begin{align*}
g &= -\frac{8x^3y-34x^2y+37xy-14y+x^6-10x^5+44x^4-89x^3+98x^2-58x+14}{x^2(x-1)^3(x-7)} \\
h &= \frac{12x^2y-35xy+26y+3x^4+11x^3-61x^2+74x-26}{x^2(9x^2-24x+13)}
\end{align*}
converts it to the Weierstrass form
$$
y^2 + xy = x^3 - 2x + 1,
$$
an elliptic curve of conductor $61$. It is isomorphic to
$X_0(61)/\langle w \rangle$, where $w$ is the Atkin-Lehner involution.

The Hilbert modular surface $Y_{-}(61)$ is an honestly elliptic
surface, since we have an evident genus-$1$ fibration over
$\Proj^1_h$. Since the coefficient of $g^4$ is a square, we convert to
the Jacobian. In Weierstrass form, we obtain
\begin{align*}
y^2 &= x^3 + h(h^5+14h^4+23h^3-102h^2+88) \, x^2 \\
    &\qquad   -h^2(110h^6+908h^5-2854h^4-1028h^3+4795h^2+120h-2000) \, x \\
   & \qquad  + h^4(1728h^7+16849h^6-24666h^5-50145h^4+52138h^3+50406h^2-29200h-20000).
\end{align*}
This elliptic surface $S$ has $\chi(\sO_S) = 3$, with bad fibers of
type $\I^*_0$ at $h = 0$, $\I_7$ at $h = \infty$, $\I_2$ at $h = 1$
and at the roots of $h^3+13 h^2+24 h+16$ (which generates the cubic
field of discriminant $-244$), and $\I_3$ at $h = -1$ and at the roots
of $3 h^4+23 h^3-64 h^2+22 h+25$ (which generates the quartic field of
discriminant $-3 \cdot 61^2$, a quadratic extension of $\Q(\sqrt{61})$
). The trivial lattice therefore has rank $26$, leaving room for at
most $4$ independent sections. We find the two sections
\begin{align*}
P_1 &= \big( h^2(-36h + 55), \, 12h^2(3h^4+23h^3-64h^2+22h+25)\big) \\
P_2 &= \big(-(36h^6+444h^5+472h^4-1492h^3-799h^2+1048h-400)/61, \\
  & \qquad  4(15h^2-2h-5)(h^3+13h^2+24h+16)(3h^4+23h^3-64h^2+22h+25)/61^{3/2}
   \big)
\end{align*}
of heights $13/21$ and $11/6$ respectively, and orthogonal to each
other under the height pairing.  By Oda's calculations, the Picard
number is $28$, and therefore the \MoW\ rank is exactly~$2$. The
sublattice of the Picard group generated by the above sections and the
trivial lattice has discriminant $41184 = 2^5 \cdot 3^2 \cdot 11 \cdot
13$. We checked that it is $2$- and $3$-saturated, and therefore it
must be the entire N\'eron-Severi lattice. Therefore the \MoW\ group
is generated by $P_1$ and $P_2$.

\subsection{Examples}

We list some points of small height and corresponding genus $2$ curves.

\begin{tabular}{l|c}
Rational point $(g,h)$ & Sextic polynomial $f_6(x)$ defining the genus $2$ curve $y^2 = f_6(x)$. \\
\hline \hline \\ [-2.5ex]
$(-3/2, 1/2)$ & $ -32x^6 - 144x^5 - 229x^4 - 24x^3 - 157x^2 - 84x + 92 $ \\
$(47/72, -1/8)$ & $ -10x^6 - 42x^5 + 39x^4 + 728x^3 + 489x^2 + 1260x - 1068 $ \\
$(-47/12, -11/4)$ & $ 756x^6 - 756x^5 + 1953x^4 - 282x^3 - 327x^2 + 1404x - 1444 $ \\
$(-86/9, -8)$ & $ -348x^6 - 972x^5 - 2661x^4 - 2326x^3 - 2205x^2 + 1260x - 140 $ \\
$(-17/18, -1/2)$ & $ 467x^6 + 1551x^5 + 3906x^4 + 3027x^3 - 495x^2 - 1800x + 400 $ \\
$(15/4, 5)$ & $ -325x^6 + 1410x^5 + 1045x^4 - 3993x^3 - 2636x^2 + 3456x + 2143 $ \\
$(-69/58, -20/29)$ & $ -128x^6 + 96x^5 - 2519x^4 + 4362x^3 + 1321x^2 + 456x + 32 $ \\
$(31/45, 14/5)$ & $ -204x^6 - 108x^5 + 2553x^4 - 946x^3 - 4683x^2 - 360x + 1200 $ \\
$(1, 4/3)$ & $ -188x^6 - 1812x^5 - 4707x^4 - 130x^3 + 6189x^2 - 1620x + 108 $ \\
$(13/18, -1/2)$ & $ -612x^6 + 3708x^5 - 6501x^4 + 5656x^3 + 693x^2 - 318x - 1126 $ \\
$(-29/6, -29/8)$ & $ -148x^6 + 2472x^5 - 1481x^4 + 5001x^3 - 6980x^2 + 1425x - 6625 $ \\
$(-37/18, -1/2)$ & $ -16x^6 - 264x^5 + 477x^4 - 2268x^3 + 4029x^2 - 6156x + 10372 $ \\
$(23/84, -1/3)$ & $ -4671x^6 - 6660x^5 - 10362x^4 + 11195x^3 + 1287x^2 + 1590x + 7621 $ \\
$(5/6, 5/3)$ & $ 821x^6 - 1896x^5 - 4922x^4 + 8588x^3 + 11341x^2 - 7674x - 8247 $ \\
$(-23/2, -10)$ & $ 2716x^6 + 84x^5 + 7107x^4 + 10642x^3 + 4803x^2 + 13764x + 10204 $ \\
$(65/18, 5/2)$ & $ -3756x^6 - 12012x^5 + 9297x^4 + 18116x^3 - 10335x^2 - 7560x + 3600 $ 
\end{tabular}

Next, we list some rational curves on the surface. The specialization
$h = -1$ gives a rational curve, but the curves of genus $2$
corresponding to the points on this curve have Jacobians with
endomorphism ring larger than just $\sO_{61}$ (they are isogenous to
the symmetric squares of elliptic curves). The section $P_1$ gives the
rational curve $g = (3h^2-7h+5)/\big(3(h-1) \big)$, for which the
Brauer obstruction vanishes identically, yielding a $1$-parameter
family of genus $2$ curves whose Jacobian have real multiplication by
$\sO_{61}$.

\section{Discriminant $65$}

\subsection{Parametrization}

We start with an elliptic surface with $E_7$, $A_4$ and $A_4$ fibers
at $t = \infty, 0, 1$ respectively, and a section of height $65/50 =
13/10 = 4 - 3/2 - 2 \cdot 3/5$.  The Weierstrass equation of such a
family is
$$
y^2 = x^3 + \big(a_0(1-t) + a_1 t (1-t) + t^2 \big)\, x^2 + 2 t^2 (t-1) e \big(b_0(1-t) + b_1t(1-t) + t^2\big)\, x  + e^2 t^4(t-1)^2 \big(c_0(1-t) + t\big),
$$
with
\begin{align*}
a_0 &= (s^2-5)^2(2rs^2+s+2r)^2/4, \\
a_1 &= 4(5s^6-8s^4-7s^2-6)r^2 + 4s(5s^4-13s^2-4)r + 5s^4-18s^2+5, \\
b_0 &= (s^2-5)(2rs^2+s+2r)(2rs^4+s^3+4rs^2+s-6r)/4, \\
b_1 &= 8(s^2-1)(s^2+1)^2r^2 + 8s(s^4 - 1)r + 2s^4-2s^2+1, \\
c_0 &= (2rs^4+s^3+4rs^2+s-6r)^2/4, \\
e &= -(s-1)(s+1)(2rs^2-4rs+s+2r-2)(2rs^2+4rs+s+2r+2).
\end{align*}

To describe an elliptic fibration with $E_8$ and $E_7$ fibers, we
identify the class of an $E_8$ fiber and move by a $2$-neighbor step
to an $E_8 A_4 A_2$ fibration.

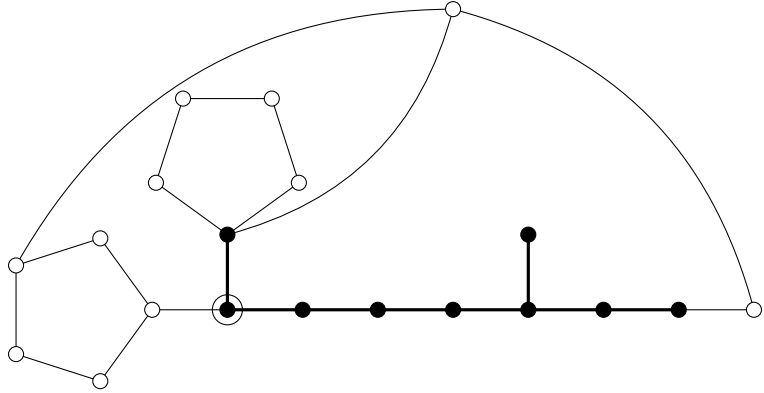
\begin{center}
\begin{tikzpicture}

\draw (2,0)--(2-0.69,0.95)--(2-1.81,0.59)--(2-1.81,-0.59)--(2-0.69,-0.95)--(2,0);
\draw (2,0)--(10,0);
\draw (7,0)--(7,1);
\draw (3,0)--(3,1);
\draw (3,1)--(2.05,1.69);
\draw (3,1)--(3.95,1.69);
\draw (2.05,1.69)--(2.41,2.81);
\draw (3.95,1.69)--(3.59,2.81);
\draw (2.41,2.81)--(3.59,2.81);
\draw [bend right] (6,4) to (0.19,0.59);
\draw [bend left] (6,4) to (3,1);
\draw [bend left] (6,4) to (10,0);
\draw [very thick] (3,1)--(3,0)--(9,0);
\draw [very thick] (7,0)--(7,1);

\fill [white] (2,0) circle (0.1);
\fill [white] (2-0.69,0.95) circle (0.1);
\fill [white] (2-0.69,-0.95) circle (0.1);
\fill [white] (2-1.81,0.59) circle (0.1);
\fill [white] (2-1.81,-0.59) circle (0.1);

\draw (3,0) circle (0.2);
\fill [black] (3,0) circle (0.1);
\fill [black] (4,0) circle (0.1);
\fill [black] (5,0) circle (0.1);
\fill [black] (6,0) circle (0.1);
\fill [black] (7,0) circle (0.1);
\fill [black] (8,0) circle (0.1);
\fill [black] (9,0) circle (0.1);
\fill [white] (10,0) circle (0.1);
\fill [black] (7,1) circle (0.1);
\fill [black] (3,1) circle (0.1);
\fill [white] (2.05,1.69) circle (0.1);
\fill [white] (3.95,1.69) circle (0.1);
\fill [white] (2.41,2.81) circle (0.1);
\fill [white] (3.59,2.81) circle (0.1);
\fill [white] (6,4) circle (0.1);

\draw (2,0) circle (0.1);
\draw (2-0.69,0.95) circle (0.1);
\draw (2-0.69,-0.95) circle (0.1);
\draw (2-1.81,0.59) circle (0.1);
\draw (2-1.81,-0.59) circle (0.1);
\draw (3,0) circle (0.1);
\draw (4,0) circle (0.1);
\draw (5,0) circle (0.1);
\draw (6,0) circle (0.1);
\draw (7,0) circle (0.1);
\draw (8,0) circle (0.1);
\draw (9,0) circle (0.1);
\draw (10,0) circle (0.1);
\draw (7,1) circle (0.1);
\draw (3,1) circle (0.1);
\draw (2.05,1.69) circle (0.1);
\draw (3.95,1.69) circle (0.1);
\draw (2.41,2.81) circle (0.1);
\draw (3.59,2.81) circle (0.1);
\draw (6,4) circle (0.1);

\end{tikzpicture}
\end{center}

The new elliptic fibration has \MoW\ rank $2$, and we
compute two generators $P$\/ and~$Q$, each of height
$32/15 = 4 - 2/3 - 6/5$, with intersection pairing $7/15$.
We draw $P$ in the figure below, as well as the class of an $A_7$ fiber $F'$.
Because $Q \cdot F' = 1$, the new fibration has a section.

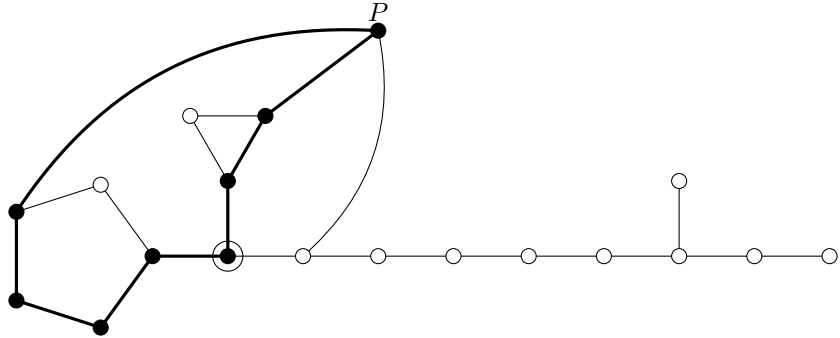
\begin{center}
\begin{tikzpicture}

\draw [very thick] (2,0)--(2-0.69,-0.95)--(2-1.81,-0.59)--(2-1.81,0.59);

\draw (2-1.81,0.59)--(2-0.69,0.95)--(2,0);
\draw [very thick] (2,0)--(3,0);
\draw (3,0)--(11,0);
\draw (9,0)--(9,1);
\draw [very thick] (3,0)--(3,1);
\draw (3,1)--(2.5,1.866)--(3.5,1.866);
\draw [very thick] (3.5,1.866)--(3,1);
\draw [very thick] (5,3)--(3.5,1.866);
\draw [very thick, bend right] (5,3) to (0.19,0.59);
\draw [bend left] (5,3) to (4,0);

\fill [black] (2,0) circle (0.1);
\fill [white] (2-0.69,0.95) circle (0.1);
\fill [black] (2-0.69,-0.95) circle (0.1);
\fill [black] (2-1.81,0.59) circle (0.1);
\fill [black] (2-1.81,-0.59) circle (0.1);

\draw (5,3) [above] node{$P$};
\draw (3,0) circle (0.2);
\fill [black] (3,0) circle (0.1);
\fill [white] (4,0) circle (0.1);
\fill [white] (5,0) circle (0.1);
\fill [white] (6,0) circle (0.1);
\fill [white] (7,0) circle (0.1);
\fill [white] (8,0) circle (0.1);
\fill [white] (9,0) circle (0.1);
\fill [white] (10,0) circle (0.1);
\fill [white] (11,0) circle (0.1);
\fill [white] (9,1) circle (0.1);
\fill [black] (3,1) circle (0.1);
\fill [white] (2.5,1.866) circle (0.1);
\fill [black] (3.5,1.866) circle (0.1);
\fill [black] (5,3) circle (0.1);

\draw (2,0) circle (0.1);
\draw (2-0.69,0.95) circle (0.1);
\draw (2-0.69,-0.95) circle (0.1);
\draw (2-1.81,0.59) circle (0.1);
\draw (2-1.81,-0.59) circle (0.1);
\draw (3,0) circle (0.1);
\draw (4,0) circle (0.1);
\draw (5,0) circle (0.1);
\draw (6,0) circle (0.1);
\draw (7,0) circle (0.1);
\draw (8,0) circle (0.1);
\draw (9,0) circle (0.1);
\draw (10,0) circle (0.1);
\draw (11,0) circle (0.1);
\draw (9,1) circle (0.1);
\draw (3,1) circle (0.1);
\draw (2.5,1.866) circle (0.1);
\draw (3.5,1.866) circle (0.1);
\draw (5,3) circle (0.1);

\end{tikzpicture}
\end{center}

The new elliptic fibration has $A_7$ and $E_8$ fibers, and a section
$P'$ of height $65/8 = 4 + 2 \cdot 3 - 3 \cdot 5/8$. We now go to $E_8
E_7$ via a $2$-neighbor step. Note that the section $P'$ intersects the
new $E_7$ fiber $F''$ in $7$, while the remaining component of the
$A_7$ fiber intersects $F''$ in $2$. Since these are coprime, the genus
$1$ fibration defined by $F''$ has a section.

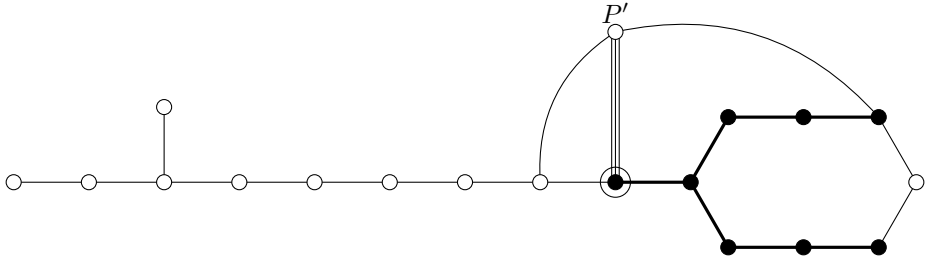
\begin{center}
\begin{tikzpicture}

\draw (-1,0)--(8,0);
\draw (1,0)--(1,1);
\draw (8,0)--(8.5,0.866)--(10.5,0.866)--(11,0)--(10.5,-0.866)--(8.5,-0.866)--(8,0);
\draw (7,2)--(7,0);
\draw (6.95,2)--(6.95,0);
\draw (7.05,2)--(7.05,0);
\draw [bend right] (7,2) to (6,0);
\draw [bend left] (7,2) to (10.5,0.866);
\draw [very thick] (7,0)--(8,0);
\draw [very thick] (10.5,0.866)--(8.5,0.866)--(8,0)--(8.5,-0.866)--(10.5,-0.866);

\fill [white] (-1,0) circle (0.1);
\fill [white] (0,0) circle (0.1);
\fill [white] (1,0) circle (0.1);
\fill [white] (2,0) circle (0.1);
\fill [white] (3,0) circle (0.1);
\fill [white] (4,0) circle (0.1);
\fill [white] (5,0) circle (0.1);
\fill [white] (6,0) circle (0.1);
\fill [white] (1,1) circle (0.1);
\fill [black] (7,0) circle (0.1);
\fill [black] (8,0) circle (0.1);
\fill [black] (8.5,0.866) circle (0.1);
\fill [black] (9.5,0.866) circle (0.1);
\fill [black] (10.5,0.866) circle (0.1);
\fill [black] (8.5,-0.866) circle (0.1);
\fill [black] (9.5,-0.866) circle (0.1);
\fill [black] (10.5,-0.866) circle (0.1);
\fill [white] (11,0) circle (0.1);
\fill [white] (7,2) circle (0.1);

\draw (7,2) [above] node{$P'$};
\draw (7,0) circle (0.2);
\draw (-1,0) circle (0.1);
\draw (0,0) circle (0.1);
\draw (1,0) circle (0.1);
\draw (2,0) circle (0.1);
\draw (3,0) circle (0.1);
\draw (4,0) circle (0.1);
\draw (5,0) circle (0.1);
\draw (6,0) circle (0.1);
\draw (1,1) circle (0.1);
\draw (7,0) circle (0.1);
\draw (8,0) circle (0.1);
\draw (8.5,0.866) circle (0.1);
\draw (9.5,0.866) circle (0.1);
\draw (10.5,0.866) circle (0.1);
\draw (8.5,-0.866) circle (0.1);
\draw (9.5,-0.866) circle (0.1);
\draw (10.5,-0.866) circle (0.1);
\draw (11,0) circle (0.1);
\draw (7,2) circle (0.1);

\end{tikzpicture}
\end{center}

We now read out the Igusa-Clebsch invariants, and compute the equation
of $Y_{-}(65)$ as a double cover of $\Proj^2_{r,s}$. It is branched
over the locus where the K3 surfaces acquire an extra $\I_2$ fiber.

\begin{theorem}
A birational model over $\Q$ for the Hilbert modular surface
$Y_{-}(65)$ as a double cover of $\Proj^2_{r,s}$ is given by the following
equation:
\begin{align*}
z^2 =& -16(s^4+2s^2+13)^2(4s^6+3s^4-10s^2-13) \, r^4 \\
 & \quad   -32s(4s^{12}+15s^{10}+127s^8-10s^6-494s^4-1253s^2-949) \, r^3 \\
 & \quad  -8(12s^{12}+33s^{10}+408s^8-898s^6-2672s^4-2023s^2+404) \, r^2 \\
 & \quad  -8s(4s^{10}+7s^8+149s^6-627s^4-641s^2+148) \, r \\
 & \quad       -(4s^{10}+3s^8+166s^6-997s^4+328s^2-80).
\end{align*}
It is an honestly elliptic surface, with arithmetic genus $2$ and
Picard number $28$.
\end{theorem}

\subsection{Analysis}

This is an honestly elliptic surface, with the extra involution
$\iota: (r,s,z) \mapsto (-r,-s,z)$ corresponding to the factorization
$65 = 5 \cdot 13$.

The branch locus is a curve of genus $1$, isomorphic to the elliptic curve
$$
y^2 + xy = x^3 - x
$$
of conductor $65$.  For lack of space we do not write down the
explicit isomorphism here, relegating the relevant formulae to the
auxiliary files. This elliptic curve is isomorphic to the quotient of
$X_0(65)$ by the Atkin-Lehner involution $w_{65}$.

We were unable to find a section of this genus-$1$ fibration. However,
for purposes of analyzing the Picard number, we study the
Jacobian of this elliptic curve over $\Q(s)$, given by
$$
y^2 = x^3 + (8s^6 + 13s^4 - 106s^2 + 101) \, x^2 + (16s^{12} +
52s^{10} - 564s^8 + 1416s^6 - 1624s^4 + 900s^2 - 196) \,x.
$$
This has reducible fibers of type $\I_8$ at $s = \pm 1$,
$\I_4$ at $s = \infty$, $\I_3$ at $s = \pm \sqrt{13}/3$,
and $\I_2$ at the four roots of $4s^4 + 29s^2 - 49$
(a dihedral extension containing $\sqrt{65}$).
The trivial lattice therefore has rank $27$, leaving room for
\MoW\ rank at most $3$. There is the obvious $2$-torsion section $(0,0)$,
and we find a non-torsion section of height $2/3$:
$$
P = \big(  (73 + 9\mu)/2  (s^2-1)^2  (s^2 + (29/8 - 5/8\mu)), \, (657 + 81\mu)/2 s (s^2-1)^2  (s^2-13/9)  (s^2 + (29/8 - 5/8\mu)) \big)
$$ 
with $\mu = \sqrt{65}$. Analysis of the quotient by $\iota$, and
its twist, shows that the \MoW\ rank is exactly $1$. Therefore the
Picard number of $Y_{-}(65)$ is $28$. The discriminant of the
sublattice of the N\'eron-Severi group generated by the trivial
lattice and these two sections is $6144 = 2^{11} \cdot 3$. We checked
that it is $2$-saturated, and so it equals the entire N\'eron-Severi
lattice. Therefore the sections above generate the \MoW\ group.

The quotient by the involution $\iota$ is given by the equation
\begin{align*}
w^2 &= -16t^2(t^2+2t+13)^2(4t^3+3t^2-10t-13)r^4 \\
  & \quad -32t^2(4t^6+15t^5+127t^4-10t^3-494t^2-1253t-949)r^3 \\
 &\quad  -8t(12t^6+33t^5+408t^4-898t^3-2672t^2-2023t+404)r^2 \\
 &\quad   -8t(4t^5+7t^4+149t^3-627t^2-641t+148)r \\
& \quad -(4t^5+3t^4+166t^3-997t^2+328t-80),
\end{align*}
where $t = s^2$.  Once again we study the Jacobian elliptic fibration:
it has the equation
$$
y^2 =x^3+(8t^4+13t^3-106t^2+101t)\,x^2
+(16t^8+52t^7-564t^6+1416t^5-1624t^4+900t^3-196t^2)\,x.
$$
This is an elliptic K3 surface with bad fibers of type $\I^*_0$ at
$t=0$, $\I_8$ at $t = 1$, $\I_3$ at $t = 13/9$, and $\I_2$ at $t=\infty$
and $t = (-29 \pm 5 \sqrt{65})/8$ .  Therefore the trivial
lattice has rank $18$, and the \MoW\ rank can be at most~$2$.
As before we have a $2$-torsion point $P_1 = (0,0)$. We also find a
non-torsion point
$$
P_2 = \big( (73-9\mu)t (t-1)^2(8t+29+5\mu)/16 ,\, (-73+9\mu)t^2 (t-1)^2 (9t - 13) (8t+29+5\mu) /16  \big)
$$ 
of height $1/3$, with $\mu = \sqrt{65}$ as before. Therefore, the
Picard number is at least $19$, and point counting modulo $11$ and
$23$ shows that the Picard number must be $19$. We verified by
checking $2$-saturation that the sections $P_1$ and $P_2$ and the
trivial lattice span the N\'eron-Severi group, which has discriminant
$64$.

We next analyze the quotient by $\iota' : (r,s,z) \mapsto (-r,-s,-z)$,
which is the quadratic twist of the elliptic K3 surface above by
$\sqrt t$\/:
$$
y^2 =x^3+(8t^3+13t^2-106t+101)\,x^2
+(16t^6+52t^5-564t^4+1416t^3-1624t^2+900t-196)\,x.
$$
This is also an elliptic K3 surface, with reducible fibers of type
$\I_2^*$ at $t = \infty$, $\I_8$ at $t = 1$, $\I_3$ at $t = 13/9$, and
$\I_2$ at $t = (-29 \pm 5 \sqrt{65})/8$. The trivial lattice has rank~$19$.
Again, point counting modulo $11$ and $23$ shows that the Picard
number is $19$, and the \MoW\ group therefore has rank $0$,
with only the $2$-torsion section $(0,0)$.

\subsection{Examples}
We list some points of small height and corresponding genus $2$ curves. 

\begin{tabular}{l|c}
Point $(r,s)$ & Sextic polynomial $f_6(x)$ defining the genus $2$ curve $y^2 = f_6(x)$. \\
\hline \hline \\ [-2.5ex]
$(40/41, -1/5)$ & $ -396x^6 + 216x^5 + 281x^4 - 889x^3 + 50x^2 + 939x + 315 $ \\
$(-40/41, 1/5)$ & $ -648x^5 + 3015x^4 - 422x^3 - 4369x^2 + 2216x - 752 $ \\
$(35/136, 1/5)$ & $ -72x^6 + 969x^5 - 3509x^4 + 847x^3 + 9373x^2 + 816x - 3724 $ \\
$(-40/143, -2/5)$ & $ -240x^6 - 384x^5 + 695x^4 + 2724x^3 + 5543x^2 - 10992x - 2736 $ \\
$(2/15, -2)$ & $ -16200x^5 + 1125x^4 - 8972x^3 - 30493x^2 + 14186x - 18974 $ \\
$(40/143, 2/5)$ & $ -4368x^6 + 420x^5 + 28144x^4 - 13235x^3 - 35846x^2 + 10080x + 14112 $ \\
$(-5/64, 7/5)$ & $ 800x^6 + 6480x^5 + 19405x^4 + 35306x^3 - 39491x^2 - 2688x - 48 $ \\
$(-49/197, 1/2)$ & $ 5088x^6 - 48648x^5 + 85307x^4 + 9352x^3 - 59071x^2 - 15690x + 730 $ \\
$(49/197, -1/2)$ & $ 546x^6 + 9798x^5 + 24115x^4 - 25228x^3 - 98531x^2 + 58920x + 38880 $ \\
$(-35/136, -1/5)$ & $ 2744x^6 - 22344x^5 - 45297x^4 - 16942x^3 + 100440x^2 + 89910x - 72900 $ \\
$(5/64, -7/5)$ & $ -332100x^6 + 344220x^5 - 54545x^4 + 106126x^3 - 68117x^2 + 3528x - 16464 $ \\
$(-2/15, 2)$ & $ -216000x^6 + 506400x^5 - 283195x^4 - 70483x^3 + 13883x^2 + 3456x + 300 $ \\
$(-1/65, 1/2)$ & $ -366600x^6 - 2197788x^5 - 64538x^4 + 11447529x^3 + 133360x^2 - 19021554x + 9447840 $ \\
$(1/65, -1/2)$ & $ -412287975x^6 - 3236837061x^5 + 5479876697x^4 + 3156545763x^3 + 1177706300x^2$ \\
& $ - 7413585000x - 1103500000 $ 
\end{tabular}

\section{Discriminant $69$}

\subsection{Parametrization}

Start with an elliptic K3 surface with fibers of type $E_6, A_8$ and
$A_1$ and a section of height $69/54 = 23/18 = 4 - 1/2 - 20/9$. We can
write down the Weierstrass equation of this family as
$$
y^2 = x^3 + (a_0 + a_1t + a_2t^2)\, x^2 + t^3 (b_0 + b_1t + b_2t^2) \, x + t^6 (c_0 + c_1t + c_2t^2),
$$
with
\begin{align*}
a_0 &= (\sigma_2-\sigma_1+2)^2/4, &
a_1 &= ( \sigma_2^2- \sigma_2(\sigma_1 - 4) - 2(\sigma_1-1))/2, &
a_2 &= ( \sigma_2^2 + 4\sigma_2 - 8)/4; \\
b_0 &= \sigma_2(\sigma_2-\sigma_1+2)^2/2, &
b_1 &=
 (\sigma_2^3 - \sigma_2^2(\sigma_1 - 4) - \sigma_1^2- \sigma_2\sigma_1
 + 2\sigma_1)/2, &
b_2 &= (\sigma_2(\sigma_1 - 6) + 2(\sigma_1 -1))/2; \\
c_0 &= \sigma_2^2(\sigma_2 - \sigma_1 + 2)^2/4,
& c_1 &= \sigma_2(\sigma_1-2)(\sigma_2 - \sigma_1)/2,
& c_2 &= (\sigma_1^2 - 4\sigma_2)/4,
\end{align*}
where
\[
\sigma_1 = r+s, \qquad \sigma_2 = rs.
\]

First we find the class of another $E_6$ fiber below and go to that
elliptic fibration via a $2$-neighbor step.

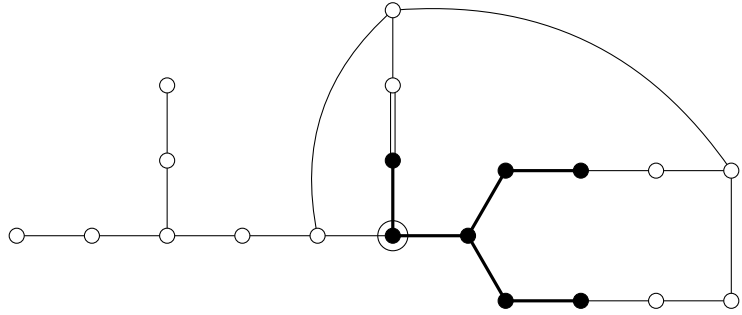
\begin{center}
\begin{tikzpicture}

\draw (0,0)--(6,0);
\draw (2,0)--(2,2);
\draw (5,0)--(5,1);
\draw (4.97,1)--(4.97,2);
\draw (5.03,1)--(5.03,2);
\draw (6,0)--(6.5,0.866)--(9.5,0.866)--(9.5,-0.866)--(6.5,-0.866)--(6,0);
\draw (5,3)--(5,2);
\draw [bend right] (5,3) to (4,0);
\draw [bend left] (5,3) to (9.5,0.866);
\draw [very thick] (5,1)--(5,0)--(6,0);
\draw [very thick] (7.5,0.866)--(6.5,0.866)--(6,0)--(6.5,-0.866)--(7.5,-0.866);

\fill [white] (0,0) circle (0.1);
\fill [white] (1,0) circle (0.1);
\fill [white] (2,0) circle (0.1);
\fill [white] (3,0) circle (0.1);
\fill [white] (4,0) circle (0.1);
\fill [black] (5,0) circle (0.1);
\fill [black] (6,0) circle (0.1);
\fill [white] (2,1) circle (0.1);
\fill [white] (2,2) circle (0.1);
\fill [black] (5,1) circle (0.1);
\fill [white] (5,2) circle (0.1);
\fill [black] (6.5,0.866) circle (0.1);
\fill [black] (7.5,0.866) circle (0.1);
\fill [white] (8.5,0.866) circle (0.1);
\fill [white] (9.5,0.866) circle (0.1);
\fill [black] (6.5,-0.866) circle (0.1);
\fill [black] (7.5,-0.866) circle (0.1);
\fill [white] (8.5,-0.866) circle (0.1);
\fill [white] (9.5,-0.866) circle (0.1);
\fill [white] (5,3) circle (0.1);

\draw (5,0) circle (0.2);
\draw (0,0) circle (0.1);
\draw (1,0) circle (0.1);
\draw (2,0) circle (0.1);
\draw (3,0) circle (0.1);
\draw (4,0) circle (0.1);
\draw (5,0) circle (0.1);
\draw (6,0) circle (0.1);
\draw (2,1) circle (0.1);
\draw (2,2) circle (0.1);
\draw (5,1) circle (0.1);
\draw (5,2) circle (0.1);
\draw (6.5,0.866) circle (0.1);
\draw (7.5,0.866) circle (0.1);
\draw (8.5,0.866) circle (0.1);
\draw (9.5,0.866) circle (0.1);
\draw (6.5,-0.866) circle (0.1);
\draw (7.5,-0.866) circle (0.1);
\draw (8.5,-0.866) circle (0.1);
\draw (9.5,-0.866) circle (0.1);
\draw (5,3) circle (0.1);

\end{tikzpicture}
\end{center}

The resulting elliptic fibration has $E_6, E_6$ and $A_3$ fibers, and
a section of height $23/12 = 4 - 4/3 - 3/4$. Now we find the class of
an $A_7$ fiber in the diagram below.

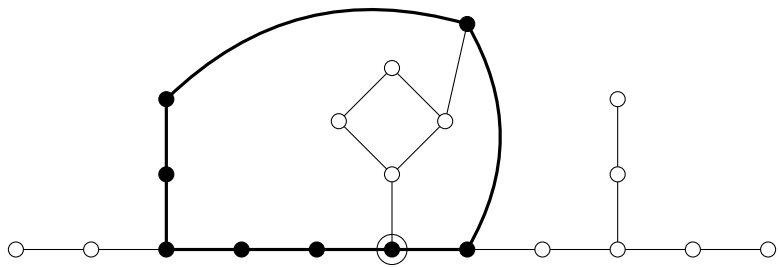
\begin{center}
\begin{tikzpicture}

\draw (0,0)--(10,0);
\draw (2,0)--(2,2);
\draw (5,0)--(5,1);
\draw (8,0)--(8,2);
\draw (5,1)--(4.293,1.707)--(5,2.414)--(5.707,1.707)--(5,1);
\draw (6,3)--(5.707,1.707);
\draw [bend right, very thick] (6,3) to (2,2);
\draw [bend left, very thick] (6,3) to (6,0);
\draw [very thick] (2,2)--(2,0)--(6,0);

\fill [white] (0,0) circle (0.1);
\fill [white] (1,0) circle (0.1);
\fill [black] (2,0) circle (0.1);
\fill [black] (3,0) circle (0.1);
\fill [black] (4,0) circle (0.1);
\fill [black] (5,0) circle (0.1);
\fill [black] (6,0) circle (0.1);
\fill [white] (7,0) circle (0.1);
\fill [white] (8,0) circle (0.1);
\fill [white] (9,0) circle (0.1);
\fill [white] (10,0) circle (0.1);
\fill [black] (2,1) circle (0.1);
\fill [black] (2,2) circle (0.1);
\fill [white] (5,1) circle (0.1);
\fill [white] (8,1) circle (0.1);
\fill [white] (8,2) circle (0.1);
\fill [white] (4.293,1.707) circle (0.1);
\fill [white] (5.707,1.707) circle (0.1);
\fill [white] (5,2.414) circle (0.1);
\fill [black] (6,3) circle (0.1);

\draw (5,0) circle (0.2);
\draw (0,0) circle (0.1);
\draw (1,0) circle (0.1);
\draw (2,0) circle (0.1);
\draw (3,0) circle (0.1);
\draw (4,0) circle (0.1);
\draw (5,0) circle (0.1);
\draw (6,0) circle (0.1);
\draw (7,0) circle (0.1);
\draw (8,0) circle (0.1);
\draw (9,0) circle (0.1);
\draw (10,0) circle (0.1);
\draw (2,1) circle (0.1);
\draw (2,2) circle (0.1);
\draw (5,1) circle (0.1);
\draw (8,1) circle (0.1);
\draw (8,2) circle (0.1);
\draw (4.293,1.707) circle (0.1);
\draw (5.707,1.707) circle (0.1);
\draw (5,2.414) circle (0.1);
\draw (6,3) circle (0.1);

\end{tikzpicture}
\end{center}

The resulting fibration has $A_7, A_5, A_2$ and $A_1$ fibers, a
$2$-torsion section, and a non-torsion section $P$ of height $69/72 =
23/24 = 4 - 2/3 - 3\cdot 3/6 - 1 \cdot 7/8$. We next identify the
class $F'$ of an $E_7$ fiber, and go to it via a $2$-neighbor
step. Note that $P \cdot F' = 1$, so the new fibration has a section.

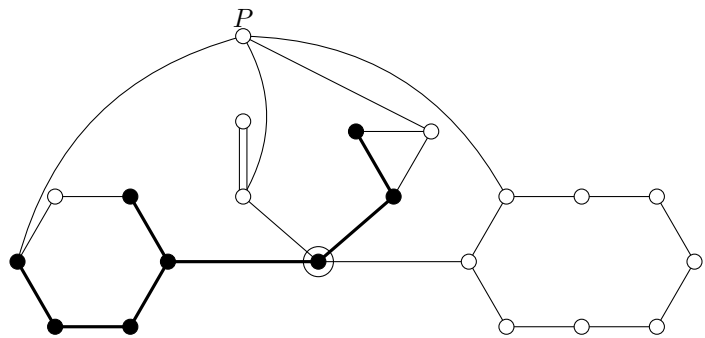
\begin{center}
\begin{tikzpicture}

\draw (2,0)--(1.5,0.866)--(0.5,0.866)--(0,0)--(0.5,-0.866)--(1.5,-0.866)--(2,0)--(6,0)--(6.5,0.866)--(8.5,0.866)--(9,0)--(8.5,-0.866)--(6.5,-0.866)--(6,0);
\draw (3,0.866)--(4,0)--(5,0.866)--(4.5,1.732)--(5.5,1.732)--(5,0.866);
\draw (2.95,0.866)--(2.95,1.866);
\draw (3.05,0.866)--(3.05,1.866);
\draw [bend left] (3,3) to (3,0.866);
\draw (3,3)--(5.5,1.732);
\draw [bend left] (3,3) to (6.5,0.866);
\draw [bend right] (3,3) to (0,0);
\draw [very thick] (0,0)--(0.5,-0.866)--(1.5,-0.866)--(2,0)--(4,0)--(5,0.866)--(4.5,1.732);
\draw [very thick] (1.5,0.866)--(2,0);

\draw (3,3) [above] node{$P$};
\draw (4,0) circle (0.2);
\fill [black] (0,0) circle (0.1);
\fill [white] (0.5,0.866) circle (0.1);
\fill [black] (0.5,-0.866) circle (0.1);
\fill [black] (1.5,0.866) circle (0.1);
\fill [black] (1.5,-0.866) circle (0.1);
\fill [black] (2,0) circle (0.1);
\fill [black] (4,0) circle (0.1);
\fill [white] (6,0) circle (0.1);
\fill [white] (6.5,0.866) circle (0.1);
\fill [white] (6.5,-0.866) circle (0.1);
\fill [white] (7.5,0.866) circle (0.1);
\fill [white] (7.5,-0.866) circle (0.1);
\fill [white] (8.5,0.866) circle (0.1);
\fill [white] (8.5,-0.866) circle (0.1);
\fill [white] (9,0) circle (0.1);
\fill [white] (3,0.866) circle (0.1);
\fill [black] (5,0.866) circle (0.1);
\fill [white] (3,1.866) circle (0.1);
\fill [white] (5.5,1.732) circle (0.1);
\fill [black] (4.5,1.732) circle (0.1);
\fill [white] (3,3) circle (0.1);

\draw (0,0) circle (0.1);
\draw (0.5,0.866) circle (0.1);
\draw (0.5,-0.866) circle (0.1);
\draw (1.5,0.866) circle (0.1);
\draw (1.5,-0.866) circle (0.1);
\draw (2,0) circle (0.1);
\draw (4,0) circle (0.1);
\draw (6,0) circle (0.1);
\draw (6.5,0.866) circle (0.1);
\draw (6.5,-0.866) circle (0.1);
\draw (7.5,0.866) circle (0.1);
\draw (7.5,-0.866) circle (0.1);
\draw (8.5,0.866) circle (0.1);
\draw (8.5,-0.866) circle (0.1);
\draw (9,0) circle (0.1);
\draw (3,0.866) circle (0.1);
\draw (5,0.866) circle (0.1);
\draw (3,1.866) circle (0.1);
\draw (5.5,1.732) circle (0.1);
\draw (4.5,1.732) circle (0.1);
\draw (3,3) circle (0.1);

\end{tikzpicture}
\end{center}

The new elliptic fibration has $E_7$, $A_7$ and $A_1$ fibers, a
$2$-torsion section $Q'$, and a non-torsion section $P'$ of height $69/8
= 4 + 2\cdot 3 - 1/2 - 7/8$. We identify a fiber $F''$ of type $E_7$
below.

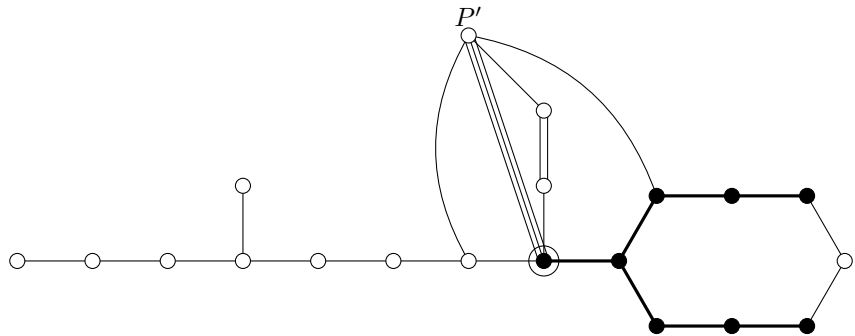
\begin{center}
\begin{tikzpicture}
\draw (0,0)--(8,0)--(8.5,0.866)--(10.5,0.866)--(11,0)--(10.5,-0.866)--(8.5,-0.866)--(8,0);
\draw (7,0)--(7,1);
\draw (7.05,1)--(7.05,2);
\draw (6.95,1)--(6.95,2);
\draw (3,0)--(3,1);
\draw [bend left] (6,3) to (8.5,0.866);
\draw [bend right] (6,3) to (6,0);
\draw (6,3)--(7,2);
\draw [very thick] (7,0)--(8,0);
\draw [very thick] (10.5,0.866)--(8.5,0.866)--(8,0)--(8.5,-0.866)--(10.5,-0.866);
\draw (6,3)--(7,0);
\draw (6.07,3)--(7.07,0);
\draw (5.93,3)--(6.93,0);

\draw (7,0) circle (0.2);
\fill [white] (0,0) circle (0.1);
\fill [white] (1,0) circle (0.1);
\fill [white] (2,0) circle (0.1);
\fill [white] (3,0) circle (0.1);
\fill [white] (4,0) circle (0.1);
\fill [white] (5,0) circle (0.1);
\fill [white] (6,0) circle (0.1);
\fill [black] (7,0) circle (0.1);
\fill [black] (8,0) circle (0.1);
\fill [black] (8.5,0.866) circle (0.1);
\fill [black] (8.5,-0.866) circle (0.1);
\fill [black] (9.5,0.866) circle (0.1);
\fill [black] (9.5,-0.866) circle (0.1);
\fill [black] (10.5,0.866) circle (0.1);
\fill [black] (10.5,-0.866) circle (0.1);
\fill [white] (11,0) circle (0.1);
\fill [white] (7,1) circle (0.1);
\fill [white] (7,2) circle (0.1);
\fill [white] (3,1) circle (0.1);
\fill [white] (6,3) circle (0.1);

\draw (0,0) circle (0.1);
\draw (1,0) circle (0.1);
\draw (2,0) circle (0.1);
\draw (3,0) circle (0.1);
\draw (4,0) circle (0.1);
\draw (5,0) circle (0.1);
\draw (6,0) circle (0.1);
\draw (7,0) circle (0.1);
\draw (8,0) circle (0.1);
\draw (8.5,0.866) circle (0.1);
\draw (8.5,-0.866) circle (0.1);
\draw (9.5,0.866) circle (0.1);
\draw (9.5,-0.866) circle (0.1);
\draw (10.5,0.866) circle (0.1);
\draw (10.5,-0.866) circle (0.1);
\draw (11,0) circle (0.1);
\draw (7,1) circle (0.1);
\draw (7,2) circle (0.1);
\draw (3,1) circle (0.1);
\draw (6,3) circle (0.1);
\draw (6,3) [above] node{$P'$};

\end{tikzpicture}
\end{center}

Note that $P' \cdot F'' = 2 \cdot 3 + 3 = 9$, while the remaining
component of the $A_7$ fiber intersects $F''$ with mulstiplicity~$2$.
Therefore the new elliptic fibration has a section.
Converting to the Jacobian, we read out the Weierstrass coefficients of
the $E_8 E_7$ form, which give us the Igusa-Clebsch invariants.

\begin{theorem}
A birational model over $\Q$ for the Hilbert modular surface
$Y_{-}(69)$ as a double cover of $\Proj^2_{r,s}$ is given by the following
equation:
\begin{align*}
z^2 =\;&  (r-1)^2r^4s^6 -2(r-1)r^2(r^3+13r^2-37r+22)s^5 \\
 & \quad  +(r^6+100r^5-439r^4+640r^3-357r^2+72r-16)s^4 \\
 & \quad -2(59r^5-320r^4+590r^3-436r^2+133r-32)s^3 \\
 & \quad  +(44r^5-357r^4+872r^3-830r^2+314r-83)s^2 \\
 & \quad  + 2(36r^4-133r^3+157r^2-65r+19)s -16r^4+64r^3-83r^2+38r-11.
\end{align*}
It is a surface of general type.
\end{theorem}

\subsection{Analysis}

This is a surface of general type, with the extra involution
$(r,s,z) \mapsto (s,r,z)$, corresponding to $69 = 3 \cdot 23$.

The branch locus is a curve of genus $2$; the transformation
\begin{align*}
r &= -\frac{x^3y+x^2y-y-3x^6-4x^5+x^4+6x^3-2x^2-2x+1}{2x^2(x^2+x-1)^2} \\
s &= \frac{x^3y+x^2y-y+3x^6+4x^5-x^4-6x^3+2x^2+2x-1}{2x^2(x^2+x-1)^2}
\end{align*}
converts it into Weierstrass form
$$
y^2 = (x^3+x^2-1)(5x^3-7x^2+4x-1).
$$
It is isomorphic to $X_0(69)/ \langle w \rangle$, where $w$ is the
Atkin-Lehner involution.

The quotient surface is (with $m = r+s, n = rs$)
\begin{align*}
z^2 &= -16m^4+4(11n^2+18n+16)m^3+(n^4-118n^3-357n^2-202n-83)m^2 \\
   &\quad  -2(n^5-50n^4-254n^3-328n^2-61n-19)m  +n^6-26n^5-203n^4-466n^3-330n^2+36n-11.
\end{align*}
The substitution $m = n + k$ makes the right hand side quartic in $n$,
with highest coefficient a square. Converting to the Jacobian, we get
(after some Weierstrass transformations and change of the parameter on
the base) the elliptic K3 surface
$$
y^2 = x^3 -(88t^3+15t^2+6t-1)\, x^2 + 8t^3(250t^3+57t^2+45t-8)\, x
 + 16t^5(1125t^3+552t^2+208t-36).
$$ 
This has fibers of type $\I_5$ at $t = 0$, $\I^*_0$ at $t = \infty$,
$\I_2$ at $t = (-3 \pm 2\sqrt{3})/4$, and $\I_3$ at the roots of $25
t^3 + 17 t^2 + 2 t - 1$ (which generates the cubic field of
discriminant $-23$). The trivial lattice has rank $18$, leaving room
for at most two independent sections.  We find the non-torsion section
$$
P_1 = \big(4 t (1-t), 4 t (25 t^3+17 t^2+2 t-1) \big)
$$ 
of height $1/5$. Counting points modulo $7$ and $13$ then shows
that the Picard number must be exactly $19$. Therefore, the
discriminant of the sublattice spanned by $P_1$ and the trivial
lattice is $432 = 2^4 \cdot 3^3$.  Looking at the contributions to the
N\'eron-Tate height from the fiber configuration, one easily sees that
there cannot be any $2$- or $3$-torsion. Similarly, it is impossible
to have a section of height $1/20$ or $1/45$. Therefore, this
sublattice must be the entire N\'eron-Severi lattice, and $P_1$ is a
generator of the \MoW\ group.

\subsection{Examples}
We list some points of small height and corresponding genus $2$ curves.

\begin{tabular}{l|c}
Point $(r,s)$ & Sextic polynomial $f_6(x)$ defining the genus $2$ curve $y^2 = f_6(x)$. \\
\hline \hline \\ [-2.5ex]
$(5/6, 5/2)$ & $ -144x^6 - 336x^5 + 491x^4 - 274x^3 + 4919x^2 - 23076x - 6476 $ \\
$(5/7, 10/3)$ & $ -108456x^6 + 89940x^5 + 3518x^4 + 11915x^3 + 29021x^2 - 40515x + 2841 $ \\
$(10/3, 5/7)$ & $ -7146x^6 + 26076x^5 + 26698x^4 - 128487x^3 - 87881x^2 + 140967x + 106899 $ \\
$(9/14, 3)$ & $ -205648x^6 - 71112x^5 + 4931x^4 - 3219x^3 - 1369x^2 + 336x + 64 $ \\
$(3, 9/14)$ & $ -47792x^6 + 212184x^5 - 134731x^4 - 131082x^3 + 58025x^2 + 39900x + 3500 $ \\
$(5/2, 5/6)$ & $ 111132x^6 + 308700x^5 + 150199x^4 - 166350x^3 - 85877x^2 + 37080x + 3208 $ \\
$(5/4, -25/3)$ & $ -203124x^6 + 537156x^5 - 1147529x^4 - 958036x^3 - 185681x^2 + 583356x - 97236 $ \\
$(31/18, 3)$ & $ -4138876x^6 - 12791196x^5 - 14043627x^4 - 2580588x^3 - 2332545x^2$ \\
& $ - 7239150x - 962750 $ \\
$(97/34, 3)$ & $ -4774900x^6 + 11612125x^5 - 1487685x^4 - 13117009x^3 + 29039993x^2$ \\
& $ + 24527448x - 2106844 $ \\
$(3, -13/9)$ & $ -1749188x^6 + 9004404x^5 - 5544841x^4 - 13022828x^3 - 36459313x^2 $ \\
& $ + 31091676x + 62509084 $ \\
$(3, 31/18)$ & $ 1490720x^6 + 34810248x^5 + 203477725x^4 - 39952362x^3 - 392594159x^2$ \\
& $ + 53751372x - 121943204 $ \\
$(31/5, 5/6)$ & $ 354444x^6 + 2968308x^5 - 37732823x^4 + 4713146x^3 + 223323505x^2$ \\
& $ + 714572220x - 955946700 $ \\
$(-25/3, 5/4)$ & $ -28660432x^6 - 9277032x^5 - 367100597x^4 + 64181262x^3 - 1142233133x^2$ \\
& $ + 827398968x - 146839168 $ \\
$(9/8, 29/13)$ & $ 953161100x^6 - 1709768900x^5 + 57159815x^4 - 997590336x^3 - 322970431x^2$ \\
& $ - 72177962x + 2546122 $ \\
$(5/6, 31/5)$ & $ 6865596x^6 - 82041816x^5 + 58608103x^4 + 1250964773x^3 + 921256891x^2$ \\
& $ - 4495870224x - 3609058132 $ \\
$(-53/7, 71/57)$ & $ -125838448x^6 + 33513120x^5 - 1068122125x^4 + 1220630640x^3 - 2407159591x^2$ \\
& $ + 4627695870x - 2358802782 $ 
\end{tabular}

By pulling back some of the low height sections of the quotient,
we produce curves of low genus on $Y_{-}(69)$.  Since $r + s = m$ and $rs = n$,
the appropriate condition is that $m^2 - 4n = (r-s)^2$ be a square
(then we can take a square root and solve for $r,s$).
In other words, $(n+k)^2 - 4n$ must be a square.
The section $P_1$ is given by $n = -(5k^2-17k+15)/(k-2)$;
it gives rise to a genus~$1$ curve
$$
y^2 = 16k^4-100k^3+237k^2-254k+105
$$
which has rational points (for instance, at infinity).
It is therefore an elliptic curve, and we calculate that it has
conductor $1711$ (prime) and \MoW\ group $\Z^2$.

Similarly, the section $-P_1$ is defined by $n = -(25k^2-92k+89)/(10k-21)$.
It also gives rise to a genus-$1$ curve
$$
y^2 = 225k^4-1130k^3+1931k^2-1350k+445
$$
with rational points (as at $k=\infty$).
It is an elliptic curve of conductor $50435 = 5 \cdot 7 \cdot 11 \cdot 131$,
with trivial torsion and rank~$1$.

\section{Discriminant $73$}

\subsection{Parametrization}

Start with a K3 elliptic surface with fibers of type $A_6, A_{9}$ and
a section of height $73/70 = 4 - 21/10 - 6/7$.  The Weierstrass equation is
$$
y^2 = x^3 + (a_0 + a_1t + a_2t^2 + a_3t^3 + a_4t^4)\, x^2
 + 2\mu t (b_0 + b_1t + b_2t^2 + b_3t^3)\, x + \mu^2t^2(c_0 + c_1t + c_2t^2),
$$
with
\begin{align*}
\mu &= -2s^4(s+2r)^2, &
a_0 &= s^2(rs-2r-2)^2, &
a_4 &= 16r^2(r+1)^2, &
b_0 &= s(rs-2r-2)^2, \\
b_3 &= 4r(r+1)^2(s+2r+2), &
c_0 &= (rs-2r-2)^2, &
c_2 &= (r+1)^2(s+2r+2)^2, & &
\end{align*}
\begin{align*}
a_3 &= 8r(r+1)\big( (r+1)s^2 + 2(4r+1)s + 4r(r-1) \big), \quad
c_1 = -2(s+2r+2)\big( s^2 + (r+2)(r-1)s - 2(r^2-1) \big), \\
b_1 &= -3s^4 - 2r(r+5)s^3 -2(r^3+4r^2-2r-2)s^2 -8(r+1)s + 8r(r+1)^2, \\
a_1 &= -2s \big(3s^4 + r(r+11)s^3 + 2(5r^2-r-1)s^2 + 4(r+1)(r^2+1)s - 8r(r+1)^2\big), \\
b_2 &= (r+1)^2s^3 + 2(r^3+7r^2+6r+2)s^2 + 4(4r^3+8r^2+7r+1)s +8(r-1)r(r+1)(r+2), \\
a_2 &= (r+1)^2s^4 + 4(4r^2+3r+1)s^3 + 4(4r^3+10r^2+10r+1)s^2 + 16r(r+1)(r-2)s + 16r^2(r+1)^2.
\end{align*}

We first identify the class of an $E_8$ fiber, and perform a $3$-neighbor
step to move to an elliptic fibration with $E_8$ and $A_7$ fibers.

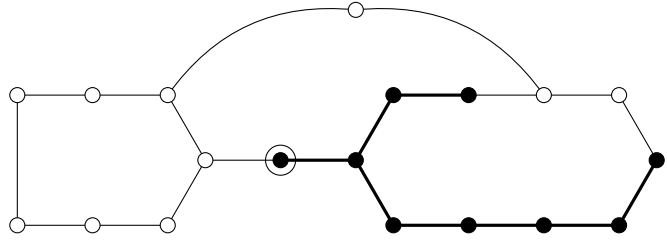
\begin{center}
\begin{tikzpicture}

\draw (0,0)--(2,0)--(2.5,0.866)--(5.5,0.866)--(6,0)--(5.5,-0.866)--(2.5,-0.866)--(2,0);
\draw (0,0)--(-0.5,0.866)--(-2.5,0.866)--(-2.5,-0.866)--(-0.5,-0.866)--(0,0);
\draw [bend right] (2,2) to (-0.5,0.866);
\draw [bend left] (2,2) to (4.5,0.866);
\draw [very thick] (1,0)--(2,0)--(2.5,0.866)--(3.5,0.866);
\draw [very thick] (2,0)--(2.5,-0.866)--(5.5,-0.866)--(6,0);

\draw (1,0) circle (0.2);
\fill [white] (0,0) circle (0.1);
\fill [black] (1,0) circle (0.1);
\fill [black] (2,0) circle (0.1);
\fill [black] (2.5,0.866) circle (0.1);
\fill [black] (3.5,0.866) circle (0.1);
\fill [white] (4.5,0.866) circle (0.1);
\fill [white] (5.5,0.866) circle (0.1);
\fill [black] (2.5,-0.866) circle (0.1);
\fill [black] (3.5,-0.866) circle (0.1);
\fill [black] (4.5,-0.866) circle (0.1);
\fill [black] (5.5,-0.866) circle (0.1);
\fill [black] (6,0) circle (0.1);

\fill [white] (-0.5,0.866) circle (0.1);
\fill [white] (-1.5,0.866) circle (0.1);
\fill [white] (-2.5,0.866) circle (0.1);
\fill [white] (-0.5,-0.866) circle (0.1);
\fill [white] (-1.5,-0.866) circle (0.1);
\fill [white] (-2.5,-0.866) circle (0.1);

\fill [white] (2,2) circle (0.1);

\draw (0,0) circle (0.1);
\draw (1,0) circle (0.1);
\draw (2,0) circle (0.1);
\draw (2.5,0.866) circle (0.1);
\draw (3.5,0.866) circle (0.1);
\draw (4.5,0.866) circle (0.1);
\draw (5.5,0.866) circle (0.1);
\draw (2.5,-0.866) circle (0.1);
\draw (3.5,-0.866) circle (0.1);
\draw (4.5,-0.866) circle (0.1);
\draw (5.5,-0.866) circle (0.1);
\draw (6,0) circle (0.1);

\draw (-0.5,0.866) circle (0.1);
\draw (-1.5,0.866) circle (0.1);
\draw (-2.5,0.866) circle (0.1);
\draw (-0.5,-0.866) circle (0.1);
\draw (-1.5,-0.866) circle (0.1);
\draw (-2.5,-0.866) circle (0.1);

\draw (2,2) circle (0.1);

\end{tikzpicture}
\end{center}

This fibration has a section of height $73/8 = 4 + 2\cdot 3 - 7/8$.
Next, we locate an $E_7$ fiber and compute a $2$-neighbor step to go
to a fibration with $E_8$ and $E_7$ fibers.

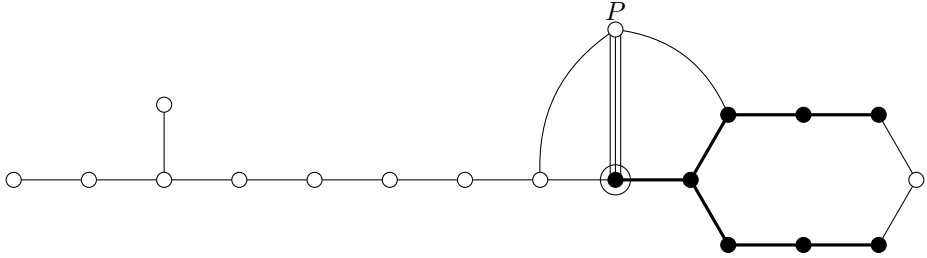
\begin{center}
\begin{tikzpicture}

\draw (7,0)--(7,2);
\draw (6.93,0)--(6.93,2);
\draw (7.07,0)--(7.07,2);
\draw (-1,0)--(8,0);
\draw (1,0)--(1,1);
\draw (8,0)--(8.5,0.866)--(10.5,0.866)--(11,0)--(10.5,-0.866)--(8.5,-0.866)--(8,0);

\draw [bend right] (7,2) to (6,0);
\draw [bend left] (7,2) to (8.5,0.866);
\draw [very thick] (7,0)--(8,0);
\draw [very thick] (10.5,0.866)--(8.5,0.866)--(8,0)--(8.5,-0.866)--(10.5,-0.866);

\draw(7,0) circle (0.2);
\fill [white] (-1,0) circle (0.1);
\fill [white] (0,0) circle (0.1);
\fill [white] (1,0) circle (0.1);
\fill [white] (2,0) circle (0.1);
\fill [white] (3,0) circle (0.1);
\fill [white] (4,0) circle (0.1);
\fill [white] (5,0) circle (0.1);
\fill [white] (6,0) circle (0.1);
\fill [white] (1,1) circle (0.1);
\fill [black] (7,0) circle (0.1);
\fill [black] (8,0) circle (0.1);
\fill [black] (8.5,0.866) circle (0.1);
\fill [black] (9.5,0.866) circle (0.1);
\fill [black] (10.5,0.866) circle (0.1);
\fill [black] (8.5,-0.866) circle (0.1);
\fill [black] (9.5,-0.866) circle (0.1);
\fill [black] (10.5,-0.866) circle (0.1);
\fill [white] (11,0) circle (0.1);
\fill [white] (7,2) circle (0.1);

\draw (-1,0) circle (0.1);
\draw (0,0) circle (0.1);
\draw (1,0) circle (0.1);
\draw (2,0) circle (0.1);
\draw (3,0) circle (0.1);
\draw (4,0) circle (0.1);
\draw (5,0) circle (0.1);
\draw (6,0) circle (0.1);
\draw (1,1) circle (0.1);
\draw (7,0) circle (0.1);
\draw (8,0) circle (0.1);
\draw (8.5,0.866) circle (0.1);
\draw (9.5,0.866) circle (0.1);
\draw (10.5,0.866) circle (0.1);
\draw (8.5,-0.866) circle (0.1);
\draw (9.5,-0.866) circle (0.1);
\draw (10.5,-0.866) circle (0.1);
\draw (11,0) circle (0.1);
\draw (7,2) circle (0.1);
\draw [above] (7,2) node {$P$};

\end{tikzpicture}
\end{center}

The intersection number of the new fiber $F'$ with the remaining
component of the $A_7$ fiber is $2$ and with the section $P$ is $9$.
Therefore, the new genus-$1$ fibration defined by $F'$ has a
section since $9$ and $2$ are coprime.

\begin{theorem}
A birational model over $\Q$ for the Hilbert modular surface
$Y_{-}(73)$ as a double cover of $\Proj^2_{r,s}$ is given by the following
equation:
\begin{align*}
z^2 &= 16(s-2)^2r^4+ 8(s-2)(17s^3-52s^2+36s-8)r^3 + (s^6+56s^5-384s^4+448s^3+432s^2-512s+64)r^2 \\
 & \quad       + 2s(s+2)(s^4-34s^3+108s^2-64s+16)r + s^2(s+2)^4 .
\end{align*}
It is an honestly elliptic surface, with arithmetic genus $2$ and
Picard number $28$.
\end{theorem}

\subsection{Analysis}

The branch locus is a curve of genus $2$; the transformation of
coordinates
\begin{align*}
r &= \frac{3x^3y-3xy+y-x^6-2x^5-4x^4-3x^2+4x-1}{2x^2(x^2+2x-1)^2} \\
s &= -2\frac{(2xy-y-x^3-7x^2+3x)}{(x+1)^2(x^2+2x-1)}
\end{align*}
converts it to Weierstrass form
$$
y^2 = x^6 + 4x^5 + 2x^4 - 6x^3 + x^2 - 2x + 1.
$$
This is a genus $2$ curve, isomorphic to the quotient of $X_0(73)$ by
the Atkin-Lehner involution.

The Hilbert modular surface $Y_{-}(73)$ is an elliptic surface. Since
the coefficient of $r^4$ is a square, this genus-$1$ curve over
$\Proj^1_s$ has a section. Computing the Jacobian, we get the
following Weierstrass equation after a change of parameter on the base
and some simple Weierstrass transformations:
\begin{align*}
y^2 &= x^3 -(83t^6-316t^5+390t^4-158t^3-21t^2+22t-1) \,x^2  \\
 &\quad  + 8(t-1)^4t^3(287t^5-1040t^4+960t^3-73t^2-217t+67) \,x \\
 &\quad  -16(t-1)^7t^6(1323t^5-5887t^4+7110t^3-1426t^2-2201t+1033).
\end{align*}
This is an honestly elliptic surface with $\chi=3$.
It has reducible fibers of type $\I_7$ at $t = 1$,
$\I_6$ at $t = 0$, $\I_5$ at $t = \infty$,
$\I_3$ at $t = (-5 \pm \sqrt{73})/4$,
and $\I_2$ at $t = -1$ and $(13 \pm \sqrt{73})/24$.
The trivial lattice therefore has rank~$24$,
leaving room for \MoW\ rank at most~$6$.

We find the following four independent sections.
\begin{align*}
P_1 &= \big( 4t^3 (t-1)^2 (9t-7), \, 4t^3 (t-1)^2(t+1)(3t-2)(2t^2+5t-6)\big),\\
P_2 &= \big( 4t^3(t-1)(7t^2-17t+8), \, 4t^3(t-1)(t+1)(9t^2-13t+5) \big), \\
P_3 &= \big( 4t^2(t-1)^3(19t-16), \, 4t^2(t-1)^3(7t^2-12t+8)(12t^2-13t+2)\big), \\
P_4 &= \big( 4t^3(7t^3-24t^2-2\mu t+42t-2\mu +9) , \\
 & \qquad (-13+\mu )t^3(t+1)(-24t+13+\mu )(-17-16t+3\mu )(-4t-5+\mu )/192 \big)
\end{align*}
(where $\mu = \sqrt{73}$), with non-degenerate height pairing matrix
$$
\left( \begin{array}{cccc}
 26/21 &  5/7  &  -1/7 & -2/3 \\
 5/7  & 68/35 &  11/7 & -1/5 \\
 -1/7 & 11/7 &  41/21 & 1/2 \\
 -2/3 & -1/5  &  1/2    & 49/30
\end{array} \right).
$$ 
Therefore, the \MoW\ rank is at least $4$.  From Oda's
calculations \cite[pg.~109]{Oda}, the Picard number is $28$, so the
\MoW\ rank is $28-24 = 4$ and our sections generate a subgroup of
finite index in the full \MoW\ group.  The sublattice of the
N\'{e}ron-Severi lattice generated by the trivial lattice and these
sections has discriminant $3916 = 2^2 \cdot 11 \cdot 89$. We checked
that this sublattice is $2$-saturated, and therefore it is the entire
N\'eron-Severi lattice.

\subsection{Examples}

We list some points of small height and corresponding genus $2$ curves.

\begin{tabular}{l|c}
Rational point $(r,s)$ & Sextic polynomial $f_6(x)$ defining the genus $2$ curve $y^2 = f_6(x)$. \\
\hline \hline \\ [-2.5ex]
$(-10/3, 4/3)$ & $ -4x^6 - 12x^5 - 23x^4 + 4x^3 + 31x^2 + 57x - 18 $ \\
$(-3, 1)$ & $ 4x^6 - 3x^4 - 35x^3 + 12x + 76 $ \\
$(5, 3)$ & $ -4x^6 - 24x^5 + 7x^4 + 83x^3 + 25x^2 - 75x - 40 $ \\
$(9/4, 3)$ & $ -15x^5 + 73x^4 - 41x^3 - 158x^2 - 12x + 36 $ \\
$(-5, 3/2)$ & $ -50x^6 + 45x^5 - 2x^4 - 159x^3 + 70x^2 + 12x - 120 $ \\
$(5/6, -2)$ & $ -48x^6 + 168x^5 - 149x^4 + 56x^3 - 53x^2 + 12x - 4 $ \\
$(-5/12, -1)$ & $ 195x^6 + 82x^5 - 75x^4 - 186x^3 - 233x^2 - 96x - 87 $ \\
$(5/3, 1/2)$ & $ -60x^6 - 105x^5 + 55x^4 + 49x^3 + 13x^2 + 252x + 116 $ \\
$(-2/3, -2/5)$ & $ -36x^6 + 39x^5 - 217x^4 + 129x^3 - 271x^2 - 108x + 84 $ \\
$(5/4, -6)$ & $ 20x^6 + 84x^5 - 15x^4 - 162x^3 + 225x^2 + 324x - 180 $ \\
$(9/5, -8)$ & $ -204x^6 + 348x^5 + 27x^4 + 34x^3 + 3x^2 + 108x - 36 $ \\
$(-1/3, 4/5)$ & $ -25x^6 + 135x^5 + 139x^4 - 383x^3 + 82x^2 + 252x - 162 $ \\
$(9/14, -2)$ & $ -48x^6 - 72x^5 - 219x^4 - 319x^3 - 159x^2 - 432x + 192 $ \\
$(-14/5, 12)$ & $ -440x^6 + 90x^5 + 324x^4 + 38x^3 - 9x^2 - 36x - 8 $ \\
$(-9/4, 1)$ & $ -455x^6 - 420x^5 + 66x^4 - 167x^3 - 15x^2 - 3x - 6 $ \\
$(6/5, 4)$ & $ -160x^6 + 450x^5 - 114x^4 - 474x^3 + 171x^2 + 162x - 27 $ 
\end{tabular}

We get many curves of genus $0$ on the surface by taking sections of
the elliptic fibration. For instance, the Brauer obstruction vanishes
for the two curves defined by $r = -(s-4)(s+2)/\big(4(s-2)\big)$ and
$r = s(s+2)/\big((s-2)(3s-2)\big)$, yielding families of genus $2$
curves parametrized by $s$, whose Jacobians have real multiplication
by $\sO_{73}$.

\section{Discriminant $76$}

\subsection{Parametrization}

Start with a K3 elliptic surface with fibers of type $A_6, A_2$ and
$D_7$, and a section of height $76/84 = 19/21 = 4 - 1 - 2/3 - 10/7$.
The Weierstrass equation of this family is
$$
y^2 =  x^3 + (a_0 + a_1 t + a_2 t^2 + a_3 t^3)\, x^2 + t^2 (b_0 + b_1 t + b_2 t^2) \, x + t^4 (c_0 + c_1 t + c_2 t^2),
$$
with
\begin{align*}
c_0 &=  (r^2-1)^2s^4(2s-r+1)^2(2s+r+1)^2/16,  \qquad \qquad b_0 =  (r^2-1)s^2(2s-r+1)(2s+r+1)/2, \\
a_3 &= (r^2-1)(s+1)^6, \qquad  \qquad a_1 = r^2s^2-5s^2+r^2s-7s+r^2-3, \\
c_2 &= (r^2-1)^2r^2s^6(s+1)^6, \qquad \qquad a_2 =  (s+1)^2(r^2s^2+3s^2-3r^2s+7s-2r^2+3), \\
a_0 &= 1, \qquad  \qquad b_2 = (r^2-1)s^2(s+1)^4(2r^2s^2+6s^2-2r^2s+6s+r^4-2r^2+1)/2, \\
c_1 &=(r^2-1)^2s^4(s+1)^2(2s-r+1)(2s+r+1)(4r^2s^2-12s^2-8s-r^4+2r^2-1)/16, \\
b_1 &=  -(r^2-1)s^2 \big( (2s^2+3s+2)r^4 - 2(4s^4+8s^3+10s^2+8s+3)r^2 +(2s+1)(16s^3+32s^2+21s+4)\big)/4 .
\end{align*}

We identify the class of an $E_7$ fiber:
\begin{center}
\begin{tikzpicture}

\draw (0,0)--(7,0)--(7.5,0.866)--(9.5,0.866)--(9.5,-0.866)--(7.5,-0.866)--(7,0);
\draw (1,0)--(1,1);
\draw (4,0)--(4,1);
\draw (6,0)--(6,1)--(5.5,1.866)--(6.5,1.866)--(6,1);
\draw [bend right] (5.5,3) to (4,1);
\draw (5.5,3)--(5.5,1.866);
\draw [bend left] (5.5,3) to (8.5,0.866);
\draw [very thick] (6.5,1.866)--(6,1)--(6,0)--(7,0)--(7.5,0.866);
\draw [very thick] (7,0)--(7.5,-0.866)--(9.5,-0.866);

\draw (6,0) circle (0.2);
\fill [white] (0,0) circle (0.1);
\fill [white] (1,0) circle (0.1);
\fill [white] (2,0) circle (0.1);
\fill [white] (3,0) circle (0.1);
\fill [white] (4,0) circle (0.1);
\fill [white] (5,0) circle (0.1);
\fill [black] (6,0) circle (0.1);
\fill [white] (1,1) circle (0.1);
\fill [white] (4,1) circle (0.1);
\fill [black] (7,0) circle (0.1);
\fill [black] (6,1) circle (0.1);
\fill [black] (7.5,0.866) circle (0.1);
\fill [white] (8.5,0.866) circle (0.1);
\fill [white] (9.5,0.866) circle (0.1);
\fill [black] (7.5,-0.866) circle (0.1);
\fill [black] (8.5,-0.866) circle (0.1);
\fill [black] (9.5,-0.866) circle (0.1);
\fill [white] (5.5,1.866) circle (0.1);
\fill [black] (6.5,1.866) circle (0.1);

\fill [white] (5.5,3) circle (0.1);

\draw (0,0) circle (0.1);
\draw (1,0) circle (0.1);
\draw (2,0) circle (0.1);
\draw (3,0) circle (0.1);
\draw (4,0) circle (0.1);
\draw (5,0) circle (0.1);
\draw (6,0) circle (0.1);
\draw (1,1) circle (0.1);
\draw (4,1) circle (0.1);
\draw (7,0) circle (0.1);
\draw (6,1) circle (0.1);
\draw (7.5,0.866) circle (0.1);
\draw (8.5,0.866) circle (0.1);
\draw (9.5,0.866) circle (0.1);
\draw (7.5,-0.866) circle (0.1);
\draw (8.5,-0.866) circle (0.1);
\draw (9.5,-0.866) circle (0.1);
\draw (5.5,1.866) circle (0.1);
\draw (6.5,1.866) circle (0.1);

\draw (5.5,3) circle (0.1);

\end{tikzpicture}
\end{center}

The resulting $3$-neighbor step gives us an elliptic fibration with
$D_8$ and $E_7$ fibers, and also a section $P$ of height
$19/2 = 4 + 2 \cdot 4 - 3/2 - 1$.
Next we take a $2$-neighbor step to go from $D_8$ to $E_8$,
keeping the $E_7$ fiber intact.

\begin{center}
\begin{tikzpicture}

\draw (0,0)--(14,0);
\draw (3,0)--(3,1);
\draw (13,0)--(13,1);
\draw (9,0)--(9,1);
\draw [bend right] (7,2) to (0,0);
\draw [bend left] (7,2) to (9,1);
\draw [very thick] (7,0)--(14,0);
\draw [very thick] (9,0)--(9,1);

\draw (6.97,0)--(6.97,2);
\draw (6.91,0)--(6.91,2);
\draw (7.03,0)--(7.03,2);
\draw (7.09,0)--(7.09,2);

\fill [white] (0,0) circle (0.1);
\fill [white] (1,0) circle (0.1);
\fill [white] (2,0) circle (0.1);
\fill [white] (3,0) circle (0.1);
\fill [white] (4,0) circle (0.1);
\fill [white] (5,0) circle (0.1);
\fill [white] (6,0) circle (0.1);
\fill [black] (7,0) circle (0.1);
\fill [black] (8,0) circle (0.1);
\fill [black] (9,0) circle (0.1);
\fill [black] (10,0) circle (0.1);
\fill [black] (11,0) circle (0.1);
\fill [black] (12,0) circle (0.1);
\fill [black] (13,0) circle (0.1);
\fill [black] (14,0) circle (0.1);
\fill [white] (3,1) circle (0.1);
\fill [white] (13,1) circle (0.1);
\fill [black] (9,1) circle (0.1);

\fill [white] (7,2) circle (0.1);

\draw (7,0) circle (0.2);
\draw (0,0) circle (0.1);
\draw (1,0) circle (0.1);
\draw (2,0) circle (0.1);
\draw (3,0) circle (0.1);
\draw (4,0) circle (0.1);
\draw (5,0) circle (0.1);
\draw (6,0) circle (0.1);
\draw (7,0) circle (0.1);
\draw (8,0) circle (0.1);
\draw (9,0) circle (0.1);
\draw (10,0) circle (0.1);
\draw (11,0) circle (0.1);
\draw (12,0) circle (0.1);
\draw (13,0) circle (0.1);
\draw (14,0) circle (0.1);
\draw (3,1) circle (0.1);
\draw (13,1) circle (0.1);
\draw (9,1) circle (0.1);

\draw (7,2) circle (0.1);
\draw (7,2) [above] node{$P$};

\end{tikzpicture}
\end{center}

The intersection number of $F'$ with the remaining component of the $D_8$ fiber
is $2$, whereas $P \cdot F' = 11$. Therefore the new
fibration has a section.

Now we may read out the Igusa-Clebsch invariants from the Weierstrass
equation of this $E_8 E_7$ fibration, and thence compute the equation
of $Y_{-}(76)$ as a double cover of the $(r,s)$-plane, following our
general method of Section \ref{method}.

\begin{theorem}
A birational model over $\Q$ for the Hilbert modular surface $Y_{-}(76
)$ as a double cover of $\Proj^2$ is given by the following equation:
$$
z^2 = -(rs-3s-2)(rs+3s+2)(32s^4+80s^3-13r^2s^2+85s^2-4r^2s+32s+4r^4).
$$ 
It is a surface of general type.
\end{theorem}

\subsection{Analysis}

The branch locus has three components; the more complicated one is
where the elliptic K3 surface has an extra $\I_2$ fiber, while the two
simpler components correspond to an extra $\I_2$ fiber as well as the
section becoming divisible by $2$, giving a section of height $19/84 =
4 - 1/2 - 2/3 - 6/7 - 7/4$. The simpler components are easily seen to
be curves of genus $0$. The last component is a curve of genus~$1$;
the transformation
$$
(r,s) = \left( \frac{2y+x+1}{x^2+x+2}, \frac{-2}{x^2 + x + 2} \right)
$$
converts it to Weierstrass form
$$
y^2 + xy + y = x^3 + x^2 + 1,
$$
which is an elliptic curve of conductor $38$. It is isomorphic to
$X_0(76)/\langle w_4, w_{19} \rangle$.

The Hilbert modular surface $Y_{-}(76)$ is a surface of general type.
The extra involution is $\iota: (r,s) \mapsto (-r,s)$.  Next,
we analyze the quotient of the surface by~$\iota$. This turns
out to be an elliptic K3 surface, and after some Weierstrass
transformations and linear shift of parameter on the base,
its Weierstrass equation may be written as
$$
y^2 =  x^3 + (13t^4-48t^3-6t^2+8t+1)\, x^2 + 64t^4(2t-1)(t^3-5t^2+7t+1) \,x.
$$
It has fibers of type $\I_8$ at $t = 0$, $\I_5$ at $t=1$, $\I_3$ at $t = -1/7$,
and $\I_2$ at $t = 1/2$ and at the roots of $t^3 - 5t^2 + 7t + 1$
(which generates the cubic field of discriminant $-76$).
The trivial lattice has rank $19$. In addition to the
obvious $2$-torsion section $P_0 = (0,0)$, we find a section
$P_1 = (16t^3(2t-1), 16t^3(t-1)(2t-1)(7t+1))$ of height $19/120$.
Therefore the K3 surface is singular.
These sections and the trivial lattice generate a sublattice
of the N\'{e}ron-Severi lattice of discriminant $-76$. It must be the
entire N\'{e}ron-Severi lattice, since otherwise, we would have either
another $2$-torsion section, a $4$-torsion section, or a section of height
$19/480$, none of which is possible with this configuration of
reducible fibers.

The quotient of $Y_{-}(76)$ by the involution $(r,s,z) \mapsto
(-r,s,-z)$ is an honestly elliptic surface with $\chi = 3$. Its
Weierstrass equation may be written as follows.
\begin{align*}
y^2 &= x^3 + (t-1)(64t^5-160t^4+53t^3+33t^2-5t-1) \, x^2 \\
   & \quad + 16(t-1)t^4(2t-1)(t^3-5t^2+7t+1)(32t^3-16t^2+21t-5) \, x.
\end{align*}
It has bad fibers of type $\I_8$ at $t = 0$,
$I_4$ at $t=\infty$ and $t=1/3$,
$\III$ at $t = 1$,
$\I_3$ at $t = -1/7$, and
$\I_2$ at $t = 1/2$, at the roots of $t^3-5t^2+7t+1$ seen above, and
at the roots of $32t^3-16t^2+21t-5$ (which generates the cubic field of discriminant $-152$).
 Hence the trivial lattice has rank $25$, leaving room for \MoW\ rank
at most~$5$. Counting points on the reduction modulo $11$ and $23$
shows that the Picard number is at most $29$. On the other hand, we
find three independent sections in addition to the $2$-torsion section
$P_0 = (0,0)$:
\begin{align*}
P_1 &= \big( 152(t-1)t^4(2t-1), 8 \mu (t-1)t^4(2t-1)(3t-1)(4t-3)(7t+1)\big) \\
P_2 &= \big( 4(t^3-5t^2+7t+1)t^3, 4t^3(2t-1)(3t-1)^2(t^3-5t^2+7t+1)\big) \\ 
P_3 &= \big(  -(t-1)^3(32t^3-16t^2+21t-5), 2 \nu (t-1)^2(3t-1)^2(32t^3-16t^2+21t-5)\big) 
\end{align*}
Here $\mu = \sqrt{19}$ and $\nu = \sqrt{-1}$. These sections have heights
$23/12$, $9/8$ and $13/8$ respectively, and are orthogonal with
respect to the height pairing. Therefore, the Mordell-Weil rank is either $3$ or $4$; we have not been able to determine it exactly.

\subsection{Examples}
We list some points of small height and corresponding genus $2$ curves.

\begin{tabular}{l|c}
Rational point $(r,s)$ & Sextic polynomial $f_6(x)$ defining the genus $2$ curve $y^2 = f_6(x)$. \\
\hline \hline \\ [-2.5ex]
$(-4/11, -8/11)$ & $ -8x^6 + 48x^5 - 196x^4 + 324x^3 - 340x^2 - 330x - 65$ \\
$(2, -4)$ & $ 375x^6 + 300x^5 + 230x^4 - 224x^3 - 76x^2 - 48x + 72$ \\
$(7/3, -1/3)$ & $ -80x^6 - 120x^5 - 109x^4 - 348x^3 - 469x^2 - 120x + 80$ \\
$(-2, -4)$ & $ -225x^6 + 600x^5 + 650x^4 + 400x^3 + 20x^2 - 8$ \\
$(13/23, -19/23)$ & $ -228x^6 + 684x^5 - 1029x^4 - 432x^3 + 525x^2 + 150x + 10$ \\
$(-7/3, -1/3)$ & $ 100x^6 - 220x^5 + 621x^4 - 528x^3 + 1699x^2 - 234x + 1478$ \\
$(19/33, -9/11)$ & $ 256x^6 - 1056x^5 - 2335x^4 + 1480x^3 + 1715x^2 - 1386x + 126$ \\
$(19/11, 1/11)$ & $ -2232x^6 - 2016x^5 + 2581x^4 + 2802x^3 - 983x^2 - 660x - 180$ \\
$(47/37, -43/37)$ & $ 3100x^6 - 540x^5 - 271x^4 - 1742x^3 + 161x^2 + 84x + 252$ \\
$(43/27, 29/27)$ & $ -592x^6 + 372x^5 + 1003x^4 + 1328x^3 - 1406x^2 - 132x - 3709$ \\
$(-19/33, -9/11)$ & $ -4404x^6 - 540x^5 - 1697x^4 - 980x^3 - 257x^2 - 240x - 64$ \\
$(22/17, -20/17)$ & $ 600x^6 - 3360x^5 + 3604x^4 + 2256x^3 + 4546x^2 + 1440x + 775$ \\
$(-22/13, -20/13)$ & $ -367x^6 - 618x^5 - 1539x^4 + 316x^3 + 1839x^2 + 4662x + 3507$ \\
$(-1/23, -31/46)$ & $ -2245x^6 - 137x^5 - 5393x^4 - 1675x^3 - 3618x^2 - 1728x - 675$ \\
$(22/13, -20/13)$ & $ 1425x^6 - 2610x^5 + 6333x^4 - 4948x^3 + 8271x^2 - 4242x + 5971$ \\
$(-1/23, -17/23)$ & $ 24x^6 + 552x^5 + 2075x^4 - 1970x^3 - 9925x^2 + 9072x + 1216$
\end{tabular}

We now describe some curves on $Y_{-}(76)$, which are useful in
producing rational points.

The specialization $s = -2/3$ gives a genus-$1$ curve
$y^2 = -81r^4+63r^2+19$.  It has rational points, such as $(r,y) = (1,1)$.
It is thus an elliptic curve; we find that it has conductor $760$
and \MoW\ group $(\Z/2\Z) \oplus \Z$.
The specialization $s = -8/7$ gives a rational curve, which we
can parametrize as $r = -5(m^2 - 1)/\big( 4(m^2+1) \big)$.

The sections $P_1, P_1 + P_0, 2P_1 + P_0$ and $3P_1 + P_0$ of the K3 quotient
give the following genus-$1$ curves, which all have rational points.
$$
\begin{array}{lcc}
\phantom{g^2 = }\textrm{Equation} & \text{conductor} & \textrm{\MoW\ group} \\
r^2 = -(8s^3 + 28s^2 + 27s +8)/s & 2 \cdot 29 & \Z \\
r^2 = -(s^3-s^2-11s-8)/s & 2^4 \, 11 \cdot 191 & \Z^2 \\
r^2 = -(8s^4+4s^3-33s^2-44s-16)/(s+2)^2 & 3^5 \, 19 & \Z^2 \\
r^2 = -(s^4+s^3-3s^2-s+1) & 2^2 \, 5 \cdot 29 & (\Z/2\Z) \oplus \Z
\end{array}
$$
The section $2P_1$ gives a genus-$0$ curve $r^2 + s^2 + 6s + 4 = 0$,
which we can parametrize as
$$
(r,s) = \left( -\frac{m^2+4m-1}{m^2+1},  -\frac{m^2+2m+5}{m^2+1} \right).
$$
The Brauer obstruction always vanishes on this locus, giving us a
$1$-parameter family of genus $2$ curves whose Jacobians have real
multiplication by $\sO_{76}$.

\section{Discriminant $77$}

\subsection{Parametrization}

We start with a family of K3 surfaces with fibers of type
$A_1, A_3, A_5$ and $D_5$, a $2$-torsion section $T$, and two orthogonal
sections $P, Q$\/ of height
$11/12 = 4 - 0 - (2 \cdot 2)/4 - (1 \cdot 5)/6 - (1 + 1/4)$ and
$7/4 = 4 - 0 - (1\cdot 3)/4 - (3\cdot 3)/6 - 0$.
The orthogonality comes from
$0 = 2 - 0 - (1\cdot 2)/4 - (1\cdot 3)/6 - 0 - 1$,
where the last term comes from the intersection number
$(P) \cdot (Q)$ on the surface.
We can write the Weierstrass equation of this family as
\begin{align*}
y^2 =& \, x^3 + \big((r s - 1)^2 -(s^2-1) (r s-4 r^2-1) \, t + r
(s^2-1) (r s^2+8 s-57 r) \, t^2/4 + 8 r^2 (s^2-1) t^3\big) \, x^2 \\
& +r^2 (s^2-1)^2 t^3 (t-1)^2 \big(16 r^2 t+ (rs -1)^2 - (s-5r)^2 \big)
\, x .
\end{align*}

We go to $E_8 E_7$ form in three steps, via $D_8 E_6$ and $E_8 E_6$.

First, we identify an $D_8$ fiber $F'$ in the figure below (we omit
drawing the node representing $Q$ and the edges connecting it to the
rest of the diagram, as it would clutter up the picture).

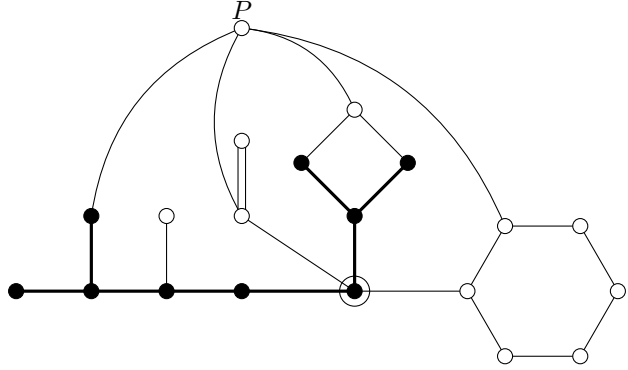
\begin{center}
\begin{tikzpicture}

\draw (0,0)--(6,0)--(6.5,0.866)--(7.5,0.866)--(8,0)--(7.5,-0.866)--(6.5,-0.866)--(6,0);
\draw (1,0)--(1,1);
\draw (2,0)--(2,1);
\draw (4.5,0)--(3,1);
\draw (4.5,0)--(4.5,1)--(5.207,1.707)--(4.5,2.414)--(3.793,1.707)--(4.5,1);
\draw (2.95,1)--(2.95,2);
\draw (3.05,1)--(3.05,2);
\draw [very thick] (0,0)--(4.5,0)--(4.5,1)--(5.207,1.707);
\draw [very thick] (4.5,1)--(3.793,1.707);
\draw [very thick] (1,0)--(1,1);

\draw [bend right] (3,3.5) to (1,1);
\draw [bend right] (3,3.5) to (3,1);
\draw [bend left] (3,3.5) to (4.5,2.414);
\draw [bend left] (3,3.5) to (6.5,0.866);

\draw (4.5,0) circle (0.2);
\fill [black] (0,0) circle (0.1);
\fill [black] (1,0) circle (0.1);
\fill [black] (2,0) circle (0.1);
\fill [black] (3,0) circle (0.1);
\fill [black] (4.5,0) circle (0.1);
\fill [white] (6,0) circle (0.1);
\fill [white] (6.5,0.866) circle (0.1);
\fill [white] (7.5,0.866) circle (0.1);
\fill [white] (6.5,-0.866) circle (0.1);
\fill [white] (7.5,-0.866) circle (0.1);
\fill [white] (8,0) circle (0.1);
\fill [black] (1,1) circle (0.1);
\fill [white] (2,1) circle (0.1);
\fill [white] (3,1) circle (0.1);
\fill [white] (3,2) circle (0.1);
\fill [black] (4.5,1) circle (0.1);
\fill [black] (5.207,1.707) circle (0.1);
\fill [black] (3.793,1.707) circle (0.1);
\fill [white] (4.5,2.414) circle (0.1);

\fill [white] (3,3.5) circle (0.1);

\draw (0,0) circle (0.1);
\draw (1,0) circle (0.1);
\draw (2,0) circle (0.1);
\draw (3,0) circle (0.1);
\draw (4.5,0) circle (0.1);
\draw (6,0) circle (0.1);
\draw (6.5,0.866) circle (0.1);
\draw (7.5,0.866) circle (0.1);
\draw (6.5,-0.866) circle (0.1);
\draw (7.5,-0.866) circle (0.1);
\draw (8,0) circle (0.1);
\draw (1,1) circle (0.1);
\draw (2,1) circle (0.1);
\draw (3,1) circle (0.1);
\draw (3,2) circle (0.1);
\draw (4.5,1) circle (0.1);
\draw (5.207,1.707) circle (0.1);
\draw (3.793,1.707) circle (0.1);
\draw (4.5,2.414) circle (0.1);

\draw (3,3.5) circle (0.1);
\draw (3,3.5) [above] node{$P$};

\end{tikzpicture}
\end{center}

The section $P$ intersects $F'$ once, and so the new fibration has a
section. It has $D_8$ and $E_6$ fibers and rank $2$.

Next, we identify the class of an $E_8$ fiber below, and go to an
elliptic fibration with $E_8$ and $E_6$ fibers, by a $2$-neighbor step.

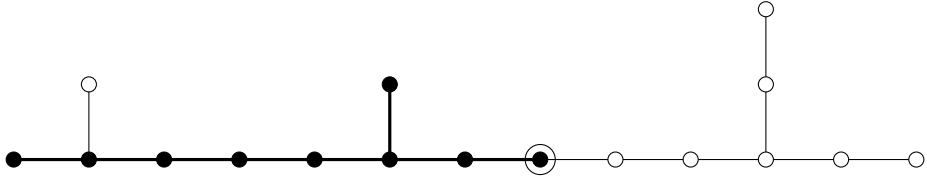
\begin{center}
\begin{tikzpicture}

\draw (0,0)--(12,0);
\draw (1,0)--(1,1);
\draw (5,0)--(5,1);
\draw (10,0)--(10,2);
\draw [very thick] (0,0)--(7,0);
\draw [very thick] (5,0)--(5,1);

\fill [black] (0,0) circle (0.1);
\fill [black] (1,0) circle (0.1);
\fill [black] (2,0) circle (0.1);
\fill [black] (3,0) circle (0.1);
\fill [black] (4,0) circle (0.1);
\fill [black] (5,0) circle (0.1);
\fill [black] (6,0) circle (0.1);
\fill [black] (7,0) circle (0.1);
\fill [white] (8,0) circle (0.1);
\fill [white] (9,0) circle (0.1);
\fill [white] (10,0) circle (0.1);
\fill [white] (11,0) circle (0.1);
\fill [white] (12,0) circle (0.1);
\fill [white] (1,1) circle (0.1);
\fill [black] (5,1) circle (0.1);
\fill [white] (10,1) circle (0.1);
\fill [white] (10,2) circle (0.1);

\draw (7,0) circle (0.2);
\draw (0,0) circle (0.1);
\draw (1,0) circle (0.1);
\draw (2,0) circle (0.1);
\draw (3,0) circle (0.1);
\draw (4,0) circle (0.1);
\draw (5,0) circle (0.1);
\draw (6,0) circle (0.1);
\draw (7,0) circle (0.1);
\draw (8,0) circle (0.1);
\draw (9,0) circle (0.1);
\draw (10,0) circle (0.1);
\draw (11,0) circle (0.1);
\draw (12,0) circle (0.1);
\draw (1,1) circle (0.1);
\draw (5,1) circle (0.1);
\draw (10,1) circle (0.1);
\draw (10,2) circle (0.1);

\end{tikzpicture}
\end{center}

The resulting elliptic fibration has \MoW\ lattice of rank $2$
and discriminant $77/3$. We can relatively easily describe a section
$P''$ of height $8/3$, which intersects a non-identity component of
the $E_6$ fiber. 

Now the $E_7$ fiber $F'''$ drawn below defines an elliptic fibration
with a section and with $E_8$ and $E_7$ fibers.

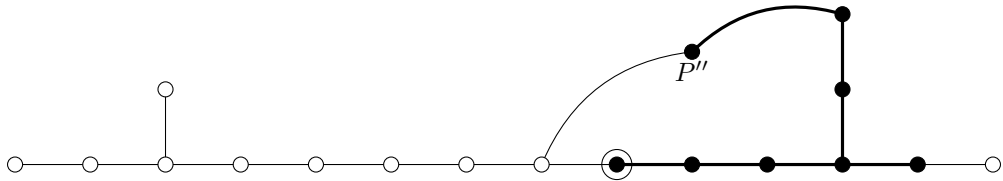
\begin{center}
\begin{tikzpicture}

\draw (0,0)--(13,0);
\draw (2,0)--(2,1);
\draw (11,0)--(11,2);
\draw (9,1.5) to [bend right] (7,0);
\draw [very thick] (9,1.5) to [bend left] (11,2);
\draw [very thick] (8,0)--(12,0);
\draw [very thick] (11,0)--(11,2);

\fill [white] (0,0) circle (0.1);
\fill [white] (1,0) circle (0.1);
\fill [white] (2,0) circle (0.1);
\fill [white] (3,0) circle (0.1);
\fill [white] (4,0) circle (0.1);
\fill [white] (5,0) circle (0.1);
\fill [white] (6,0) circle (0.1);
\fill [white] (7,0) circle (0.1);
\fill [black] (8,0) circle (0.1);
\fill [black] (9,0) circle (0.1);
\fill [black] (10,0) circle (0.1);
\fill [black] (11,0) circle (0.1);
\fill [black] (12,0) circle (0.1);
\fill [white] (13,0) circle (0.1);
\fill [black] (11,1) circle (0.1);
\fill [black] (11,2) circle (0.1);
\fill [white] (2,1) circle (0.1);
\fill [black] (9,1.5) circle (0.1);

\draw (8,0) circle (0.2);
\draw (0,0) circle (0.1);
\draw (1,0) circle (0.1);
\draw (2,0) circle (0.1);
\draw (3,0) circle (0.1);
\draw (4,0) circle (0.1);
\draw (5,0) circle (0.1);
\draw (6,0) circle (0.1);
\draw (7,0) circle (0.1);
\draw (8,0) circle (0.1);
\draw (9,0) circle (0.1);
\draw (10,0) circle (0.1);
\draw (11,0) circle (0.1);
\draw (12,0) circle (0.1);
\draw (13,0) circle (0.1);
\draw (11,1) circle (0.1);
\draw (11,2) circle (0.1);
\draw (2,1) circle (0.1);
\draw (9,1.5) circle (0.1);

\draw (9,1.5) [below] node{$P''$};

\end{tikzpicture}
\end{center}

We now read out the Igusa-Clebsch invariants and proceed as in Section
\ref{method} to compute the equation of $Y_{-}(77)$ as a double cover
of the Humbert surface $\sH_{77}$.

\begin{theorem}
A birational model over $\Q$ for the Hilbert modular surface
$Y_{-}(77)$ as a double cover of $\Proj^2$ is given by the following
equation:
\begin{align*}
z^2 =& \, r^2 (r-1)^2  (r+1)^2 s^6 + 2 r (r-1) (r+1) (21 r^2-13) s^5 \\
& -(14 r^6-433 r^4+328 r^2+27) s^4 -4 r (281 r^4-522 r^2-79) s^3 \\
& -(343 r^6+6268 r^4+1763 r^2-54) s^2 + 2 r (3997 r^4+2446 r^2-171) s \\
& -(1372 r^6+4531 r^4-362 r^2+27).
\end{align*}
It is a surface of general type.
\end{theorem}

\subsection{Analysis}
It is a surface of general type. It has an extra involution $(r,s)
\mapsto (-r,-s)$. The branch locus is a curve of genus $2$. The change
of coordinates
$$
(r,s) = \left(  \frac{4xy-5x^4-22x^2-21}{x(x^2+7)^2} , \frac{-(2x^2y+2y-x^5-6x^3-9x)}{(x^2-1)^2} \right)
$$
converts it to Weierstrass form
$$
y^2 = x^6 + 5x^4 + 3x^2 + 7.
$$
It is isomorphic to the quotient of $X_0(77)$ by the Atkin-Lehner
involution $w_{77}$.

\subsection{Examples}

We list some points of small height and corresponding genus $2$
curves.

\begin{tabular}{l|c}
Rational point $(r,s)$ & Sextic polynomial $f_6(x)$ defining the genus $2$ curve $y^2 = f_6(x)$. \\
\hline \hline \\ [-2.5ex]
$(-1/2, -23/5)$ & $ 1581x^6 - 25965x^5 + 128199x^4 - 124655x^3 - 282240x^2 - 92400x - 9478 $ \\
$(-3/5, -47/9)$ & $ 926711x^6 - 351913x^5 - 2531791x^4 + 699677x^3 + 2176646x^2 - 359536x - 579882 $ \\
$(1/2, 23/5)$ & $ 2038350x^6 + 1601640x^5 - 6288456x^4 + 116705x^3 - 5729115x^2$ \\
& $ - 8845053x - 2931103 $ \\
$(-13/2, 6)$ & $ 10433826x^6 - 16243110x^5 + 25749477x^4 - 25899800x^3 + 5297523x^2$ \\
& $ - 6454140x - 9577876 $ \\
$(13/2, -6)$ & $ -749865564x^6 + 4317895148x^5 + 1178682897x^4 - 6739621816x^3$ \\
& $ - 800208729x^2 + 2973824982x + 128797182 $ \\
$(-33/65, -239/63)$ & $ -5738303278944x^6 + 6551435295576x^5 + 28045528925148x^4 - 5100723398753x^3$ \\
& $ - 23656013198837x^2 + 3333165270637x + 3904336668117 $ 
\end{tabular}

\section{Discriminant $85$}

\subsection{Parametrization}

We start with a K3 elliptic surface with fibers of type $E_6, D_5$ and $A_4$,
with a section of height $85/60 = 17/12 = 4 - 4/3 -5/4$.

The Weierstrass equation is
$$
y^2 = x^3 + t (a_0 + a_1 t) \, x^2
     + 2 t^2 (t-1) (b_0 + b_1 t )\, x + t^3 (t-1)^4 (c_0 + c_1 t),
$$
with
\begin{align*}
a_0 &= -4(e^2-1)(3ef^2+2e^2f+2f-2e^2+e+2), \quad
 c_0 = -64 e^2 (e^2-1)^3(f^2-1)^2(f+2e+1)(ef-e+2), \\
a_1 &=  4(e^2-1)^2 f^2, \qquad \qquad \qquad \quad \qquad \qquad  \qquad
 b_0 = 8 e (e^2-1)^2(f^2-1)(3ef^2+4e^2f+4f-4e^2+5e+4), \\
b_1 &=  -8 f(e^2-1)^2(f^2-1) (e^4f+e^3f-ef+f-e^4+e^3+e+1), \\
c_1 &= 16(e^2-1)^2(f^2-1)^2(e^4f+e^3f-ef+f-e^4+e^3+e+1)^2.
\end{align*}

We identify the class of a $D_8$ fiber below, and move to the new
elliptic fibration via a $2$-neighbor step.

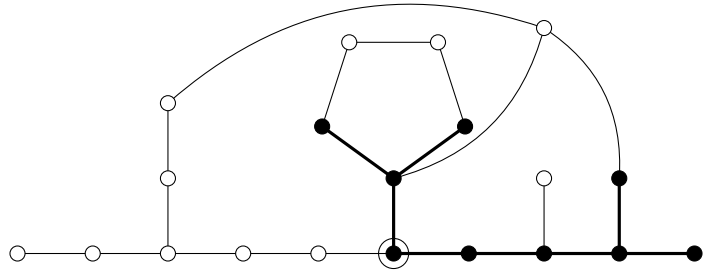
\begin{center}
\begin{tikzpicture}

\draw (0,0)--(9,0);
\draw (2,0)--(2,2);
\draw (7,0)--(7,1);
\draw (8,0)--(8,1);
\draw (5,0)--(5,1);
\draw (5,1)--(4.05,1.69);
\draw (5,1)--(5.95,1.69);
\draw (4.05,1.69)--(4.41,2.81);
\draw (5.95,1.69)--(5.59,2.81);
\draw (4.41,2.81)--(5.59,2.81);
\draw [very thick] (5,1)--(5,0)--(9,0);
\draw [very thick] (8,0)--(8,1);
\draw [very thick] (4.05,1.69)--(5,1)--(5.95,1.69);
\draw [bend right] (7,3) to (2,2);
\draw [bend left] (7,3) to (8,1);
\draw [bend left] (7,3) to (5,1);

\fill [white] (7,3) circle (0.1);
\fill [white] (0,0) circle (0.1);
\fill [white] (1,0) circle (0.1);
\fill [white] (2,0) circle (0.1);
\fill [white] (3,0) circle (0.1);
\fill [white] (4,0) circle (0.1);
\fill [black] (5,0) circle (0.1);
\fill [white] (2,1) circle (0.1);
\fill [white] (2,2) circle (0.1);
\fill [black] (6,0) circle (0.1);
\fill [black] (7,0) circle (0.1);
\fill [black] (8,0) circle (0.1);
\fill [black] (9,0) circle (0.1);
\fill [white] (7,1) circle (0.1);
\fill [black] (8,1) circle (0.1);

\fill [black] (5,1) circle (0.1);
\fill [black] (4.05,1.69) circle (0.1);
\fill [black] (5.95,1.69) circle (0.1);
\fill [white] (4.41,2.81) circle (0.1);
\fill [white] (5.59,2.81) circle (0.1);

\draw (5,0) circle (0.2);
\draw (7,3) circle (0.1);
\draw (0,0) circle (0.1);
\draw (1,0) circle (0.1);
\draw (2,0) circle (0.1);
\draw (3,0) circle (0.1);
\draw (4,0) circle (0.1);
\draw (5,0) circle (0.1);
\draw (2,1) circle (0.1);
\draw (2,2) circle (0.1);
\draw (6,0) circle (0.1);
\draw (7,0) circle (0.1);
\draw (8,0) circle (0.1);
\draw (9,0) circle (0.1);
\draw (7,1) circle (0.1);
\draw (8,1) circle (0.1);

\draw (5,1) circle (0.1);
\draw (4.05,1.69) circle (0.1);
\draw (5.95,1.69) circle (0.1);
\draw (4.41,2.81) circle (0.1);
\draw (5.59,2.81) circle (0.1);

\end{tikzpicture}
\end{center}

The new elliptic fibration has reducible fibers of type $D_8$ and
$E_6$, and \MoW\ rank $2$.  Next, we move to an elliptic fibration
with $E_8$ and $E_6$ fibers by another $2$-neighbor step, using the
$E_8$ fiber shown below.

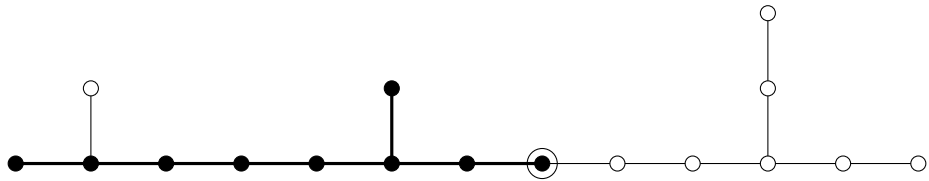
\begin{center}
\begin{tikzpicture}

\draw (0,0)--(12,0);
\draw (1,0)--(1,1);
\draw (5,0)--(5,1);
\draw (10,0)--(10,2);
\draw [very thick] (0,0)--(7,0);
\draw [very thick] (5,0)--(5,1);

\draw (7,0) circle (0.2);
\fill [black] (0,0) circle (0.1);
\fill [black] (1,0) circle (0.1);
\fill [black] (2,0) circle (0.1);
\fill [black] (3,0) circle (0.1);
\fill [black] (4,0) circle (0.1);
\fill [black] (5,0) circle (0.1);
\fill [black] (6,0) circle (0.1);
\fill [black] (5,1) circle (0.1);
\fill [white] (1,1) circle (0.1);
\fill [black] (7,0) circle (0.1);
\fill [white] (8,0) circle (0.1);
\fill [white] (9,0) circle (0.1);
\fill [white] (10,0) circle (0.1);
\fill [white] (11,0) circle (0.1);
\fill [white] (12,0) circle (0.1);
\fill [white] (10,1) circle (0.1);
\fill [white] (10,2) circle (0.1);

\draw (0,0) circle (0.1);
\draw (1,0) circle (0.1);
\draw (2,0) circle (0.1);
\draw (3,0) circle (0.1);
\draw (4,0) circle (0.1);
\draw (5,0) circle (0.1);
\draw (6,0) circle (0.1);
\draw (5,1) circle (0.1);
\draw (1,1) circle (0.1);
\draw (7,0) circle (0.1);
\draw (8,0) circle (0.1);
\draw (9,0) circle (0.1);
\draw (10,0) circle (0.1);
\draw (11,0) circle (0.1);
\draw (12,0) circle (0.1);
\draw (10,1) circle (0.1);
\draw (10,2) circle (0.1);

\end{tikzpicture}
\end{center}

The new elliptic fibration has $E_8$ and $E_6$ fibers. We can find an
explicit section $P$ of the $E_8 E_6$ fibration which intersects a
non-identity component of the $E_6$ fiber and does not intersect the
zero section. 

Let $F$ be an $E_7$ fiber enclosed by the box in the picture below.
We move to it by a $2$-neighbor step, to recover an elliptic fibration
with $E_8$ and $E_7$ fibers.

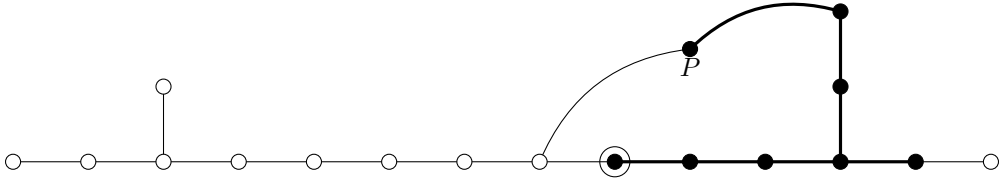
\begin{center}
\begin{tikzpicture}

\draw (0,0)--(13,0);
\draw (2,0)--(2,1);
\draw (11,0)--(11,2);
\draw (9,1.5) to [bend right] (7,0);
\draw [very thick] (9,1.5) to [bend left] (11,2);
\draw [very thick] (8,0)--(12,0);
\draw [very thick] (11,0)--(11,2);

\draw (8,0) circle (0.2);
\fill [white] (0,0) circle (0.1);
\fill [white] (1,0) circle (0.1);
\fill [white] (2,0) circle (0.1);
\fill [white] (3,0) circle (0.1);
\fill [white] (4,0) circle (0.1);
\fill [white] (5,0) circle (0.1);
\fill [white] (6,0) circle (0.1);
\fill [white] (7,0) circle (0.1);
\fill [black] (8,0) circle (0.1);
\fill [black] (9,0) circle (0.1);
\fill [black] (10,0) circle (0.1);
\fill [black] (11,0) circle (0.1);
\fill [black] (12,0) circle (0.1);
\fill [white] (13,0) circle (0.1);
\fill [black] (11,1) circle (0.1);
\fill [black] (11,2) circle (0.1);
\fill [white] (2,1) circle (0.1);
\fill [black] (9,1.5) circle (0.1);

\draw (0,0) circle (0.1);
\draw (1,0) circle (0.1);
\draw (2,0) circle (0.1);
\draw (3,0) circle (0.1);
\draw (4,0) circle (0.1);
\draw (5,0) circle (0.1);
\draw (6,0) circle (0.1);
\draw (7,0) circle (0.1);
\draw (8,0) circle (0.1);
\draw (9,0) circle (0.1);
\draw (10,0) circle (0.1);
\draw (11,0) circle (0.1);
\draw (12,0) circle (0.1);
\draw (13,0) circle (0.1);
\draw (11,1) circle (0.1);
\draw (11,2) circle (0.1);
\draw (2,1) circle (0.1);
\draw (9,1.5) circle (0.1);

\draw (9,1.5) [below] node{$P$};

\end{tikzpicture}
\end{center}

We may now read out the Igusa-Clebsh invariants, and work
out the equation of $Y_{-}(85)$ as a double cover of $\Proj^2_{e,f}$.

\begin{theorem}
A birational model over $\Q$ for the Hilbert modular surface
$Y_{-}(85)$ as a double cover of $\Proj_{e,f}^2$ is given by the following
equation:
\begin{align*}
z^2 = & -(e^2-e-1)^2(8e^4+11e^2+8)\, f^4  + 4\, (e^4-1)(18e^4-11e^3+27e^2+11e+18)f^3 \\
 & -2 \, (82e^8-118e^7-9e^6-60e^5-173e^4+60e^3-9e^2+118e+82) \, f^2 \\
 &  + 4 \, (e^4-1)(4e^2-5e-4)(9e^2-13e-9) \, f -(2e^2-e-2)^2(11e^4-34e^3+5e^2+34e+11).
\end{align*}
It is an honestly elliptic surface, with arithmetic genus $3$ and Picard
number $37$ or $38$.
\end{theorem}

\subsection{Analysis}
It is an honestly elliptic surface. The extra involution is $\iota:
(e,f,z) \mapsto (-1/e,-f,z/e^4)$.

The branch locus is a hyperelliptic curve of genus $3$. The change of
coordinates
$$
(e,f) = \left( \frac{y+x^4+x^3-x+1}{2x^2+2x-1} , \frac{(3x^4+x^3-3x^2-x+3)y -(x^2-1)(x^2+1)^3}{(x^2+x-1)^2(2x^4-x^2+2)} \right)
$$
converts it to Weierstrass form:
$$
y^2 = x^8 + 2 x^7 - x^6 - 8 x^5 + x^4 + 8 x^3 - x^2 - 2 x + 1.
$$ 
It is isomorphic to the quotient of $X_0(85)$ by the Atkin-Lehner
involution $w_{85}$.

The equation of $Y_{-}(85)$ describes it as an elliptic surface over
$\Proj^1_e$. So far, we are unable to find a section.

The Jacobian of this genus-$1$ curve over $\Q(e)$ can be
written (after some Weierstrass transformations) as
\begin{align*}
y^2 &= x^3 + (e^8-10e^7+3e^6+84e^5+85e^4-84e^3+3e^2+10e+1) \, x^2\\
 &\qquad  -8e^5(23e^6-133e^5-420e^4-64e^3+420e^2-133e-23) \, x \\
 &\qquad  -16e^9(108e^6-637e^5-1944e^4-26e^3+1944e^2-637e-108).
\end{align*}
It has $\chi = 4$. There are bad fibers of type $\I_9$ at $e = 0$ and
$e =\infty$, $\I_3$ at $e = \pm 1$, $\I_2$ at the roots of $e = 4 \pm
\sqrt{17}$, and $\I_3$ at the roots of
$3e^6-16e^5-54e^4+5e^3+54e^2-16e-3$ (which generates the compositum of
$\Q(\sqrt{85})$ and the cubic field of discriminant $-3 \cdot 5 \cdot 17$).
The trivial lattice has rank~$36$, leaving room for at most four
independent sections. So far, we can only say from the ensuing
analysis that the rank is either $1$ or $2$.

Next, we analyze the quotient of this surface by the involution~$\iota$.
In terms of $g = e - 1/e$ and $h = f/(e+1/e)$ (which are
invariant under $\iota$), its equation is given by
\begin{align*}
z^2 &= -(g-1)^2(g^2+4)^2(8g^2+27)\, h^4 + 4g(g^2+4)^2(18g^2-11g+63)\, h^3 \\
& \qquad -2(g^2+4)(82g^4-118g^3+319g^2-414g-27) \, h^2 \\
& \qquad  + 4g(4g-5)(9g-13)(g^2+4) \, h -(2g-1)^2(11g^2-34g+27).
\end{align*}
This is also an honestly elliptic surface, this time with $\chi = 3$.
Its Jacobian is given by
\begin{align*}
y^2 &=  x^3 + (g^2+4)(g^4-10g^3+7g^2+66g-27) \, x^2 \\
 & \qquad + 8g(g-8)(g^2+4)^2(g^3-12g^2-12g+27) \, x \\
 & \qquad + 16g^2 (g-8)^2 (g^2+4)^3 (g^2-14g-27).
\end{align*}
It has bad fibers of type $\I_9$ at $g = \infty$, $\I_3$ at $g = 0$,
$\I_2$ at $g = 8$, $\I^*_0$ at $g = \pm 2 \sqrt{-1}$ and $\I_3$ at the
roots of $3g^3 - 16g^2 - 45g - 27$ (which generates the cubic field of
discriminant $-255$).  The trivial lattice has rank $27$, leaving room
for \MoW\ rank at most~$3$. Counting points modulo $11$ and $19$ shows
that the Picard number is at most $29$. On the other hand, we are able
to find the non-torsion section
$$
P_1 = \big( -4(g-8)(g^2+4)(9g^3+3g^2+g+72)/85, 4(g-8)(21g-4)(g^2+4)^2(3g^3-16g^2-45g-27)/85^{3/2} \big)
$$
of height $3/2$. Therefore, the \MoW\ rank is either $1$ or $2$. 

Next, we consider the quadratic twist of the quotient elliptic surface,
which is obtained by simply removing the factors of $(g^2+4)$ in the
Weierstrass equation above (recalling that $g^2 + 4 = (e + 1/e)^2$).
We get a K3 surface with a genus-$1$ fibration, whose Jacobian has
Weierstrass equation
\begin{align*}
y^2 &=  x^3 + (g^4-10g^3+7g^2+66g-27) \, x^2 \\
 & \qquad + 8g(g-8)(g^3-12g^2-12g+27) \, x \\
 & \qquad + 16g^2 (g-8)^2 (g^2-14g-27).
\end{align*}
It has reducible fibers of type $\I_9$ at $g = \infty$,
$\I_2$ at $g = 8$, and $\I_3$ at $g = 0$ and at the roots of
$3g^3 - 16g^2 - 45g - 27$.
Therefore the trivial lattice has rank $19$, and the
\MoW\ rank can be $0$ or $1$.  We find a $3$-torsion section
$$
P_0 = \big( 8g + 36, 4(3g^3-16g^2-45g-27) \big).
$$
Counting points modulo $7$ and $19$ shows that the Picard number is
exactly $19$. The $3$-torsion section and trivial lattice span a
sublattice of discriminant $162 = 2 \cdot 3^4$ of the N\'{e}ron-Severi
lattice of discriminant. It is easy to check that this sublattice is
$3$-saturated, and therefore must form the entire N\'eron-Severi
lattice.

\subsection{Examples}

We list some points of small height and corresponding genus $2$ curves.

\begin{tabular}{l|c}
Rational point $(e,f)$ & Sextic polynomial $f_6(x)$ defining the genus $2$ curve $y^2 = f_6(x)$. \\
\hline \hline \\ [-2.5ex]
$(1/2, -29/15)$ & $ 576x^6 + 432x^5 + 927x^4 + 81x^3 + 171x^2 - 72x - 208 $ \\
$(4/3, 1/7)$ & $ -1344x^6 - 672x^5 - 2233x^4 - 3026x^3 - 997x^2 - 2196x - 548 $ \\
$(7/2, 8/17)$ & $ -1566x^6 - 7704x^5 - 4056x^4 - 8581x^3 - 5841x^2 - 2055x - 2395 $ \\
$(-2, 29/15)$ & $ -3500x^6 - 2100x^5 + 11205x^4 + 2422x^3 - 11295x^2 + 1080x + 2160 $ \\
$(-2, 23/21)$ & $ -316x^6 + 3048x^5 + 14649x^4 + 10547x^3 - 13509x^2 - 1296x + 1728 $ \\
$(-6/7, -13)$ & $ -5028x^6 - 10620x^5 - 2605x^4 - 16750x^3 + 5255x^2 - 6600x + 2832 $ \\
$(-1/2, -7)$ & $ 8964x^6 - 3132x^5 + 18927x^4 + 6286x^3 + 6655x^2 + 11300x - 500 $ \\
$(2/5, -7/3)$ & $ 21006x^6 - 45414x^5 + 16263x^4 - 20048x^3 - 7227x^2 + 960x - 3200 $ \\
$(-3/4, -3/5)$ & $ 5500x^6 + 30300x^5 + 19835x^4 + 20174x^3 - 46885x^2 + 2340x - 380 $ \\
$(5/2, 7/3)$ & $ -12852x^6 - 15876x^5 + 40383x^4 + 49976x^3 - 30231x^2 - 43650x + 2250 $ \\
$(11/20, -71/49)$ & $ -66020x^6 + 43980x^5 + 10001x^4 + 1154x^3 - 5899x^2 - 1464x + 1096 $ \\
$(-2/5, -7/3)$ & $ -72620x^6 + 37884x^5 - 12135x^4 + 29302x^3 - 4107x^2 + 1848x - 2672 $ \\
$(-5/2, 7/3)$ & $ 20x^6 + 180x^5 - 3879x^4 - 34668x^3 + 44937x^2 + 62856x - 73296 $ \\
$(7/6, 13)$ & $ -5442x^6 + 3630x^5 - 7079x^4 - 93460x^3 + 35059x^2 - 420x + 9212 $ \\
$(-4, 11/9)$ & $ -16964x^6 - 33804x^5 + 53325x^4 + 100170x^3 - 35163x^2 - 81540x - 1116 $ \\
$(11/2, 9/13)$ & $ -102046x^6 + 130482x^5 + 61857x^4 + 9504x^3 - 74697x^2 - 38412x - 14036 $ 
\end{tabular}

\section{Discriminant $88$}

\subsection{Parametrization}

We start with an elliptic K3 surface with fibers of type $A_9$, $D_4$
and $A_2$, and a section of height $11/15 = 4 - 2/3 - 1- 16/10$. The
Weierstrass equation for this family is

$$
y^2 = x^3 + (a_0 + a_1 t + a_2 t^2 + a_3 t^3) \, x^2
 + 2 t^2 (\lambda t - \mu) (b_0 + b_1t + b_2t^2 + b_3t^3) \, x
 +  t^4 (\lambda t - \mu)^2 (c_0 + c_1 t)^2,
$$

with
\begin{flalign*}
\qquad a_0 &= 1,  &  \mu &= rs + 2s + 1, &  \lambda &= s (r+2)^2 (2s+1)^2, & \\
\qquad c_0 &= -2, &  b_0 &= 2,           &      a_1 &= -8s(rs+2r+1), & 
\end{flalign*}
\vspace*{-4ex}
\begin{flalign*}
\qquad c_1 &= 8(r+2)s^2+8(r+1)s+r^2, & \\
\qquad b_3 &= 2r(r+2)s(2s+1)(8s+r^2)(8s^2+r), & \\
\qquad b_2 &= 32r(r+2)s^4 + 32(3r^2+4r+2)s^3 -4(r^3-20r^2-24r-8)s^2 + 4(r-1)r^2s, & \\
\qquad b_1 &= -16(r+1)s^2-8(3r+2)s-r^2, &\\
\qquad a_3 &= 4rs \big(
    64rs^4+16(r^3+3r^2+12r+4)s^3 + 2r(r^2+4)s^2  + r(r^3+12r^2+12r+16)s+r^3
   \big), &\\
\qquad a_2 &= 4s \big( 4r^2s^3+8r(2r-1)s^2 -4(r^3-4r^2-4r-1)s-r^2(r+4) \big). &
\end{flalign*}

First we identify an $E_7$ fiber, and make a $3$-neighbor move to it.
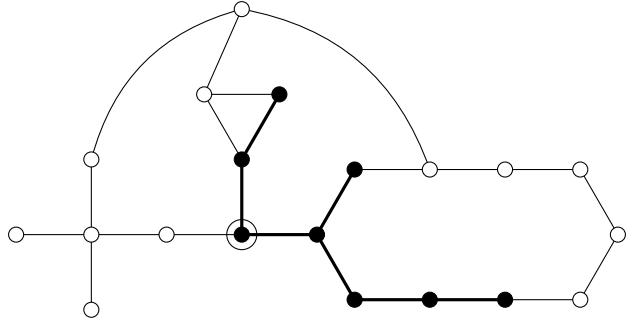
\begin{center}
\begin{tikzpicture}

\draw [bend right] (3,3) to (1,1);
\draw (3,3)--(2.5,1.866);
\draw [bend left] (3,3) to (5.5,0.866);

\draw (0,0)--(4,0)--(4.5,0.866)--(7.5,0.866)--(8,0)--(7.5,-0.866)--(4.5,-0.866)--(4,0);
\draw (1,-1)--(1,1);
\draw (3,0)--(3,1);
\draw (3,1)--(2.5,1.866)--(3.5,1.866)--(3,1);
\draw [very thick] (3.5,1.866)--(3,1)--(3,0)--(4,0)--(4.5,0.866);
\draw [very thick] (4,0)--(4.5,-0.866)--(6.5,-0.866);

\draw (3,0) circle (0.2);
\fill [white] (0,0) circle (0.1);
\fill [white] (1,0) circle (0.1);
\fill [white] (2,0) circle (0.1);
\fill [black] (3,0) circle (0.1);
\fill [black] (4,0) circle (0.1);
\fill [white] (1,1) circle (0.1);
\fill [white] (1,-1) circle (0.1);
\fill [black] (3,1) circle (0.1);
\fill [white] (2.5,1.866) circle (0.1);
\fill [black] (3.5,1.866) circle (0.1);
\fill [black] (4.5,0.866) circle (0.1);
\fill [white] (5.5,0.866) circle (0.1);
\fill [white] (6.5,0.866) circle (0.1);
\fill [white] (7.5,0.866) circle (0.1);
\fill [black] (4.5,-0.866) circle (0.1);
\fill [black] (5.5,-0.866) circle (0.1);
\fill [black] (6.5,-0.866) circle (0.1);
\fill [white] (7.5,-0.866) circle (0.1);
\fill [white] (8,0) circle (0.1);
\fill [white] (3,3) circle (0.1);

\draw (0,0) circle (0.1);
\draw (1,0) circle (0.1);
\draw (2,0) circle (0.1);
\draw (3,0) circle (0.1);
\draw (4,0) circle (0.1);
\draw (1,1) circle (0.1);
\draw (1,-1) circle (0.1);
\draw (3,1) circle (0.1);
\draw (2.5,1.866) circle (0.1);
\draw (3.5,1.866) circle (0.1);
\draw (4.5,0.866) circle (0.1);
\draw (5.5,0.866) circle (0.1);
\draw (6.5,0.866) circle (0.1);
\draw (7.5,0.866) circle (0.1);
\draw (4.5,-0.866) circle (0.1);
\draw (5.5,-0.866) circle (0.1);
\draw (6.5,-0.866) circle (0.1);
\draw (7.5,-0.866) circle (0.1);
\draw (8,0) circle (0.1);
\draw (3,3) circle (0.1);

\end{tikzpicture}
\end{center}

This gives us an elliptic fibration with $E_7$, $D_5$ and $A_3$
fibers, and a section of height $88/32 = 11/4 = 4 - 5/4$. Then we can
identify a $D_8$ fiber $F'$ below, and move to the associated genus
$1$ fibration by a $2$-neighbor step.

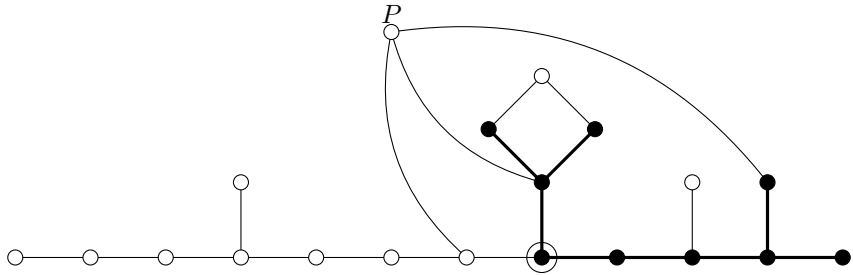
\begin{center}
\begin{tikzpicture}

\draw (0,0)--(11,0);
\draw (3,0)--(3,1);
\draw (9,0)--(9,1);
\draw (10,0)--(10,1);
\draw (7,0)--(7,1)--(7.707,1.707)--(7,2.414)--(6.293,1.707)--(7,1);
\draw [bend right] (5,3) to (6,0);
\draw [bend right] (5,3) to (7,1);
\draw [bend left] (5,3) to (10,1);
\draw [very thick] (7,1)--(7,0)--(11,0);
\draw [very thick] (10,0)--(10,1);
\draw [very thick] (7.707,1.707)--(7,1)--(6.293,1.707);

\draw (7,0) circle (0.2);
\fill [white] (0,0) circle (0.1);
\fill [white] (1,0) circle (0.1);
\fill [white] (2,0) circle (0.1);
\fill [white] (3,0) circle (0.1);
\fill [white] (3,1) circle (0.1);
\fill [white] (4,0) circle (0.1);
\fill [white] (5,0) circle (0.1);
\fill [white] (6,0) circle (0.1);
\fill [black] (7,0) circle (0.1);
\fill [black] (8,0) circle (0.1);
\fill [black] (9,0) circle (0.1);
\fill [white] (9,1) circle (0.1);
\fill [black] (10,0) circle (0.1);
\fill [black] (10,1) circle (0.1);
\fill [black] (11,0) circle (0.1);
\fill [black] (7,1) circle (0.1);
\fill [black] (7.707,1.707) circle (0.1);
\fill [white] (7,2.414) circle (0.1);
\fill [black] (6.293,1.707) circle (0.1);
\fill [white] (5,3) circle (0.1);

\draw (0,0) circle (0.1);
\draw (1,0) circle (0.1);
\draw (2,0) circle (0.1);
\draw (3,0) circle (0.1);
\draw (3,1) circle (0.1);
\draw (4,0) circle (0.1);
\draw (5,0) circle (0.1);
\draw (6,0) circle (0.1);
\draw (7,0) circle (0.1);
\draw (8,0) circle (0.1);
\draw (9,0) circle (0.1);
\draw (9,1) circle (0.1);
\draw (10,0) circle (0.1);
\draw (10,1) circle (0.1);
\draw (11,0) circle (0.1);
\draw (7,1) circle (0.1);
\draw (7.707,1.707) circle (0.1);
\draw (7,2.414) circle (0.1);
\draw (6.293,1.707) circle (0.1);
\draw (5,3) circle (0.1);
\draw (5,3) [above] node{$P$};

\end{tikzpicture}
\end{center}

To see that the genus-$1$ fibration defined by this fiber $F'$ has a
section, note that $P \cdot F' = 3$, while $F'$ intersects the near
leaf of the $D_5$ fiber with multiplicity $2$. Therefore we may
replace the genus $1$ fibration by its Jacobian.

Finally, we go by another $2$-neighbor move to a fibration with $E_8$
and $E_7$ fibers. We identify the class of an $E_8$ fiber $F''$ below.
The \MoW\ group is generated by a section $P'$ of height
$88/(2 \cdot 4) = 11 = 4 + 2 \cdot 4 - 1$, so the section must intersect
the zero section with multiplicity $4$, and it must intersect the near leaf
of the $D_8$ fiber. Therefore $P' \cdot F'' = 2 \cdot 4 + 3 = 11$,
whereas the omitted far leaf of the $D_8$ fiber intersects $F''$
with multiplicity $2$. So the new fibration has a section.

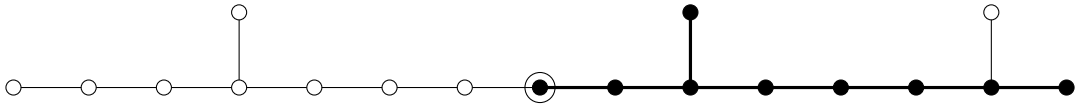
\begin{center}
\begin{tikzpicture}

\draw (0,0)--(14,0);
\draw (3,0)--(3,1);
\draw (9,0)--(9,1);
\draw (13,0)--(13,1);
\draw [very thick] (7,0)--(14,0);
\draw [very thick] (9,0)--(9,1);

\fill [white] (0,0) circle (0.1);
\fill [white] (1,0) circle (0.1);
\fill [white] (2,0) circle (0.1);
\fill [white] (3,0) circle (0.1);
\fill [white] (3,1) circle (0.1);
\fill [white] (4,0) circle (0.1);
\fill [white] (5,0) circle (0.1);
\fill [white] (6,0) circle (0.1);
\fill [black] (7,0) circle (0.1);
\fill [black] (8,0) circle (0.1);
\fill [black] (9,0) circle (0.1);
\fill [black] (10,0) circle (0.1);
\fill [black] (11,0) circle (0.1);
\fill [black] (12,0) circle (0.1);
\fill [black] (13,0) circle (0.1);
\fill [black] (9,1) circle (0.1);
\fill [white] (13,1) circle (0.1);
\fill [black] (14,0) circle (0.1);

\draw (7,0) circle (0.2);
\draw (0,0) circle (0.1);
\draw (1,0) circle (0.1);
\draw (2,0) circle (0.1);
\draw (3,0) circle (0.1);
\draw (3,1) circle (0.1);
\draw (4,0) circle (0.1);
\draw (5,0) circle (0.1);
\draw (6,0) circle (0.1);
\draw (7,0) circle (0.1);
\draw (8,0) circle (0.1);
\draw (9,0) circle (0.1);
\draw (10,0) circle (0.1);
\draw (11,0) circle (0.1);
\draw (12,0) circle (0.1);
\draw (13,0) circle (0.1);
\draw (14,0) circle (0.1);
\draw (9,1) circle (0.1);
\draw (13,1) circle (0.1);

\end{tikzpicture}
\end{center}

We may now read out the Igusa-Clebsch invariants and compute the
equation of the branch locus for $Y_{-}(88) \mapsto \Proj^2_{r,s}$.

\begin{theorem}
A birational model over $\Q$ for the Hilbert modular surface
$Y_{-}(88)$ as a double cover of $\Proj^2_{r,s}$ is given by the following
equation:
\begin{align*}
z^2 &= (8rs^2+16s^2+8s+r^2)(8r^3s^4+16r^2s^4+96r^3s^3+472r^2s^3 +544rs^3-27r^4s^2 \\
    & \qquad  -120r^3s^2 +64r^2s^2+472rs^2+16s^2-46r^3s-120r^2s+96rs+8s-27r^2).
\end{align*}
It is a surface of general type.
\end{theorem}

\subsection{Analysis}

The branch locus has two components. Both correspond to elliptic K3
surfaces with an extra $\I_2$ fiber, and the simpler component to
having a $2$-torsion section in addition. The simpler component of the
branch locus has genus $1$; the change of coordinates $r = 2y/x^2, s =
-1/(2x)$ converts it to Weierstrass form
$$
y^2 + y = x^3 - x^2
$$
which is an elliptic curve of conductor $11$ (isomorphic to $X_1(11)$).

The other component has genus $2$. The transformation
$$
(r,s) = \left( \frac{-(x-1)y + (x+1)(x^3-3x^2-3x-1)}{3x^2+2x+1} , \frac{(3x+1)y -(x+1)(3x^3+3x^2+3x+1)}{4x(x^2-2x-1)} \right)
$$
converts it to Weierstrass form
$$
y^2 = x^6-2x^5+11x^4+20x^3+15x^2+6x+1.
$$

The Hilbert modular surface $Y_{-}(88)$ is a surface of general
type. We now analyze its quotient by the involution $\iota: (r,s,z)
\mapsto (1/s,1/r, z/(rs)^3)$. Writing $h = -(s+1/r), m = (s-1/r)^2$,
we find the equation
\begin{align*}
z^2 &= (h^4-2h^3-2mh^2+2mh+m^2+1)(9h^6-30h^5-26mh^4+16h^4+58mh^3 \\
   & \qquad +30h^3 +25m^2h^2-8mh^2-25h^2-28m^2h -30mh-8m^3-8m^2-2m).
\end{align*}
The invertible transformation
$m = 2 + 1/t + 8/(nt) + 4/(nt)^2, \, h = -1 - 2/(nt)$
makes this a quartic in $n$,
$$
z^2 = (4t-n-4) \big(
 t(2t+1)(6t^2-13t+8) n^3 + 4t(2t-3)^2n^2 -4(2t-1)^2n + 16(t-1)
 \big).
$$
with an obvious section $n = 4t - 4$.  Converting to the Jacobian,
we get an elliptic K3 surface with the following equation (after some
Weierstrass transformations and a change of parameter $t \mapsto 1-t$
on the base):
$$
y^2 = x^3 + (28t^4-24t^3-8t^2+4t+1) \, x^2
  - 16 t^3(t-1)^2(t^3-10t^2+4t+1)\,x
  + 64 t^6 (t-1)^4(29t^2-10t-3).
$$
This has bad fibers of type $\I_6$ at $t = 0$, $\I_5$ at $t = 1$,
$\I_2$ at $t= 1/2$ and $t = -1/6$, and $\I_3$ at $t = 1/4 \pm
\sqrt{33}/12$. The trivial lattice has rank $17$.
We find the independent sections
\begin{align*}
P_1 &= \big(-4t(t-1)(7t^2-2t-1), \quad 4t(t-1)(t+1)(2t-1)(6t^2-3t-1)\big), \\
P_2 &= \big(4t(t-1)^2(5t+1),  \quad 4t(t-1)^2(6t+1)(6t^2-3t-1)\big), \\
P_3 &= \big(4t^3(6-13t), \quad 12 \sqrt{-3}\, t^4(2t-1)(6t+1)\big), \\
\end{align*}
with height matrix
$$
\left(\begin{array}{ccc}
8/15 & 1/10 & 0 \\
1/10 & 2/15 & 0 \\
0 & 0 & 3/2
\end{array} \right).
$$
Therefore the K3 surface is singular, and an easy argument shows
that these sections and the trivial lattice must span the
N\'{e}ron-Severi lattice, which therefore has rank $20$ and discriminant $-99$.

\subsection{Examples}

We list some points of small height and corresponding genus $2$
curves.

\begin{tabular}{l|c}
Rational point $(r,s)$ & Sextic polynomial $f_6(x)$ defining the genus $2$ curve $y^2 = f_6(x)$. \\
\hline \hline \\ [-2.5ex]
$(-2/3, -7/10)$ & $ 25x^6 + 120x^5 - 291x^4 - 1292x^3 + 987x^2 + 588x - 497 $ \\
$(-4, 3/2)$ & $ 486x^6 - 810x^5 + 1323x^4 - 800x^3 + 585x^2 - 48x + 64 $ \\
$(-10/7, -3/2)$ & $ 515x^6 + 1314x^5 - 3120x^4 - 1332x^3 + 2292x^2 + 720x - 200 $ \\
$(-4/7, -3/10)$ & $ 20x^6 + 180x^5 - 159x^4 - 3276x^3 + 249x^2 + 1980x - 1100 $ \\
$(2/3, -1/4)$ & $ 4608x^6 + 6048x^5 + 3771x^4 - 1026x^3 + 351x^2 - 36x + 4 $ \\
$(-20/7, -5/6)$ & $ 356x^6 - 4980x^5 + 6373x^4 + 2580x^3 - 4409x^2 - 4170x - 790 $ \\
$(8/21, -2/5)$ & $ 1664x^6 + 624x^5 + 3747x^4 - 5222x^3 + 5511x^2 - 1140x + 15020 $ \\
$(-6/5, -7/20)$ & $ 6260x^6 - 21060x^5 + 7009x^4 - 1254x^3 - 239x^2 - 540x - 100 $ \\
$(-10/3, -7/4)$ & $ 2x^6 + 54x^5 + 45x^4 + 1080x^3 - 2961x^2 - 44352 $ \\
$(-4/3, -13/6)$ & $ -14388x^6 - 86076x^5 - 115441x^4 + 70272x^3 + 86417x^2 - 10794x + 314 $ \\
$(-22/21, -13/70)$ & $ 7865x^6 - 9750x^5 + 62049x^4 - 2788x^3 + 162759x^2 - 4350x + 119375 $ \\
$(-5/2, 21/8)$ & $ 363300x^6 - 50652x^5 + 128541x^4 + 2266x^3 + 19257x^2 + 1008x + 896 $ \\
$(38/65, -13/40)$ & $ -1106244x^6 + 336780x^5 + 23283x^4 + 248770x^3 - 101625x^2 - 33072x - 28736 $ \\
$(-8/39, -1/42)$ & $ -48600x^6 + 483840x^5 - 1386285x^4 - 264482x^3 - 282489x^2 + 883404x - 1658988 $ \\
$(-70/13, -21/22)$ & $ 599697x^6 - 445662x^5 + 824913x^4 - 838612x^3 + 2057823x^2 - 1620774x + 957519 $ \\
$(-10/7, -13/36)$ & $ -1112220x^6 + 2309556x^5 - 397465x^4 - 269262x^3 - 847153x^2 - 265908x + 612 $ 
\end{tabular}

We describe some curves on the surface which are a source of rational
points (some more may be produced by applying the involution $\iota$).
The specialization $s = -5/6$ gives a genus-$1$ curve
$$
y^2 = -(9r^2+50r+40)(243r^2+670r-40).
$$
It has rational points, such as $(r,y) = (0,40)$.
It is thus an elliptic curve; we find that it has conductor
$2 \cdot 3 \cdot 5^2 \cdot 29 \cdot 53$
and \MoW\ group $(\Z/2\Z) \oplus \Z^2$.

Pulling back sections of the elliptic fibration on the quotient
surface gives us some more curves of genus $1$, each with a rational point
and rank~$1$:

$$
\begin{array}{lccc}
\phantom{g^2 = }\textrm{Equation}  & \textrm{point}
 & \textrm{conductor} & \textrm{group} \\[5pt]
h = -\frac{4t^2-5t+2}{2t(t-1)}\, , \;
 g^2 = \frac{(2t-1)(14t^3-25t^2+16t-4)}{4t^2(t-1)^2} &
 t = \frac12 & 53 & \Z \\
h = -\frac{12t^2-5t+2}{2t(3t-2)}\, , \;
 g^2 = \frac{252t^4-192t^3+73t^2-20t+4}{4t^2(3t-2)^2} &
 t = 0 & 2 \cdot 3 \cdot 4391 & \Z \\
h = -\frac{8t^2-15t+8}{2t(t-1)}\, , \;
 g^2 = \frac{92t^4-320t^3+433t^2-268t+64}{4t^2(t-1)^2} &
 t = 0 & 7 \cdot 977 & (\Z/2\Z) \oplus \Z \\
h = -\frac{3t^2-3t+2}{3t(t-1)}\, , \;
 g^2 = \frac{18t^4-27t^3+24t^2-15t+4}{9t^2(t-1)^2} &
 t = 0 & 2^5 3^2 7 & (\Z/2\Z) \oplus \Z
\end{array}
$$

\section{Discriminant $89$}

\subsection{Parametrization}

We start with an elliptic K3 surface with fibers of type $A_8 A_7$,
and a section of height $89/72 = 4 -(1\cdot 8)/9 - (3 \cdot 5)/8$.

The Weierstrass equation for this family is
$$
y^2 = x^3
     + (a_0 + a_1t + a_2t^2 + a_3t^3 + a_4t^4)\, x^2
     + 2 \mu t^2(b_0 + b_1t + b_2t^2 + b_3t^3) \, x
     + \mu^2 t^4 (c_0 + c_1t + c_2t^2),
$$
with
\begin{align*}
a_0 &= (rs+1)^2, &
b_0 &= -(rs+1)^2, &
c_0 &= (rs+1)^2, \\
\mu &= 4rs(r+1)^2, &
b_3 &= s(s^2-rs-2s+1)^2, &
a_4 &= s^2(s^2-rs-2s+1)^2,
\end{align*}
\begin{align*}
c_2 &= (s^2-rs-2s+1)^2, \qquad c_1 = 2rs^3 - 2(r^2+4r+1)s^2+4(r+1)^2s-2, \\
b_2 &= (2r-1)s^4-(3r^2+7r-2)s^3+(r^3+6r^2+7r-1)s^2-2r(r+1)s+r, \\
a_1 &= -2(r^2-r)s^3+2(r^3-6r-1)s^2+2(4r^2+6r+1)s+2r, \\
b_1 &= r(r-2)s^3-(r^3-r^2-10r-2)s^2-(6r^2+10r+3)s-r+1, \\
a_3 &= 2s((r-1)s^4-(2r^2+3r-3)s^3+(r^3+4r^2+3r-3)s^2-(2r^2+2r-1)s+r), \\
a_2 &= (r^2-4r+1)s^4-2(r^3-2r^2-9r)s^3+(r^4-16r^2-22r-3)s^2-2(r^3-r-1)s+r^2.
\end{align*}

To obtain an $E_8 E_7$ elliptic fibration on these K3 surfaces, we
first move by a $2$-neighbor step to one with $E_7$ and $A_8$ fibers.

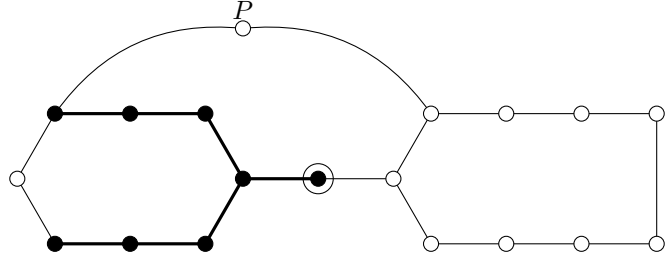
\begin{center}
\begin{tikzpicture}

\draw (0,0)--(0.5,0.866)--(2.5,0.866)--(3,0)--(2.5,-0.866)--(0.5,-0.866)--(0,0);
\draw (3,0)--(5,0)--(5.5,0.866)--(8.5,0.866)--(8.5,-0.866)--(5.5,-0.866)--(5,0);

\draw [bend right] (3,2) to (0.5,0.866);
\draw [bend left] (3,2) to (5.5,0.866);
\draw [very thick] (0.5,0.866)--(2.5,0.866)--(3,0)--(2.5,-0.866)--(0.5,-0.866);
\draw [very thick] (3,0)--(4,0);

\draw (4,0) circle (0.2);
\fill [white] (0,0) circle (0.1);
\fill [black] (0.5,0.866) circle (0.1);
\fill [black] (1.5,0.866) circle (0.1);
\fill [black] (2.5,0.866) circle (0.1);
\fill [black] (0.5,-0.866) circle (0.1);
\fill [black] (1.5,-0.866) circle (0.1);
\fill [black] (2.5,-0.866) circle (0.1);
\fill [black] (3,0) circle (0.1);
\fill [black] (4,0) circle (0.1);
\fill [white] (5,0) circle (0.1);
\fill [white] (5.5,0.866) circle (0.1);
\fill [white] (6.5,0.866) circle (0.1);
\fill [white] (7.5,0.866) circle (0.1);
\fill [white] (8.5,0.866) circle (0.1);
\fill [white] (5.5,-0.866) circle (0.1);
\fill [white] (6.5,-0.866) circle (0.1);
\fill [white] (7.5,-0.866) circle (0.1);
\fill [white] (8.5,-0.866) circle (0.1);

\fill [white] (3,2) circle (0.1);

\draw (0,0) circle (0.1);
\draw (0.5,0.866) circle (0.1);
\draw (1.5,0.866) circle (0.1);
\draw (2.5,0.866) circle (0.1);
\draw (0.5,-0.866) circle (0.1);
\draw (1.5,-0.866) circle (0.1);
\draw (2.5,-0.866) circle (0.1);
\draw (3,0) circle (0.1);
\draw (4,0) circle (0.1);
\draw (5,0) circle (0.1);
\draw (5.5,0.866) circle (0.1);
\draw (6.5,0.866) circle (0.1);
\draw (7.5,0.866) circle (0.1);
\draw (8.5,0.866) circle (0.1);
\draw (5.5,-0.866) circle (0.1);
\draw (6.5,-0.866) circle (0.1);
\draw (7.5,-0.866) circle (0.1);
\draw (8.5,-0.866) circle (0.1);

\draw (3,2) circle (0.1);
\draw (3,2) [above] node{$P$};

\end{tikzpicture}
\end{center}

The elliptic fibration defined by this new fiber $F'$ has a section,
since $P \cdot F' = 1$. Also, the new elliptic fibration must have a
section $P'$ of height $89/18 = 4 + 2\cdot 2 - 3/2 - (2 \cdot 7)/9$.

Finally, we go to $E_8 E_7$ by a $3$-neighbor step.

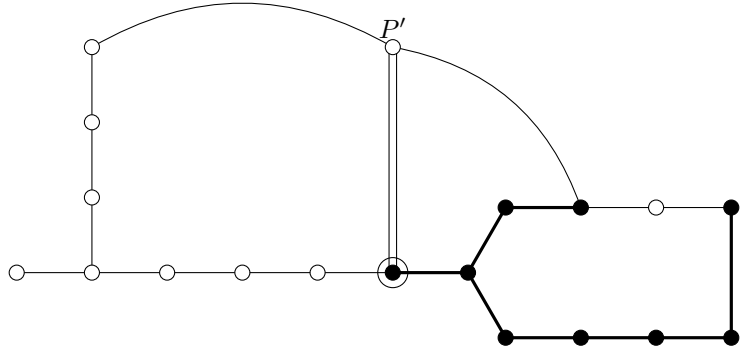
\begin{center}
\begin{tikzpicture}

\draw (3,0)--(3,3);
\draw (2,0)--(8,0)--(8.5,0.866)--(11.5,0.866)--(11.5,-0.866)--(8.5,-0.866)--(8,0);

\draw [bend right] (7,3) to (3,3);
\draw (6.95,3)--(6.95,0);
\draw (7.05,3)--(7.05,0);
\draw [bend left] (7,3) to (9.5,0.866);
\draw [very thick] (7,0)--(8,0);
\draw [very thick] (9.5,0.866)--(8.5,0.866)--(8,0)--(8.5,-0.866)--(11.5,-0.866)--(11.5,0.866);

\draw (7,0) circle (0.2);
\fill [white] (3,3) circle (0.1);
\fill [white] (3,2) circle (0.1);
\fill [white] (2,0) circle (0.1);
\fill [white] (3,0) circle (0.1);
\fill [white] (4,0) circle (0.1);
\fill [white] (5,0) circle (0.1);
\fill [white] (6,0) circle (0.1);
\fill [black] (7,0) circle (0.1);
\fill [black] (8,0) circle (0.1);
\fill [white] (3,1) circle (0.1);
\fill [black] (8.5,0.866) circle (0.1);
\fill [black] (9.5,0.866) circle (0.1);
\fill [white] (10.5,0.866) circle (0.1);
\fill [black] (11.5,0.866) circle (0.1);
\fill [black] (8.5,-0.866) circle (0.1);
\fill [black] (9.5,-0.866) circle (0.1);
\fill [black] (10.5,-0.866) circle (0.1);
\fill [black] (11.5,-0.866) circle (0.1);

\fill [white] (7,3) circle (0.1);

\draw (3,3) circle (0.1);
\draw (3,2) circle (0.1);
\draw (2,0) circle (0.1);
\draw (3,0) circle (0.1);
\draw (4,0) circle (0.1);
\draw (5,0) circle (0.1);
\draw (6,0) circle (0.1);
\draw (7,0) circle (0.1);
\draw (8,0) circle (0.1);
\draw (3,1) circle (0.1);
\draw (8.5,0.866) circle (0.1);
\draw (9.5,0.866) circle (0.1);
\draw (10.5,0.866) circle (0.1);
\draw (11.5,0.866) circle (0.1);
\draw (8.5,-0.866) circle (0.1);
\draw (9.5,-0.866) circle (0.1);
\draw (10.5,-0.866) circle (0.1);
\draw (11.5,-0.866) circle (0.1);

\draw (7,3) circle (0.1);

\draw [above] (7,3) node {$P'$};

\end{tikzpicture}
\end{center}

The new fiber $F''$ satisfies $P' \cdot F'' = 2 + 2 \cdot 3 = 8$, and
the identity component of the $E_7$ fiber intersects $F'$ in
$3$. Since these have greatest common divisor $1$, the genus $1$
fibration defined by $F'$ has a section.

\begin{theorem}
A birational model over $\Q$ for the Hilbert modular surface
$Y_{-}(89)$ as a double cover of $\Proj_{r,s}^2$ is given by the following
equation:
\begin{align*}
z^2 &= s^4 r^6 -2 s^3 (2 s^2+3 s+2) r^5 + s^2 (6 s^4+16 s^3-49 s^2-26 s+6) r^4 -2 s (2 s^6+6 s^5-50 s^4\\
 & \quad  +26 s^3  +73 s^2-35 s+2) r^3 + (s^8-36 s^6+26 s^5+273 s^4-514 s^3+271 s^2-38 s+1) r^2 \\
& \quad + 2 (s-1)^2 s (s^5-4 s^4+25 s^3-107 s^2+147 s-44) r + (s-4)^3 (s-1)^4 s.
\end{align*}
It is a surface of general type.
\end{theorem}

\subsection{Analysis}
The branch locus has genus $1$; one can give an explicit isomorphism
(see the auxiliary files) to the elliptic curve of conductor $89$
given by the Weierstrass equation
$$
y^2 + xy + y = x^3 + x^2 - x.
$$
It is isomorphic to $X_0(89)/\langle w \rangle$, where $w$ is the
Atkin-Lehner involution.

The Hilbert modular surface $Y_{-}(89)$ is a surface of general
type. Note that the change of coordinates $r = s + g$ simplifies the
equation a bit further, making the degree of the right hand side equal
to $6$ in each variable. However, it complicates the original defining
Weierstrass equation of the family of K3 surfaces, so we have chosen
the $(r,s)$ coordinate system.

\subsection{Examples}

We list some points of small height and corresponding genus $2$
curves.

\begin{tabular}{l|c}
Point $(r,s)$ & Sextic polynomial $f_6(x)$ defining the genus $2$ curve $y^2 = f_6(x)$. \\
\hline \hline \\ [-2.5ex]
$(-31/3, -1/6)$ & $ -334084x^6 + 65892x^5 + 847841x^4 - 156012x^3 - 1036555x^2 - 453867x - 525 $ \\
$(-49/9, -2/3)$ & $ 632x^6 - 480x^5 + 43475x^4 - 97578x^3 - 1030393x^2 + 855708x - 1045044 $ \\
$(-19, -1)$ & $ 126905x^6 + 2388081x^5 - 2600778x^4 - 3075787x^3 - 5448045x^2 - 3683352x - 709200 $ \\
$(-31/10, 5/2)$ & $ -83300x^6 + 168420x^5 + 5079215x^4 - 6586832x^3 + 584735x^2 + 70020x - 8100 $ \\
$(-5/6, 11/6)$ & $ 2185004x^6 - 12346980x^5 + 10798163x^4 + 732660x^3 + 47975267x^2$ \\
& $ + 21406020x + 27911916 $ \\
$(13/9, -5/9)$ & $ -14966100x^6 - 43598124x^5 + 25890735x^4 + 105396908x^3 - 44422995x^2$ \\
& $ - 65750574x + 34674550 $ \\
$(-40/7, -1/2)$ & $ 2754000x^6 + 86434200x^5 + 150411025x^4 - 14830346x^3 - 49970411x^2$ \\
& $ + 242599308x + 131021492 $ \\
$(16/33, 11/3)$ & $ 25329267x^6 - 96789717x^5 + 223774305x^4 - 449560367x^3 - 46904988x^2$ \\
& $ - 772810308x + 413626230 $ 
\end{tabular}

We find two elliptic curves of positive rank on the surface.
The specialization $s = 25/22$ gives a curve of genus $1$
$$
y^2 = 1210000r^4-19157600r^3-17065736r^2+678600r-8575
$$
with rational points (as at infinity), conductor
$3 \cdot 5 \cdot 11 \cdot 163 \cdot 191 \cdot 881$,
and rank at least $2$.
The locus $s = r + 101/50$ gives another curve of genus $1$
$$
y^2 = -3739190000r^4-21451957600r^3-43018833576r^2-36551728152r-11227811551
$$
with rational points (as at $(r,y)=(-1,\pm 65^2)$), conductor
$2^4 \cdot 5 \cdot 17 \cdot 19 \cdot 463 \cdot 58787$,
and rank at least $3$.  We were not able to determine the exact rank
of either curve, but the global root numbers indicate that the rank
should be even for the former curve and odd for the latter, so one might
guess that the lower bounds $2$ and $3$ on their ranks are sharp.

\section{Discriminant $92$}

\subsection{Parametrization}

We start with an elliptic K3 surface with fibers of type $A_8, A_1$
and $D_6$, and a section of height $92/72 = 23/18 = 4 - (1 \cdot 1)/2
- (4 \cdot 5)/9$.

The Weierstrass equation may be written as
$$
y^2 = x^3 + (a_0+a_1 t+a_2 t^2+a_3 t^3) \, x^2
     + 2 t^2(\lambda  t-\mu) (b_0+b_1 t+b_2 t^2) \, x
     + t^4(\lambda  t-\mu)^2 (c_0+c_1 t),
$$
with
\begin{flalign*}
\lambda &= (r+s)^2,  &  \mu &= r+2s,    &  a_3 &= -4rs(s+r^2+r)(s^2+rs-r), \\
a_0 &=  (rs-r-1)^2, &  b_0 &= -4rs^2(rs-r-1)^2, & c_0 &= 16r^2s^4(rs-r-1)^2, \\
a_1 &=  -2r( (r+1)^2(s-1)^2 + s^2), &
  b_1 &=  4r^2s^2(rs+2s-r-1)^2, &
  c_1 &=  -64r^3(r+1)(s-1)s^5,
\end{flalign*}
\vspace*{-5ex}
\begin{flalign*}
b_2 &=  -8r^2s^3( (r+2)s^2 + (2r^2 + 2r- 1)s - 2r(r+1)), &
a_2 &=  r(rs+4s-r-1)(4s^2+r^2s+4rs-r^2-r).
\end{flalign*}

As in the case of discriminant $56$, we first go to an $E_7 A_8$
fibration using the $E_7$ fiber $F'$ identified below. Note that $F'
\cdot P = 3$ while the component of the $D_6$ fiber which is not
included in $F'$ intersects $F'$ with multiplicity~$2$.
Since $\gcd(2,3) = 1$, the fibration defined by $F'$ has a section.

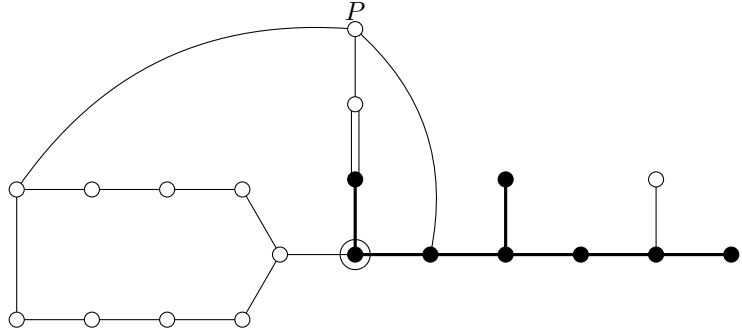
\begin{center}
\begin{tikzpicture}

\draw (1,0)--(0.5,0.866)--(-2.5,0.866)--(-2.5,-0.866)--(0.5,-0.866)--(1,0)--(7,0);
\draw (2,0)--(2,1);
\draw (1.95,1)--(1.95,2);
\draw (2.05,1)--(2.05,2);
\draw (4,0)--(4,1);
\draw (6,0)--(6,1);
\draw (2,3)--(2,2);
\draw [bend right] (2,3) to (-2.5,0.866);
\draw [bend left] (2,3) to (3,0);
\draw [very thick] (2,1)--(2,0)--(7,0);
\draw [very thick] (4,0)--(4,1);

\draw (2,3) [above] node{$P$};
\draw (2,0) circle (0.2);
\fill [white] (0.5,0.866) circle (0.1);
\fill [white] (0.5,-0.866) circle (0.1);
\fill [white] (-0.5,0.866) circle (0.1);
\fill [white] (-0.5,-0.866) circle (0.1);
\fill [white] (-1.5,0.866) circle (0.1);
\fill [white] (-1.5,-0.866) circle (0.1);
\fill [white] (-2.5,0.866) circle (0.1);
\fill [white] (-2.5,-0.866) circle (0.1);

\fill [white] (1,0) circle (0.1);
\fill [black] (2,1) circle (0.1);
\fill [white] (2,2) circle (0.1);
\fill [black] (2,0) circle (0.1);
\fill [black] (3,0) circle (0.1);
\fill [black] (4,0) circle (0.1);
\fill [black] (5,0) circle (0.1);
\fill [black] (6,0) circle (0.1);
\fill [black] (7,0) circle (0.1);
\fill [black] (4,1) circle (0.1);
\fill [white] (6,1) circle (0.1);
\fill [white] (2,3) circle (0.1);

\draw (2,3) circle (0.1);

\draw (0.5,0.866) circle (0.1);
\draw (0.5,-0.866) circle (0.1);
\draw (-0.5,0.866) circle (0.1);
\draw (-0.5,-0.866) circle (0.1);
\draw (-1.5,0.866) circle (0.1);
\draw (-1.5,-0.866) circle (0.1);
\draw (-2.5,0.866) circle (0.1);
\draw (-2.5,-0.866) circle (0.1);

\draw (1,0) circle (0.1);
\draw (2,1) circle (0.1);
\draw (2,2) circle (0.1);
\draw (2,0) circle (0.1);
\draw (3,0) circle (0.1);
\draw (4,0) circle (0.1);
\draw (5,0) circle (0.1);
\draw (6,0) circle (0.1);
\draw (7,0) circle (0.1);
\draw (4,1) circle (0.1);
\draw (6,1) circle (0.1);

\end{tikzpicture}
\end{center}

The new elliptic fibration has a section $P'$ of height $92/(2 \cdot
9) = 46/9 = 4 + 2 \cdot 1 - 8/9$, which must therefore intersect the
zero section, the identity component of the $E_7$ fiber and component
$1$ of the $A_8$ fiber.

We identify an $E_8$ fiber and compute its Weierstrass equation by a
$3$-neighbor move. Note that it intersects $P'$ in $7$ and the
excluded component of the $A_8$ fiber in $3$. Therefore the fibration
it defines has a section, and we may convert to the Jacobian.

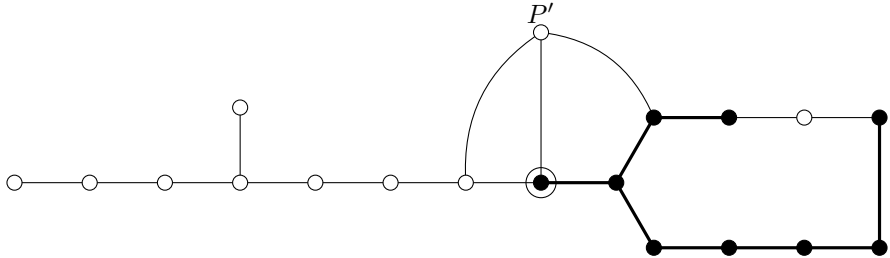
\begin{center}
\begin{tikzpicture}

\draw (-1,0)--(7,0)--(7.5,0.866)--(7.5,0.866)--(10.5,0.866)--(10.5,-0.866)--(7.5,-0.866)--(7,0);
\draw (2,0)--(2,1);
\draw (6,2)--(6,0);
\draw [bend right] (6,2) to (5,0);
\draw [bend left] (6,2) to (7.5,0.866);
\draw [very thick] (6,0)--(7,0)--(7.5,0.866)--(8.5,0.866);
\draw [very thick] (7,0)--(7.5,-0.866)--(10.5,-0.866)--(10.5,0.866);

\draw (6,2) [above] node{$P'$};
\draw (6,0) circle (0.2);
\fill [white] (-1,0) circle (0.1);
\fill [white] (0,0) circle (0.1);
\fill [white] (1,0) circle (0.1);
\fill [white] (2,0) circle (0.1);
\fill [white] (3,0) circle (0.1);
\fill [white] (4,0) circle (0.1);
\fill [white] (5,0) circle (0.1);
\fill [black] (6,0) circle (0.1);
\fill [white] (2,1) circle (0.1);
\fill [black] (7,0) circle (0.1);
\fill [black] (7.5,0.866) circle (0.1);
\fill [black] (7.5,-0.866) circle (0.1);
\fill [black] (8.5,0.866) circle (0.1);
\fill [black] (8.5,-0.866) circle (0.1);
\fill [white] (9.5,0.866) circle (0.1);
\fill [black] (9.5,-0.866) circle (0.1);
\fill [black] (10.5,0.866) circle (0.1);
\fill [black] (10.5,-0.866) circle (0.1);
\fill [white] (6,2) circle (0.1);

\draw (-1,0) circle (0.1);
\draw (0,0) circle (0.1);
\draw (1,0) circle (0.1);
\draw (2,0) circle (0.1);
\draw (3,0) circle (0.1);
\draw (4,0) circle (0.1);
\draw (5,0) circle (0.1);
\draw (6,0) circle (0.1);
\draw (2,1) circle (0.1);
\draw (7,0) circle (0.1);
\draw (7.5,0.866) circle (0.1);
\draw (7.5,-0.866) circle (0.1);
\draw (8.5,0.866) circle (0.1);
\draw (8.5,-0.866) circle (0.1);
\draw (9.5,0.866) circle (0.1);
\draw (9.5,-0.866) circle (0.1);
\draw (10.5,0.866) circle (0.1);
\draw (10.5,-0.866) circle (0.1);
\draw (6,2) circle (0.1);

\end{tikzpicture}
\end{center}

Now we read out the Igusa-Clebsch invariants and compute the branch locus.

\begin{theorem}
A birational model over $\Q$ for the Hilbert modular surface
$Y_{-}(92)$ as a double cover of $\Proj_{r,s}^2$ is given by the following
equation:
\begin{align*}
z^2 &= (s+r^2+r)(s^2+rs-r) \big( (s-1)^3r^5 + (s-1)^2(s^2-15s-3)r^4 -(s-1)(42s^3-27s^2-31s-3)r^3 \\
 &\qquad  -(27s^5-30s^4-77s^3+69s^2+17s+1)r^2 + s(46s^3-30s^2-42s-1)r -27s^3 \big).
\end{align*}
It is a surface of general type.
\end{theorem}

\subsection{Analysis}

The branch locus has three components. Points of the two simpler
components correspond to elliptic K3 surfaces where the $D_6$ fiber is
promoted to an $E_7$ fiber, while the more complicated component
corresponds to an extra $\I_2$ fiber. All the three components are
rational (genus $0$). This is obvious from inspection for the simpler
components, and for the last we have the parametrization
$$
(r,s) = \left( \frac{ t(t^2+2)(t+1)^2}{ (t+2)(t^3+2t^2+2t+2)} , \frac{-t(t+1)(t^3+2t^2+4t+4)}{(t+2)^2(t^2+2)} \right).
$$

The surface $Y_{-}(92)$ is a surface of general type.
The extra involution is $(r,s,z) \mapsto (-1/s, -1/r, z/(rs)^4)$.
We now analyze the quotient of the Hilbert modular surface by this involution.
Because the involution fixes $r/s$, we obtain the quotient by setting
$r = st$ and writing everything in terms of $m = s - 1/(ts)$.
We find the equation
\begin{align*}
y^2 &= \big(t^2(t+1)m -t^3+t^2+2t+1 \big)
  \big(t^3(t+1)m^3 -t(3t^3+17t^2+42t+27)m^2  \\
 &\quad  + t(3t^3+31t^2+72t+30)m -(t^4+15t^3+30t^2+7t+8) \big),
\end{align*}
which expresses the quotient as a genus-$1$ curve over $\Q(t)$.  Since
there is an obvious section (where the first factor vanishes), we may
convert to the Jacobian, which has the Weierstrass equation (after
shifting $t$ by $1$ and performing some Weierstrass transformations)
$$
y^2 = x^3-(2t+1)(8t^3+8t^2-6t-1)\, x^2-8(t-1)t^4(t+1)(10t^2-32t-5)\,x -16(t-1)^2t^8(t+1)(47t+7).
$$ 
This is an elliptic K3 surface. It has reducible fibers of type
$\I_8$ at $t = 0$, $\I_3$ at $1$ and $(-7 \pm 3 \sqrt{3})/11$, and
$\I_2$ at the roots of $11t^3-10t^2+5t+1$ (which generates the cubic
field of discriminant $-23$). Therefore the trivial lattice has rank
$18$. We easily identify a non-torsion section $P$ of height $5/8$
with $x$-coordinate $4t^3(6t+1)$.  On the other hand, counting points
modulo $13$ and $17$ shows that the Picard number cannot be $20$.
Therefore the Picard number of this quotient surface is~$19$. The
sublattice of the N\'eron-Severi group spanned by $P$ and the trivial
lattice has discriminant $360 = 2^3 \cdot 3^2 \cdot 5$. It is easy to
see from height calculations that there cannot be any $2$- or
$3$-torsion sections, and that $P$ cannot be divisible by $2$ or $3$
in the \MoW\ group. Hence, this sublattice is the entire
N\'eron-Severi lattice.

\subsection{Examples}

We list some points of small height and corresponding genus $2$ curves.

\begin{tabular}{l|c}
Point $(r,s)$ & Sextic polynomial $f_6(x)$ defining the genus $2$ curve $y^2 = f_6(x)$. \\
\hline \hline \\ [-2.5ex]
$(3/10, -39/70)$ & $ 5981584x^6 - 4016376x^5 + 1699985x^4 + 313485x^3 - 168322x^2 + 49665x + 21175 $ \\
$(70/39, -10/3)$ & $ 2916x^6 + 591516x^5 + 6670933x^4 + 12740602x^3 - 44084051x^2 + 8704740x + 16105100 $ 
\end{tabular}

Specializations of $r$ or $s$, and pullbacks of sections of the
quotient, do not seem to yield any genus $0$ or $1$ curves on the
surface (at any rate, none corresponding to abelian surfaces with ``honest''
real multiplication by $\sO_{92}$ and not a larger endomorphism ring).

\section{Discriminant $93$}

\subsection{Parametrization}

Start with an elliptic K3 surface with fibers of type $A_{10}$, $A_3$
and $A_2$, with a section of height $31/44 = 4 - 3/4 - 28/11$. The
extra involution comes from flipping the $A_2$ fiber.

This family has the Weierstrass equation
\begin{eqnarray*}
y^2 &\!\! = \!\! & x^3 + (a_0 + a_1t + a_2t^2 + a_3t^3 + a_4t^4) \, x^2\\
&& {} + 2 t(\lambda t-\mu) (b_0 + b_1t + b_2t^2 + b_3t^3) \, x
   + t^2(\lambda t-\mu)^2  (c_0 + c_1t + c_2t^2),
\end{eqnarray*}
with
\begin{flalign*}
\lambda &= -(n^2-1), & \mu &= (n^2-mn-n-m)(n^2+mn+n-m), & a_0 &= (m+1)^2n^8, \\
b_0 &= m^4 (m+1)^2 n^8, & c_0 &= m^8(m+1)^2n^8, & a_4 &= 1, \\
c_2 &= m^8 n^4, & b_3 &= m^4 n^2, & a_3 &= (m^2+2m+4)n^2-3m^2, \\
\end{flalign*}
\vspace*{-7ex}
\begin{flalign*}
c_1 &= m^8 n^4( (m^2+2m+2)n^2-m^2), \qquad \qquad \qquad \qquad \qquad
  b_2 = m^4 n^2 \big((m^2+2m+3)n^2-2m^2 \big), & \\
a_2 &= 3(m^2+2m+2)n^4-2m^2(m^2+3m+3)n^2+3m^4, & \\
b_1 &= m^4 n^2 \big( (2m^2+4m+3)n^4 -m^2(m+1)(m+2)n^2+m^4 \big), & \\
a_1 &= (3m^2+6m+4)n^6-3m^2(m+1)^2n^4+m^4(m+1)(m+3)n^2-m^6.
\end{flalign*}

We first identify the class of an $E_8$ fiber below, and move to it by a
$3$-neighbor step.

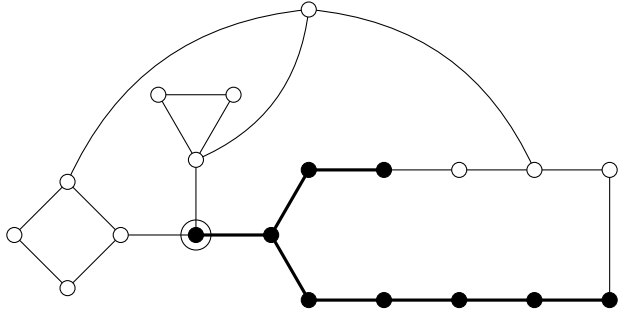
\begin{center}
\begin{tikzpicture}

\draw (0,0)--(2,0)--(2.5,0.866)--(6.5,0.866)--(6.5,-0.866)--(2.5,-0.866)--(2,0);
\draw (1,0)--(1,1)--(0.5,1.866)--(1.5,1.866)--(1,1);
\draw (0,0)--(-0.707,0.707)--(-1.4142,0)--(-0.707,-0.707)--(0,0);
\draw [bend right] (2.5,3) to (-0.707,0.707);
\draw [bend left] (2.5,3) to (1,1);
\draw [bend left] (2.5,3) to (5.5,0.866);
\draw [very thick] (1,0)--(2,0)--(2.5,0.866)--(3.5,0.866);
\draw [very thick] (2,0)--(2.5,-0.866)--(6.5,-0.866);

\fill [white] (2.5,3) circle (0.1);

\fill [white] (0,0) circle (0.1);
\fill [black] (1,0) circle (0.1);
\fill [black] (2,0) circle (0.1);
\fill [black] (2.5,0.866) circle (0.1);
\fill [black] (3.5,0.866) circle (0.1);
\fill [white] (4.5,0.866) circle (0.1);
\fill [white] (5.5,0.866) circle (0.1);
\fill [white] (6.5,0.866) circle (0.1);
\fill [black] (2.5,-0.866) circle (0.1);
\fill [black] (3.5,-0.866) circle (0.1);
\fill [black] (4.5,-0.866) circle (0.1);
\fill [black] (5.5,-0.866) circle (0.1);
\fill [black] (6.5,-0.866) circle (0.1);
\fill [white] (1,1) circle (0.1);
\fill [white] (0.5,1.866) circle (0.1);
\fill [white] (1.5,1.866) circle (0.1);
\fill [white] (-0.707,0.707) circle (0.1);
\fill [white] (-0.707,-0.707) circle (0.1);
\fill [white] (-1.4142,0) circle (0.1);

\draw (2.5,3) circle (0.1);

\draw (1,0) circle (0.2);
\draw (0,0) circle (0.1);
\draw (1,0) circle (0.1);
\draw (2,0) circle (0.1);
\draw (2.5,0.866) circle (0.1);
\draw (3.5,0.866) circle (0.1);
\draw (4.5,0.866) circle (0.1);
\draw (5.5,0.866) circle (0.1);
\draw (6.5,0.866) circle (0.1);
\draw (2.5,-0.866) circle (0.1);
\draw (3.5,-0.866) circle (0.1);
\draw (4.5,-0.866) circle (0.1);
\draw (5.5,-0.866) circle (0.1);
\draw (6.5,-0.866) circle (0.1);
\draw (1,1) circle (0.1);
\draw (0.5,1.866) circle (0.1);
\draw (1.5,1.866) circle (0.1);
\draw (-0.707,0.707) circle (0.1);
\draw (-0.707,-0.707) circle (0.1);
\draw (-1.4142,0) circle (0.1);

\end{tikzpicture}
\end{center}

This gives us an elliptic fibration with $E_8$, $A_5$ and $A_2$
fibers, and a section $P$ of height $93/18 = 31/6 = 4 + 2 \cdot 1 -
5/6$. We then identify an $E_7$ fiber and move to it by a $2$-neighbor
step. Since the new fiber intersects the section $P$ in $7$ and the
excluded component of the $A_5$ fiber in $3$, we see that the new
fibration has a section.

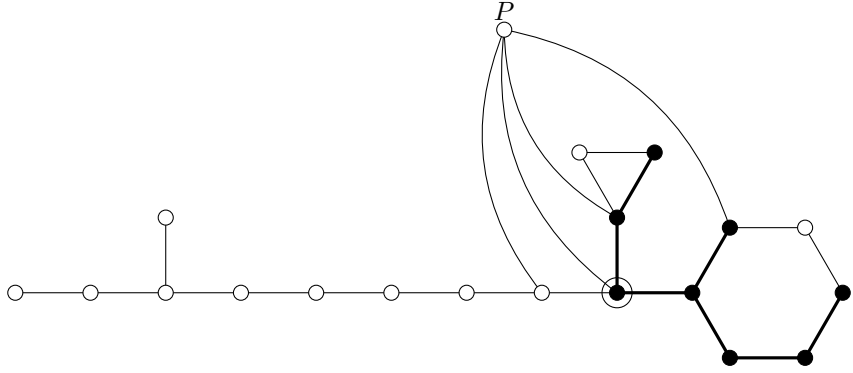
\begin{center}
\begin{tikzpicture}

\draw [bend right] (6.5,3.5) to (8,1);
\draw [bend right] (6.5,3.5) to (8,0);
\draw [bend right] (6.5,3.5) to (7,0);
\draw [bend left] (6.5,3.5) to (9.5,0.866);
\draw [very thick] (8.5,1.866)--(8,1)--(8,0)--(9,0)--(9.5,-0.866)--(10.5,-0.866)--(11,0);
\draw [very thick] (9,0)--(9.5,0.866);

\draw (0,0)--(9,0)--(9.5,0.866)--(10.5,0.866)--(11,0)--(10.5,-0.866)--(9.5,-0.866)--(9,0);

\draw (2,0)--(2,1);
\draw (8,0)--(8,1)--(7.5,1.866)--(8.5,1.866)--(8,1);

\draw (8,0) circle (0.2);
\fill [white] (0,0) circle (0.1);
\fill [white] (1,0) circle (0.1);
\fill [white] (2,0) circle (0.1);
\fill [white] (2,1) circle (0.1);
\fill [white] (3,0) circle (0.1);
\fill [white] (4,0) circle (0.1);
\fill [white] (5,0) circle (0.1);
\fill [white] (6,0) circle (0.1);
\fill [white] (7,0) circle (0.1);
\fill [black] (8,0) circle (0.1);
\fill [black] (8,1) circle (0.1);
\fill [black] (9,0) circle (0.1);
\fill [black] (9.5,0.866) circle (0.1);
\fill [white] (10.5,0.866) circle (0.1);
\fill [black] (9.5,-0.866) circle (0.1);
\fill [black] (10.5,-0.866) circle (0.1);
\fill [black] (11,0) circle (0.1);

\fill [white] (7.5,1.866) circle (0.1);
\fill [black] (8.5,1.866) circle (0.1);

\fill [white] (6.5,3.5) circle (0.1);

\draw (0,0) circle (0.1);
\draw (1,0) circle (0.1);
\draw (2,0) circle (0.1);
\draw (2,1) circle (0.1);
\draw (3,0) circle (0.1);
\draw (4,0) circle (0.1);
\draw (5,0) circle (0.1);
\draw (6,0) circle (0.1);
\draw (7,0) circle (0.1);
\draw (8,0) circle (0.1);
\draw (8,1) circle (0.1);
\draw (9,0) circle (0.1);
\draw (9.5,0.866) circle (0.1);
\draw (10.5,0.866) circle (0.1);
\draw (9.5,-0.866) circle (0.1);
\draw (10.5,-0.866) circle (0.1);
\draw (11,0) circle (0.1);

\draw (7.5,1.866) circle (0.1);
\draw (8.5,1.866) circle (0.1);

\draw (6.5,3.5) circle (0.1);

\draw [above] (6.5,3.5) node {$P$};

\end{tikzpicture}
\end{center}

\begin{theorem}
A birational model over $\Q$ for the Hilbert modular surface
$Y_{-}(93)$ as a double cover of $\Proj^2_{m,n}$ is given by the
following equation:
\begin{align*}
z^2 &= 16(n^2-1)^2n^2m^6 + 8(n^2-1)(21n^4+22n^2-27)m^5 -(27n^8-684n^6-1246n^4+1620n^2+27)m^4  \\
& \quad -8n^2(27n^6-109n^4-471n^2+41)m^3 -8n^2(81n^6+135n^4-273n^2-7)m^2 \\
& \quad -96n^4(9n^2-1)(n^2+3)m -16n^4(n^2+3)(27n^2+1).
\end{align*}
It is a surface of general type.
\end{theorem}

\subsection{Analysis}

The extra involution is $\iota: (m,n) \mapsto (m,-n)$.

The branch locus is a curve of genus $4$, isomorphic to
$X_0(93)/\langle w_{93} \rangle$, where $w_{93}$ is the Atkin-Lehner
involution. We do not give the explicit isomorphism here, but the
formulas are available in the auxiliary computer files. Setting $k =
n^2$, we can write it as a double cover of a genus $2$ curve, which
can be transformed to the Weierstrass form
$$
y^2 -(9x^3+11x-3)y + x^2(69x^3-56x^2+81x-22) = 0.
$$

This Hilbert modular surface is a surface of general type.
The quotient by this involution $\iota$ has the equation (with $k = n^2$)
\begin{align*}
z^2 &= -27(m+2)^4\, k^4 + 4(4m^6+42m^5+171m^4+218m^3-270m^2-624m-328)\, k^3 \\
    &\quad   -2(16m^6-4m^5-623m^4-1884m^3-1092m^2-144m+24) \,k^2 \\
    &\quad  + 4m^2(4m^4-98m^3-405m^2-82m+14)\, k + 27m^4(8m-1)
\end{align*}
This has a genus-$1$ fibration over $\Q(m)$.
The fibration has a section at infinity defined over $\Q(\sqrt{-3})$;
we do not know whether there is a section defined over~$\Q$.
The Jacobian has the Weierstrass equation
\begin{align*}
y^2 &= x^3 + (m^6+20m^5+118m^4+186m^3+33m^2+18m-3) \, x^2 \\
    & \quad + 8m^2(m+10)(9m^4+39m^3+57m^2-1) \, x + 16m^4(m+10)^2(4m^3+13m^2+18m-3).
\end{align*}
This is an honestly elliptic surface with $\chi = 3$. It has reducible
fibers of type $\I_9$ at $m = \infty$, $\I_4$ at $m = 0$, $\I_2$ at $m
= -10$, and $\I_3$ at the roots of $m^6+11m^5+16m^4+32m^3+17m^2-9m+1$
(whose splitting field is a dihedral extension of degree $12$
containing $\sqrt{93}$). The trivial lattice has rank $26$, leaving
room for \MoW\ rank up to $4$. Counting points modulo $13$ and $17$
shows that the Picard number is at most $29$. On the other hand, we
are (so far) able to produce the sections
\begin{align*}
P_0 &= \big( -16(m^3+m^2+6m-1), 32(m^6+11m^5+16m^4+32m^3+17m^2-9m+1) \big) \\
P_1 &= \big( -16(m+10)(m^5+m^4+6m^3+3m^2+18m-3)/31, \\
 & \qquad 288 (m+10)(3m^2+2m+9)(m^6+11m^5+16m^4+32m^3+17m^2-9m+1)/93^{3/2} \big),
\end{align*}
of which $P_0$ is $3$-torsion, while $P_1$ has height
$3/2$. Therefore, the \MoW\ rank is between $1$ and $3$.

\subsection{Examples}

We list some points of small height and corresponding genus $2$ curves.

\begin{tabular}{l|c}
Point $(m,n)$ & Sextic polynomial $f_6(x)$ defining the genus $2$ curve $y^2 = f_6(x)$. \\
\hline \hline \\ [-2.5ex]
$(2, 1/3)$ & $ -2112x^6 - 5184x^5 + 5451x^4 + 2593x^3 - 4596x^2 - 2223x - 101 $ \\
$(-10, 5/3)$ & $ -7452x^6 - 4860x^5 - 24039x^4 - 4540x^3 - 17205x^2 + 4686x - 302 $ \\
$(1/5, 1/7)$ & $ -24786x^6 + 25272x^5 + 90900x^4 - 73885x^3 - 107482x^2 + 54020x + 40286 $ \\
$(-10, 5)$ & $ -31752x^5 - 48825x^4 + 52868x^3 - 175537x^2 + 91124x - 80644 $ \\
$(2, -1/3)$ & $ -594x^6 + 10962x^5 - 154233x^4 + 391936x^3 + 265521x^2 + 330228x - 71068 $ \\
$(-10, -5)$ & $ 43756x^6 + 110088x^5 + 463887x^4 - 609201x^3 + 208770x^2 - 6300x - 211000 $ \\
$(-10, -5/3)$ & $ -3008x^6 + 270048x^5 - 773739x^4 - 611989x^3 - 2150523x^2 + 631152x - 342144 $ \\
$(1/5, -1/7)$ & $ -253800x^6 - 1186380x^5 - 1627302x^4 + 4611739x^3 + 1795017x^2 - 2139291x + 2480233 $ 
\end{tabular}

\section{Discriminant $97$}

\subsection{Parametrization}

Start with an elliptic K3 surface with fibers of type $D_5 A_4 A_6$,
with a section of height $97/140 = 4 - (1+1/4) - 6/5 - 6/7$.

We can write the Weierstrass equation as
$$
y^2 = x^3 + (a_0+a_1t+a_2t^2+a_3t^3) \, x^2
     + 2t^2(t-1)^2 (b_0+b_1t+b_2t^2) \, x
      +t^4(t-1)^4 (c_0+c_1t),
$$
with
\begin{align*}
a_0 &=  (r+1)^2(rs^2+s^2+r^2s+r)^2, \\
a_1 &=  2(r+1)\big( (r+1)^2s^5 + 2(r+1)(r^2-r-1)s^4 + r(r^3-4r^2-4r+2)s^3 \\
  & \quad  -r(2r^3+r^2+2r+5)s^2 -r^2(r^2+3r+3)s -r^2(r+1) \big), \\
a_2 &=  (r+1)^2s^6 + 2(r+1)(r^2-4r-2)s^5 + (r^4-16r^3-6r^2+18r+4)s^4 \\
   & \quad -2r(4r^3-6r^2-4r+11)s^3 + r(6r^3-6r^2-3r+20)s^2
    + 2r^2(r+3)(2r+3)s + r^2(r+1)^2, \\
a_3 &= -4r(s-1)s(s+r)\big(s^3+(r-2)s^2-(r-2)s-2r-3 \big), \\
b_0 &=  4r(r+1)^2(s-1)^2s(s+r)(rs^2+s^2+r^2s+r)^2, \\
b_1 &=  4r(r+1)(s-1)^2s(s+r)\big( (r+1)^2s^5 + 2(r+1)(r^2-2r-1)s^4 \\
  & \quad  + r(r^3-8r^2-6r+4)s^3 -r(4r^3+r^2+2r+7)s^2
    -r^2(r^2+5r+5)s - r^2(r+1)\big), \\
b_2 &=  -8r^2(r+1)(s-1)^2s^2(s+r)^2\big(2s^3+2(r-2)s^2-2(r-2)s-r-3 \big), \\
c_0 &=  16r^2(r+1)^2(s-1)^4s^2(s+r)^2(rs^2+s^2+r^2s+r)^2, \\
c_1 &=  -64r^3(r+1)^2(s-1)^4s^3(s+r)^3(s^2+rs-s+1).
\end{align*}

First we identify a $D_8$, and move to the associated elliptic
fibration (which also has an $A_6$ fiber) by a $2$-neighbor step.

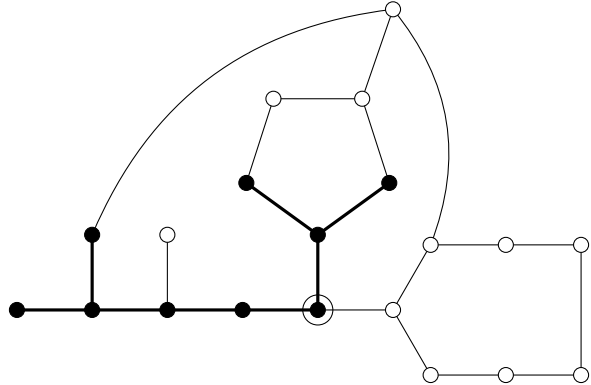
\begin{center}
\begin{tikzpicture}

\draw (4,0)--(4,1);
\draw (4,1)--(3.05,1.69);
\draw (4,1)--(4.95,1.69);
\draw (3.05,1.69)--(3.41,2.81);
\draw (4.95,1.69)--(4.59,2.81);
\draw (3.41,2.81)--(4.59,2.81);
\draw (0,0)--(5,0)--(5.5,0.866)--(7.5,0.866)--(7.5,-0.866)--(5.5,-0.866)--(5,0);
\draw (1,0)--(1,1);
\draw (2,0)--(2,1);
\draw [bend right] (5,4) to (1,1);
\draw (5,4)--(4.59,2.81);
\draw [bend left] (5,4) to (5.5,0.866);
\draw [very thick] (0,0)--(4,0)--(4,1);
\draw [very thick] (1,0)--(1,1);
\draw [very thick] (3.05,1.69)--(4,1)--(4.95,1.69);

\draw (4,0) circle (0.2);
\fill [black] (4,0) circle (0.1);
\fill [black] (4,1) circle (0.1);
\fill [black] (3.05,1.69) circle (0.1);
\fill [black] (4.95,1.69) circle (0.1);
\fill [white] (3.41,2.81) circle (0.1);
\fill [white] (4.59,2.81) circle (0.1);
\fill [black] (0,0) circle (0.1);
\fill [black] (1,0) circle (0.1);
\fill [black] (2,0) circle (0.1);
\fill [black] (3,0) circle (0.1);
\fill [white] (5,0) circle (0.1);
\fill [white] (5.5,0.866) circle (0.1);
\fill [white] (6.5,0.866) circle (0.1);
\fill [white] (7.5,0.866) circle (0.1);
\fill [white] (5.5,-0.866) circle (0.1);
\fill [white] (6.5,-0.866) circle (0.1);
\fill [white] (7.5,-0.866) circle (0.1);
\fill [black] (1,1) circle (0.1);
\fill [white] (2,1) circle (0.1);
\fill [white] (5,4) circle (0.1);

\draw (4,0) circle (0.1);
\draw (4,1) circle (0.1);
\draw (3.05,1.69) circle (0.1);
\draw (4.95,1.69) circle (0.1);
\draw (3.41,2.81) circle (0.1);
\draw (4.59,2.81) circle (0.1);
\draw (0,0) circle (0.1);
\draw (1,0) circle (0.1);
\draw (2,0) circle (0.1);
\draw (3,0) circle (0.1);
\draw (5,0) circle (0.1);
\draw (5.5,0.866) circle (0.1);
\draw (6.5,0.866) circle (0.1);
\draw (7.5,0.866) circle (0.1);
\draw (5.5,-0.866) circle (0.1);
\draw (6.5,-0.866) circle (0.1);
\draw (7.5,-0.866) circle (0.1);
\draw (1,1) circle (0.1);
\draw (2,1) circle (0.1);
\draw (5,4) circle (0.1);

\end{tikzpicture}
\end{center}

This elliptic fibration has $D_8$ and $A_6$ fibers, and two
independent sections $P,Q$ with height matrix
$$
\left(\begin{array}{cc}
8/7 & -5/14 \\
-5/14 & 22/7
\end{array} \right).
$$

We identify a fiber $F'$ of type $E_8$ and move to the associated
fibration by a $2$-neighbor step. Note that $Q \cdot F' = 3$, while
the remaining component of the $D_8$ fiber has intersection $2$ with $F'$.
Therefore the new genus one fibration has a section.

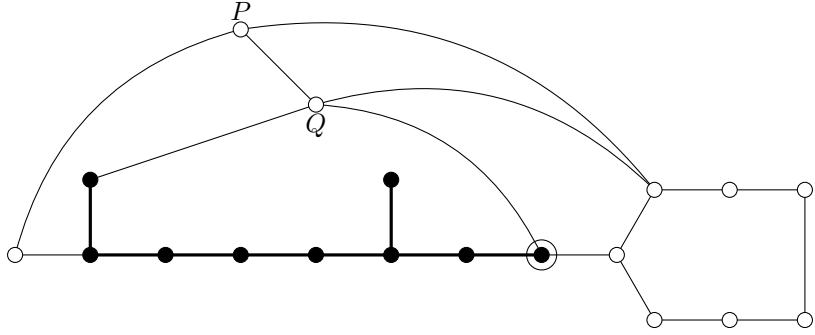
\begin{center}
\begin{tikzpicture}

\draw (0,0)--(8,0)--(8.5,0.866)--(10.5,0.866)--(10.5,-0.866)--(8.5,-0.866)--(8,0);
\draw (1,0)--(1,1);
\draw (5,0)--(5,1);
\draw [bend right] (3,3) to (0,0);
\draw [bend left] (3,3) to (8.5,0.866);
\draw (4,2)--(1,1);
\draw [bend left] (4,2) to (8.5,0.866);
\draw [bend left] (4,2) to (7,0);
\draw (3,3)--(4,2);
\draw [very thick] (1,1)--(1,0)--(7,0);
\draw [very thick] (5,0)--(5,1);

\draw (7,0) circle (0.2);
\fill [white] (0,0) circle (0.1);
\fill [black] (1,0) circle (0.1);
\fill [black] (2,0) circle (0.1);
\fill [black] (3,0) circle (0.1);
\fill [black] (4,0) circle (0.1);
\fill [black] (5,0) circle (0.1);
\fill [black] (6,0) circle (0.1);
\fill [black] (7,0) circle (0.1);
\fill [white] (8,0) circle (0.1);
\fill [black] (1,1) circle (0.1);
\fill [black] (5,1) circle (0.1);
\fill [white] (8.5,0.866) circle (0.1);
\fill [white] (8.5,-0.866) circle (0.1);
\fill [white] (9.5,0.866) circle (0.1);
\fill [white] (9.5,-0.866) circle (0.1);
\fill [white] (10.5,0.866) circle (0.1);
\fill [white] (10.5,-0.866) circle (0.1);
\fill [white]  (3,3) circle (0.1);
\fill [white] (4,2) circle (0.1);

\draw (0,0) circle (0.1);
\draw (1,0) circle (0.1);
\draw (2,0) circle (0.1);
\draw (3,0) circle (0.1);
\draw (4,0) circle (0.1);
\draw (5,0) circle (0.1);
\draw (6,0) circle (0.1);
\draw (7,0) circle (0.1);
\draw (8,0) circle (0.1);
\draw (1,1) circle (0.1);
\draw (5,1) circle (0.1);
\draw (8.5,0.866) circle (0.1);
\draw (8.5,-0.866) circle (0.1);
\draw (9.5,0.866) circle (0.1);
\draw (9.5,-0.866) circle (0.1);
\draw (10.5,0.866) circle (0.1);
\draw (10.5,-0.866) circle (0.1);
\draw (3,3) circle (0.1);
\draw (4,2) circle (0.1);

\draw (3,3) [above] node{$P$};
\draw (4,2) [below] node{$Q$};

\end{tikzpicture}
\end{center}

The new elliptic fibration has bad fibers of types $E_8$ and $A_7$,
and a section $P'$ of height $97/8 = 2 + 2\cdot 6 - 3 \cdot 5/8$. We
identify a fiber $F''$ of type $E_7$, and move to the associated
elliptic fibration by a $2$-neighbor step. Note that $P' \cdot F'' =
13$, while the remaining component of the $A_7$ fiber has intersection
$2$ with $F''$. Therefore the elliptic fibration associated to $F''$
has a section, and is the of the desired type $E_8 E_7$.

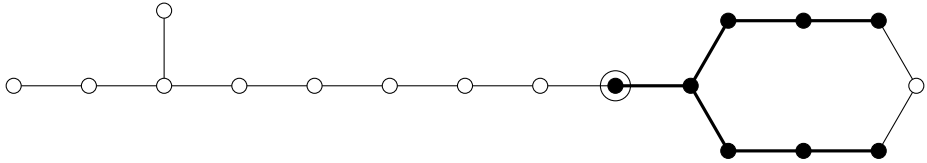
\begin{center}
\begin{tikzpicture}

\draw (-1,0)--(8,0);
\draw (1,0)--(1,1);
\draw (8,0)--(8.5,0.866)--(10.5,0.866)--(11,0)--(10.5,-0.866)--(8.5,-0.866)--(8,0);
\draw [very thick] (7,0)--(8,0);
\draw [very thick] (10.5,0.866)--(8.5,0.866)--(8,0)--(8.5,-0.866)--(10.5,-0.866);

\fill [white] (-1,0) circle (0.1);
\fill [white] (0,0) circle (0.1);
\fill [white] (1,0) circle (0.1);
\fill [white] (2,0) circle (0.1);
\fill [white] (3,0) circle (0.1);
\fill [white] (4,0) circle (0.1);
\fill [white] (5,0) circle (0.1);
\fill [white] (6,0) circle (0.1);
\fill [white] (1,1) circle (0.1);
\fill [black] (7,0) circle (0.1);
\fill [black] (8,0) circle (0.1);
\fill [black] (8.5,0.866) circle (0.1);
\fill [black] (9.5,0.866) circle (0.1);
\fill [black] (10.5,0.866) circle (0.1);
\fill [black] (8.5,-0.866) circle (0.1);
\fill [black] (9.5,-0.866) circle (0.1);
\fill [black] (10.5,-0.866) circle (0.1);
\fill [white] (11,0) circle (0.1);

\draw (7,0) circle (0.2);
\draw (-1,0) circle (0.1);
\draw (0,0) circle (0.1);
\draw (1,0) circle (0.1);
\draw (2,0) circle (0.1);
\draw (3,0) circle (0.1);
\draw (4,0) circle (0.1);
\draw (5,0) circle (0.1);
\draw (6,0) circle (0.1);
\draw (1,1) circle (0.1);
\draw (7,0) circle (0.1);
\draw (8,0) circle (0.1);
\draw (8.5,0.866) circle (0.1);
\draw (9.5,0.866) circle (0.1);
\draw (10.5,0.866) circle (0.1);
\draw (8.5,-0.866) circle (0.1);
\draw (9.5,-0.866) circle (0.1);
\draw (10.5,-0.866) circle (0.1);
\draw (11,0) circle (0.1);

\end{tikzpicture}
\end{center}

We now read out the Igusa-Clebsch invariants and work out the equation
of the branch locus of $Y_{-}(97)$ as a double cover of $\sH_{97}$.

\begin{theorem}
A birational model over $\Q$ for the Hilbert modular surface
$Y_{-}(97)$ as a double cover of $\Proj^2_{r,s}$ is given by the following
equation:
\begin{align*}
z^2 &= s^2 (s^2+14 s+1) r^6 + 2 s (2 s^4+27 s^3-13 s^2+15 s+1) r^5 \\
& \qquad + (6 s^6+80 s^5-75 s^4+128 s^3-54 s^2+18 s+1) r^4  \\
& \qquad + 2 (2 s^7+28 s^6-32 s^5+84 s^4-74 s^3+48 s^2-13 s+1) r^3 \\
& \qquad  + (s^8+18 s^7-11 s^6+68 s^5-101 s^4+112 s^3-69 s^2+22 s+1) r^2 \\
& \qquad  + 2 s^2 (s^6+3 s^5-5 s^4+7 s^3+3 s^2-16 s+12) r + (s-2)^4 s^4.
\end{align*}
It is a surface of general type.
\end{theorem}

\subsection{Analysis}

The branch locus is a genus $3$ curve, isomorphic to the quotient of
$X_0(97)$ by the Atkin-Lehner involution. We omit the formulas for the
isomorphism, but they are available in the auxiliary computer files.

The Hilbert modular surface itself is of general type. The
substitution $r = u - s$ simplifies the equation of the double cover
somewhat, making it degree $6$ in each variable.  However, it
complicates the original Weierstrass equation of the K3 family, so we
have chosen the $(r,s)$-coordinates on the moduli space.

\subsection{Examples}

We list some points of small height and corresponding genus $2$ curves.

\begin{tabular}{l|c}
Rational point $(r,s)$ & Sextic polynomial $f_6(x)$ defining the genus $2$ curve $y^2 = f_6(x)$. \\
\hline \hline \\ [-2.5ex]
$(5/2, -2)$ & $ -100x^6 + 180x^5 + 3x^4 + 12x^3 - 207x^2 + 54x + 54 $ \\
$(-5/3, 2/3)$ & $ 115x^6 - 120x^5 - 692x^4 - 42x^3 + 643x^2 + 36x - 156 $ \\
$(8/15, -6/5)$ & $ -418x^6 - 99x^5 + 700x^4 + 130x^3 - 401x^2 - 45x + 81 $ \\
$(-14/3, 6)$ & $ 528x^6 - 792x^5 + 311x^4 - 26x^3 - 205x^2 + 60x - 20 $ \\
$(-20/7, 5/14)$ & $ -72x^6 - 72x^5 + 669x^4 + 706x^3 - 1623x^2 - 60x + 500 $ \\
$(13/6, -3/2)$ & $ 1236x^6 - 852x^5 - 1919x^4 + 1702x^3 + 1473x^2 - 940x - 700 $ \\
$(17/6, -1/3)$ & $ 200x^6 - 420x^5 + 1918x^4 - 1455x^3 + 2968x^2 + 1740x - 1175 $ \\
$(-13/7, 11/21)$ & $ -1872x^5 - 3540x^4 + 1021x^3 + 2331x^2 - 1185x + 145 $ \\
$(23/35, -5/14)$ & $ 370x^6 + 1084x^5 - 2510x^4 - 683x^3 - 32x^2 - 752x - 3822 $ \\
$(-11/30, 5/6)$ & $ -1225x^6 + 3570x^5 - 3266x^4 - 176x^3 + 3463x^2 + 1446x + 5868 $ \\
$(1/2, 5/2)$ & $ -1938x^6 + 3132x^5 + 1730x^4 - 855x^3 + 609x^2 - 9065x + 5145 $ \\
$(-19/30, 3/10)$ & $ 4900x^6 + 5320x^5 - 11751x^4 - 4255x^3 + 2867x^2 + 4515x - 1596 $ \\
$(-29/18, 10/9)$ & $ 16048x^6 - 7524x^5 - 11096x^4 + 16107x^3 - 4244x^2 - 5652x - 864 $ \\
$(13/4, -1/4)$ & $ -1140x^6 + 4820x^5 - 3105x^4 + 3366x^3 - 16681x^2 - 6468x - 12348 $ \\
$(-7/2, 6)$ & $ 14076x^6 - 20748x^5 + 11899x^4 + 1252x^3 - 125x^2 + 2676x + 380 $ \\
$(9/10, -2/5)$ & $ -6688x^6 + 9840x^5 - 8271x^4 + 24640x^3 - 5373x^2 + 12150x - 7290$
\end{tabular}

We find a few curves with infinitely many rational points.
For instance, $r = 1-s$ gives a rational curve, with parametrization
$$
(r,s) = \left( \frac{(m+1)(m+3)}{m^2+7}, \frac{-4(m-1)}{m^2+7} \right).
$$ 
The Brauer obstruction vanishes identically along this
curve. However, it turns out to be a modular curve: the corresponding
abelian surfaces have endomorphism ring a (split) quaternion algebra.

Another curve of genus $0$ is given by $r = -(3s^2+8s+4)/(3s)$.
Again, the Brauer obstruction vanishes, and this time we get a family
of abelian surfaces with ``honest'' real multiplication.

The locus $r = 1/2 - s$ gives a genus-$1$ curve
$$
y^2 = (2s+1)(2s^3-39s^2+28s+36)
$$
with conductor $5862 = 2 \cdot 3 \cdot 977$ and \MoW\ group $\Z^2$.

\section{Acknowledgements}

We thank Jennifer Balakrishnan, Henri Darmon, Lassina Demb\'{e}l\'{e},
Eyal Goren, Kiran Kedlaya, Ronen Mukamel, George Pappas, Bjorn Poonen,
Frithjof Schulze, Matthias Sch\"utt and Andrew Sutherland for helpful
comments. We also thank the anonymous referee for a careful reading of
the paper and several useful remarks. The computer algebra systems
PARI/gp, Maple, Maxima, and Magma were used in the calculations for
this paper. We also made heavy use of the programs \texttt{mwrank},
\texttt{ratpoints}, the Maple package \texttt{algcurves}, and the
Magma program \texttt{ConicsFF.m} for finding points on conics over
function fields.  We thank the authors of these programs as well.

\end{document}